\documentclass[11pt,a4paper,reqno]{amsart}
\usepackage[T1]{fontenc}
\usepackage{amsmath}
\usepackage{phaistos}
\usepackage{protosem}
\usepackage{amsthm}
\usepackage{amsfonts}
\usepackage{amssymb}
\usepackage{graphicx}
\usepackage{color}
\usepackage[colorlinks=true,linkcolor=black,citecolor=blue]{hyperref}
\usepackage{amsbsy}
\usepackage{mathrsfs}
\usepackage{bbm}
\usepackage{todonotes}

\newtheorem{theorem}{Theorem}[section]
\newtheorem{lemma}[theorem]{Lemma}

\newtheorem{proposition}[theorem]{Proposition}
\newtheorem{corollary}[theorem]{Corollary}

\theoremstyle{remark}
\newtheorem{remark}[theorem]{\it \bf{Remark}\/}

\numberwithin{equation}{section}
\catcode`@=11
\def\section{\@startsection{section}{1}%
  \z@{1.5\linespacing\@plus\linespacing}{.5\linespacing}%
  {\normalfont\bfseries\large\centering}}
\catcode`@=12
\newcommand{\be}{\begin{equation}}
\newcommand{\ee}{\end{equation}}
\newcommand{\bea}{\begin{eqnarray}}
\newcommand{\eea}{\end{eqnarray}}
\newcommand{\bee}{\begin{eqnarray*}}
\newcommand{\eee}{\end{eqnarray*}}
\newcommand{\norm}[1]{{\left\vert\kern-0.25ex\left\vert\kern-0.25ex\left\vert #1 
    \right\vert\kern-0.25ex\right\vert\kern-0.25ex\right\vert}}

\def\pa{\partial}
\def\pr{\partial}

\def\RR{\mathbb{R}}

\def\TT{\mathcal{T}}

\def\Pe{P_{\hskip -.1pc\peye}}

\def\wt{\tilde{w}}
\def\sigmat{\tilde{\sigma}}
\def\ga{\gamma}
\def\b{\beta}

\def\de{\delta}
\def\om{\omega}

\def\uh{\hat{u}}

\catcode`@=11
\def\supess{\mathop{\operator@font Sup\,ess}}
\catcode`@=12

\def\RR{\mathbb{R}}

\def\e{\varepsilon}

\def\Wt{\tilde{W}}
\def\Sigmat{\tilde{\Sigma}}

\def\R2+{\RR ^2_+}

\def\re{{\cal R} \hspace{-2pt}{\textit e}\,}

\def\dt{\tilde{d}}
\def\et{\tilde{e}}

\def\pa{\partial}

\def\lim{\mathop{\rm lim}}
\def\T{\mathcal T}

\def\sup{\mathop{\rm sup}}

\def\l{\lambda}

\def\th{{\rm th}}
\def\log{{\rm log}}

\def\rhoh{\hat{\rho}}

\def\T{\Theta}

\def\Phit{\widetilde{\Phi}}

\def\th{\tilde{H}}
\def\pa{\partial}

\def\psit{\tilde{\psi}}

\def\Psit{\tilde{\Psi}}

\def\pa{\partial}

\def\matchal{\mathcal}

\def\NL{\textrm{NL}}

\def\rhoh{\hat{\rho}}

\def\th{\tilde{h}}

\def\Et{\tilde{E}}

\def\mt{\tilde{m}}

\def\Dt{\tilde{D}}
\def\T{\mathcal T}

\def\NLt{\widetilde{\NL}}

\def\xit{\tilde{\xi}}
\def\th{\tilde{h}}

\def\Thetam{\Theta_{\rm main}}
\def\uh{\hat{u}}
\def\At{\tilde{A}}
\def\Bt{\tilde{B}}
\DeclareMathOperator{\peye}{\textproto{o}}
\renewcommand{\re}{r_{\hskip -.1pc\peye}}
\def\Wte{\Wt_{\hskip -.1pc\peye}}
\def\Sigmate{\Sigmat_{\hskip -.1pc\peye}}
\def\wte{\wt_{\hskip -.1pc\peye}}

\def\psite{\psit_{\hskip -.1pc\peye}}
\def\Thetat{\tilde{\Theta}}

\begin{document}

\title[]{On smooth self-similar solutions to the compressible Euler equations}
\author[F. Merle]{Frank Merle}
\address{AGM, Universit\'e de Cergy Pontoise and IHES, France}
\email{merle@math.u-cergy.fr}
\author[P. Rapha\"el]{Pierre Rapha\"el}
\address{Department of Pure Mathematics and Mathematical Statistics, Cambridge, UK}
\email{praphael@maths.cam.ac.uk}
\author[I. Rodnianski]{Igor Rodnianski}
\address{Department of Mathematics, Princeton University, Princeton, NJ, USA}
\email{irod@math.princeton.edu}
\author[J. Szeftel]{Jeremie Szeftel}
\address{CNRS \& Laboratoire Jacques Louis Lions, Sorbonne Universit\'e, Paris, France}
\email{jeremie.szeftel@upmc.fr}
\maketitle

\begin{abstract} 
We consider the barotropic Euler equations in dimension $d\ge 2$ with decaying density at spatial infinity. The phase portrait of the nonlinear ode governing the equation for spherically symmetric
 self-similar solutions has been introduced in  the pioneering work of Guderley \cite{guderley}. It allows to construct 
 global  profiles of the self-similar problem, which however turn out to be generically non-smooth across the associated light (acoustic) cone. In a suitable range of barotropic laws and for a sequence of {\em quantized speeds} accumulating to a critical value, we prove the existence of non-generic 
 $\mathcal C^\infty$  self-similar solutions with suitable decay at infinity. The $\mathcal C^\infty$ regularity is used in a {\em fundamental way}  in the companion papers \cite{MRRSnls}, \cite{MRRSfluid}  to control the associated linearized operator, and construct finite energy blow up solutions of respectively the defocusing nonlinear Schr\"odinger equation in dimension $5\le d\le 9$, and the isentropic ideal compressible Euler and Navier-Stokes equations in dimensions $d=2,3$.
\end{abstract}

\tableofcontents


\section{Introduction}



\subsection{Setting of the problem}


We consider in this paper the isentropic compressible Euler equations in dimension $d\ge 2$, $y\in \Bbb R^d$,
\be
\label{eulercomp}
(\rm{Euler})\ \ \left|\begin{array}{lll}\pa_t\rho+\nabla\cdot(\rho u)=0\\
\rho\pa_tu+\rho u\cdot\nabla u+\nabla p=0\\
p=\frac{\gamma-1}{\gamma}\rho^\gamma\\
\rho(t,y)>0.
\end{array}\right.
\ee 
 Our goal is the construction of a family of {\it \underline{smooth}, \underline{global}, self-similar} profiles corresponding to 
spherically symmetric solutions of \eqref{eulercomp} which arise from smooth initial data and blow up at a chosen
point $(T,0)$. We are specifically interested in solutions which do not exhibit {\it growth} at infinity. In fact, our solutions will obey
\be
\label{vneivenoenenevnove}
 \lim_{|y|\to +\infty}(\rho(t,y),u(t,y))=0.
\ee 
Existence of self-similar solutions with spherical symmetry for the Euler equation with {\em a continuum} of admissible blow up speeds has been known since the pioneering work of Guderley \cite{guderley} and Sedov \cite{sedov}. However, 
the known solutions are either non-global or non-smooth (in self-similar variables). 
The former solutions are ubiquitous in the physics literature and 
describe physical phenomena of denotation, implosion, flame propagation etc. and by design contain a shock, discontinuously
connecting a smooth solution to another state. We refer to them as non-global even if, more accurately, they should be called discontinuous, to differentiate them from the second class of solutions -- non-smooth ones. The latter appear to be much less known but can be easily constructed from the phase portrait analysis introduced in \cite{guderley}. The solutions are global 
in a sense that they connect the behavior \eqref{vneivenoenenevnove} at infinity with the regular behavior at the center of symmetry. In the process, they have to cross the so called {\it sonic line} -- a point on which represents a backward acoustic 
cone from the singular point in the original variables. It turns out that {\it generically} this crossing is non-smooth. 
The regularity of the associated solution depends on the values of various parameters (dimension, equation of state, scaling)
but the standard Lyapunov analysis suggests that it is always finite (although it can be made arbitrarily high for a particular 
choice of parameters.) In principle, such solutions, together with the finite speed of propagation,  immediately lead to the existence of {\em finite energy} well localized blow up solutions of the Euler equations. Note however that for the reason
explained above these solutions do not arise from smooth initial data.\\

The aim of this paper is to show the existence of non-generic $\matchal C^\infty$ self-similar solutions for {\em quantized values of the blow up speed in the vicinity of a certain critical value}. This is contrary to the Lyapunov analysis which 
would suggest that it might never be possible.  In the companion papers \cite{MRRSnls,MRRSfluid}, we will use 
these solutions as the leading order blow up profiles for respectively the energy super-critical defocusing nonlinear Schr\"odinger equations and the compressible Navier-Stokes equation (as well as its inviscid Euler limit) to produce blowing 
up solutions arising from smooth initial data. The {\em $\mathcal C^\infty$ regularity}  
of the profile is needed not only for the regularity of initial 
data but much more importantly, in fact, crucially, for the stability analysis. The stability analysis itself is needed not only to 
establish existence of a whole finite co-dimensional manifold (in the moduli space of initial data) of blowing up solutions but 
also for the existence result of even just one such solution. These profiles are merely {\it approximate} solutions for the Schr\"odinger and Navier-Stokes problems and their stability and thus their {\em $\mathcal C^\infty$ regularity} are essential to ensuring that the approximation holds until the blow up time.\\

Self-similar motion has long been recognized as an important concept in hydrodynamics (see e.g. \cite{sedov} and the references therein). It could be said that it originated in the dimensional analysis of Reynolds and eventually crystallized as a model 
for both the simplest and universal behaviors in fluid and gas dynamics. The simplicity stems from the fact 
that the assumption of self-similar motion (together with spherical or cylindrical symmetry in higher dimensions) 
reduces the Euler equations to a system of ode's. The universality is supported by the ubiquity of self-similar solutions 
as well as the belief that self-similar motions act as an attractor for many different phenomena in hydro/gas dynamics. 
In that respect, two types of self-similar 
motions have been discussed in the physics literature, \cite{Z}. In the first kind, all self-similar parameters are 
determined from the dimensional analysis, while in the second kind, an undetermined (free) parameter is fixed 
by the boundary conditions or some other physical requirements on the solution. Self-similar solutions have been 
extensively used and analyzed in the study of problems involving detonation and implosion waves, combustion, reflected shocks, etc. The current approach has been pioneered by Guderley \cite{guderley} and has been given numerous treatments, see \cite{CF,sedov}. In that approach, the study of 
self-similar (spherically symmetric) solutions of the Euler equations is reduced to the system
\be\label{gud}
\left|\begin{array}{l}
\frac{dw}{dx}=-\frac{\Delta_1}{\Delta}\\
\frac{d\sigma}{dx}=-\frac{\Delta_2}{\Delta}
\end{array}\right.
\ee
where $\Delta, \Delta_{1,2}$ are polynomials in $w,\sigma$ and the similarity variable $Z=e^x$ is related to the original $(t,y)$ via 
$$
Z=\frac{|y|}{(T-t)^{\frac 1r}}
$$
The $1/r$ parameter is the similarity exponent (free parameter) and is the inverse of what in this paper we call {\it speed $r$.}
The equations \eqref{gud} are an autonomous 
system of ode's. Its phase portrait and, specifically, the set where $\Delta$, $\Delta_1$, $\Delta_2$ vanish 
determine the qualitative properties of all solutions. 


\subsection{The self-similar equation}


Let 
\be
\label{vnovnenneone}
\ell=\frac{2}{\gamma-1}>0,  \ \ r>1,
\ee\
then the self-similar renormalization
\be
\label{selfimislairna}
\left|\begin{array}{l}
\rho(t,y)=\left(\frac{\l}{\nu}\right)^{\frac{2}{\gamma-1}}\rhoh(\tau,Z)\\
u(t,y)=\frac{\l}{\nu}\hat{u}(\tau,Z)\\
Z=\frac{y}{\l}, \ \ \frac{d\tau}{dt}=\frac{1}{\nu}\\
-\frac{\l_\tau}{\l}=1, \ \ -\frac{\nu_\tau}{\nu}=r
\end{array}\right.
\ee
maps \eqref{eulercomp} on $[0,T)$ onto the global in time $\tau$ renormalized flow
\be
\label{renormalizedflow}
\left|\begin{array}{l}
\pa_\tau \rhoh+\ell(r-1)\rhoh+\Lambda\rhoh+\nabla \cdot(\rhoh\uh)=0\\
\pa_\tau \uh+(r-1)\uh+\Lambda \uh+\uh\cdot\nabla \uh+\nabla (\rhoh^{\gamma-1})=0\\
\Lambda =Z\cdot \nabla
\end{array}\right.       
\ee
A self-similar profile is a stationary solution to \eqref{renormalizedflow}:
\be
\label{renormalizedflowstaionray}
\left|\begin{array}{l}
\ell(r-1)\rhoh+\Lambda\rhoh+\nabla \cdot(\rhoh \uh)=0\\
(r-1)u+\Lambda \uh+\uh \cdot\nabla \uh+\nabla (\rhoh^{\gamma-1})=0\\
\end{array}\right.
\ee
which produces a blow up solution for \eqref{eulercomp} with the rate of concentration $$\l(t)=\l_0(T-t)^{\frac1r}, \ \ \nu(t)=r(T-t).$$


\subsection{Emden transform and Guderley's phase portrait}


In the pioneering work \cite{guderley}, see also \cite{sedov,MS}, all solutions to \eqref{renormalizedflowstaionray} with spherical symmetry are mapped through the Emden transform
\be
\label{emdentransform}
\left|\begin{array}{l}(\rhoh(Z))^{\frac{\gamma-1}{2}}=\sqrt{\frac \ell 2}Z\sigma(x)\\
\uh(Z)=- Zw(x)\\
Z=e^x
\end{array}\right.
\ee
onto the {\em autonomous} system of nonlinear ode's:
\be
\label{systemedefoc}
\left|\begin{array}{ll}
(w-1)w'+\ell \sigma\sigma'+(w^2-rw+\ell \sigma^2)=0\\
\frac{\sigma}{\ell}w'+(w-1)\sigma'+\sigma\left[w\left(\frac{d}{\ell}+1\right)-r\right]=0
\end{array}
\right.
\Leftrightarrow \left|\begin{array}{l}
\Delta w'=-\Delta_1\\
\Delta \sigma'=-\Delta_2
\end{array}\right.
\ee
with the explicit 
\be
\label{veluadeternte}
\left|\begin{array}{lll}
\Delta=(w-1)^2-\sigma^2\\
\Delta_1=w(w-1)(w-r)-d(w-w_e)\sigma^2\\
\Delta_2=\frac{\sigma}{\ell}\left[(\ell+d-1)w^2-w(\ell+d+\ell r-r)+\ell r-\ell \sigma^2\right]
\end{array}\right.
\ee
and
\be
\label{defwlimting}
w_e=\frac{\ell(r-1)}{d}.
\ee
The above system can be fully analyzed through {\em the phase portrait in the $(\sigma,w)$ plane}. The shape of the phase portrait is highly dependent on the values of the parameters $(r,\ell,d)$. Let us introduce the following critical speeds which will play a fundamental role in the forthcoming analysis
\be
\label{valuerstar}
\left|\begin{array}{l}
r^*(d,\ell)=\frac{d+\ell}{\ell+\sqrt{d}},\\
r_+(d,\ell)=1+\frac{d-1}{(1+\sqrt{\ell})^2},
\end{array}\right.
\ee
where an explicit computation\footnote{see \eqref{vlaueifnioegn}} shows $$1<r^*(d,\ell)\le r_+(d,\ell)\ \ \mbox{with equality iff} \ \ \ell=d.$$
Let us denote 
\be
\label{defrinfty}
\re(d,\ell)=\left|\begin{array}{l} r^*(d,\ell)\ \ \mbox{for}\ \ 0<\ell<d\\
r_+(d,\ell)\ \ \mbox{for}\ \ \ell>d.
\end{array}\right.
\ee
Then, in the range of parameters 
\be
\label{rangeparameters}
d\ge 2, \ \ \ell>0, \ \ \left|\begin{array}{l} 1<r<r^*(d,\ell) \ \ \mbox{for}\ \ \ell<d\\r^*(d,\ell)<r<r_+(d,\ell)\ \ \mbox{for}\ \ \ell>d,
\end{array}\right.
\ee
we will prove that the phase portrait contains five fundamental points $(P_i)_{1\le i\le 5}$, see figure \ref{fig:solutioncurve}, figure \ref{fig:solutioncurvebis} and section \ref{geometryphaseportrait}, as well as two {\it sonic lines}.\\

\noindent\underline{Sonic lines} $w-1=\pm \sigma$. This is exactly the set where $\Delta=0$. In the original variables
of the Euler equations, the sonic lines correspond to the equation
\be\label{sonic}
\frac{|y|}{r(T-t)}=-\left(u\pm c\right),\qquad c=\sqrt{\frac{\pa p}{\pa \rho}}
\ee
where $c$ is the sound speed and the right hand side is a function of $\frac {|y|}{\lambda_0 (T-t)^{\frac 1r}}$.
On the other hand, the equation 
$$
\frac {d|y|}{dt} =-\left(u\pm c\right)
$$
describes solutions $(t, |y|(t))$ -- {\it acoustical cones} (radial characteristics) of the metric associated with a solution of the 
compressible Euler equations. It then follows that for any point $Z^*$ on the sonic line, the set
$(T-t, Z^*\lambda_0 (T-t)^{\frac 1r})$ is the backward acoustical cone from the singular point $(0,0)$. On the phase diagram, the points to the right of the sonic line $w-1=\sigma$ will correspond to the $(t,y)$ points in the interior of the 
acoustical (light) cone, while to the points to the left of the sonic line lie in the exterior of the cone. Moreover, the (absolute value of)
characteristic speed $u+c$ -- particle velocity plus the sound speed -- is smaller in the exterior and larger in the interior of the cone.

\noindent\underline{$P_5$ point}. The point $$P_5=\left(\sigma_5=\frac{r\sqrt{d}}{d+\ell},w_5=\frac{\ell r}{d+\ell}\right)$$ is an endpoint of the dynamical system \eqref{systemedefoc}, i.e. $$\left|\begin{array}{l}
\Delta_1(P_5)=\Delta_2(P_5)=0\\
\Delta(P_5)\neq 0
\end{array}\right.
$$ and the relative position of $P_5$ with respect to the sonic line is given by $$w_5+\sigma_5<1\Leftrightarrow r<r^*(d,\ell).$$
\noindent\underline{Points $P_2, P_3$}. Trajectories can only cross the sonic line at the triple points, where $$\Delta=\Delta_1=\Delta_2,$$ which  are $(0,0), (1,0), (r,0)$ and two other points on the sonic line $w+\sigma=1$ which we refer to as $P_2,P_3$ and which exist thank to the constraint \eqref{rangeparameters}.\\
\noindent\underline{$P_6$ point}. The point $P_6=(w=w_e, \sigma=+\infty)$ is a saddle point at $+\infty$
and corresponds to $x=-\infty$, i.e., $Z=0$.

\noindent\underline{$P_4$ point}. The point $P_4=(0,0)$ attracts solutions which vanish near $x\to +\infty$.

A classical analysis of the phase portrait, figure \ref{fig:solutioncurve}, figure \ref{fig:solutioncurvebis}, yields the following result (see Lemma \ref{lememarompfoe} and \ref{lemmaconnection} for a proof).

\begin{lemma}[General structure of spherically symmetric self-similar solutions]
\label{vnioneneno}
Assume \eqref{rangeparameters}. Then,
\begin{enumerate}
\item Solutions near the origin: there exists a unique trajectory of \eqref{systemedefoc} which connects $P_6$ to $P_2$. This trajectory corresponds to the unique (up to scaling)  (local) spherically symmetric solution to \eqref{renormalizedflowstaionray} which is $\mathcal C^\infty$ on the interval $|Z|\in [0,Z_2)$ with $0$ corresponding 
to $P_6$ and $Z_2$ to $P_2$.
\vskip .4pc
\item Solutions near infinity: there exists a one parameter family of trajectories $P_2-P_4$. These curves correspond to the spherically symmetric solution to \eqref{renormalizedflowstaionray} and are in $\mathcal C^\infty$
of the interval $|Z|\in (Z_2,\infty)$ with $\infty$ corresponding to $P_4$.
\vskip .4pc
\item Connection at $P_2$: in both cases, $P_2$ is reached in finite time, i.e., $0<Z_2<\infty$, 
and the solutions constructed in (1) and (2) can be glued continuously to each other.
\end{enumerate}
\end{lemma}

In other words, the unique\footnote{up to scale invariance}, spherically symmetric solution $(\hat\rho(Z), \hat u(Z))$ to \eqref{renormalizedflowstaionray}, which is smooth at $Z=0$, extends to the point $Z_2$ (which corresponds to $P_2$
on the phase diagram) where it can be {\em glued} to any of the one parameter family of solutions $(\hat\rho(Z), \hat u(Z))$ defined for $[Z_2,\infty)$ (which correspond to the trajectories connecting $P_2$ to $P_4$). This procedure yields spherically symmetric solutions to \eqref{renormalizedflowstaionray} which are $\mathcal C^\infty(\Bbb R^d\backslash \{Z=Z_2\})$ and vanish as $Z\to +\infty$.

\subsection{Regularity at $P_2$ and the reconnection problem} 

It remains to understand the regularity at $P_2$. The above gluing procedure produces a solution with limited $\mathcal C^{k(r)}$ regularity\footnote{with $k(r)\sim \frac{c(d,\ell)}{\re(d,\ell)-r}$ as $r\uparrow \re(d,\ell)$.} at the degenerate point $P_2$, see Remark \ref{rem:reg}. As it turns out, this limited regularity has a dramatic effect on the spectral structure of the linearized operator for \eqref{renormalizedflow} close to this self-similar profile, and we do not know how to use these non $\mathcal C^\infty$ solutions to produce finite energy imploding solutions to the structural perturbations of \eqref{eulercomp} studied in \cite{MRRSnls,MRRSfluid}.\\

On the other hand, $P_2$ is a regular singular point of \eqref{systemedefoc}. As a result, there always\footnote{at least away from critical integer values, see Lemma \ref{prop:behavioroftheflownearP2:smooth}.} exists one trajectory which is $\mathcal C^\infty$ at $P_2$. Thus, the problem becomes: {\em can we find  parameters $(d,\ell,r)$ for which the $P_6-P_2$ solution is $\matchal C^\infty$ at $P_2$  {\em and} can be glued to a  $P_2-P_4$ 
solution which is {\it also} $\mathcal C^\infty$ at $P_2$ (see figure \ref{fig:solutioncurve})?} Such a solution would produce a $\mathcal C^\infty$ self-similar profile with vanishing density and velocity as $Z\to +\infty$.\\

The Lyapunov type (linear) analysis at $P_2$ predicts that the regularity of solutions passing through is determined by the 
eigenvalues $\lambda_\pm$ of the Jacobian matrix. Both $\lambda_-<\lambda_+<0$, so that $P_2$ is a stable node. As a result, all curves through $P_2$  but one will have limited regularity $\mathcal C^{\frac{\lambda_-}{\lambda_+}}$. The exceptional solution is $\mathcal C^\infty$ but it does not go to $P_6$ and is thus inadmissible as a global profile. 
The $\mathcal C^{\frac{\lambda_-}{\lambda_+}}$ regularity is {\em sharply} insufficient for our purposes (linear stability analysis of these solutions as solutions of the Euler equations misses exactly one derivative, see \cite{MRRSnls}). Instead, we consider the regime where 
$\lambda_+\uparrow 0$ (and $\lambda_-$ stays uniformly bounded away from $0$ and $-\infty$) which corresponds to a limiting degeneracy of the phase portrait for the critical speed $r\uparrow \re(d,\ell)$ given by  \eqref {defrinfty} where a vanishing ``$\Large\peye$'' structure appears to the left of $P_2$. We perform a  careful
{\it nonlinear} and {\it global} analysis near the eye configuration to show the existence of a discrete sequence of $\mathcal C^\infty$ solutions as $r=r_n\uparrow r_{\hskip -.1pc\peye}$.


\subsection{Statement of the main result}


Our main result in this paper is the existence of such  $\matchal C^\infty$ profiles for {\em quantized values of the speed $r$} near the critical speed {\em whose value depends on $\ell$} given by \eqref{vnovnenneone}. We introduce the set 
$\mathcal O_d\subset (0,+\infty)\setminus\{d\}$ defined as the union of the interval $(d,+\infty)$ and a subset 
$O_d^*$ of $(0,d)$
defined 
by the condition that $\ell\in \mathcal O_d$ for $\ell<d$ if the value of the function 
$\nu_\infty(\ell,d)$ in \eqref{eq:algebricformulafornu} is $>0$. In dimensions $d=2,3$, $\mathcal O_d=(0,+\infty)\setminus\{d\}$ and 
for $5\le d\le 9$ the subset $\mathcal O_d^*$ is a finite collection of subintervals and it is non-empty, see \eqref{signofnu}.

\begin{theorem}[Existence of $\mathcal C^\infty$ asymptotically vanishing self-similar profiles] 
\label{thmmain}
Let $d\ge 2$. Let the critical speed $\re(d,\ell)$ be given by \eqref{defrinfty}. Then there exist a function 
$$
S_\infty(d,\ell):\Bbb N^*\backslash\{1\}\times \mathcal O_d\to \Bbb R
$$ such that for any $\ell\in \mathcal O_d$ obeying the condition
\be
\label{conditionell}
S_\infty(d,\ell)\neq 0,
\ee
there exists a discrete sequence $$
\left|\begin{array}{l}
r_n<\re(d,\ell), \ \ |r_n-\re(d,\ell)|\ll1\\
\lim_{n\to +\infty} r_n=\re(d,\ell)
\end{array}\right.
$$ such that \eqref{renormalizedflowstaionray} with $r=r_n$ admits a global $\matchal C^\infty$ spherically symmetric solution $(\rhoh(Z),\uh(Z))$ which terminates at $P_4$ at spatial infinity (i.e. as $Z\to +\infty$), see figure \ref{fig:solutioncurve}. 
\end{theorem}

\begin{corollary}
\label{corm}
Under the assumptions of Theorem \ref{thmmain}, the Euler equations \eqref{eulercomp} admit a family of spherically symmetric self-similar solutions, \underline{smooth} away from the concentration point $(T,0)$:
$$
\left|\begin{array}{l}
\rho(t,y)=\frac 1{(T-t)^{\frac{2(r-1)}{r(\gamma-1)}}} \rhoh(Z)\\
u(t,y)=\frac 1{(T-t)^{\frac{(r-1)}{r}}} \uh(Z)\\
Z=\frac{y}{\lambda_0(T-t)^{\frac 1r}}
\end{array}\right.
$$
with the asymptotics 
$$
\rho(t,y)= \frac {\rho_*(1+o_{|Z|\to +\infty}(1))}{|y|^{\frac{2(r-1)}{r(\gamma-1)}}},\qquad u(t,y)\sim  \frac {u_*(+o_{|Z|\to +\infty}(1))}{|y|^{\frac{(r-1)}{r}}},
$$
for some $\rho^*>0$.  In particular, these solutions decay at infinity but do not have finite energy.
\end{corollary}

The function $S_\infty(d,\ell)$ appears in the asymptotic analysis of the flow near $P_2$. It can be explicitly expressed as a normally convergent series 
$$S_\infty(d, \ell)=\sum_{n=0}^{+\infty} u_n(d,\ell), \ \ |u_n(d,\ell)|\leq\frac{c_{d,\ell}}{1+n^2}$$ 
where the series $u_n(d,\ell)$ satisfies an explicit though complicated non linear induction relation. The proof of convergence of the series yields the analyticity of the mapping $\ell\mapsto S_\infty(d,\ell)$ in a suitable open set of the complex plane.\\

More precisely, the function $S_\infty(d,\ell)$ is in fact composed of two different functions $S^*_\infty(d,\ell)$, defined for the 
critical value of the speed $r_*$ on the subset $\mathcal O_d^*\subset (0,d)$, and $S^+_\infty(d,\ell)$, defined for the 
critical value of the speed $r^+$ for $\ell>d$. It turns out that the {\it explicit} definitions of these functions allow us 
to extend them holomorphically.

\begin{lemma}[Holomorphic extension]
\label{lemmaisolated} 
The following hold
\begin{enumerate}
\item For $d\ge 4$ the function $\ell\to S_\infty^{*}(d,\ell)$ extends holomorphically to a complex neighborhood of the $\mathcal O_d^*$ .
\item For $d\ge 4$ the function $\ell\to S_\infty^{+}(d,\ell)$ extends holomorphically to a complex neighborhood of the $(d,+\infty)$.
\item For $d=2,3$ the function $\ell\to S_\infty^{*}(d,\ell)$ extends holomorphically to a complex neighborhood of the $(0,d)$.
\item For $d=2,3$, the function $\ell\to S_\infty^+(d,\ell)$ extends holomorphically to an open set $\Omega_d^+$ of the complex plane,  with $\Bbb R^*_+\backslash\{d\}\subset \Omega^+_d$,and $(0,d)$ and $(d,+\infty)$ belong to the same connected component of $\Omega^+_d$.
\end{enumerate}
\end{lemma}

We do not know how to compute analytically the zeroes of $S_\infty(d,\ell)$. However, for $d=2,3$, Lemma \ref{lemmaisolated} ensures that, unless the function vanishes identically, the possible zeroes are isolated with possible accumulation points $(0,d,+\infty)$ only\footnote{Note that all the conclusions of Lemma \ref{lemmaisolated} can be extended to higher dimension $d\ge 4$.}. For $d\ge 4$ the same conclusions can be reached about each of the intervals 
in $\mathcal O_d^*$ and $(d+\infty)$. 

Moreover, since $u_n(d,\ell)$ is given by an explicit induction relation on the coefficients, we can perform an elementary numerical computation of the series. We first give the results in dimension $d=2$ and $d=3$ which will be used for the study of the compressible Euler and Navier-Stokes equations in \cite{MRRSfluid}, see Appendix \ref{appendixnumericsofSinftyinell}. The assertion in (4) allows us to check the non-vanishing condition for small values 
of $\ell$ only.\\

\noindent{\bf Numerical claim} [Numerical study of the zeroes of $S_\infty(d,\ell)$, case $d=2,3$]
{\em In the case $d=2$ and $d=3$, we have 
\be
\label{estiamtesfison}
\left|\begin{array}{l}
S^*_\infty(2,\ell)>0\  \ \mbox{for}\ \ \ell=0.1,\\
S^*_\infty(3,\ell)>0\  \ \mbox{for}\ \ \ell=0.1,
\end{array}\right.
\ee
and 
\be
\label{estiamtesfisonbis}
\left|\begin{array}{l}
S^+_\infty(2,\ell)>0\  \ \mbox{for}\ \ \ell=0.1,\\
S^+_\infty(3,\ell)>0\  \ \mbox{for}\ \ \ell=0.1.
\end{array}\right.
\ee}

The case of higher dimensions will be relevant for the study of the energy super critical defocusing (NLS) equation in \cite{MRRSnls}, for which the power nonlinearity involves the real number $p$ given by 
\be
\label{nolienartiyt}
p=1+\frac{4}{\ell}.
\ee
We numerically check the non degeneracy $S_\infty(d,\ell)\neq 0$ for a range of dimensions and integer nonlinearities, see Appendix \ref{appendixnumericsofSinftyinell}.\\

\noindent{\bf Numerical claim} [Numerical study of the zeroes of $S_\infty(d,\ell)$, case $d\ge 5$]\,\, \newline{\em Let $p(\ell)$ be given by \eqref{nolienartiyt}. Then the condition \eqref{conditionell} holds for 
\be
\label{sufffiicnetconisit}
(d,p)\in \{(5,9), (6,5), (7,4), (8,3), (9,3)\}. 
\ee
}

\noindent{\bf Comments on the results}.\\

\noindent{\em 1. The set $\mathcal O_d$}. It is defined by the requirement $\nu_\infty(d,\ell)>0$ with $\nu_\infty $ explicitly given by the  formulas \eqref{eq:algebricformulafornu} for $\ell<d$ and \eqref{eq:algebricformulafornubis} for $\ell>d$. It turns out that $\nu_\infty>0$ for all $\ell \in \Bbb R^*_+\backslash\{d\}$ in dimension $d=2,3$. 
In higher dimensions, $\mathcal O_d$ also includes the interval $\ell>d$ and its intersection with $\ell<d$ consists 
of a finite union of open intervals. The latter intersection is non-empty for $5\le d\le 9$.
The condition $\nu_\infty(d,\ell)>0$ is convenient for our analysis but is clearly not sharp and could be relaxed. Let us emphasize that \eqref{estiamtesfison}, \eqref{sufffiicnetconisit} merely provide explicit examples of admissible couples $(d,\ell)$.\\

\begin{figure}
\centering
\includegraphics[width=13cm]{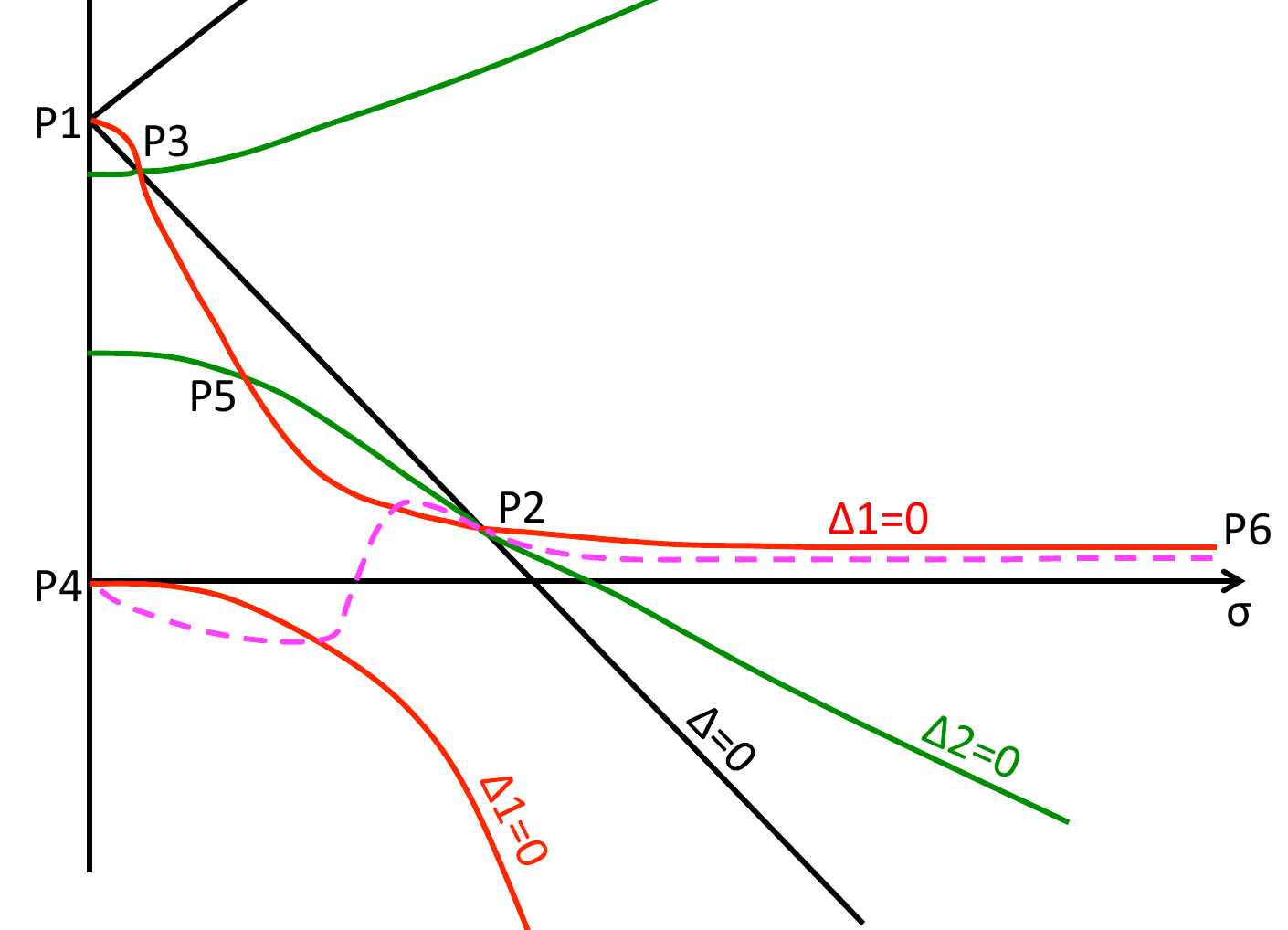}
\caption{Phase portrait in the range $1<r<r^*(d,\ell)$. Dashed curve is the trajectory of the solution constructed in Theorem \ref{thmmain}.}
\label{fig:solutioncurve}
\end{figure}

\begin{figure}
\centering
\includegraphics[width=13cm]{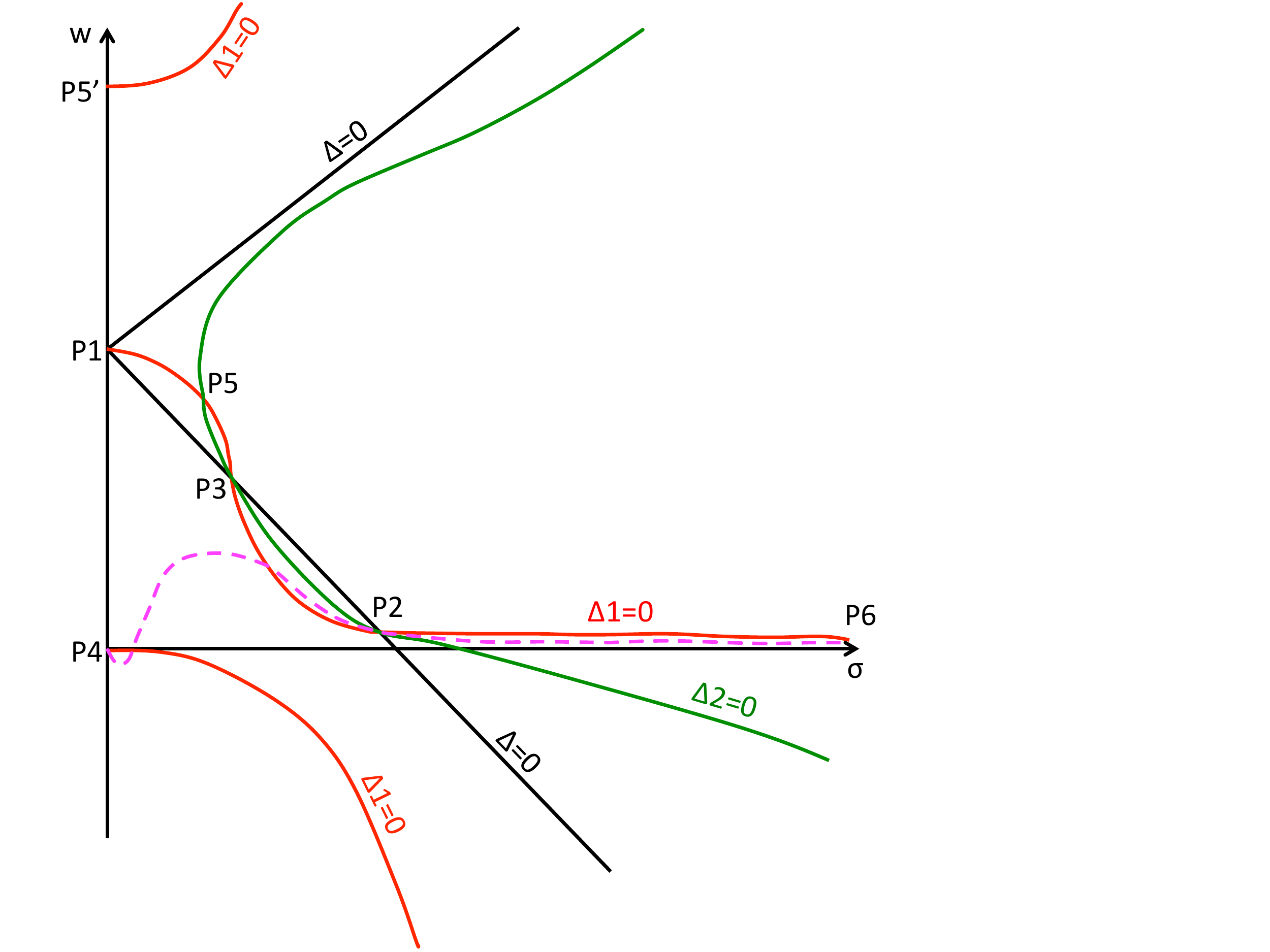}
\caption{Phase portrait in the range $r^*(d,\ell)<r<r_+(d,\ell)$, $\ell>d$. Dashed curve is the trajectory of the solution constructed in Theorem \ref{thmmain}.}
\label{fig:solutioncurvebis}
\end{figure}


\noindent{\em 2. The asymptotic analysis $r\uparrow \re(d,\ell)$.} The proof of Theorem \ref{thmmain} involves a careful renormalization of the flow \eqref{systemedefoc} for $|r-\re|\ll 1$, $r<\re$. The reason for this choice is the presence of an "eye" of the phase portrait at the left of $P_2$: as $r\uparrow \re$, the critical point at the left of $P_2$ which is $P_5$ for $\ell<d$ and $P_3$ for $\ell>d$, converges to $P_2$, and this convergence happens with a suitable ordering of the points $P_2,P_3,P_5$ which depends on whether $\ell>d$ or $\ell<d$, see Lemma \ref{phasperotrai}. The choice of the critical values $r_+$ or $r^*$ in \eqref{defrinfty} is dictated by this geometry of the phase portrait. The fact that two of the points $P_2,P_3,P_5$ collide at $r_{\hskip -.1pc\peye}$ induces a degeneracy of the flow at the critical value which is the starting point of a renomalization process in terms of a small parameter $b=o_{r\uparrow \re}(1)$. We would however like to stress the fact that, given the smallness parameter $b$, the understanding of the $\mathcal C^\infty$ regularity of the solution involves formally the expansion of the solution {\em to the order $\frac 1b$} which is too large for a WKB type analysis. Instead, we rely on a holomorphic expansion of the solution at $P_2$ to extract a formal limit involving the limiting series $S_\infty(d,\ell)$. The condition $S_\infty(d, \ell)\ne 0$ turns out to be sufficient to prove the existence of a $\mathcal C^\infty$ profile, but we do not know if it is also necessary, as this would require pushing the asymptotic expansions to the next order. We refer to section \ref{stegerguerproof} for a detailed strategy of the proof of Theorem \ref{thmmain} and the description of the role played by the limiting series $S_\infty(d,\ell)$.\\

\noindent{\em 3. The sequence $r_n$.} There is an abundant literature devoted to the existence of self-similar solutions in various nonlinear problems. For example for the nonlinear heat equation $$\pa_tu=\Delta u+u^p,$$ there is in a suitable regime of nonlinearities a countable family of self-similar solutions decaying at infinity, but the self-similar speed is determined uniquely by the scaling symmetry of the problem, see \cite{CRS,lepin}. We should also mention the study of 
self-similar solutions for the energy super-critical wave equation in \cite{KS} in both the focusing and defocusing cases, where the regularity of solutions across the light cone also plays a fundamental role. The situation of Theorem \ref{thmmain} is however different: the Euler equations possess a 2-parameter family of scalings, and, as a result, the blow up speed $r$ of the self-similar profile \eqref{selfimislairna} is a free parameter. Theorem \ref{thmmain} thus provides a family of admissible\footnote{selected by the fact that their  $\mathcal C^\infty$ regularity and the decay at $+\infty$ guarantee that
 the linearized operator displays  dynamical stability properties, \cite{MRRSnls}.} blow up profiles which {\em accumulate} at the critical speed $\re(d,\ell)$. This is, to our knowledge,  a completely new phenomenon.\\

\noindent{\em 4. The case $\ell= d$}. The case $d=\ell=3$ is $\gamma=\frac 53$, i.e., the law for the monoatomic gas, and is always degenerate since for $d=\ell$:  $$r=r^*(d,d)=r_+(d,d)\Rightarrow P_2=P_3=P_5$$ and the phase portrait has a triple point degeneracy which requires a separate treatment.\\

\noindent{\em 5. Decay of self-similar solutions.} As stated in Corollary \ref{corm}, the constructed self-similar solutions 
$(\rho(t,y), u(t,y))$ decay as $y\to \infty$ but do not have finite energy. Nonetheless, this decay is sufficient for us to 
 construct, in our companion paper \cite{MRRSfluid}, solutions to the Euler and Navier-Stokes equations which arise from 
smooth, well localized initial data (in particular, of finite energy) and form a singularity in finite time via a profile given in this paper. 
We also are able to utilize these profiles to produce blowing up solutions to the super-critical defocusing nonlinear Schr\"odinger 
equations (in $d\ge 5$) which also arise from smooth, well localized (in particular, finite energy) initial data, \cite{MRRSnls}.

\subsection{Further qualitative properties of the solution}


We emphasize again that the main motivation of Theorem \ref{thmmain} is the {\em dynamical study} of \eqref{renormalizedflow}, central for the construction of well localized smooth blow up solutions in \cite{MRRSnls}, \cite{MRRSfluid}. The study of the linearized operator close to the solution given by Theorem \ref{thmmain} and performed in \cite{MRRSnls} requires, in an essential way, the $\mathcal C^\infty$ regularity at $P_2$ as well as suitable positivity properties which we now collect. These properties are responsible for the coercivity of the linearized operator, which is why 
we refer to them as {\it repulsivity.}\\

\noindent\underline{Positivity inside the light cone}. We first claim the following positivity property inside the light cone $\sigma>\sigma(P_2)$ ($Z<Z_2$) for the $P_6-P_2$ trajectory.

\begin{lemma}[Repulsivity inside the light cone]
\label{lemmainside}
 Let $d\ge 2$ then there exists $0<\e\ll1$ such that for all $(\ell,r)$ in the range
 \be
\label{suaempeonpoengopngoe}
\left|\begin{array}{l}
  d\ge 2\\
0<\ell<d\\
r^*(d,\ell)-\e(d,\ell)<r<r^*(d,\ell)
\end{array}\right.\ \ \mbox{or}\ \ \left|\begin{array}{l}
  d=2,3\\
\ell>d\\
r_+(d,\ell)-\e(d,\ell)<r<_+(d,\ell)
\end{array}\right.
\ee
 there exists $c_r>0$ such that the $P_6-P_2$ trajectory $(\sigma(x),w(x))$ given by Lemma \ref{vnioneneno} satisfies the following bound. Let 
\be
\label{difinitotniof}
F=\sigma+\sigma',
\ee 
then
\be
\label{coercivityquadrcouplinginside}
\forall \sigma\ge \sigma(P_2), \ \ \left|\begin{array}{l}
(1-w-w')^2-F^2\ge c_r\\
1-w- w'-\frac{(1-w)F}{\sigma}\ge c_r\\
1-w-w'\ge c_{r}
\end{array}\right.
\ee
\end{lemma}

Taking $r=r_n$ for $n$ large enough ensures that \eqref{coercivityquadrcouplinginside} holds for the solution of Theorem \ref{thmmain}. Property \eqref{coercivityquadrcouplinginside} is sharp: the constant $c_r\to 0$ as $r\uparrow \re$. Note that the fact that the repulsivity property \eqref{coercivityquadrcouplinginside}, which is at the heart of the stability of the linearized flow, can be proved in the full range $\ell>0$ is surprising, and will require a substantial amount of algebra.\\

\noindent\underline{Positivity outside the light cone}. We claim another positivity property in the exterior of the light cone ($Z>Z_2$) 
in a range of parameters which includes the solutions of Theorem \ref{thmmain}.

\begin{lemma}[Repulsivity outside the light cone]
\label{repsusoutides} Let 
\be
\label{assumptionsdl}
\left|\begin{array}{l}
d=3, \ \ \ell>\sqrt{3}\ \ \mbox{satisfying} \ \ \eqref{conditionell},\\
\mbox{or}\ \ (d,\ell)\ \ \mbox{satisfies} \ \ \eqref{sufffiicnetconisit}.
\end{array}\right.
\ee
Let $r_n$ be given by Theorem \ref{thmmain} and $(w,\sigma)$ be the corresponding $\mathcal C^\infty$ integral curve. Then for all $n$ large enough, there exists $c_{r_n}>0$ such that
\be
\label{propertyptobeproved} \forall 0<\sigma\le \sigma(P_2), \ \ \left|\begin{array}{l} (1-w-w')^2-F^2>c_{r_n},\\ 1-w- w'>c_{r_n}.
\end{array}\right.
\ee
\end{lemma}

\subsection{Organization of the paper} In section \ref{geometryphaseportrait}, we establish the main geometric properties of  the phase diagram \ref{fig:solutioncurve}, in particular related to the location of the roots of the polynomials $\Delta, \Delta_1,\Delta_2$ in the suitable range of parameters. In section \ref{sec:3}, we recall the main dynamical properties of solutions to \eqref{systemedefoc}, and in particular discuss the existence of the $P_6-P_2$ trajectory, and the regularity of integral curves passing through $P_2$. Both sections \ref{geometryphaseportrait} and \ref{sec:3} are classical and we recall most details for the convenience of the reader. In section \ref{semiclassical}, we start the semi classical analysis $r\uparrow \re(d,\ell)$, provide an overview of the strategy of the proof of Theorem \ref{thmmain} in section \ref{stegerguerproof}
and  complete the proof in sections \ref{semiclassical}, \ref{sectionlimit}, \ref{bbounds}, \ref{sec:studyCinftysolution}. The positivity properties of Lemma \ref{lemmainside} and Lemma \ref{repsusoutides} are proved in  section \ref{sec:posotivitynecessaryforcontrollinearizationinsidelightcone} and section \ref{sec:posotivitynecessaryforcontrollinearizationoutsidelightcone}, respectively. The holomorphic dependence on $\ell$ of $S_\infty(d,\ell)$ is proved in Appendix \ref{appendixanalyticityofSinftyinell}, and numerical results for the computation of $S_\infty(d,\ell)$ are collected in Appendix \ref{appendixnumericsofSinftyinell}.

\subsection{Acknowledgements}  P.R. is supported by the ERC-2014-CoG 646650 SingWave and would like to thank the Universit\'e de la C\^ote d'Azur where part of this work was done for its kind hospitality. I.R. is partially supported by the NSF 
grant DMS \#1709270 and a Simons Investigator Award. J.S  is supported by the ERC grant  ERC-2016 CoG 725589 EPGR.



\section{The geometry of the phase portrait}
\label{geometryphaseportrait}


The aim of this section is to start the analysis of the nonlinear ODE system \eqref{systemedefoc} 
\be
\label{systemedefoc'}
\left|\begin{array}{ll}
(w-1)w'+\ell \sigma\sigma'+(w^2-rw+\ell \sigma^2)=0\\
\frac{\sigma}{\ell}w'+(w-1)\sigma'+\sigma\left[w\left(\frac{d}{\ell}+1\right)-r\right]=0
\end{array}
\right.
\Leftrightarrow \left|\begin{array}{l}
\Delta w'=-\Delta_1\\
\Delta \sigma'=-\Delta_2
\end{array}\right.
\ee
by examining  the roots of the polynomials $\Delta,\Delta_1,\Delta_2$ given by \eqref{veluadeternte}
\be
\label{veluadeternte'}
\left|\begin{array}{lll}
\Delta=(w-1)^2-\sigma^2\\
\Delta_1=w(w-1)(w-r)-d(w-w_e)\sigma^2\\
\Delta_2=\frac{\sigma}{\ell}\left[(\ell+d-1)w^2-w(\ell+d+\ell r-r)+\ell r-\ell \sigma^2\right]
\end{array}\right.
\ee
\be
\label{defwlimting'}
w_e=\frac{\ell(r-1)}{d}.
\ee
Their location is heavily dependent on the values of the parameters $(d,\ell,r)$. From the beginning we restrict to the case 
\be
\label{limtedcase}
\left|\begin{array}{l}
d\ge 2\\
0<w_e<1
\end{array}\right.
\ee
where we recall \eqref{defwlimting}.


\subsection{Roots of $\Delta, \Delta_1\Delta_2$}


$\Delta$ has been normalized to vanish on the sonic lines
$$
\{\Delta=0\}=\{w=1+\sigma\}\cup\{w=1-\sigma\}
$$
which are independent of the parameters. We now study the roots of $\Delta_2$. 

\begin{lemma}[Roots of $\Delta_2$] 
\label{rootsdelta2}
Assume \eqref{limtedcase}. There exists $\sigma_2^{(0)}(d,\ell)\in[0,+\infty)$ such that the roots of $\Delta_2$ in the range $\sigma \ge 0$ are given by
\be
\label{def:rootsofDelta2}
\left|\begin{array}{l}
w^{\pm}_2(\sigma)=\frac{1}{2(\ell+d-1)}\left[2\ell+d-1-\frac{dw_e(1-\ell)}{\ell}\pm\sqrt{I(\sigma)}\right]\\
\sigma\ge \sigma_2^{(0)}
\end{array}\right.
\ee
where
\be
\label{defjroots}
\left|\begin{array}{lll}
J(w_e) &=& \displaystyle d^2\left(\frac{1-\ell}{\ell}\right)^2w_e^2-\frac{2d(d-1)(\ell+1)}{\ell}w_e+(d-1)^2,\\[3mm]
 I(\sigma) &=& J(w_e)-4dw_e+4\ell(\ell+d-1)\sigma^2.
 \end{array}\right.
 \ee 
 Moreover, 
\be
\label{notonocnonty}
\forall \sigma> \sigma_2^{(0)}, \ \  (w_2^-)'(\sigma)<0, \ \ (w_2^+)'(\sigma)>0.
\ee
\end{lemma}

\begin{remark} The value $\sigma_2^{(0)}$ is explicitly given by \eqref{enoieoiennoeg} if $J(w_e)-4dw_e<0$ and $\sigma_2^{(0)}=0$ otherwise. In the range of parameters \eqref{limtedcase}, both cases $\sigma_2^{(0)}=0$ and $\sigma_2^{(0)}>0$ are possible which means that the parabola defining the set of zeroes of $\Delta_2$ may or may not touch the line $\sigma=0$  in figure \ref{fig:solutioncurve}. This will play no role in our qualitative study of the flow.
\end{remark}

\begin{proof}[Proof of Lemma \ref{rootsdelta2}]
We solve $\Delta_2=0$ for $\sigma\neq 0$ which is
\bee
0& =& (\ell +d-1)w^2-(\ell+d+\ell r-r)w+\ell r-\ell\sigma^2\\
& = & (\ell+d-1)w^2-\left(\ell+d+\ell\left(1+\frac{dw_e}{\ell}\right)-1-\frac{dw_e}{\ell}\right)w+\ell\left(1+\frac{dw_e}{\ell}\right)-\ell\sigma^2\\
& = &  (\ell+d-1)w^2-\left(2\ell+d-1+dw_e\left(1-\frac1\ell\right)\right)w+\left(\ell+dw_e\right)-\ell\sigma^2.
\eee
The discriminant is given by
\bee
I(\sigma)&=&\left(2\ell+d-1+dw_e\left(1-\frac1\ell\right)\right)^2-4\left[\left(\ell+dw_e\right)-\ell\sigma^2\right](\ell+d-1)\\
& = & 4\ell(\ell+d-1)\sigma^2+\frac{d^2(1-\ell)^2}{\ell^2}w_e^2-\frac{2d}{\ell}(\ell d+\ell+d-1)w_e+(d-1)^2\\
& = & J(w_e)-4dw_e+4\ell(\ell+d-1)\sigma^2
\eee
which justifies the formula for $w^{\pm}_2(\sigma)$. We now study the sign of $J(w_e)-4dw_e$. The equation $J(w_e)-4dw_e=0$ has real roots given by 
$$w_{\ell,\pm}^{**}=\frac{\ell}{d(1-\ell)^2}\left[\ell d+\ell+d-1\pm2\sqrt{\ell d(d+\ell-1)}\right].$$
Hence $J(w_e)-4dw_e<0$ for $w_{\ell,-}^{**}<w_e<w_{\ell,+}^{**}$ in which case $I(\sigma)\ge 0$ for 
\be
\label{enoieoiennoeg}
\sigma\ge \sigma_2^{(0)}=\sqrt{\frac{4dw_e-J(w_e)}{4\ell(\ell+d-1)}}.
\ee Next $J(w_e)-4dw_e\geq 0$ for $w_e\geq w_{\ell,+}^{**}$ or $w_e\leq w_{\ell,-}^{**}$ in which case $I(\sigma)\ge 0$ for all $\sigma\ge \sigma_2^{(0)}=0$. The monotonicity \eqref{notonocnonty} follows directly from the fact that we have $I'(\sigma)> 0$ for all $\sigma>0$.
\end{proof}

\begin{lemma}[Roots of $\Delta _1$] 
\label{lem:Delta1} 
Assume \eqref{limtedcase}. For all $\sigma\ge 0$, the equation $\Delta_1(w,\sigma)=0$ has exactly three distinct root branches $w_1(\sigma)<w_2(\sigma)<w_3(\sigma)$ which satisfy the following:\\
\noindent\underline{relative positions}: $\forall \sigma \ge 0$, 
\be
\label{spacelocalization}
 -w_1(\sigma)\leq 0<w_e<w_2(\sigma)\leq 1<r\leq w_3(\sigma).
\ee
\noindent\underline{monotonicity}: $\forall \sigma>0$, \ \ 
\be
\label{mototnicnityroots}
w_1'(\sigma)<0, \ \ w'_2(\sigma)<0, \ \ w_3'(\sigma)>0\ \ \mbox{for}\ \ \sigma>0.
\ee
\noindent\underline{asymptotics}:
\be
\label{neinvootrigin}
\left|\begin{array}{lll}
\displaystyle w_1(\sigma)=-\frac{dw_e}{r}\sigma^2+O(\sigma^3),\\
\displaystyle w_2(\sigma)=1-\frac{d(1-w_e)\sigma^2}{r-1}+O(\sigma^3),\\
\displaystyle w_3(\sigma)=r+\frac{d(r-w_e)}{r(r-1)}\sigma^2+O(\sigma^3),
\end{array}\right. \ \ \mbox{as}\ \ \sigma \to 0
\ee
and 
\be
\label{bisneinvootrigin}
\left|\begin{array}{lll}
\displaystyle w_1(\sigma)=-\sqrt{d}\sigma +O(1),\\[1mm]
\displaystyle w_2(\sigma)=w_e+O(\sigma^{-2}),\\[1mm]
\displaystyle w_3(\sigma)=\sqrt{d}\sigma +O(1)
\end{array}\right.\ \ \mbox{as}\ \ \sigma\to +\infty.
\ee
\end{lemma}

\begin{proof}[Proof of Lemma \ref{lem:Delta1}] At $\sigma=0$, we have the three obvious roots
\bee
w_1(0)=0,\ \ w_2(0)=1,\ \ w_3(0)=r
\eee
where we recall that \eqref{limtedcase}, \eqref{defwlimting} imply 
\be
\label{loewrebvi}
r>1.
\ee
Also, for each fixed $\sigma>0$, we have 
$$\left|\begin{array}{l}
\lim_{w\to+\infty}\Delta_1(w,\sigma)=+\infty, \ \ \lim_{w\to-\infty}\Delta_1(w,\sigma)=-\infty\\
\Delta_1(0,\sigma) = dw_e\sigma^2> 0,\ \ \Delta_1(w_e,\sigma)= w_e(w_e-1)(w_e-r)>0\\
\Delta_1(1,\sigma) = -d(1-w_e)\sigma^2<0,\ \ \Delta_1(r,\sigma) = -d(r-w_e)\sigma^2<0,
\end{array}\right.
$$
where we used from \eqref{defwlimting}, \eqref{loewrebvi} that $w_e<1<r$. Thus, for all $\sigma\geq 0$, $\Delta_1(w,\sigma)$ has exactly three distinct simple roots which then satisfy \eqref{spacelocalization} and 
\be
\label{firstrealtion}
\pr_w\Delta_1(w_1(\sigma), \sigma)>0,\ \ \pr_w\Delta_1(w_2(\sigma), \sigma)<0, \ \ \pr_w\Delta_1(w_3(\sigma), \sigma)>0.
\ee
We may now apply the implicit function theorem and conclude that the roots $w_1(\sigma)$, $w_2(\sigma)$, $w_3(\sigma)$ are smooth function of $\sigma$ for $\sigma\geq 0$. Furthermore, we have the formula for $j=1,2,3$
\be
\label{cneioneoncnoie}
w_j'(\sigma)=-\frac{\pr_\sigma\Delta_1(w_j(\sigma), \sigma)}{\pr_w\Delta_1(w_j(\sigma), \sigma)}=\frac{2d\sigma(w_j(\sigma) - w_e)}{\pr_w\Delta_1(w_j(\sigma), \sigma)}, \ \  \sigma\geq 0
\ee
which together with \eqref{firstrealtion} and the location of the roots \eqref{spacelocalization} ensures the monotonicity \eqref{mototnicnityroots}.\\
We now compute the limiting asymptotics. Near $\sigma=0$, we compute from \eqref{cneioneoncnoie}:
\bee
w_j'(0)=0, \ \ w_j''(0)=\frac{2d(w_j(0) - w_e)}{\pr_w\Delta_1(w_j(0), 0)}, \ \ j=1,2,3,
\eee
which together with the above explicit values $w_j(0)$, $j=1,2,3$, and the fact that
\bee
\pr_w\Delta_1(w, 0) = w(w-1)+w(w-r)+(w-1)(w-r)
\eee
yields the expansion \eqref{neinvootrigin} as $\sigma\to 0$. To compute the expansion of the three roots near $+\infty$, we notice that
\bee
\Delta_1(\pm\sqrt{d}\sigma+c_1, \sigma) &=&  d(2c_1-r-1+w_e)\sigma^2+O(\sigma),\\
\Delta_1(w_e+c_2\sigma^{-2}, \sigma) &=& w_e(w_e-1)(w_e-r)-dc_2+O(\sigma^{-2}),
\eee
and an appropriate choice of the constants $c_1$, $c_2$ yields \eqref{bisneinvootrigin} from the mean value theorem.
\end{proof}


\subsection{Double roots}


We now discuss the double roots $\Delta_1=\Delta_2=0$ which play a {\em fundamental} role in the study of \eqref{systemedefoc}.

\begin{lemma}[Double roots]
\label{doubleroots}
Assume \eqref{limtedcase}. The solutions to $\Delta_1=\Delta_2=0$ are:
\be
\label{doublepoints}
\left|\begin{array}{lll} 
P_1=(0,1), \ \  P_2=(1-w_-,w_-)\ \ P_3=(1-w_+,w_+),\\
P_4=(0,0), \ \ P_5=(\sigma_5,w_5), \ \ P'_5=(0,r),
\end{array}\right.
\ee
where the points are defined as follows:\\
\noindent\underline{$P_5$ point}. 
\be
\label{cooridntate}
P_5=\left(w_5=\frac{\ell r}{d+\ell}, \sigma_5=\frac{r\sqrt{d}}{d+\ell}\right)
 \ee
 \noindent\underline{$P_2$, $P_3$ points}. Let $J(w_e)$ given \eqref{defjroots}, then
\be
\label{defjevknl}
w_{\pm}=\frac{1}{2(d-1)}\left(dw_e+d-1-\frac{dw_e}{\ell}\pm\sqrt{J(w_e)}\right)
\ee 
and $P_2$, $P_3$ are on the phase portrait iff  
\be
\label{condtionptwt}
w_\pm  \ \ \mbox{real} \Leftrightarrow (w_e<w_{\ell}^- \ \ \mbox{or}\ \ w_e>w_\ell^+)
\ee
with 
\be
\label{dfnienfieoneinomoien}
w_{\ell}^{\pm}=\frac{\ell(d-1)}{d(1-\ell)^2}\left[\ell+1\pm2\sqrt{\ell}\right].
\ee \noindent\underline{Location}.  When defined, $P_2,P_3,P_5$ are located on the curve of the middle root $(\sigma, w_2(\sigma))$ of $\Delta_1$. Moreover, $P_2,P_5$ are on the curve of the lower root $w_2^-$ of $\Delta_2$.\\
\vskip .3pc
\noindent\underline{Position of the middle root}. Let $w_e<w_\ell^-$ and $w_2(\sigma)$ be the middle root of $\Delta _1$, then the relative position of the middle root with respect to the sonic line is:
\be
\label{relativepsotion}
\sigma+w_2(\sigma)\left|\begin{array}{l}
>1\ \ \mbox{for}\ \ 0<\sigma<\sigma(P_3)\\
<1\ \ \mbox{for}\ \ \sigma(P_3)<\sigma<\sigma(P_2)\\
>1\ \ \mbox{for}\ \ \sigma>\sigma(P_2).
\end{array}\right.
\ee

\end{lemma}

\begin{proof}[Proof of Lemma \ref{doubleroots}] It relies on the factorization
\be
\label{calculdeterm}
\left|\begin{array}{lll}
\Delta = a_1b_2-b_1a_2\\[1mm]
\Delta_1 = -b_1d_2+b_2d_1\\[1mm]
\Delta_2 = \displaystyle d_2a_1-d_1a_2.
\end{array}\right.
\ee
with
\be
\label{defvaluesboinedone}
\left|\begin{array}{ll}
a_1=w-1, \ \ b_1=\ell\sigma, \ \ d_1=w^2-rw+\ell\sigma^2,\\[1mm]
a_2=\frac{\sigma}{\ell}, \ \ b_2=w-1, \ \ d_2=\sigma\left[\left(1+\frac{d}{\ell}\right)w-r\right].
\end{array}\right.
\ee

\noindent{\bf step 1} Computation of the triple points. From \eqref{calculdeterm}:
\be
\label{realigenoe}
( \Delta_1=\Delta_2=0) \Leftrightarrow \left(b_1d_2=b_2d_1\ \ \mbox{and}\ \ d_2a_1=d_1a_2\right).
\ee
Also, by definition, $a_1=b_2$, and we argue below according to the cases $a_1=b_2=0$ and $a_1=b_2\neq 0$. \\

\noindent\underline{Case $a_1=b_2=0$, i.e. $w=1$}. If $\sigma=0$, we have the point $P_1$. If $\sigma\neq 0$, $b_1\neq 0$ and hence from \eqref{defvaluesboinedone} \eqref{realigenoe}
$$0=d_2=\sigma\left(1+\frac{d}{\ell}-r\right)=\frac{\sigma d}{\ell}(1-w_e)\neq 0,$$ 
a contradiction.\\

\noindent\underline{Case $a_1=b_2\neq 0$.} If $\sigma=0$, then $d_2=b_1=0$ and hence $0=d_1=w(w-r)$, and hence the points $P_4=(0,0)$ and $P'_5=(0,r)$. If $\sigma\neq 0$, then $a_1,b_2,b_1,a_2\neq 0$. If $d_2=0$, then $d_1=0$ and hence 
$$
\left|
\begin{array}{ll}
w=w_5=\frac{r}{1+\frac{d}{\ell}}, \\
 \sigma^2_5=\frac{w_5(r-w_5)}{\ell}=\frac{dr^2}{(d+\ell)^2}, \ \ \sigma_5=\frac{r\sqrt{d}}{d+\ell}.
 \end{array}\right.
 $$
 We observe $$w_5<1\Leftrightarrow w_e<1$$ and hence \eqref{limtedcase} and \eqref{spacelocalization} implies that $P_5$ lies on the middle root $w_2(\sigma)$ of $\Delta_1$.\\
 If $d_2\neq 0$ then $$\frac{a_1}{a_2}=\frac{b_1}{b_2}=\frac{d_1}{d_2}, \ \ \Delta_1=\Delta_2=\Delta=0.$$ Hence $w=1+\sigma$ or $w+\sigma=1$ which we consider separately.\\
 
 \noindent\underline{Points on the lower sonic line, i.e., $w=1-\sigma$}. First, note that it suffices to consider the solutions to $\Delta_2(\sigma,1-\sigma)=0$. Indeed, since $a_1=b_2\neq 0$ and $d_2\neq 0$, then, since we have $\Delta_2=\Delta=0$, we infer
 \bee
 \frac{b_1}{b_2}=\frac{a_1}{a_2}=\frac{d_1}{d_2}.
 \eee
 It then follows that  $\Delta_1=0$. Thus, we now consider the solutions to $\Delta_2(\sigma,1-\sigma)=0$. Since $\sigma\neq 0$:
\bea
\label{defonfeoneo}
\nonumber 0=P(w) & = & (\ell+d-1)w^2-(\ell+d+\ell r-r)w+\ell r-\ell(1-w)^2\\
\nonumber & = & (\ell+d-1)w^2-(\ell+d+\ell r-r)w+\ell r-\ell(w^2-2w+1)\\
\nonumber & = & (d-1)w^2-(\ell(r-1)+d-r)w+\ell(r-1)\\
& = & (d-1)w^2-\left(dw_e+d-1-\frac{dw_e}{\ell}\right)w+dw_e.
\eea
The roots of $P$ are real iff 
\bea
\label{eq:J} J&=& \left(dw_e+d-1-\frac{dw_e}{\ell}\right)^2-4d(d-1)w_e\\ \nonumber
& = & d^2\left(\frac{1-\ell}{\ell}\right)^2w_e^2-\frac{2d(d-1)(\ell+1)}{\ell}w_e+(d-1)^2\geq 0.
\eea
The discriminant of $J$, as a second order polynomial in $w_e$, is $>0$, and the roots of $J$ are given by
\bee
\nonumber w_{\ell}^{\pm}&=&\frac{\ell^2}{2d^2(1-\ell)^2}\left[\frac{2d(d-1)(\ell+1)}{\ell}\pm\frac{4d(d-1)}{\sqrt{\ell}}\right]= \frac{\ell(d-1)}{d(1-\ell)^2}\left[\ell+1\pm2\sqrt{\ell}\right].
\eee
Hence 
\be
\label{dfnoinonoen}
J\geq 0 \ \ \Longleftrightarrow \ \ (w_e\leq w_{\ell}^-\ \ \mbox{or}\ \ w_e\geq w_{\ell}^+).
\ee
If $J\geq 0$, the roots of $P$ are given by 
$$w_{\pm}=\frac{1}{2(d-1)}\left(dw_e+d-1-\frac{dw_e}{\ell}\pm\sqrt{J}\right).$$ We compute from \eqref{defwlimting}:
$$\left|\begin{array}{l}
P(1)=r-1>0\\
P'(1)=d-1+(1-\ell)(r-1)>r-1>0
\end{array}\right.
$$
and 
$$\left|\begin{array}{l}
P(0)=\ell(r-1)>0\\
P'(0)=-\left[d-1+(1-\ell)(r-1)\right]<0
\end{array}\right.
$$
which, since $P$ is the second order polynomial with non negative second order term, ensures $$0<w_{\pm}<1.$$ This implies from \eqref{spacelocalization} that $P_2,P_3$ lie on the middle root $w_2(\sigma)$ of $\Delta_1$ which is a non increasing function of $\sigma$ from \eqref{mototnicnityroots}, and hence necessarily $$w_e<w_{\pm}.$$

\noindent\underline{Points on the upper sonic line, i.e., $w=1+\sigma$.} We consider the solutions to $\Delta_2(\sigma,1+\sigma)=0$ or, equivalently, $$Q(\sigma)=(\ell+d-1)(1+\sigma)^2-(\ell+d+\ell r-r)(1+\sigma)+\ell r-\ell \sigma^2=0.$$ We compute $$Q(0)=d+r-1, \ \ Q'(0)=d+r-\ell(r-1)>r$$ from \eqref{limtedcase}, and hence the second order polynomial $Q$ is $>0$ for $\sigma>0$ and there is no intersection point on the upper sonic line.\\

\noindent{\bf step 2} Location of $P_2,P_3,P_5$. We have established that $P_2,P_3,P_5$ lie on the middle root of $\Delta_1$. 
Next, note that $w=1-\sigma$ is decreasing and $w_2^+(\sigma)$ in increasing by \eqref{notonocnonty}, so they can intersect at most once, and hence at least one point among $P_2, P_3$ must be on  the root $(\sigma, w^-_2(\sigma))$ of $\Delta_2$. Since $w_-<w_+$, we infer that $P_2$ is on the root $(\sigma, w^-_2(\sigma))$ of $\Delta_2$.\\ 

\noindent{\bf step 3} Proof of \eqref{relativepsotion}. Assume now $w_e<w_\ell^-$, then $J(w_e)>0$ ensures $w_-<w_+$ and, since $P_2,P_3$ lie on the sonic line $w+\sigma=1$, $$\sigma(P_3)<\sigma(P_2).$$
Hence $0<\sigma(P_3)<\sigma(P_2)$ are three distinct roots of 
$$R(\sigma)=\Delta_1(1-\sigma,\sigma)=(1-\sigma)(-\sigma)(1-\sigma-r)-d(1-\sigma-w_e)\sigma^2$$
which is an order three polynomial, and hence these are the only roots which are simple. Since $w_2(\sigma)$ lies above the sonic line $\sigma+w=1$ near $\sigma=0$ and near $\sigma=+\infty$ from \eqref{neinvootrigin}, the ordering \eqref{relativepsotion} follows \end{proof}


\subsection{Relative positions of $P_2,P_3,P_5$}


We now discuss a very important property for our forthcoming analysis regarding the relative position of $P_2$ and $P_5$. We recall the definition \eqref{valuerstar}, \eqref{defrinfty} of the critical speeds. Let us start with comparing these values.

\begin{lemma}[Comparison of $r_+$ and $r^*$]
\label{camorimofnego}
Let $d\ge 2$, then $$\left|\begin{array}{l} w_\ell^-=w_e(r_+)<1\\
r^*(d,\ell)<r_+(d,\ell).
\end{array}\right.
$$ and 
\be
\label{cneioneonoen}
w_e<w_{\ell}^-\Leftrightarrow r\le r_{+}(d,\ell).
\ee
\end{lemma}

\begin{proof}[Proof of Lemma \ref{camorimofnego}] We compute from \eqref{dfnienfieoneinomoien}:
$$w_\ell^-=\frac{\ell(d-1)}{d(1-\ell)^2}\left[\ell+1\pm2\sqrt{\ell}\right]=\frac{\ell(d-1)}{d(1+\sqrt{\ell})^2}<1$$
 and hence from \eqref{valuerstar}: $$w_e<w_{\ell}^-\Leftrightarrow \frac{\ell(r-1)}{d}<\frac{\ell(d-1)}{d(1+\sqrt{\ell})^2}\Leftrightarrow r\le r_{+}(d,\ell).$$
Then \bea
\label{vlaueifnioegn}
&&\nonumber r_+(d,\ell)-r^*(d,\ell)=1+\frac{d-1}{(1+\sqrt{\ell})^2}-\frac{d+\ell}{\ell+\sqrt{d}}=\frac{(d-1)(\ell+\sqrt{d})-(d-\sqrt{d})(1+2\sqrt{\ell}+\ell)}{(\ell+\sqrt{d})(1+\sqrt{\ell})^2}\\
& = & (\sqrt{d}-1)\frac{(\sqrt{\ell}-\sqrt{d})^2}{(\ell+\sqrt{d})(1+\sqrt{\ell})^2}>0.
\eea
\end{proof}

We now design an admissible portrait as follows.

\begin{lemma}[Admissible phase portrait, see figure \ref{fig:signofDeltasinphaseportrait}]
\label{phasperotrai}
Assume
\be
\label{aummptitonrdel}
d\ge 2, \ \ \ell>0, \ \ 1<r\le r_+(d,\ell).
\ee 
then the conclusions of Lemma \ref{doubleroots} hold with $0<w_e\le w_\ell^-$ and $P_2,P_3,P_5$ are well defined. More precisely:\\
\noindent\underline{1. below the $r^*$ speed}: for $1<r<r^*(d,\ell)$,  $P_5$ lies strictly below the sonic line 
$0<w_5+\sigma_5<1$ 
and 
\be
\label{ordering}
0<\sigma_3<\sigma_5<\sigma_2.
\ee
\noindent\underline{2. $r^*$ speed}: for $r=r^*(d,\ell)$, \be
\label{poitnprtwppthre}
\left|\begin{array}{l}
\ell<d\Leftrightarrow P_5=P_2\\
\ell>d\Leftrightarrow P_5=P_3\\
\ell=d\Leftrightarrow P_5=P_2=P_3.
\end{array}\right.
\ee
\noindent\underline{3. below the $r_+$ speed}: for $r^*(d,\ell)<r<r_+(d,\ell)$ and $\ell>d$,
\be
\label{reaieoveio} 
\sigma_5<\sigma_3<\sigma_2
\ee 
and $$P_3\to P_2\ \ \mbox{as}\ \ r\uparrow r_+.$$
\end{lemma}

\begin{remark} We note here the fundamental role payed by the case $\ell=d$ which corresponds to $\gamma=1+\frac 2d$ and is a degenerate triple point configuration $$r=r^*(d,d)=r_+(d,d)\Leftrightarrow P_2=P_3=P_5.$$
\end{remark}

\begin{proof}[Proof of Lemma \ref{phasperotrai}] We observe $$w_\ell^-<1\Leftrightarrow \frac{\ell(d-1)(\sqrt{\ell}-1)^2}{d(1-\ell)^2}<1\Leftrightarrow \frac{\ell(d-1)}{d(1+\sqrt{\ell})^2}<1\Leftrightarrow d(2\sqrt{\ell}+1)+\ell>0$$ which holds, and this together with \eqref{cneioneonoen} ensures that the conclusions of Lemma \ref{doubleroots} hold and $P_2,P_3$ are well defined, distinct and on the sonic line.\\
\noindent\underline{Subcritical speed}. For $1<r<r^*(d,\ell)$, we compute from \eqref{cooridntate}:
\be
\label{venioneineoneovi}
w_5+\sigma_5=\frac{\ell r}{d+\ell}+\frac{r\sqrt{d}}{d+\ell}=\frac{r}{r^*(d,\ell)}<1
\ee 
and hence $P_5$ lies strictly below the sonic line $w+\sigma=1$. Since $P_5$ lies on the curve  $(\sigma,w_2(\sigma))$ where $w_2(\sigma)$ is the middle root of $\Delta_1$, the ordering \eqref{relativepsotion} implies \eqref{ordering}.\\

\noindent\underline{Critical speed}. Let now $r=r^*(d,\ell)$. Then $P_5$ is on the sonic line and since $(0,1),P_2,P_3$ are the only  intersections of $w_2(\sigma)$ with the sonic line, $P_5$ coincides necessarily  with $P_2$ or $P_3$. For $d=\ell$, $$r^*(d,\ell)=r_+(d,\ell)\Rightarrow w_e=w_{\ell}^-\Rightarrow P_2=P_3$$ and hence $P_2=P_3=P_5$. For $d\neq \ell$, let $J=J(w_e)\ge0$ be given by \eqref{defjroots} and, from  \eqref{defjevknl}, let 
\bee
\sigma_2&\equiv& \sigma(P_2)=1-w_-=1-\frac{1}{2(d-1)}\left(dw_e+d-1-\frac{dw_e}{\ell}-\sqrt{J}\right)\\
& = & \frac{1}{2(d-1)}\left[d-1+d\left(\frac{1}{\ell}-1\right)w_e+\sqrt{J}\right]
\eee
and
\bee
\sigma_3&\equiv& \sigma(P_3)=1-w_+=1-\frac{1}{2(d-1)}\left(dw_e+d-1-\frac{dw_e}{\ell}+\sqrt{J}\right)\\
& = & \frac{1}{2(d-1)}\left[d-1+d\left(\frac{1}{\ell}-1\right)w_e-\sqrt{J}\right]
\eee
so that from \eqref{cooridntate}:
\be
\label{formualugf}
\left|\begin{array}{l}
\sigma_2-\sigma_5=\frac{\sqrt{J}-A}{2(d-1)}\\
\sigma_3-\sigma_5=\frac{-\sqrt{J}-A}{2(d-1)}
\end{array}\right.
\ee
with
$$\left|\begin{array}{l}
A=  2(d-1)\sqrt{d}\frac{r}{d+\ell}-d+1-(1-\ell)(r-1)=\left(\frac{2(d-1)\sqrt{d}}{d+\ell}-1+\ell\right)(r-r_0(\ell))\\
r_0(\ell)=\frac{d+\ell-2}{\frac{2(d-1)\sqrt{d}}{d+\ell}-1+\ell}>0,
\end{array}\right.
$$
where we used that $$\forall \ell>0, \ \ \forall d\ge 2, \ \ \frac{2(d-1)\sqrt{d}}{d+\ell}-1+\ell>0.$$ We compute
\bee
\nonumber &&r^*(d,\ell)-r_0=\frac{d+\ell}{\sqrt{d}+\ell}-\frac{d+\ell-2}{\frac{2(d-1)\sqrt{d}}{d+\ell}+\ell-1}\\
\nonumber & =& \frac{d+\ell}{(\sqrt{d}+\ell)(2(d-1)\sqrt{d}+(\ell-1)(d+\ell))}\left[2(d-1)\sqrt{d}+(\ell-1)(d+\ell)-(\sqrt{d}+\ell)(d+\ell-2)\right]\\
& = & \frac{(\sqrt{d}-1)(d+\ell)}{(\sqrt{d}+\ell)(2(d-1)\sqrt{d}+(\ell-1)(d+\ell))}(d-\ell)
\eee
Therefore,
\be
\label{nioenieneno}
A(r^*)\left|\begin{array}{l} >0\ \ \mbox{for}\ \ \ell<d\\ <0\ \ \mbox{for}\ \ \ell>d.
\end{array}\right.
\ee
Since $P_5$ coincides with $P_2$ or $P_3$, \eqref{formualugf} implies 
\be
\label{neineenonvoe}
J(r^*)=A^2(r^*)
\ee and hence the sign of $A(r^*)$ is given by \eqref{nioenieneno} and \eqref{formualugf} yield \eqref{poitnprtwppthre}.\\

\noindent{\bf step 3} We now turn to the case $r^*<r<r_+$ for $\ell>d$. We claim 
\be
\label{firsestiamte}
r_+-r_0<0
\ee
and 
\be
\label{formual}
J-A^2=-c_1(d,\ell)(r+c_2(d,\ell))(r-r^*), \ \ c_1,c_2>0.
\ee
Assume \eqref{firsestiamte}, \eqref{formual}, then for $r^*<r<r_+$, $A<0$ and $J-A^2<0$ implies from \eqref{formualugf} that $\sigma_3-\sigma_5>0$, and since $\sigma_2>\sigma_3$ by definition, \eqref{reaieoveio} is proved.\\
\noindent{\em Proof of \eqref{firsestiamte}}. We compute
\bee
\nonumber &&r^*(d,\ell)-r_0=\frac{d+\ell}{\sqrt{d}+\ell}-\frac{d+\ell-2}{\frac{2(d-1)\sqrt{d}}{d+\ell}+\ell-1}\\
\nonumber & =& \frac{d+\ell}{(\sqrt{d}+\ell)(2(d-1)\sqrt{d}+(\ell-1)(d+\ell))}\left[2(d-1)\sqrt{d}+(\ell-1)(d+\ell)-(\sqrt{d}+\ell)(d+\ell-2)\right]\\
& = & \frac{(\sqrt{d}-1)(d+\ell)}{(\sqrt{d}+\ell)(2(d-1)\sqrt{d}+(\ell-1)(d+\ell))}(d-\ell)
\eee
and hence recalling \eqref{vlaueifnioegn}:
\bee
&&r_+-r_0=r_+-r^*+r^*-r_0\\
&=&\frac{(\sqrt{d}-1)(d+\ell)(d-\ell)}{(\sqrt{d}+\ell)(2(d-1)\sqrt{d}+(\ell-1)(d+\ell))}+(\sqrt{d}-1)\frac{(\sqrt{\ell}-\sqrt{d})^2}{(\ell+\sqrt{d})(1+\sqrt{\ell})^2}\\
& = & \frac{(\sqrt{d}-1))(\sqrt{d}-\sqrt{\ell})}{\ell+\sqrt{d}}\left[\frac{(d+\ell)(\sqrt{\ell}+\sqrt{d})}{2(d-1)\sqrt{d}+(\ell-1)(d+\ell))}-\frac{\sqrt{\ell}-\sqrt{d}}{(1+\sqrt{\ell})^2}\right]
\eee
and hence the sign is dictated by
\bee
&&P(d,\ell)=(d+\ell)(\sqrt{\ell}+\sqrt{d})(1+\sqrt{\ell})^2-(\sqrt{\ell}-\sqrt{d})\left[2(d-1)\sqrt{d}+(\ell-1)(d+\ell))\right]\\
& = & (d+\ell)\left[(\sqrt{\ell}+\sqrt{d})(1+2\sqrt{\ell}+\ell)-(\sqrt{\ell}-\sqrt{d})(\ell-1)\right]-2(d-1)\sqrt{d}(\sqrt{\ell}-\sqrt{d})\\
& = & (d+\ell)\left[\sqrt{d}+\sqrt{\ell}+2\sqrt{\ell}\sqrt{d}+2\ell+\ell\sqrt{d}+\ell\sqrt{\ell}-(\ell\sqrt{\ell}-\sqrt{\ell}-\ell\sqrt{d}+\sqrt{d})\right]\\
& - & 2(d-1)\sqrt{d}(\sqrt{\ell}-\sqrt{d})\\
& = & (d+\ell)[(2\sqrt{d}+2)\sqrt{\ell}+(\sqrt{d}+3)\ell]-2(d-1)\sqrt{d}(\sqrt{\ell}-\sqrt{d})\\
& = & (d+\ell)(\sqrt{d}+3)\ell+2\sqrt{d}(d-1)+\sqrt{\ell}[(d+\ell)(2\sqrt{d}+2)-2(d-1)\sqrt{d}]>0
\eee
and \eqref{firsestiamte} is proved.\\
\noindent{\em Proof of \eqref{formual}}. We compute
\bee
J& = & d^2\left(\frac{1-\ell}{\ell}\right)^2w_e^2-\frac{2d(d-1)(\ell+1)}{\ell}w_e+(d-1)^2\\
& = & (1-\ell)^2(r-1)^2-2(d-1)(\ell+1)(r-1)+(d-1)^2\\
&= & (1-\ell)^2r^2-2r\left[(1-\ell)^2+(\ell+1)(d-1)\right]+(d-1)^2+(1-\ell)^2+2(d-1)(\ell+1).
\eee
and hence injecting the value of $A$:
\bee
J-A^2& = &-\left\{ r^2\left[\left(\frac{2(d-1)\sqrt{d}}{d+\ell}-1+\ell\right)^2-(1-\ell)^2\right]\right.\\
& +& \left.2r\left[(1-\ell)^2+(\ell+1)(d-1)-\left(\frac{2(d-1)\sqrt{d}}{d+\ell}-1+\ell\right)^2r_0\right]r\right.\\
& - & \left.(d-1)^2-(1-\ell)^2-2(d-1)(\ell+1)+\left(\frac{2(d-1)\sqrt{d}}{d+\ell}-1+\ell\right)^2r_0^2\right\}\\
& = & -(ar^2+2br-c)
\eee
with using $\ell>d>1$:
$$
a=  \left(\frac{2(d-1)\sqrt{d}}{d+\ell}+\ell-1\right)^2-(\ell-1)^2=\frac{4(d-1)\sqrt{d}}{d+\ell}\left[\frac{(d-1)\sqrt{d}}{d+\ell}+\ell-1\right]>0$$
\bee
b&=&(1-\ell)^2+(\ell+1)(d-1)-\left(\frac{2(d-1)\sqrt{d}}{d+\ell}-1+\ell\right)^2r_0\\
&=&(1-\ell)^2+(\ell+1)(d-1)-(d+\ell-2)\left(\frac{2(d-1)\sqrt{d}}{d+\ell}-1+\ell\right)\\
& = & 1-2\ell+\ell^2+\ell d-\ell+d-1-\frac{2(d-1)(d+\ell-2)\sqrt{d}}{d+\ell}-(\ell d+\ell^2-2\ell-d-\ell+2)\\
& = &2(d-1)\left[1-\frac{(d+\ell-2)\sqrt{d}}{d+\ell}\right]
\eee
and
\bee
c& = & (d-1)^2+(1-\ell)^2+2(d-1)(\ell+1)-\left(\frac{2(d-1)\sqrt{d}}{d+\ell}-1+\ell\right)^2r_0^2\\
& = & (d-1)^2+(1-\ell)^2+2(d-1)(\ell+1)-(d+\ell-2)^2=4(d-1).
\eee
Since $a,c>0$, the roots are given by 
$$r^*\pm=\frac{1}{a}(-b\pm\sqrt{b^2+ac}), \ \ r^*_-<0.$$ We now observe that \eqref{neineenonvoe} implies 
\be
\label{vnioneioneov}
r^*_+=r^*
\ee 
which can also be checked directly, and \eqref{formual} is proved.
\end{proof}


\subsection{Slopes at $P_2$}


The point $P_2$ will play a fundamental role in the proof of Theorem \ref{thmmain}. In this section we collect the main geometric properties of the phase portrait near $P_2$ in the regime \eqref{aummptitonrdel}. We note $$\sigma_2=\sigma(P_2), \ \ w_2=1-\sigma_2.$$ 

\noindent\underline{Definition of the slopes} We compute the slopes of root curves 
$w_2(\sigma)$ and $w_2^-(\sigma)$  at $P_2$ by defining the following coefficients 
\be
\label{defvalueci}
\left|\begin{array}{llll}
c_1=\pa_w\Delta_1(P_2)=3w_2^2-2(r+1)w_2+r-d\sigma_2^2\\
c_2=\pa_w\Delta_2(P_2)=\frac{\sigma_2}{\ell}[2w_2(\ell+d-1)-(\ell+d+\ell r-r)]\\
c_3=\pa_\sigma \Delta_1(P_2)=-2d\sigma_2w_2+2\ell(r-1)\sigma_2\\
c_4=\pa_\sigma\Delta_2(P_2)=-2\sigma_2^2,
\end{array}\right.
\ee
so that the corresponding slopes are $\frac {c_1}{c_1}$ and $\frac {c_4}{c_2}$. We now claim:

\begin{lemma}[Sign of the slopes] 
\label{signslopes}
Assume $d\ge 2$ and 
\be
\label{nondegeneraterange}
\left|\begin{array}{l}
1<r<r^*(d,\ell)\ \ \mbox{for}\ \ \ell<d\\
r^*(d,\ell)<r<r_+(d,\ell)\ \ \mbox{for}\ \ \ell>d
\end{array}\right.
\ee
then
\be
\label{signsci}
\left|\begin{array}{l}
c_i<0, \ \ 1\le i\le 4\\
c_2c_3-c_1c_4<0.
\end{array}\right.
\ee
\end{lemma}

\begin{proof}[Proof of Lemma \ref{signslopes}] The argument relies solely on the consideration of the relative positions of the red and green curves locally near $P_2$ which is the same in the range \eqref{nondegeneraterange}. Indeed, in view of the discussion of the roots of $\Delta_1$, and since we have established that $P_2$ is on the middle root of $\Delta_1$, we have
\bee
c_1=\pr_\om\Delta_1(P_2)<0.
\eee
Since $P_2$ corresponds to the smallest root $w_2^-$ of $\Delta_2$ and $\Delta_2$ is a second order polynomial in $w$ with a strictly positive coefficient in front of the $w^2$ term, we have
\bee
c_2=\pr_\om\Delta_2(P_2)<0.
\eee
Also, since $\sigma>0$ at $P_2$, and since $w_e<w_-<1$, we have
\bee
&& c_3=\pr_\sigma\Delta_1(P_2)=-2d(w_--w_e)(1-w_-)<0, \\ 
&& c_4=\pr_\sigma\Delta_2(P_2)=-2(1-w_-)^2<0.
\eee

Finally, we compute
\bee
c_2c_3- c_1c_4 &=& \pr_\om\Delta_2(P_2)\pr_\sigma\Delta_1(P_2) - \pr_\om\Delta_1(P_2)\pr_\sigma\Delta_2(P_2)\\
&=& \pr_\om\Delta_2(P_2)\pr_\om\Delta_1(P_2)\left(\frac{\pr_\sigma\Delta_1(P_2)}{\pr_\om\Delta_1(P_2)}-\frac{\pr_\sigma\Delta_2(P_2)}{\pr_\om\Delta_2(P_2)}\right)\\
&=& c_1c_2\Big(({w^-_2})'(\sigma)-w_1'(\sigma)\Big)_{|_{\sigma=1-w_-}}
\eee
where we used the fact that $\sigma=1-w_-$ and $w=w_-$ at $P_2$, the fact that $w_2$ is both the middle root of $\Delta_1$, i.e., $w_1$, and the smallest root of $\Delta_2$, i.e., $w_2^-$, with the formula for ${w^-_2}'(\sigma)$ and $w_1'(\sigma)$ following from the implicit function theorem. Now, at $P_2$, the slope of $\Delta_2$ is strictly more negative than the slope of $\Delta_1$ since $P_2$ is the last intersection in $\sigma$ and $\Delta_1$ asymptotes to $w_e$ while $\Delta_2$ goes to $-\infty$. Thus, we have $({w^-_2})'(1-w_-)-w_2'(1-w_-)<0$ and hence
\bee
c_2c_3- c_1c_4 &< & 0
\eee
as desired.
\end{proof}

\noindent\underline{Slopes and eigenvalues}. In additions to the slopes of the root curves $w_2(\sigma)$ and
$w_2^-(\sigma)$ we will also compute the slopes of any integral curve passing through $P_2$. it turns out that there are only two possible values:  
\be
\label{defslpodeplus}
c_\pm=\frac{c_4-c_1\pm\sqrt{(c_1-c_4)^2+4c_2c_3}}{2|c_2|}
\ee 
This follows since $c_\pm$ are the solutions of the equation
\be
\label{equationcminus}
c_{\pm}=\frac{c_1c_\pm+c_3}{c_2c_\pm+c_4}.
\ee 
The characteristic matrix $$\mathcal A(P_2)=\left(\begin{array}{ll} c_1 &c_3\\ c_2&c_4\end{array}\right)$$ 
possesses the following eigenvalues:  
\be
\label{deflplus}
\l_{\pm}=\frac{c_1+c_4\pm \sqrt{(c_1-c_4)^2+4c_2c_3}}{2}
\ee 
It may be diagonalized as follows:
$$P^{-1}\left[\mathcal A(P_2)\right]P=\left(\begin{array}{ll} \l_+&0\\ 0&\l_-\end{array}\right)
$$
with
\be
\label{defp}
P=\left(\begin{array}{ll} c_-&c_+\\ 1&1\end{array}\right), \ \ P^{-1}=\frac{1}{c_+-c_-}\left(\begin{array}{ll} -1&c_+\\1&-c_-\end{array}\right).
\ee
\begin{lemma}[Estimates on the slopes]
\label{neonvineoinve}
Assume \eqref{nondegeneraterange} and let 
\be
\label{eq:defbntionofA}
 A =\frac{\l_-}{\l_+}= \frac{c_1c_4-c_2c_3}{(c_4+c_2c_-)^2},
\ee
Then,
\be
\label{tionveiogbngo3o}
\left|\begin{array}{l}
c_-<0<c_+\\
c_4+c_2c_- <0\\
A>1\\
-1<-\frac{c_4}{c_2}< c_- < -\frac{c_3}{c_1}<0\\
\l_-<\l_+<0
\end{array}\right.
\ee
\end{lemma}

\begin{proof}[Proof of Lemma \ref{neonvineoinve}] From $c_2c_3>0$:
\bee
 \frac{c_4-c_1 -\sqrt{(c_1-c_4)^2+4c_2c_3}}{2|c_2|} <0< \frac{c_4-c_1 +\sqrt{(c_1-c_4)^2+4c_2c_3}}{2|c_2|}
\eee
and hence
\bee
c_-<0<c_+.
\eee

Next, we compute
$$
c_4+c_2c_- = \frac{c_4+c_1 + \sqrt{(c_1-c_4)^2+4c_2c_3}}{2}= \frac{2(c_2c_3- c_1c_4)}{-c_4-c_1 + \sqrt{(c_1-c_4)^2+4c_2c_3}}< 0.
$$
We now observe
\be
\label{realtionslopeweignefuncitons}
\left|\begin{array}{ll}
c_2c_-+c_4=c_4-\frac{c_4-c_1-\sqrt{\Delta}}{2}=\frac{c_4+c_1+\sqrt{\Delta}}{2}=\l_+\\
c_2c_++c_4=c_4-\frac{c_4-c_1+\sqrt{\Delta}}{2}=\frac{c_4+c_1-\sqrt{\Delta}}{2}=\l_-.
\end{array}\right.
\ee
and hence $\l_-<\l_+<0.$ We now estimate $A$:
\bee
&&c_4c_1-c_3c_2-(c_4+c_2c_-)^2 \\
&=& c_4c_1-c_3c_2 - \frac{4(c_2c_3- c_1c_4)^2}{(-c_4-c_1 + \sqrt{(c_1-c_4)^2+4c_2c_3})^2}\\
&=& -(c_4+c_2c_-)\frac{\left[(-c_4-c_1 + \sqrt{(c_1-c_4)^2+4c_2c_3})^2-4(c_4c_1-c_2c_3)\right]}{2(-c_4-c_1 + \sqrt{(c_1-c_4)^2+4c_2c_3})}\\
&=& -(c_4+c_2c_-)\frac{\left[(c_4+c_1)^2+2(-c_4-c_1)\sqrt{(c_1-c_4)^2+4c_2c_3}+(c_1-c_4)^2+4c_2c_3-4(c_4c_1-c_2c_3)\right]}{2(-c_4-c_1 + \sqrt{(c_1-c_4)^2+4c_2c_3})}\\
&=& -(c_4+c_2c_-)\frac{\left[(-c_4-c_1)\sqrt{(c_1-c_4)^2+4c_2c_3}+(c_1-c_4)^2+4c_2c_3\right]}{(-c_4-c_1 + \sqrt{(c_1-c_4)^2+4c_2c_3})^2}
\eee
and hence
\bee
A  - 1 &=& \frac{c_1c_4-c_2c_3}{(c_4+c_2c_-)^2} -1\\
& =&- \frac{\left((-c_4-c_1)\sqrt{(c_1-c_4)^2+4c_2c_3}+(c_1-c_4)^2+4c_2c_3\right)}{(c_4+c_2c_-){(-c_4-c_1 + \sqrt{(c_1-c_4)^2+4c_2c_3})}}.
\eee
Using in particular that $c_1, c_2, c_3, c_4<0$, this yields
\bee
A>1.
\eee
Next, note that we have
\bee
-\frac{c_3}{c_1}<0
\eee
since $c_3<0$ and $c_1<0$. Also, since $\{\Delta_2=0\}$ intersects $w=1-\sigma$ at $P_2$  and $P_3$, $(\sigma, w_2^-(\sigma))$ is above $w=1-\sigma$ for $\sigma>1-w_-$ and below for $\sigma<1-w_-$. Since $-c_4/c_2$ is the slope of $w=w_2^-(\sigma)$ at $P_2$, we infer 
\bee
-\frac{c_4}{c_2}>-1.
\eee
It remains to compare $c_-$ to $-c_3/c_1$. We compute
\bee
c_-+\frac{c_3}{c_1} &=& c_-+\frac{|c_3|}{|c_1|}\\
&=& \frac{|c_1|(c_4-c_1 -\sqrt{(c_1-c_4)^2+4c_2c_3})+2c_2c_3}{2|c_1||c_2|}\\ 
&=& \frac{c_1(c_1-c_4)+2c_2c_3 -|c_1|\sqrt{(c_1-c_4)^2+4c_2c_3}}{2|c_1||c_2|}. 
\eee
Now, we have
\bee
&& \Big(c_1(c_1-c_4)+2c_2c_3\Big)^2 - \Big(|c_1|\sqrt{(c_1-c_4)^2+4c_2c_3}\Big)^2\\
&=&  c_1^2(c_1-c_4)^2+4c_2c_3c_1(c_1-c_4)+4c_2^2c_3^2 -c_1^2(c_1-c_4)^2-4c_1^2c_2c_3\\
&=& 4c_2c_3(c_2c_3-c_1c_4) < 0
\eee
and hence
\bee
c_-+\frac{c_3}{c_1}< 0.
\eee
This concludes the proof of \eqref{tionveiogbngo3o}.
\end{proof}


\section{General properties of the dynamical system \eqref{systemedefoc}}
\label{sec:3}


In this section we establish the general properties of Lemma \ref{vnioneneno} for the dynamical system \eqref{systemedefoc}.  We assume 
\be
\label{assumptionparameters}
d\ge 2, \ \ \left|\begin{array}{l} 1<r<r^*(\ell,d)\ \ \mbox{for}\ \ \ell<d\\
r^*(d,\ell)<r<r_+(d,\ell) \ \ \mbox{for}\ \ \ell>d
\end{array}\right.
\ee
so that the shape of the phase portrait is given by respectively figure \ref{fig:solutioncurve} or figure \ref{fig:solutioncurvebis}. We recall that  $w_2(\sigma)$ is the middle root of $\Delta_1$ and $w_2^{-}(\sigma)$ is the smallest root of $\Delta_2$ given by \eqref{def:rootsofDelta2}.
\begin{figure}
\centering
\includegraphics[width=13cm]{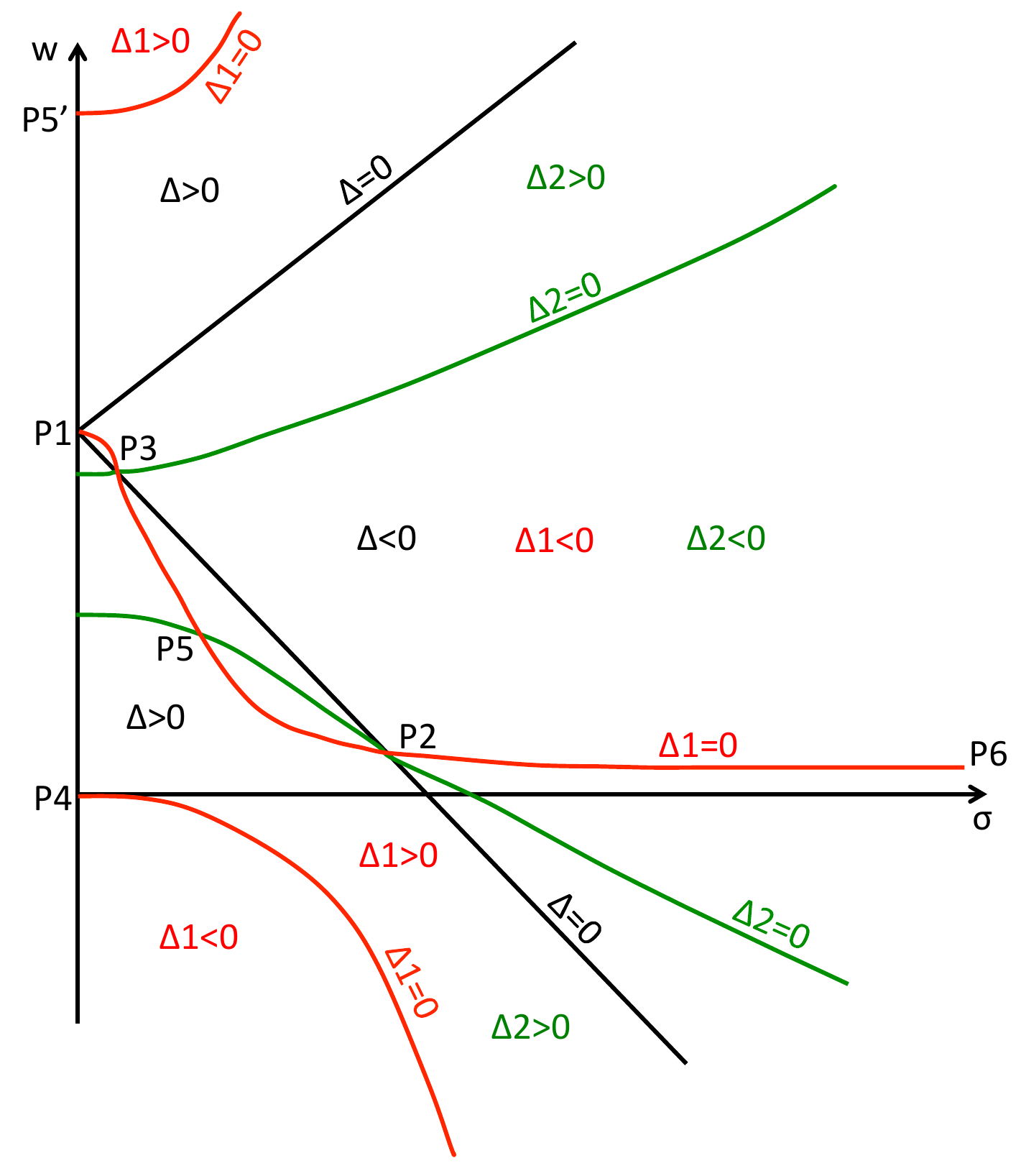}
\caption{Shape of the phase portrait for $1<r<r^*(\ell,d)$ }
\label{fig:signofDeltasinphaseportrait}
\end{figure}
\begin{figure}
\centering
\includegraphics[width=13cm]{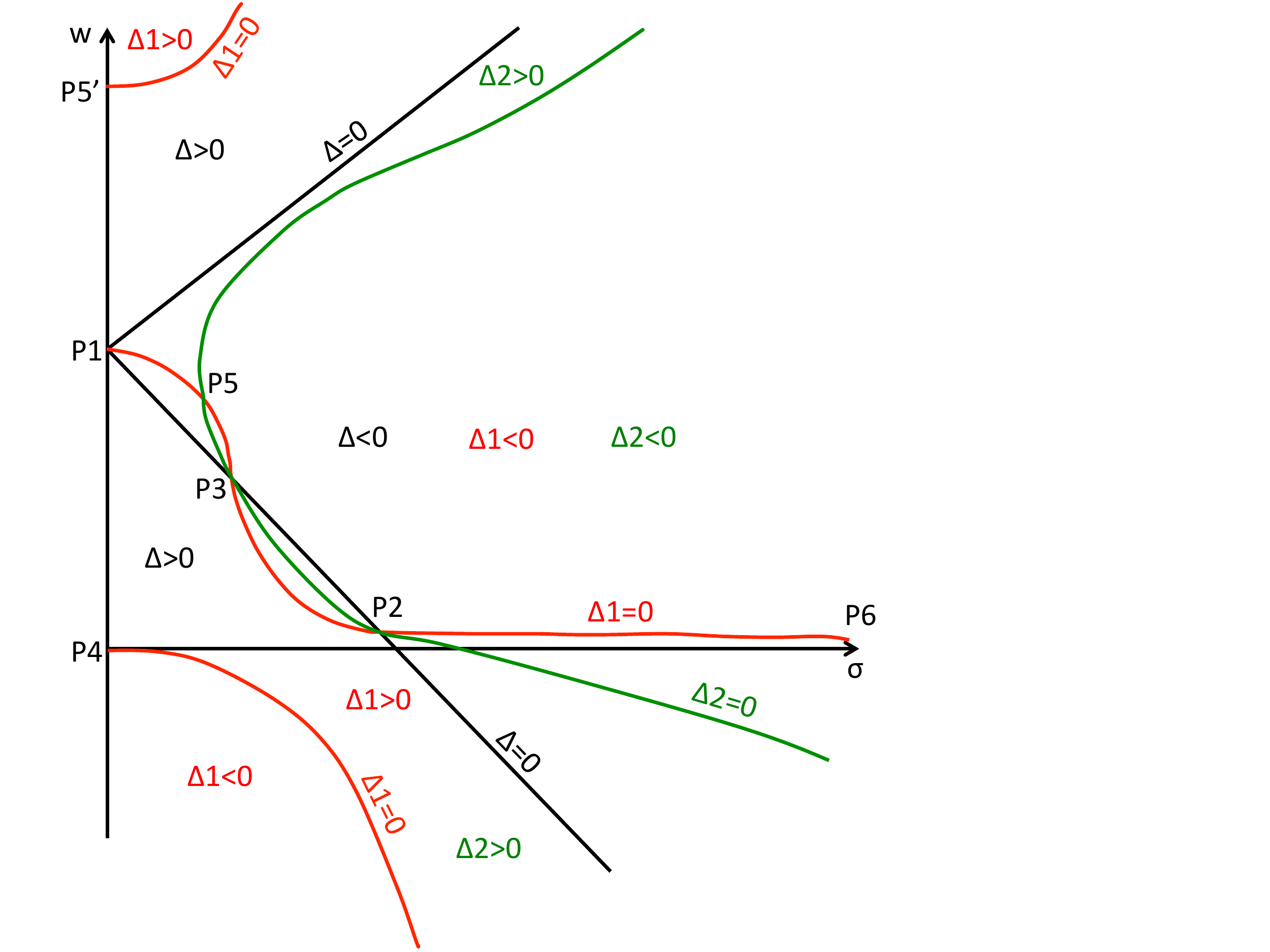}
\caption{Shape of the phase portrait for $r^*(\ell,d)<r<r_+(d,\ell)$, $\ell>d$ }
\label{fig:signofDeltasinphaseportraitbis}
\end{figure}
The arguments in this section are classical and are given for the reader's convenience.\\
We note $$P_{\hskip -.1pc\peye}=\left|\begin{array}{l} P_5\ \ \mbox{for}\ \ \ell<d\\ P_3\ \ \mbox{for}\ \ \ell>d.
\end{array}\right.
$$ 



\subsection{The spherically symmetric solution emerging from the origin}


We first claim the existence and uniqueness (up to the scaling symmetry) of a spherically symmetric solution to \eqref{renormalizedflowstaionray} which exists on the interval $[0,Z_2]$ and, in the variables of Emden 
transform, corresponds to the integral curve $P_6-P_2$.

\begin{lemma}[The solution emerging from $P_6$]
\label{lememarompfoe}
Assume \eqref{assumptionparameters}.
\begin{enumerate}
\item Existence: there is $\sigma_0>0$ large enough and a unique curve solution $w(\sigma)$ 
to \eqref{systemedefoc} with 
\be
\label{equationseparatrxi}
\lim_{\sigma\to +\infty}w(\sigma)=w_e.
\ee
It admits the asymptotic expansion as $\sigma\to +\infty$:
\be
\label{asymptoticiosisp6}
w(\sigma)=w_e+\frac{w_e(w_e-1)(w_e-r)}{d+2}\frac{1}{\sigma^2}+O\left(\frac{1}{\sigma^4}\right).
\ee
\item Original variables: The curve corresponds to a spherically symmetric solution of \eqref{renormalizedflowstaionray} defined on the interval $|Z|\in [0,Z_0]$. This solution belongs to 
$\mathcal C^\infty(|Z|\le Z_0)$.\\
\vskip .3pc
\item Reaching $P_2$: we have $w\in\mathcal C^\infty(\sigma_2,+\infty)$ with  
\be
\label{leinigngnieng}
\left|\begin{array}{l}
\forall \sigma_2<\sigma<+\infty, \ \ w_2^-(\sigma)<w(\sigma)<w_2(\sigma)\\
\lim_{\sigma\downarrow  \sigma_2}w(\sigma)=w_2,
\end{array}\right.
\ee
see figure \ref{fig:solutioncurve}.
\end{enumerate}
\end{lemma}

\begin{proof}[Proof of Lemma \ref{lememarompfoe}] This follows from the asymptotic behavior of the polynomials $\Delta_1,\Delta_2$.\\

\noindent{\bf step 1} Flow near $P_6$. We decompose 
\bee
w=w_e+\widehat{w},\quad \lim_{\sigma\to +\infty}\widehat{w}(\sigma)=0.
\eee
Then, $\widehat{w}$ satisfies 
\bee
\widehat{w}' -\frac{d}{\sigma}\widehat{w} &=& \frac{w(w-1)(w-r)-d\widehat{w}\sigma^2}{\frac{\sigma}{\ell}\Big[(\ell+d-1)w^2-w(\ell+d+\ell r-r)+\ell r-\ell \sigma^2\Big]} -\frac{d}{\sigma}\widehat{w}\\
\nonumber &=& \left(\frac{\widehat{w}-\frac{w(w-1)(w-r)}{d\sigma^2}}{1+\frac{-(\ell+d-1)w^2+w(\ell+d+\ell r-r)-\ell r}{\ell\sigma^2}} -\widehat{w}\right)\frac{d}{\sigma}\\
\nonumber &=& \left(\frac{-\ell w(w-1)(w-r)+d\Big((\ell+d-1)w^2-w(\ell+d+\ell r-r)+\ell r\Big)\widehat{w}}{1+\frac{-(\ell+d-1)w^2+w(\ell+d+\ell r-r)-\ell r}{\ell\sigma^2}} \right)\frac{1}{\ell\sigma^3}
\eee
and hence
\bea
\label{boebbeibebeviev}
\left(\frac{\widehat{w}}{\sigma^d}\right)' &=& \frac{1}{\sigma^d}\left(\widehat{w}' -\frac{d}{\sigma}\widehat{w}\right)\\
\nonumber&=& \left(\frac{-\ell w(w-1)(w-r)+d\Big((\ell+d-1)w^2-w(\ell+d+\ell r-r)+\ell r\Big)\widehat{w}}{1+\frac{-(\ell+d-1)w^2+w(\ell+d+\ell r-r)-\ell r}{\ell\sigma^2}} \right)\frac{1}{\ell\sigma^{d+3}}.
\eea
\noindent\underline{Existence}. We solve
\bea
\label{fiecnieniengo}
&&\widehat{w}\\
\nonumber &=&-\sigma^d\int_\sigma^{+\infty}\left(\frac{-\ell w(w-1)(w-r)+d\Big((\ell+d-1)w^2-w(\ell+d+\ell r-r)+\ell r\Big)\widehat{w}}{1+\frac{-(\ell+d-1)w^2+w(\ell+d+\ell r-r)-\ell r}{\ell{\sigma'}^2}} \right)\frac{d\sigma'}{\ell\sigma^{d+3}}.
\eea
using an elementary fixed point argument which yields the existence and uniqueness on $\sigma_0\leq\sigma\leq  +\infty$ of $\widehat{w}$ such that
\be
\label{ceioneioneononeonvi}
|\widehat{w}| \lesssim \frac{1}{\sigma^2}\textrm{ for }\sigma\geq\sigma_0
\ee
provided $\sigma_0$ has been chosen large enough.\\
\vskip .3pc
 \noindent\underline{Uniqueness}. Let now $w(\sigma)$ be a solution to \eqref{boebbeibebeviev} on $[\sigma_0,+\infty)$ with $\lim_{\sigma\to +\infty} w(\sigma)=w_e$. We integrate \eqref{boebbeibebeviev}. The integral converges at $+\infty$ from the a priori bound $|w|\lesssim 1$, and hence $\widehat{w}$ satisfies \eqref{fiecnieniengo} and the a priori bound \eqref{ceioneioneononeonvi}. The uniqueness claim follows.\\
 \vskip .3pc
\noindent\underline{Asymptotics as $\sigma\to +\infty$}. The integral equation \eqref{fiecnieniengo} for $\widehat{w}$ and the fact that
\bee
&& \frac{-\ell w(w-1)(w-r)+d\Big((\ell+d-1)w^2-w(\ell+d+\ell r-r)+\ell r\Big)\widehat{w}}{1+\frac{-(\ell+d-1)w^2+w(\ell+d+\ell r-r)-\ell r}{\ell{\sigma'}^2}}\\
&=& -\ell w_e(w_e-1)(w_e-r)+O\left(\frac{1}{\sigma^2}\right)
\eee
and 
\bee
&& -\sigma^d\int_\sigma^{+\infty}\left(-\ell w_e(w_e-1)(w_e-r)+O\left(\frac{1}{{\sigma'}^2}\right)\right)\frac{d\sigma'}{\ell{\sigma'}^{d+3}}\\
 &=& \frac{w_e(w_e-1)(w_e-r)}{d+2}\frac{1}{\sigma^2}+O\left(\frac{1}{\sigma^4}\right).
\eee
yield \eqref{asymptoticiosisp6}.\\
\vskip .3pc
\noindent\underline{Spherical symmetry and regularity at the origin}. From \eqref{asymptoticiosisp6}, on the solution, 
 as $\sigma\to+\infty$:
$$\left|\begin{array}{l}
\Delta=-\sigma^2\left[1+O\left(\frac{1}{\sigma^2}\right)\right]\\
\Delta_2=-\sigma^3\left[1+O\left(\frac{1}{\sigma^2}\right)\right]\\
\end{array}\right.
$$
As a result,
$$\frac{dx}{d\sigma}=-\frac{\Delta}{\Delta_2}=-\frac{1}{\sigma}\left[1+O\left(\frac{1}{\sigma^2}\right)\right].$$  
Up to a constant 
$$x=-\log\sigma+O\left(\frac{1}{\sigma^2}\right)\Rightarrow \sigma=e^{-x}\left[1+O(e^{2x})\right] \ \ \mbox{as}\ \ x\to -\infty.$$ Recalling the Emden transform formula \eqref{emdentransform}, we obtain from \eqref{asymptoticiosisp6} the asymptotics as $Z=e^{x}\to 0$:
\be
\label{boudnaryoriigin}
\left|\begin{array}{l}
(\rhoh(Z))^{\frac{\gamma-1}{2}}=\sqrt{\frac \ell 2}Z\sigma(x)=\sqrt{\frac \ell 2}(1+O(Z^2))\\ 
\uh(Z)=- Zw(x)=-Z\left[w_e+O(Z^2)\right]
\end{array}\right.
\ee
We also observe that the system \eqref{systemedefoc'} can be rewritten with respect to the variables 
$$
s:=Z^2,\qquad f:=w-w_e,\qquad g=Z\sigma
$$
in the form 
$$
\left|\begin{array}{l}
\frac {df}{ds} =-\frac d{2s} f + \mathcal F(s,f,g)\\ 
\frac {dg}{ds}  =\mathcal G(s,f,g)
\end{array}\right.
$$
for some functions $\mathcal F, \mathcal G$ with the property that they are smooth functions of all the variables in 
the neighborhood of the point $(0,0,\sqrt{\frac\ell 2})$. As above, writing the above in the form 
$$
\left|\begin{array}{l}
f(s) = s^{-\frac d2}\int_0^s (\tilde s)^{\frac d2}\mathcal F(\tilde s,f,g) d\tilde s\\ 
g(s)  =\sqrt{\frac\ell 2}+\int_0^s\mathcal G(\tilde s,f,g)d\tilde s
\end{array}\right.
$$
produces a unique fixed point solution which is $\mathcal C^1$ in $s$ and, obviously, coincides with the 
solution constructed from \eqref{fiecnieniengo}. Rewriting the above again as
$$
\left|\begin{array}{l}
f(s) = s\int_0^1 (\kappa)^{\frac d2}\mathcal F(s\kappa,f(s\kappa),g(s\kappa)) d\kappa\\ 
g(s)  =\sqrt{\frac\ell 2}+s\int_0^1\mathcal G(\kappa,f(s\kappa),g(s\kappa)) d\kappa
\end{array}\right.
$$
implies that 
\bee
f'(s) &=& \int_0^1 (\kappa)^{\frac d2}\mathcal F(s\kappa,f(s\kappa),g(s\kappa)) d\kappa+
s\int_0^1 (\kappa)^{\frac d2}\frac d{ds}\left(\mathcal F(s\kappa,f(s\kappa),g(s\kappa))\right) d\kappa\\ 
&=&\int_0^1 (\kappa)^{\frac d2}\mathcal F(s\kappa,f(s\kappa),g(s\kappa)) d\kappa+
\int_0^1 (\kappa)^{\frac d2+1}\frac d{d\kappa}\left(\mathcal F(s\kappa,f(s\kappa),g(s\kappa))\right) d\kappa\\
&=&\int_0^1 (\kappa)^{\frac d2}\mathcal F(s\kappa,f(s\kappa),g(s\kappa)) d\kappa+\mathcal F(s,f(s),g(s))-
\frac {d+2}2\int_0^1 (\kappa)^{\frac d2}\mathcal F(s\kappa,f(s\kappa),g(s\kappa)) d\kappa
\eee
Similarly, for $g$. This expresses $f'(s)$ (and $g'(s)$) in terms of regular kernels involving $f$ and $g$ and immediately 
implies (by step by step differentiation) that both $f$ and $g$ are $\mathcal C^\infty$ with respect to $s$. Recalling that $s=Z^2$ and the 
Emden transform relation \eqref{boudnaryoriigin} to the original variables $\hat u, \hat\rho$, yields the
desired local (defined in a neighborhood of the origin) $\mathcal C^\infty(\Bbb R^d,\Bbb R^2)$ 
spherically symmetric  solution.\\

\noindent{\bf step 2} Reaching $P_2$. Let $w_r(\sigma)$ be the unique curve entering $P_6$ constructed in step 1. From \eqref{asymptoticiosisp6},
\bee
\Delta_1(\sigma, w_r(\sigma)) &=& w_r(\sigma)(w_r(\sigma)-1)(w_r(\sigma)-r)-d(w_r(\sigma)-w_e)\sigma^2\\
&=& w_e(w_e-1)(w_e-r)\left(1-\frac{d}{d+2}\right)+O\left(\frac{1}{\sigma^2}\right)\\
&=& w_e(w_e-1)(w_e-r)\frac{2}{d+2}+O\left(\frac{1}{\sigma^2}\right)>0
\eee
near $+\infty$. By Lemma \ref{lem:Delta1}, $\Delta_1<0$ for $w<w_1(\sigma)$ and for $w\in (w_2(\sigma),w_3(\sigma))$,
and $\Delta_1>0$ for $w>w_3(\sigma)$ and for $w\in (w_1(\sigma),w_2(\sigma))$. Similarly, by Lemma \ref{rootsdelta2},
$\Delta_2>0$ for $w\in (w_2^-(\sigma),w_2^+(\sigma))$ and $\Delta_2<0$ for $w<w_2^-(\sigma)$ and for 
$w>w_2^+(\sigma)$. From the asymptotic expansion of $w_r(\sigma)$ we can conclude that $w_r'(\sigma)<0$ for $\sigma>\sigma_0$.
Given that $w_2^\pm(\sigma)\to \pm \infty$ as $\sigma\to \infty$, it follows that $w_r(\sigma)\in (w_2^-(\sigma),w_2^+(\sigma))$ for $\sigma>\sigma_0$ and thus $\Delta_2(\sigma,w_r(\sigma))>0$. Since, 
$$
0>w_r'(\sigma)=\frac {\Delta_1}{\Delta_2},
$$
it follows that $\Delta_1(\sigma,w_r(\sigma))<0$. Combining this with the fact that $w_r(\sigma)\to w_e=\lim_{\sigma\to\infty} w_2(\sigma)$ and examining the regions of constant signs of $\Delta_1$ described above, imply 
that $w_r(\sigma)<w_2(\sigma)$ for all $\sigma>\sigma_0$.
Thus, we have obtained
$$
w_2^-(\sigma)<w_r(\sigma)<w_2(\sigma)\textrm{ for }\sigma\geq\sigma_0.
$$
In view of the phase portrait of figure \ref{fig:signofDeltasinphaseportrait} {\em and the strict monotonicity} $w_2'<0$ given by \eqref{mototnicnityroots}, this implies that the curve $w_r$ reaches $P_2$ with 
\be
\label{neioninoinionegor}
\forall \sigma>\sigma_2, \ \ w_2^-(\sigma)<w_r(\sigma)<w_2(\sigma).
\ee 
Let us give a quick proof. From \eqref{neioninoinionegor}, we consider the region below the middle root of $\Delta_1$, above the smallest root of $\Delta_2$ and to the right of $P_2$. 
$$
\mathcal{R}_r = \{(\sigma, w),\, w_2^-(\sigma)\leq w\leq w_2(\sigma), \sigma\geq 1-w_-\}
$$
The point $P_2$ is the utmost right joint root of
$\Delta_1$ and $\Delta_2$ . The curves $(\sigma,w_2(\sigma))$ and $(\sigma,w_2^-(\sigma))$ intersect there and
$\lim_{\sigma\to\infty} w_2(\sigma)=w_e$, $\lim_{\sigma\to\infty} w^-_2(\sigma)=-\infty$. Therefore, 
$w_2^-(\sigma)<w_2(\sigma)$ to the right of $P_2$. Moreover, in that region $w_1(\sigma)<w_2^-(\sigma)$
and $w_2(\sigma)<w_2^+(\sigma)$. As a consequence of the above, in the region $\mathcal{R}_r$

\bee
\Delta_1\le 0, \ \ \Delta_2\ge 0.
\eee
Now, assume by contradiction that there exists $\sigma_1>1-w_-$ such that $(\sigma, w_r(\sigma))$ is in $\mathcal{R}_r$   for $\sigma> \sigma_1$ and $(\sigma_1, w_r(\sigma_1))$ is the last point in that region. Then either $(\sigma_1, w_r(\sigma_1))$ is on the middle root of $\Delta_1$ to the right of $P_2$, in which case we have
\bee
w_r'(\sigma_1)=0
\eee
But  in view of \eqref{mototnicnityroots}, $w_2'(\sigma_1)<0$, and we have a contradiction since $w_r(\sigma)<w_2(\sigma)$
for all $\sigma>\sigma_1$. Or $(\sigma_1, w_r(\sigma_1))$ is on the smallest root of $\Delta_2$ to the right of $P_2$, in which case we have
\bee
w_r'(\sigma_1) = -\infty
\eee
This is again a contradiction since $w_r(\sigma)> w^-_2(\sigma)$  all $\sigma>\sigma_1$.
Therefore, the curve can only exit the region $\mathcal{R}_r$ at $\sigma=1-w_-$, i.e., through $P_2$. Since there are no other accumulation points in this zone,
the curve must approach $P_2$ as $\sigma\downarrow  \sigma_2$. Therefore, 
the curve $w_r(\sigma)$ is defined on $(\sigma_2,+\infty)$, is $\mathcal C^\infty$ on this interval and satisfies \eqref{leinigngnieng}. 
\end{proof}


\subsection{Solutions crossing red between $P_2$ and $P_{\hskip -.1pc\peye}$}


We now analyze trajectories that cross the middle root $w_2(\sigma)$ of $\Delta_1$ between $P_2$ and $P_{\hskip -.1pc\peye}$.

\begin{lemma}[Solutions crossing red between $P_2$ and $P_{\hskip -.1pc\peye}$]
\label{lemmaconnection}
 Assume \eqref{assumptionparameters}. Let $\sigma_5<\sigma^*<\sigma_2$ and $w(\sigma)$ be the solution to \eqref{systemedefoc} with the data $w(\sigma^*)=w_2(\sigma^*)$. Then:\\
 \vskip .3pc
\noindent{\em 1. backward flow}: $w\in \mathcal C^\infty(0,\sigma^*])$ and
 $$
 \lim_{\sigma\downarrow 0}w(\sigma)=0.
 $$
 Moreover, $\sigma\to 0$ corresponds to $x\to +\infty$ and there exist $(w_\infty,\sigma_\infty)\in \Bbb R\times \Bbb R_+^*$ such that for $x\to +\infty$
\be
\label{exofjejeope}
\left|\begin{array}{l} 
\sigma(x) = \sigma_\infty e^{-rx}\left(1+O(e^{-rx})\right),\\
w(x) = w_\infty e^{-rx}\left(1+O(e^{-rx})\right).
\end{array}\right.
\ee
{\em 2. forward flow}: $w\in\mathcal C^\infty[\sigma^*,\sigma_2)$ and 
\be
\label{leinigngniengbis}
\left|\begin{array}{l}
\forall \sigma^*<\sigma<\sigma_2, \ \ w_2(\sigma)<w(\sigma)<w^-_2(\sigma)\\
\lim_{\sigma\downarrow  \sigma_2}w(\sigma)=w_2,
\end{array}\right.
\ee
 see figure \ref{fig:solutioncurve}.
\end{lemma}
\begin{remark}
The above Lemma shows that all such solutions provide admissible $P_2-P_4$ connections.
\end{remark}
\begin{proof}[Proof of Lemma \ref{lemmaconnection}] The fact that the solution generates a $P_4-P_2$ connection with forward flow trapped in the region \eqref{leinigngniengbis} follows again directly from the phase portrait of figure \ref{fig:signofDeltasinphaseportrait} and the monotonicity \eqref{mototnicnityroots}. We leave that to the reader, while we focus on the proof of the asymptotic expansion \eqref{exofjejeope} near $\sigma=0$.\\

\noindent{\bf step 1} Behavior of $w(\sigma)$. Let
\bea
\varphi=\frac{w}{\sigma}
\eea
so that from \eqref{systemedefoc}, \eqref{veluadeternte}
\bee
\frac{dw}{d\sigma}=\frac{\varphi(\sigma\varphi-1)(\sigma\varphi-r)-d(\sigma\varphi-w_e)\sigma}{r -\frac{\ell+d+\ell r-r}{\ell}\sigma\varphi+\frac{\ell+d-1}{\ell}\sigma^2\varphi^2- \sigma^2}.
\eee
Since
\bee
\frac{dw}{d\sigma}=\sigma\frac{d\varphi}{d\sigma}+\varphi,
\eee
$\varphi$ solves
\bee
\sigma\frac{d\varphi}{d\sigma} &=& \frac{\varphi(\sigma\varphi-1)(\sigma\varphi-r)-d(\sigma\varphi-w_e)\sigma}{r -\frac{\ell+d+\ell r-r}{\ell}\sigma\varphi+\frac{\ell+d-1}{\ell}\sigma^2\varphi^2- \sigma^2}-\varphi\\
&=& \frac{ -(1+r) \sigma\varphi^2 +dw_e\sigma +\frac{\ell+d+\ell r-r}{\ell}\sigma\varphi^2 +\sigma^2\varphi^3 -d\sigma^2\varphi  -\frac{\ell+d-1}{\ell}\sigma^2\varphi^3+ \sigma^2\varphi}{r -\frac{\ell+d+\ell r-r}{\ell}\sigma\varphi+\frac{\ell+d-1}{\ell}\sigma^2\varphi^2- \sigma^2}
\eee
i.e.
\bea
\frac{d\varphi}{d\sigma} = \frac{ -(1+r)\varphi^2 +dw_e +\frac{\ell+d+\ell r-r}{\ell}\varphi^2 +\sigma\varphi^3 -d\sigma\varphi  -\frac{\ell+d-1}{\ell}\sigma\varphi^3+ \sigma\varphi}{r -\frac{\ell+d+\ell r-r}{\ell}\sigma\varphi+\frac{\ell+d-1}{\ell}\sigma^2\varphi^2- \sigma^2}.
\eea
This is a regular ODE at $\sigma=0$ and an elementary fixed point argument ensures the behavior 
\bee
\varphi(\sigma)=\varphi(0)+O(\sigma)\textrm{ as }\sigma\downarrow 0
\eee
where $\varphi(0)$ is a real constant. Hence
\be
\label{ceioveoiheneiogoegeoi}
w(\sigma)=\sigma\varphi=\varphi(0)\sigma+O(\sigma^2)\textrm{ as }\sigma\downarrow 0.
\ee
The fact that the solution $w(\sigma)$ belongs to $\mathcal C^\infty([0,\sigma_2)$ follows from above since 
$P_4$ is the only singular point on this interval.\\

\noindent{\bf step 2} Behavior in $x$ and $Z$.  From \eqref{ceioveoiheneiogoegeoi}:
 \bee
 \frac{dx}{d\sigma}&=& -\frac{\Delta(\sigma, w(\sigma))}{\Delta_2(\sigma, w(\sigma))}= -\frac{(w(\sigma)-1)^2-\sigma^2}{\frac{\sigma}{\ell}\Big[(\ell+d-1)w(\sigma)^2-(\ell+d+\ell r-r)w(\sigma)+\ell r-\ell \sigma^2\Big]}\\
&=& -\frac{1}{r\sigma} +  O(1)
 \eee
 and hence
 \bee
x &=& -\frac{1}{r}\log \sigma + x_2+ O(\sigma)
  \eee
for some real constant $x_2$. Thus, $x\to +\infty$ as $\sigma\downarrow 0$, and we have 
\bee
e^x=\frac{e^{x_2}}{\sigma^{\frac{1}{r}}}\left(1+O(\sigma)\right).
\eee
This yields the following expansion for $\sigma$ as $x\to +\infty$
\bee
\sigma(x) &=& \sigma_\infty e^{-rx}\left(1+O(e^{-rx})\right)
\eee
for some real constant $\sigma_\infty>0$. Plugging in the expansion for $w(\sigma)$, we infer
\bee
w(x) &=& w_\infty e^{-rx}\left(1+O(e^{-rx})\right)\textrm{ as }x\to +\infty
\eee
for some real constant $w_\infty$, and \eqref{exofjejeope} is proved.
\end{proof}
The results of the previous two sections provide the proof of all the statements of Lemma \ref{vnioneneno}. 
We summarize them as follows. We have constructed the unique spherically symmetric smooth solution of 
\eqref{renormalizedflowstaionray} on the interval $[0,Z_2)$ and a one parameter of spherically symmetric smooth 
solutions of 
\eqref{renormalizedflowstaionray} on the interval $(Z_2,\infty)$. These solutions agree at $Z_2$ and thus can 
be glued to each other continuously. Our goal however is to construct a {\it global} $\mathcal C^\infty$ solution.
At this point it is already clear that the crux of the matter is the point $P_2$.


\subsection{Diagonalized system at $P_2$}


The point $P_2$ will play an essential role in the proof of Theorem \ref{thmmain}. The dynamical properties of this point in the regime \eqref{assumptionparameters} can only be seen after passing to the diagonalized variables \eqref{defp}. We recall the values of the slopes \eqref{defvalueci}, \eqref{defslpodeplus}, \eqref{deflplus}, the diagonalization matrices \eqref{defp} and the non degeneracy properties of Lemma \ref{signslopes} and Lemma \ref{neonvineoinve} in the range \eqref{assumptionparameters}. We rewrite the system in coordinates which diagonalize its linear part. We will also introduce
a time variable and recast the system as a dynamical flow approaching the point $P_2$ (from either side) as $t\to\infty$.

\begin{lemma}[Equations in the diagonal form]
\label{rehivnioeei}
Assume \eqref{assumptionparameters}. Let 
\be
\label{varibakey}
\left|\begin{array}{l}
w=w_2+W\\
\sigma=\sigma_2+\Sigma\\
\frac{dt}{dx}=-\frac{1}{\Delta}
\end{array}\right., \ \ X=\left|\begin{array}{l} W\\\Sigma\end{array}\right., \ \ Y=P^{-1}X=\left|\begin{array}{ll} \Wt\\\Sigmat \end{array}\right.
\ee
then \eqref{systemedefoc} becomes:
\be
\label{normalizedbasisssystem}
\frac{dY}{dt}=\frac{1}{c_+-c_-}\left|\begin{array}{l}
\mathcal G_1\\
\mathcal G_2
\end{array}\right.
\ee
with
\bea
\label{defgaone}
\nonumber \mathcal G_1&=&(c_+-c_-)\l_+\Wt+\dt_{20}\Wt^2+\dt_{11}\Wt\Sigmat+\dt_{02}\Sigmat^2+\dt_{30}\Wt^3+\dt_{21}\Wt^2\Sigmat+\dt_{12}\Wt\Sigmat^2+\dt_{03}\Sigma^3\\
& = & -\Delta_1+c_+\Delta_2,
\eea
\bea
\label{defgaonetwo}
\nonumber \mathcal G_2&=&(c_+-c_-)\l_-\Sigmat+\et_{20}\Wt^2+\et_{11}\Wt\Sigmat+\et_{02}\Sigmat^2+\et_{30}\Wt^3+\et_{21}\Wt^2\Sigmat+\et_{12}\Wt\Sigmat^2+\et_{03}\Sigma^3\\
& = & \Delta_1-c_-\Delta_2
\eea
and where the values of the coefficients are collected in \eqref{valueparameters}, \eqref{ienveovnovne}.
\end{lemma}

\begin{proof}[Proof of Lemma \ref{rehivnioeei}] This is a direct computation.\\

\noindent{\bf step 1} Reexpressing $\Delta,\Delta_1,\Delta_2$. Let $w=w_2+W$, $\sigma=\sigma_2+\Sigma,$ we compute the nonlinear terms
\bea
\label{expressionfodetij}
\nonumber &&\Delta_1=w^3-(r+1)w^2+rw-dw\sigma^2+\ell(r-1)\sigma^2\\
\nonumber& = & w_2^3+3w_2^2W+3w_2W^2+W^3-(r+1)(w_2^2+2w_2W+W^2)+r(w_2+W)\\
\nonumber&-&d(w_2+W)(\sigma^2_2+2\sigma_2\Sigma+\Sigma^2)+\ell(r-1)(\sigma_2^2+2\sigma_2\Sigma+\Sigma^2)\\
\nonumber& = & W(3w_2^2-2(r+1)w_2+r-d\sigma_2^2)+\Sigma(-2d\sigma_2w_2+2\ell(r-1)\sigma_2)\\
\nonumber& + & W^2(3w_2-(r+1))+\Sigma^2(\ell(r-1)-dw_2)+\Sigma W(-2d\sigma_2)\\
\nonumber& + & W^3-dW\Sigma^2\\
& = & c_1W+c_3\Sigma+d_{20}W^2+d_{11}W\Sigma+d_{02}\Sigma^2+W^3-dW\Sigma^2
\eea
and 
\bea
\label{expressionfodetijbis}
\nonumber &&\Delta_2=\frac{\sigma}{\ell}\Big[(\ell+d-1)w^2-w(\ell+d+\ell r-r)+\ell r-\ell \sigma^2\Big]\\
\nonumber & = & \frac{\sigma_2+\Sigma}{\ell}\left[(\ell+d-1)(w_2^2+2w_2W+W^2)-(\ell+d+\ell r-r)(w_2+W)+\ell r-\ell(\sigma_2^2+2\sigma_2\Sigma+\Sigma^2)\right]\\
\nonumber & = & \frac{\sigma_2+\Sigma}{\ell}\left[(\ell+d-1)(2w_2W+W^2)-(\ell+d+\ell r-r)W-\ell(2\sigma_2\Sigma+\Sigma^2)\right]\\
\nonumber & = & W\frac{\sigma_2}{\ell}[2w_2(\ell+d-1)-(\ell+d+\ell r-r)]-2\sigma_2^2\Sigma\\
\nonumber & + & W^2\frac{\sigma_2(\ell+d-1)}{\ell}+\Sigma^2[-3\sigma_2]+W\Sigma\left[\frac{2w_2(\ell+d-1)-(\ell+d+\ell r-r)}{\ell}\right]\\
\nonumber & + & W^2\Sigma\left[\frac{\ell+d-1}{\ell}\right]+\Sigma^3[-1]\\
& = & c_2W+c_4\Sigma+e_{20}W^2+e_{11}W\Sigma+e_{02}\Sigma^2+e_{21}W^2\Sigma+\Sigma^3[-1]
\eea
where the parameters are given by \eqref{valueparameters}.\\

\noindent{\bf step 2} $Y$ variable. We now pass to the $Y$ variable:
\bee
&&\Delta_1(X)=c_1W+c_3\Sigma+d_{20}W^2+d_{11}W\Sigma+d_{02}\Sigma^2+W^3-dW\Sigma^2\\
& = & c_1(c_-\Wt+c_+\Sigmat )+c_3(\Wt+\Sigmat)+d_{20}(c_-\Wt+c_+\Sigmat )^2+d_{11}(c_-\Wt+c_+\Sigmat )(\Wt+\Sigmat)\\
&+& d_{02}(\Wt+\Sigmat)^2+(c_-\Wt+c_+\Sigmat )^3-d(c_-\Wt+c_+\Sigmat )(\Wt+\Sigmat)^2\\
& = & (c_1c_-+c_3)\Wt+(c_1c_++c_3)\Sigmat+d_{20}(c_-^2\Wt^2+2c_-c_+\Wt\Sigmat+c_+^2\Sigmat^2)\\
& + & d_{11}(c_-\Wt^2+(c_-+c_+)\Wt\Sigmat+c_+\Sigmat^2)+d_{02}(\Wt^2+2\Wt\Sigmat+\Sigmat^2)\\
& + & (c^3_-\Wt^3+3c_-^2\Wt^2c_+\Sigmat+3c_-\Wt c_+^2\Sigmat^2+c_+^3\Sigmat^3)-  d(c_-\Wt+c_+\Sigmat)(\Wt^2+2\Wt\Sigmat+\Sigmat^2)\\
& = & (c_1c_-+c_3)\Wt+(c_1c_++c_3)\Sigmat\\
& + & (d_{20}c_-^2+d_{11}c_-+d_{02})\Wt^2+(2c_-c_+d_{20}+(c_-+c_+)d_{11}+2d_{02})\Wt\Sigmat+(d_{20}c_+^2+d_{11}c_++d_{02})\Sigmat^2\\
& + & (c_-^3-dc_-)\Wt^3+(3c_-^2c_+-2dc_--dc_+)\Wt^2\Sigmat+(3c_-c_+^2-dc_--2dc_+)\Wt\Sigmat^2+(c_+^3-dc_+)\Sigma^3
\eee
and 
\bee
&& \Delta_2(X)=c_2W+c_4\Sigma+e_{20}W^2+e_{11}W\Sigma+e_{02}\Sigma^2+\left(\frac{\ell+d-1}{\ell}\right)W^2\Sigma-\Sigma^3\\
& = & c_2(c_-\Wt+c_+\Sigmat )+c_4(\Wt+\Sigmat)\\
&+& e_{20}(c_-\Wt+c_+\Sigmat )^2+e_{11}(c_-\Wt+c_+\Sigmat )(\Wt+\Sigmat)+e_{02}(\Wt+\Sigmat)^2\\
& + & \left(\frac{\ell+d-1}{\ell}\right)(c_-\Wt+c_+\Sigmat )^2(\Wt+\Sigmat)-(\Wt+\Sigmat)^3\\
& = & (c_2c_-+c_4)\Wt+(c_2c_++c_4)\Sigmat+e_{20}(c_-^2\Wt^2+2c_-c_+\Wt\Sigmat+c_+^2\Sigmat^2)\\
& + & e_{11}(c_-\Wt^2+(c_-+c_+)\Wt\Sigmat+c_+\Sigmat^2)+e_{02}(\Wt^2+2\Wt\Sigmat+\Sigmat^2)\\
& + & \frac{\ell+d-1}{\ell}\left[c_-^2\Wt^3+(c_-^2+2c_+c_-)\Wt^2\Sigmat+(2c_-c_++c_+^2)\Wt\Sigmat^2+c_+^2\Sigmat^3\right]\\
& - & (\Wt^3+3\Wt^2\Sigmat+3\Wt\Sigmat^2+\Sigmat^3)\\
& = & (c_2c_-+c_4)\Wt+(c_2c_++c_4)\Sigmat\\
& + & (e_{20}c_-^2+e_{11}c_-+e_{02})\Wt^2+(2e_{20}c_-c_++e_{11}(c_-+c_+)+2e_{02})\Wt\Sigmat+  (e_{20}c_+^2+e_{11}c_++e_{02})\Sigmat^2\\
& + & \left[\frac{\ell+d-1}{\ell}c_-^2-1\right]\Wt^3+\left[\frac{\ell+d-1}{\ell}(c_-^2+2c_-c_+)-3\right]\Wt^2\Sigmat\\
&+&\left[\frac{\ell+d-1}{\ell}(2c_-c_++c_+^2)-3\right]\Wt\Sigmat^2+\left[\frac{\ell+d-1}{\ell}c_+^2-1\right]\Sigmat^3.
\eee
Linear terms yield
\bee
&&\frac{1}{c_+-c_-}\left|\begin{array}{l}-((c_1c_-+c_3)\Wt+(c_1c_++c_3)\Sigmat)+c_+[(c_2c_-+c_4)\Wt+(c_2c_++c_4)\Sigmat]\\ 
(c_1c_-+c_3)\Wt+(c_1c_++c_3)\Sigmat-c_-[(c_2c_-+c_4)\Wt+(c_2c_++c_4)\Sigmat]
\end{array}\right.\\
& = &\frac{1}{c_+-c_-}\left|\begin{array}{l} 
[-(c_1c_-+c_3)+c_+(c_2c_-+c_4)]\Wt+[-(c_1c_++c_3)+c_+(c_2c_++c_4)]\Sigmat\\
((c_1c_-+c_3)-c_-(c_2c_-+c_4))\tilde{W}+[(c_1c_++c_3)-c_-(c_2c_++c_4)]\Sigmat
\end{array}\right.
\eee
We then recall the $c_\pm$ equation \eqref{defslpodeplus} which kills the off diagonal term
$$\left|\begin{array}{l}
-(c_1c_++c_3)+c_+(c_2c_++c_4)=0\\
(c_1c_-+c_3)-c_-(c_2c_-+c_4)=0
\end{array}\right.
$$ 
and compute using \eqref{realtionslopeweignefuncitons}:
$$
\left|\begin{array}{ll}
\frac{-(c_1c_-+c_3)+c_+(c_2c_-+c_4)}{c_+-c_-}=\frac{-c_-(c_2c_-+c_4)++c_+(c_2c_-+c_4)}{c_+-c_-}=c_2c_-+c_4=\l_+\\
\frac{(c_1c_++c_3)-c_-(c_2c_++c_4)}{c_+-c_-}=\frac{c_+(c_2c_++c_4)-c_-(c_2c_++c_4)}{c_+-c_-}=c_2c_++c_4=\l_-\\
\end{array}\right.
$$
For the quadratic terms, we compute for the first coordinate
\bea
\label{formuladtotwo}
\nonumber \dt_{20}&=&-(d_{20}c_-^2+d_{11}c_-+d_{02})+c_+(e_{20}c_-^2+e_{11}c_-+e_{02})\\
& = & (c_+e_{20}-d_{20})c_-^2+(c_+e_{11}-d_{11})c_-+c_+e_{02}-d_{02},
\eea
and
\bee
\dt_{11}&=&-(2c_-c_+d_{20}+(c_-+c_+)d_{11}+2d_{02})+c_+(2e_{20}c_-c_++e_{11}(c_-+c_+)+2e_{02})\\
& = & 2c_-c_+(c_+e_{20}-d_{20})+(c_-+c_+)(c_+e_{11}-d_{11})+2(c_+e_{02}-d_{02})
\eee
and
\bee
\nonumber \dt_{02}&=&-(d_{20}c_+^2+d_{11}c_++d_{02})+c_+(e_{20}c_+^2+e_{11}c_++e_{02})\\
& = & (c_+e_{20}-d_{20})c_+^2+(c_+e_{11}-d_{11})c_++c_+e_{02}-d_{02}
\eee
and
\bee
\dt_{30}=-(c_-^3-dc_-)+c_+\left(\frac{\ell+d-1}{\ell}c_-^2-1\right)
\eee
and
\bee
\dt_{21}=-(3c_-^2c_+-2dc_--dc_+)+c_+\left(\frac{\ell+d-1}{\ell}(c_-^2+2c_-c_+)-3\right)
\eee
and
\bee
\dt_{12}& = & -(3c_-c_+^2-dc_--2dc_+)+c_+\left(\frac{\ell+d-1}{\ell}(2c_-c_++c_+^2)-3\right).
\eee
Similarly for the second coordinate:
\bea
\label{defettwozero}
\nonumber \et_{20}&=&(d_{20}c_-^2+d_{11}c_-+d_{02})-c_-(e_{20}c_-^2+e_{11}c_-+e_{02})\\
& = & (d_{20}-c_-e_{20})c_-^2+(d_{11}-c_-e_{11})c_-+d_{02}-c_-e_{02}
\eea
and
\bee
\et_{11}&=&2c_-c_+d_{20}+(c_-+c_+)d_{11}+2d_{02}-c_-(2e_{20}c_-c_++e_{11}(c_-+c_+)+2e_{02})\\
& = & 2c_-c_+(d_{20}-c_-e_{20})+(c_-+c_+)(d_{11}-c_-e_{11})+2(d_{02}-c_-e_{02})
\eee
and
\bee
\et_{02}&=&d_{20}c_+^2+d_{11}c_++d_{02}-c_-(e_{20}c_+^2+e_{11}c_++e_{02})\\
& = & (d_{20}-c_-e_{20})c_+^2+(d_{11}-c_-e_{11})c_++d_{02}-c_-e_{02}
\eee
and
\bee
\et_{21}=(3c_-^2c_+-2dc_--dc_+)-c_-\left(\frac{\ell+d-1}{\ell}(c_-^2+2c_-c_+)-3\right),
\eee
this is \eqref{ienveovnovne}.\\
Introducing the time variable $t$ (we note that $t\to +\infty$ corresponds to both $Z\uparrow\downarrow Z_2$): 
$$\frac{dt}{dx}={-}\frac{1}{\Delta}$$ yields
$$
 \frac{dY}{dt}= \left|\begin{array}{l}
\frac{(c_+-c_-)\l_+\Wt+\dt_{20}\Wt^2+\dt_{11}\Wt\Sigmat+\dt_{02}\Sigmat^2+\dt_{30}\Wt^3+\dt_{21}\Wt^2\Sigmat+\dt_{12}\Wt\Sigmat^2+\dt_{03}\Sigma^3}{c_+-c_-}\\
\frac{(c_+-c_-)\l_-\Sigmat+\et_{20}\Wt^2+\et_{11}\Wt\Sigmat+\et_{02}\Sigmat^2+\et_{30}\Wt^3+\et_{21}\Wt^2\Sigmat+\et_{12}\Wt\Sigmat^2+\et_{03}\Sigma^3}{c_+-c_-}.
\end{array}\right.
$$
this is \eqref{defgaone}, \eqref{defgaonetwo}.
\end{proof}


\subsection{Integral curves passing through $P_2$}


\begin{lemma}[Slope of the curves converging to $P_2$]\label{curvespraitapjoi04ptwo}
Assume \eqref{assumptionparameters}.  Let $c_+$, $c_-$, $A$ be given by \eqref{defslpodeplus}, \eqref{eq:defbntionofA}.  Then, 
\begin{enumerate}
\item all  integral curves of \eqref{systemedefoc}  converging to $P_2$ have slope given  either by $c_+$ or $c_-$, 
 
\item there is only one curve with slope $c_+$, while all the others curves have slope $c_-$,

\item the unique curve converging to $P_2$ with the slope $c_+$ exists on $0<\sigma<\sigma_2$ and converges to $P_4$ as $\sigma \to 0$.
\end{enumerate}
\end{lemma}

\begin{proof}
The Jacobian matrix of the autonomous system \eqref{normalizedbasisssystem} at the equilibrium $(\Wt, \Sigmat)=(0,0)$ is diagonal with two negative eigenvalues $\l_-<\l_+<0$. Thus, standard results imply that $(0,0)$ is an asymptotically stable node, and the following holds for the  trajectories converging to $(0,0)$ as $t\to +\infty$
\begin{enumerate}
\item there exists a unique  trajectory tangent to the eigenvector of the Jacobian matrix corresponding to the smallest eigenvalue $\l_-$, i.e. these trajectories satisfy
\bee
\lim_{t\to +\infty}\frac{\Wt}{\Sigmat}=0,
\eee
\item all the other trajectories are tangent to the eigenvector of the Jacobian matrix corresponding to the largest eigenvalue $\l_+$, i.e. this trajectory satisfies 
\bee
\lim_{t\to +\infty}\frac{\Sigmat}{\Wt}=0.
\eee
\end{enumerate}
Coming back to $(W, \Sigma)$, we infer that the slope of any curve converging to $P_2$ is either $c_+$ or $c_-$, and there is a unique one with slope $c_+$, while all the others have slope $c_-$. 

The fact that the unique curve converging to $P_2$ with slope $c_+$ is defined all the way to $\sigma=0$ and attracted to $P_4$ is a straightforward consequence of the phase portrait in Figure  \ref{fig:signofDeltasinphaseportrait}.
\end{proof}

\begin{lemma}
\label{prop:behavioroftheflownearP2:smooth}
Assume \eqref{assumptionparameters}. Let $c_-$, $A$ be given by \eqref{defslpodeplus}, \eqref{eq:defbntionofA} and assume that $$A=K+\alpha, \ \ K\in \Bbb N\backslash\{0,1\}, \ \ 0<\alpha<1.$$ 
Then, there exists a unique solution curve which is $\mathcal C^\infty$ at $P_2$ with slope $c_-$. 
\end{lemma}

\begin{remark}
\label{rem:reg}
One can show that all the other curves converging to $P_2$ with slope $c_-$, see Lemma \ref{curvespraitapjoi04ptwo}, are  $\mathcal{C}^{K+\alpha}$ both on the right and on the left of $P_2$. 
\end{remark}

\begin{proof} 
Recall from Lemma \ref{curvespraitapjoi04ptwo} that the curves converging to $P_2$ with slope $c_-$ satisfy $\tilde{W}=o(\Sigmat)$ near $P_2$, and hence, we have at $P_2$, i.e., at $\Wt=0$,  
\bea
\label{voingneiongegnibis}
\Psit:=\frac{\Sigmat}{\Wt}, \ \ \ \ \ \ \ \ \Psit(0)=0.
\eea
Now 
\bee
&&\Wt\frac{d\Psit}{d\Wt}+\Psit=\frac{d\Sigmat}{d\Wt}\\
&=& \frac{(c_+-c_-)\l_-\Sigmat+\et_{20}\Wt^2+\et_{11}\Wt\Sigmat+\et_{02}\Sigmat^2+\et_{30}\Wt^3+\et_{21}\Wt^2\Sigmat+\et_{12}\Wt\Sigmat^2+\et_{03}\Sigma^3}{(c_+-c_-)\l_+\Wt+\dt_{20}\Wt^2+\dt_{11}\Wt\Sigmat+\dt_{02}\Sigmat^2+\dt_{30}\Wt^3+\dt_{21}\Wt^2\Sigmat+\dt_{12}\Wt\Sigmat^2+\dt_{03}\Sigma^3}\\
& =& \frac{(c_+-c_-)\l_-\Psit+\et_{20}\Wt+\et_{11}\Wt\Psit+\et_{02}\Psit^2\Wt+\et_{30}\Wt^2+\et_{21}\Wt^2\Psit+\et_{12}\Wt^2\Psit^2+\et_{03}\Psit^3\Wt^2}{(c_+-c_-)\l_++\dt_{20}\Wt+\dt_{11}\Wt\Psit+\dt_{02}\Psit^2\Wt+\dt_{30}\Wt^2+\dt_{21}\Wt^2\Psit+\dt_{12}\Wt^2\Psi^2+\dt_{03}\Psit^3\Wt^2}
\eee
and hence 
$$\Wt\frac{d\Psit}{d\Wt}-\left[\frac{\l_-}{\l_+}-1\right]\Psit=\frac{\Wt G_1(\Wt,\Psit)}{G_2(\Wt,\Psit)}$$
with $G_2(0,0)\neq 0$.  Since $A=\l_-/\l_+$, we infer 
$$\Wt\frac{d\Psit}{d\Wt}-\left[A-1\right]\Psit=\frac{\Wt G_1(\Wt,\Psit)}{G_2(\Wt,\Psit)}.$$
Let $k\geq 1$. Assuming that $\Wt$ is $\mathcal{C}^\infty$, we differentiate $k$ times, and evaluate at $\Wt=0$. We infer
\bea\label{eq:uniqunessTaylorexpansionforPsit}
\left[k+1-A\right]\Psit^{(k)}(0)=k\left[\frac{G_1(\Wt,\Psit)}{G_2(\Wt,\Psit)}\right]^{(k-1)}_{|_{(\Wt, \Psit)=(0,0)}},\ \ k\geq 1, \ \ \ \ \ \  \Psit(0)=0.
\eea
Since $A\notin\mathbb{N}$, \eqref{eq:uniqunessTaylorexpansionforPsit} yields the uniqueness of the Taylor expansion of $\Psit$ at any order.  

Let $k\in\mathbb{N}^*$ with $k>A-2$. We denote by $P_k(\Wt)$ the unique Taylor polynomial of degree $k$ provided by solving the iteration \eqref{eq:uniqunessTaylorexpansionforPsit} for $j=1, \cdots, k$. Then, $P_k(\Wt)$ satisfies 
$$\Wt\frac{dP_k}{d\Wt}-\left[A-1\right]P_k=\frac{\Wt G_1(\Wt,P_k(\Wt))}{G_2(\Wt,P_k(\Wt))}+\Wt^{k+1}Q_k,$$
where $Q_k$ is a polynomial uniquely prescribed by $k$, $A$, $G_1$ and $G_2$. We decompose
\bee
\Psit &=& P_k+\tilde{\varepsilon}_k, \ \ \ \ \ \ \tilde{\varepsilon}_k=O(\Wt^{k+1}).
\eee
Plugging this decomposition into the equation for $\Psit$, we infer
\bee
\Wt\frac{d\tilde{\varepsilon}_k}{d\Wt}-\left[A-1\right]\tilde{\varepsilon}_k &=& \frac{\Wt G_1(\Wt,P_k(\Wt)+\tilde{\varepsilon}_k)}{G_2(\Wt,P_k(\Wt)+\tilde{\varepsilon}_k)} -\left(\Wt\frac{dP_k}{d\Wt}-\left[A-1\right]P_k\right)\\
&=& -\Wt^{k+1}Q_k+\Wt  \tilde{\varepsilon}_k H_k(\Wt, \tilde{\varepsilon}_k)
\eee
where $H_k$ is a smooth function near $(0,0)$ uniquely determined by $k$, $A$, $G_1$ and $G_2$. We deduce
\bee
\frac{d}{d\Wt}\Big(\Wt^{1-A}\tilde{\varepsilon}_k\Big) &=& -\Wt^{k+1-A}Q_k+\frac{\tilde{\varepsilon}_k}{\Wt^{A-1}}H_k(\Wt, \tilde{\varepsilon}_k).
\eee
Now, since $\tilde{\varepsilon}_k=O(\Wt^{k+1})$ and $k>A-2$, we deduce that if  a solution curve  is $\mathcal C^\infty$ at $P_2$ with slope $c_-$, then it must satisfy 
\bea\label{enioghoneneioneneoin}
\tilde{\varepsilon}_k &=& \Wt^{A-1}\int_0^{\Wt}\left(-\wt^{k+1-A}Q_{k}+\frac{\tilde{\varepsilon}_k}{\wt^{A-1}}H_k(\wt, \tilde{\varepsilon}_k)\right)d\wt.
\eea
Using a fixed point argument, there exists a unique solution to  \eqref{enioghoneneioneneoin} defined in a neighborhood of $\Wt=0$. Uniqueness follows.

We now focus on  the existence of a solution curve  $\mathcal C^\infty$ at $P_2$ with slope $c_-$. Let 
\bee
\Psit &:=& P_{\lceil A\rceil -2}+\tilde{\varepsilon}_{\lceil A\rceil -2},
\eee
where $\tilde{\varepsilon}_{\lceil A\rceil -2}$ is the unique solution to the fixed point argument fixed  \eqref{enioghoneneioneneoin} for $k=\lceil A\rceil -2>A-2$. With this definition, $\Psit$ yields 
 a solution curve   with slope $c_-$ at $P_2$, and it remains to prove that it is $\mathcal{C}^\infty$ at $P_2$. 
 In fact, it suffices to prove that 
 \bea\label{eq:remainingstufftoprovetoconculudeexistencesmoothatP2}
P_k+\tilde{\varepsilon}_k &=& P_{\lceil A\rceil -2}+\tilde{\varepsilon}_{\lceil A\rceil -2}\textrm{ for all }k\geq \lceil A\rceil -2,
\eea
as it will then imply that $\Psit$ has a Taylor expansion at $P_2$ at any order and is hence $\mathcal{C}^\infty$. 

From now on, we focus on proving \eqref{eq:remainingstufftoprovetoconculudeexistencesmoothatP2}. For $k>A-2$, we  define $\de\Psit_{k,A}$ by
\bee
\de\Psit_{k,A} &:=& P_k-P_{\lceil A\rceil -2}+\tilde{\varepsilon}_k-\tilde{\varepsilon}_{\lceil A\rceil -2}
\eee
which satisfies in view of the properties of $P_{\lceil A\rceil -2}$, $P_k$, $\tilde{\varepsilon}_{\lceil A\rceil -2}$ and $\tilde{\varepsilon}_k$
\bee
|\de\Psit_{k,A}| &=& O( \Wt^{\lceil A\rceil -1})
\eee
and
\bee
\Wt\frac{d\de\Psit_{k,A}}{d\Wt}-\left[A-1\right]\de\Psit_{k,A} &=& \Wt  \de\Psit_{k,A} J_{k, A}(\Wt, P_k, P_{\lceil A\rceil -2},\tilde{\varepsilon}_k, \tilde{\varepsilon}_{\lceil A\rceil -2})
\eee
where $J_{k, A}$ is a smooth function near $(0,0)$ uniquely prescribed by $k$, $A$, $G_1$ and $G_2$. In view of the above control  for $\de\Psit_{k,A}$, we deduce
\bee
|\de\Psit_{k,A}| &\lesssim&  \Wt^{A-1}\int_0^{\Wt}\frac{|\de\Psit_{k,A}|}{\wt^{A-1}}d\wt.
\eee
Hence $\de\Psit_{k,A}$ vanishes identically, so that  \eqref{eq:remainingstufftoprovetoconculudeexistencesmoothatP2} holds as desired.
\end{proof}


\section{Semi classical renormalization of the flow near $P_2$}
\label{semiclassical}

Our aim in this section is to start the proof of Theorem \ref{thmmain} with the renormalization of the flow \eqref{systemedefoc} for $0<\re-r\ll 1$ near $P_2$ with $\re$ given by \eqref{defrinfty}. We first detail the strategy of the proof in section \ref{thmmain}, and then proceed to the expected renormalization. The proof of Theorem \ref{thmmain} will then be completed in sections \ref{sectionlimit},  \ref{bbounds}, \ref{cinftysolution}. For the rest of this paper, we assume $$\ell\in \Bbb R^*\backslash \{d\}, \ \ 0<\re(d,\ell)-r\ll1.$$ We will also denote $$\Pe=\left|\begin{array}{l} P_5\ \ \mbox{for}\ \ \ell<d\\P_3 \ \ \mbox{for}\ \ \ell<d.\end{array}\right.$$ 

\subsection{Strategy of the proof of Theorem \ref{thmmain}}
\label{stegerguerproof}

We describe the main steps of the proof of Theorem \ref{thmmain}.\\

\noindent{\bf step 1} Renormalization. We introduce a suitable semi classical parameter, see \eqref{defbriri},  $$0<b=o_{r\uparrow \re}(1)$$ and use the geometry of the ``eye'' $P_{\hskip -.1pc\peye},P_2$ to produce a suitable renormalization. Here we use a fundamental degeneracy of the phase portrait, see Lemma \ref{phasperotrai}, which gives the eye property \be
\label{neioenenoevno}
\lim_{r\uparrow \re}|\Pe(r)-P_2(r)|=0,
\ee
{\em and} the slopes of $w_2(\sigma)$ and $w_2^-(\sigma)$ converge to the same value, see figure \ref{fig:signofDeltasinphaseportrait}, \ref{fig:signofDeltasinphaseportraitbis}. Passing to the diagonalized variables \eqref{defp} and after explicit suitable reductions, this degeneracy leads to a quadratic cancellation, \eqref{fundvonnnrelation}. Solving for $w(\sigma)$, we are left with the study of a problem of the form:
 \be
 \label{nneionienoene}
 u(1-u)\Theta' -\left[\gamma-2+(\nu_b+3)u\right]\Theta=-\frac{\mathcal G}{b}.
 \ee
 Here we introduced the parameters which appear in the renormalization process
 \be\label{gara}
 \left|\begin{array}{l}
 \gamma=\frac{c_\infty(d,\ell)(1+o_{b\to 0}(1))}{b}, \ \ c_\infty>0\\
 \nu_b=\nu_\infty(d,\ell)+o_{b\to 0}(1), \ \ \nu_\infty\neq 0
\end{array}\right.
\ee
The trajectory is $\Theta(u)$ with $u=0$ at $\sigma_2$, $u=1$ at $\sigma_{\hskip -.1pc\peye}$.  $\mathcal G$ is an explicit nonlinear term. An analogous reduction can be performed to the right of $P_2$.\\
  
\noindent{\bf step 2} Main part of the solution. The nonlinear ode \eqref{nneionienoene} has a regular singular point at the origin. Therefore, it admits a $\mathcal C^\infty$ solution with an holomorphic expansion at the origin $$\Theta(u)=\sum_{k=0}^{+\infty}\theta_k(b,d,\ell) u^k$$ where $\theta_k(b,d,\ell)$ is given by an explicit $b$ dependent induction relation. We let 
\be
\label{veiovneoneenoe}
\gamma-1=K+\alpha_\gamma, \ \ 0<\alpha_\gamma<1
\ee
and truncate the holomorphic expansion at the critical frequency\footnote{w®hich corresponds to the limit of regularity of a generic solutions, see Lemma \ref{prop:behavioroftheflownearP2:smooth}.}:
\be
\label{fmeneneneoe}
\Theta(u) =\sum_{k=0}^{K-2}\theta_k u^k+\theta_{K-1}u^{K-1}+(-1)^{K-1}S_{K-1}\Theta_{\rm main}(u)-\T(r_\mathcal G)
\ee 
where $\Theta_{\rm main}=O(u^K)$ is an explicit integral, and $\T(r_\mathcal G)$ is a remainder which is of higher order. Our first fundamental observation is that there exists a strong limit 
\be
\label{limitexists}
\lim_{b\to 0}S_{K-1}= S_\infty(d,\ell).
\ee
The proof relies on bounding the formal series solution to a limiting problem with $b=0$ first. This is done in
 Proposition \ref{proposnvoivone}, which belongs to the realm of nonlinear Maillet theorems, \cite{malgrange, sibuya}. 
 The original problem \eqref{nneionienoene} can be thought of as a $b$-deformation of the limiting problem. 
 The challenge however is that we need uniform estimates for all frequencies up to the critical value $K$ {\em which itself is of size $\sim \frac 1b$}\\

\noindent{\bf step 3} Non vanishing of $S_\infty(d,\ell)$. The proof of finiteness of $S_\infty(d,\ell)$ implies the analycity of the mapping $\ell\to S_\infty(d,\ell)$ away from the critical point i.e., $\ell<d$. This number can be reexpressed as an explicit normally convergent series, but we do not know how to prove analytically that it is non zero. We therefore perform a numerical study of this convergent series which allows us to provide windows of parameters $(d,\ell)$ for which 
\be
\label{veioeonveonoe}
S_\infty(d,\ell)\ne 0.
\ee

\noindent{\bf step 4} No oscillation at the left of $P_2$. In the variable $\Theta$, it is easily seen that the $P_2-P_{\hskip -.1pc\peye}$ separatrix satisfies $|\Theta|\lesssim_{\ell,d} 1$, Lemma \ref{lemmaseprpatira}. Hence we pick a large enough (in absolute value) constant $\Theta^*\gg1 $ and aim at reaching the value 
\be
\label{ceioneoneon}
\Theta(u^*)=\Theta^*, \ \ 0<u^*<1.
\ee 
The second fundamental observation is that the function $\Theta_{\rm main}$ can be analyzed explicitly near $u=0$ 
\bea
\label{formalurlajrioje}
\Theta_{\rm main}&=& \Gamma(\alpha_\gamma)\Gamma(1-\alpha_\gamma)K^{\nu_b+3-\alpha_\gamma}u^{K-1}\\
\nonumber &\times & \left\{\left[1+o_{b\to 0}(1)\right]\left[\frac{1}{\Gamma(1-\alpha_\gamma)}+\frac{K+\nu_b+2}{\Gamma(2-\alpha_\gamma)}u\right]+{\rm lot}\right\}.
\eea
In a boundary layer close to the integer values
\be
\label{neoineneoen}
\alpha_\gamma\in (\e_1,2\e_1)\cup (1-2\e_2,1-\e_2), \ \ \e_i=o_{b\to 0}(1),
\ee 
we can ensure that \eqref{ceioneoneon} happens for a small $0<u^*(\alpha_\gamma)\ll 1$.  For the  interval $2\e_1<\alpha_\gamma<1-2\e_2$, we need to understand $\Theta_{\rm main}$ away from $u=0$ where the truncated Taylor expansion no longer dominates, and here we use the {\em explicit integral representation of $\Theta_{\rm main}$} to show that \eqref{ceioneoneon} happens for $u^*(\alpha_\gamma)<\frac 12$.
The conclusion is that for every $\alpha_\gamma\in(0,1)$ except maybe a very small $b$ dependent boundary layer around the integer values, the solution to \eqref{fmeneneneoe} reaches \eqref{ceioneoneon} in time $0<u^*<\frac 12$. This means that we are leaving a large neighborhood around the $P_2-P_5$ separatrix with a prescribed sign $\Theta^*\gg 1$. A further use of monotonicity properties of the flow \eqref{systemedefoc} allows us to conclude that the integral curve will intersect 
either the root branch $w_2(\sigma)$ or $w_2^-(\sigma)$ for some $\sigma_5<\sigma^*<\sigma_2$. 
The $(-1)^K$ prefactor in \eqref{fmeneneneoe} dictates that the former happens when $K$ is even while the latter holds 
when $K$ is odd (if $S_\infty>0$ and the other way around if $S_\infty<0$.) Once the trajectory reaches $w_2(\sigma)$, by Lemma \ref{lemmaconnection}, it then continues on to $P_4$, as desired.\\

\noindent{\bf step 5} Oscillations at the right of $P_2$. The analysis of the flow to the right of $P_2$  produces the same decomposition \eqref{formalurlajrioje} but with $u<0$. We then observe since $\lim_{\alpha_\gamma\uparrow 1} \Gamma(1-\alpha_\gamma)=+\infty $ that by choosing $\alpha_\gamma$ in a boundary layer close to respectively $0$ or $1$, the sign $u<0$ allows us to reach 
$$\Theta(u^*)= \left|\begin{array}{l}
\Theta^*\ \ \mbox{for}\ \ \e_1<\alpha_\gamma<2\e_1\\
-\Theta^*\ \ \mbox{for}\ \ \e_2<1-\alpha_\gamma<2\e_2
\end{array}\right., \ \ u^*<0, \ \ |u^*|\ll1.$$ 
Then, if $K$ is even, the curve will exist through $w_2(\sigma)$ in the first case, and through the branch 
$w_2^-(\sigma)$ in the second case. For odd $K$ the situation is reversed (Again, this holds for $S_\infty>0$.
For $S_\infty<0$ the picture is reversed.) \\

\noindent{\bf step 6} Conclusion by continuity. Given $K$ with the suitable parity (depending on the sign of $S_\infty(d,\ell)\ne 0$), we vary the parameter $\alpha_\gamma\in(\e_1,1-\e_2)$ continuously and conclude that the $\mathcal C^\infty$ curve through $P_2$ crosses the red at the left of $P_2$ for all $\alpha_\gamma\in(e_1,1-\e_2)$, while it also crosses the red at the right of $P_2$  at the beginning of the $\alpha_\gamma$ interval, green at the end. Hence an elementary continuity argument implies the existence of at least one value $\alpha_\gamma^*\in(\e_1,1-\e_2)$ such that the solution curve intersects the $P_2-P_6$ trajectory given by Lemma \ref{lememarompfoe}. The two curves intersect each other away 
from singular points and thus, by uniqueness, must coincide.  It follows easily that this constructed solution satisfies the conclusions of Theorem \ref{thmmain}. In other words, as long as the non degeneracy condition \eqref{veioeonveonoe} is satisfied, the integer interval $\gamma\in [K+1,K+2]$, with $K$ given by \eqref{veiovneoneenoe} large enough and of suitable parity, contains at least one $\mathcal C^\infty$ $P_6-P_4$ solution. Since $\gamma$ is related to $b$ through 
\eqref{gara} and $b==o_{r\uparrow \re}(1)$, the above construction produces an infinite family of global $\mathcal C^\infty$ solutions 
parametrized by the speeds $r_k\uparrow \re(d,\ell)$ for each $(d,\ell)$ such that $S_\infty(d,\ell)\ne 0$.


\subsection{Degeneracy of the geometry at $\re(d,\ell)$}


We use Lemma \ref{rehivnioeei} and the $Y$ variable \eqref{varibakey} to map \eqref{systemedefoc} onto \eqref{normalizedbasisssystem}. The starting point of the renormalization procedure is the following fundamental degeneracy property as $r\uparrow \re(d,\ell)$.

\begin{remark}[Notation for the parameters]
\label{renkonvnvoine}
 From now on and for the rest of this paper we adopt the following notation: all slopes, characteristic eigenvalues and geometrical parameters involved in the renormalization of the flow, Lemma \ref{rehivnioeei}, depend of $r$, and will be noted with an $^\infty$ subscript when evaluated at $r=\re$. The non degeneracy and signs of some of these limiting values will be crucial in the forthcoming analysis, and all relevant values are collected in Appendix \ref{appendixconstants}.
\end{remark}

\begin{lemma}[Degeneracy in the diagonalized system]
Let 
\be
\label{defbriri}
b=\left|\begin{array}{l}
r^*-r\ \ \mbox{for}\ \ \ell<d\\
\sqrt{r_+-r}\ \ \mbox{for}\ \ \ell>d
\end{array}\right.
\ee
and define 
\be
\label{dnwineineo}
\mu_+=\frac{\l_+}{b}
\ee
then
\be
\label{fundvonnnrelation}
\left|\begin{array}{l}
\Wte=-b\frac{(c_+-c_-)\mu_+}{\dt_{20}}+O(b^2)\\
\Sigmate=-\frac{\et_{20}\Wte^2}{(c_+-c_-)\l_-}+O(b^3)\\
\mu_+=\mu_+^\infty+O(b)
\end{array}\right.
\ee
where the non degenerate limiting values are computed in Appendix \ref{appendixconstants}.
\end{lemma}

\begin{proof} The value of $\sigma_2(r)$ is computed from \eqref{defjevknl} and hence $\sigma_2\in \mathcal C^\infty(1,r_+)$. Moreover, $J(r)$ is from \eqref{eq:J} a second order polynomial with roots $r_+=1+\frac{d-1}{(1+\sqrt{\ell})^2}<r_-=1+\frac{d-1}{(1-\sqrt{\ell})^2}$ and hence the root $r_+$ is simple. Since $r^*<r_+$, we conclude that with the definition \eqref{defbriri}, $\sigma_2(r)$ and the slopes coefficients $c_i(r)$ given by \eqref{defvalueci} are smooth functions of $b$ on $[0,b^*]$, $0<b^*(d,\ell)\ll1$ universal small enough. We now explicitly check that the determinant $(c_1-c_4)^2+4c_2c_3$ which appears in the definition of the slopes and eigenfunctions \eqref{defslpodeplus}, \eqref{deflplus} is non degenerate at $\re$, see limiting values in Appendix \ref{appendixconstants} and the non degeneracy of $\l_-$, which ensures that $c_\pm,\l_\pm$ are smooth functions of $b$ all the way to $b=0$. The eye property \eqref{neioenenoevno} thus implies $$|W_{\hskip -.1pc\peye}|+|\Sigma_{{\hskip -.1pc\peye}}|=|w_{\hskip -.1pc\peye}-w_2|+|\sigma_{\rm{eye}}-\sigma_2|\lesssim Cb.$$
Since the coefficients of the matrix $P^{-1}$ given by \eqref{defp} do not degenerate at $\re$ from direct check, we conclude
$$|\Wte|+|\Sigmate|\le Cb.$$
The coefficients $\dt_{ij},\et_{ij}$ of the polynomials of the RHS of \eqref{normalizedbasisssystem} are computed from \eqref{ienveovnovne} and are $O(1)$ at $\re$. Moreover $\matchal G_1, \mathcal G_2$  vanish at $P_{\hskip -.1pc\peye}$ from \eqref{defgaone}, \eqref{defgaonetwo}, and this forces \eqref{fundvonnnrelation}.
\end{proof}


\subsection{Renormalization}  


We now proceed to the renormalization of \eqref{normalizedbasisssystem} for $0<b\ll1$.

\begin{lemma}[Renormalization and quasilinear formulation]
\label{lemmarenriannowr}
Let 
\be
\label{changevariables}
\left|\begin{array}{l}
\Wt=-b\wt\\
\Sigmat=b^2\sigmat
\end{array}\right., \ \ 
\left|\begin{array}{l}
\psit=\frac{\sigmat}{\wt}=\psit_{\hskip -.1pc\peye} \phi\\
\wt=\wt_{\hskip -.1pc\peye} u
\end{array}\right., \ \ \phi(u)=u+(1-u)\Psi(u)
\ee
then \eqref{normalizedbasisssystem} is mapped to the quasilinear problem
\bea
\label{equaitoncompee}
 &&\left[1+H_2+G_2\Psi+\NLt_2\right]u(1-u)\frac{d\Psi}{du}\\
\nonumber&+& \left[(1-2u)(1+H_2+G_2\Psi+\NLt_2)-\gamma(1+G_1)+2uG_2\right]\Psi\\
\nonumber& = &u\left[\gamma bH_1-2(1+H_2)+\frac{\gamma b\NLt_1}{x}-2\NLt_2\right]
\eea
where $H_1,H_2,G_1,G_2$ are explicit polynomials in $(b,u)$ given by \eqref{decopmositionhtwo}, \eqref{polynomialsgi}, and the nonlinear terms $(\NLt_{i})_{i=1,2}$ are given by \eqref{nolinearterms}.
Moreover, 
\be
\label{valuelimitsfhihs}
\left|\begin{array}{l}
\wt_{\hskip -.1pc\peye}=\frac{(c^\infty_+-c^\infty_-)\mu_+^\infty}{\dt_{20}^\infty}+O(b)\\
\psit_{\hskip -.1pc\peye}=-\frac{\et^\infty_{20}\mu^\infty_+}{\dt^\infty_{20}\l^\infty_-}+O(b).
\end{array}\right.
\ee
\end{lemma}

\begin{remark} Unfortunately, we need to keep track of {\em all} terms in \eqref{equaitoncompee} since they will create the limiting problem which, in turn, will give rise to the $S_\infty(d,\ell)$ function, evaluated numerically.
\end{remark}

\begin{remark} In the quasilinear formulation \eqref{changevariables}, \eqref{equaitoncompee}, $u=0$ is $P_2$ and $u=1$ is $P_{\hskip -.1pc\peye}$.
\end{remark}

\begin{proof}[Proof of Lemma \ref{lemmarenriannowr}] This is a brute force computation.\\

\noindent{\bf step 1} $b$ renormalization. We renormalize \eqref{normalizedbasisssystem}:
$$\left|\begin{array}{l}
\Wt=-b\wt\\
\Sigmat=b^2\sigmat\\
\l_+=b\mu_+\\
\frac{d\tau}{dt}=b\\
\end{array}\right.$$ 
and define the variable 
\be
\label{cneionevoneoenv}
\psit=\frac{\sigmat}{\wt}.
\ee
We compute from \eqref{fundvonnnrelation}, \eqref{fpormoaurmalf}, \eqref{fpormoaurmalfbis} the expansion as $b\to 0$:
$$
\left|\begin{array}{l}
\mu_+(r)=\frac{\l_+(r)}{b}=\mu_+^\infty+O(b)<0\\
\wte(r)=-\frac{\Wte(r)}{b}=\frac{(c^\infty_+-c^\infty_-)\mu_+^\infty}{\dt_{20}^\infty}+O(b)\\
\psit_{\hskip -.1pc\peye}=\frac{\sigmat_{\hskip -.1pc\peye}}{\wt_{\hskip -.1pc\peye}}=-\frac{\Sigmat_{\hskip -.1pc\peye}}{b^2\frac{\Wte}{b}}=\frac{\et_{20}\Wt_{\hskip -.1pc\peye}}{b(c_+-c_-)\l_-}=-\frac{\et^\infty_{20}\mu^\infty_+}{\dt^\infty_{20}\l^\infty_-}+O(b)
\end{array}\right.
$$
and \eqref{valuelimitsfhihs} is proved. We now compute the flow from \eqref{normalizedbasisssystem}.\\
\noindent\underline{First equation}. We compute \bee
&&\frac{\mathcal G_1}{c_+-c_-}=-b^2\frac{d\wt}{d\tau}\\
&=&\frac{-b^2(c_+-c_-)\mu_+\wt+b^2\dt_{20}\wt^2-b^3\dt_{11}\wt\sigmat+b^4\dt_{02}\sigmat^2-b^3\dt_{30}\wt^3+b^4\dt_{21}\wt^2\sigmat-b^5\dt_{12}\wt\sigmat^2+b^6\dt_{03}\sigmat^3}{c_+-c_-}
\eee
i.e.
\bee
\frac{d\wt}{d\tau}=\frac{-(c_+-c_-)|\mu_+|\wt+|\dt_{20}|\wt^2+b(\dt_{11}\wt\sigmat+\dt_{30}\wt^3)-b^2(\dt_{02}\sigmat^2+\dt_{21}\wt^2\sigmat)+b^3\dt_{12}\wt\sigmat^2-b^4\dt_{03}\sigmat^3}{c_+-c_-}.
\eee
We insert \eqref{cneionevoneoenv} and compute:
\bea
\label{lineftwogone}
  &&\mathcal F_2=-\frac{\mathcal G_1}{b^2}\\
\nonumber   &=&-(c_+-c_-)|\mu_+|\wt+|\dt_{20}|\wt^2+b(\dt_{11}\wt\sigmat+\dt_{30}\wt^3)-b^2(\dt_{02}\sigmat^2+\dt_{21}\wt^2\sigmat)+b^3\dt_{12}\wt\sigmat^2-b^4\dt_{03}\sigmat^3\\
\nonumber   & = &\wt\left[-(c_+-c_-)|\mu_+|+|\dt_{20}|\wt+b(\dt_{11}\wt\psit+\dt_{30}\wt^2)-b^2(\dt_{02}\psit^2\wt+\dt_{21}\wt^2\psit)\right.\\
 \nonumber  &+& \left.b^3\dt_{12}\wt^2\psit^2-b^4\dt_{03}\wt^2\psit^3\right]\\
 \nonumber  & = & \wt\left\{-(c_+-c_-)|\mu_+|+\left[|\dt_{20}|+b\dt_{11}\psit-b^2\dt_{02}\psit^2\right]\wt+\left[b\dt_{30}-b^2\dt_{21}\psit+b^3\dt_{12}\psit^2-b^4\dt_{03}\psit^3\right]\wt^2\right\}
  \eea
\noindent\underline{Second equation}. We compute
\bee
&&\frac{\mathcal G_2}{c_+-c_-}=b^3\frac{d\sigmat}{d\tau}\\
& = & \frac{b^2(c_+-c_-)\l_-\sigmat+b^2\et_{20}\wt^2-b^3\et_{11}\wt\sigmat+b^4\et_{02}\sigmat^2-b^3\et_{30}\wt^3+b^4\et_{21}\wt^2\sigmat-b^5\et_{12}\wt\sigmat^2+b^6\et_{03}\sigma^3}{c_+-c_-}.
\eee
and hence
\bee
&&b\frac{d\sigmat}{d\tau}=\frac{1}{c_+-c_-}\\
&\times&\left[-(c_+-c_-)|\l_-|\sigmat+\et_{20}\wt^2-b(\et_{11}\wt\sigmat+\et_{30}\wt^3)+b^2(\et_{02}\sigmat^2+\et_{21}\wt^2\sigmat)-b^3\et _{12}\wt\sigmat^2+ b^4\et_{03}\sigmat^3\right] 
\eee
and
\bea
\label{defttoneone}
&&\mathcal F_1=\frac{\mathcal G_2}{b^2}\\
\nonumber & = & -(c_+-c_-)|\l_-|\sigmat+\et_{20}\wt^2-b(\et_{11}\wt\sigmat+\et_{30}\wt^3)+b^2(\et_{02}\sigmat^2+\et_{21}\wt^2\sigmat)-b^3\et _{12}\wt\sigmat^2+ b^4\et_{03}\sigmat^3\\
\nonumber& = &\wt\left\{-(c_+-c_-)|\l_-|\psit+\et_{20}\wt-b(\et_{11}\wt\psit+\et_{30}\wt^2)+b^2(\et_{02}\wt\psit^2+\et_{21}\wt^2\psit)-b^3\et _{12}\wt^2\psit^2\right.\\
\nonumber&+&\left.b^4\et_{03}\wt^2\psit^3)\right\}\\
\nonumber& = &\wt\left\{ -(c_+-c_-)|\l_-|\psit+\left[\et_{20}-b\et_{11}\psit+b^2\et_{02}\psit^2\right]\wt+\left[-b\et_{30}+b^2\et_{21}\psit-b^3\et_{12}\psit^2+b^4\et_{03}\psit^3\right]\wt^2\right\}.
\eea

\noindent\underline{Conclusion}. We have obtained the system

\bea
\label{vneiovneineonenoeenenslals}
&&\frac{d\wt}{d\tau}=\frac{\mathcal F_2}{c_+-c_-}\\
\nonumber & = & \frac{\wt\left\{-(c_+-c_-)|\mu_+|+\left[|\dt_{20}|+b\dt_{11}\psit-b^2\dt_{02}\psit^2\right]\wt+\left[b\dt_{30}-b^2\dt_{21}\psit+b^3\dt_{12}\psit^2-b^4\dt_{03}\psit^3\right]\wt^2\right\}}{c_+-c_-}
\eea 
and
\bea
\label{cenocneno3n3o}
&&b\frac{d\sigmat}{d\tau}=\frac{\mathcal F_1}{c_+-c_-}\\
\nonumber & = & \frac{\wt\left\{ -(c_+-c_-)|\l_-|\psit+\left[\et_{20}-b\et_{11}\psit+b^2\et_{02}\psit^2\right]\wt+\left[-b\et_{30}+b^2\et_{21}\psit-b^3\et_{12}\psit^2+b^4\et_{03}\psit^3\right]\wt^2\right\}}{c_+-c_-}.
\eea

\begin{remark} Note that  from \eqref{fundvonnnrelation} $$\psite=-\frac{\Sigmate}{b\Wte}=O_{b\to 0}(1)$$  \eqref{cenocneno3n3o} then forces the relations for $b=0$: 
\be
\label{nioneinevioohve}
\left|\begin{array}{l}
\psite^\infty=\frac{\et^\infty_{20}}{(c^\infty_+-c^\infty_-)|\l^\infty_-|}\wte^\infty\\
|\mu^\infty_+|=\frac{|\dt^\infty_{20}|}{c^\infty_+-c^\infty_-}\wte^\infty.
\end{array}\right.
\ee
{Note also that we have the signs, valid for all $d\geq 2$ and $\ell\neq d$:
$$\dt^\infty_{20}<0, \ \ \et^\infty_{20}>0,$$
see \eqref{signofdt20andet20}.}
\end{remark}

\noindent{\bf step 2} Normalization of $\wt$. We further renormalize the flow to obtain the leading order size 1 constants leading the nonlinear dynamics as $b\to 0$. Let 
\be
\label{mepehjoenvoenoe}
\wt=\wte u, \ \ \psit=\psite\phi
\ee 
and define 
\be
\label{defgamma}
{\gamma=\frac{|\l_-|}{|\l_+|}=\frac{|\l_-|}{b|\mu_+|}}
\ee 
then recalling \eqref{vneiovneineonenoeenenslals}
\bea
\label{vnlkvenvnoenoenenv}
&&\frac{du}{d\tau}=\frac{1}{\wte}\frac{d\wt}{d\tau}=\frac{\mathcal F_2}{\wte(c_+-c_-)}=|\mu_+|u\\
\nonumber &\times & \left[-1+\frac{|\dt_{20}|+b\dt_{11}\psite\phi-b^2\dt_{02}\psite^2\phi^2}{|\mu_+|(c_+-c_-)}\wte u\right.\\
&+& \left.\frac{b\dt_{30}-b^2\dt_{21}\psite\phi+b^3\dt_{12}\psite^2\phi^2-b^4\dt_{03}\psite^3\phi^3}{|\mu_+|(c_+-c_-)}\wte^2u^2\right]\\
\nonumber & = & |\mu_+|u\\
\nonumber &\times & \left[-1+\left[|\Dt_{20}|+b\Dt_{11}\phi-b^2\Dt_{02}\phi^2\right]u+b\left[\Dt_{30}-b\Dt_{21}\phi+b^2\Dt_{12}\phi^2-b^3\Dt_{03}\phi^3\right]u^2\right]
\eea
with $\Dt_{ij}$ given by \eqref{defdtij}. Observe that by definition of $P_{\hskip -.1pc\peye}$: 
$$-1+|\Dt_{20}|+b\Dt_{11}-b^2\Dt_{02}+b\Dt_{30}-b^2\Dt_{21}+b^3\Dt_{12}-b^4\Dt_{03}=0.$$
We then compute 
\bee
&&\frac{d\psit}{d\tau}=\frac{1}{\wt}\frac{d\sigmat}{d\tau}-\frac{\sigmat}{\wt^2}\frac{d\wt}{d\tau}=\frac{1}{\wt}\frac{d\sigmat}{d\tau}-\psit\frac{1}{\wt}\frac{d\wt}{d\tau}\\
&\Leftrightarrow& \frac{d\phi}{d\tau}+\frac{1}{u}\frac{du}{d\tau}\phi=\frac{1}{\psite\wt}\frac{d\sigmat}{d\tau}
\eee
and recalling \eqref{cenocneno3n3o}:
\bea
\label{veniovneoenonevenv}
&&\frac{b}{\wt}\frac{d\sigmat}{d\tau}=\frac{\mathcal F_1}{(c_+-c_-)\wt}=|\l_-|\psite\\
\nonumber &\times& \left[ -\phi+\frac{\et_{20}-b\et_{11}\psite\phi+b^2\et_{02}\psite^2\phi^2}{(c_+-c_-)|\l_-|\psite}\wte u\right.\\
&+& \left.\frac{-b\et_{30}+b^2\et_{21}\psite\phi-b^3\et_{12}\psite^2\phi^2+b^4\et_{03}\psite^3\phi^3}{(c_+-c_-)|\l_-|\psite}\wte^2u^2\right]\\
\nonumber &\Leftrightarrow&\frac{1}{|\mu_+|\psite\gamma}\frac{1}{\wt}\frac{d\sigmat}{d\tau}\\
\nonumber & = & \left[ -\phi+\left(\Et_{20}-b\Et_{11}\phi+b^2\Et_{02})\phi^2\right)u+\left(-b\Et_{30}+b^2\Et_{21}\phi-b^3\Et_{12}\phi^2+b^4\Et_{03}\phi^3\right)u^2\right]
\eea
with $\Et_{ij}$ given by \eqref{defdtij}, and again by definition of $P_{\hskip -.1pc\peye}$:
$$ -1+\Et_{20}-b\Et_{11}+b^2\Et_{02}-b\Et_{30}+b^2\Et_{21}-b^3\Et_{12}+b^4\Et_{03}=0.$$ This yields the renormalized $\phi$ equation
\bee
&&\frac{d\phi}{d\tau}+\frac{1}{u}\frac{du}{d\tau}\phi\\
& = & |\mu_+|\gamma\left[ -\phi+\left(\Et_{20}-b\Et_{11}\phi+b^2\Et_{02}\phi^2\right)u+b\left(-\Et_{30}+b\Et_{21}\phi-b^2\Et_{12}\phi^2+b^3\Et_{03}\phi^3\right)u^2\right].
\eee

\noindent{\bf step 3} Quasilinear formulation. Let $$\Lambda=Z\frac{d}{dZ}=-\frac{1}{|\mu_+|}\frac{d}{d\tau},$$ then equivalently:
$$
 \left|\begin{array}{l}
 \Lambda u=u\left[1-\left(|\Dt_{20}|+b\Dt_{11}\phi-b^2\Dt_{02}\phi^2\right)u-b\left(\Dt_{30}-b\Dt_{21}\phi+b^2\Dt_{12}\phi^2-b^3\Dt_{03}\phi^3\right)u^2\right]\\
\Lambda \phi+\frac{\Lambda u}{u}\phi= \gamma\left[ \phi-\left(\Et_{20}-b\Et_{11}\phi+b^2\Et_{02}\phi^2\right)u-b\left(-\Et_{30}+b\Et_{21}\phi-b^2\Et_{12}\phi^2+b^3\Et_{03}\phi^3\right)u^2\right].
\end{array}\right.
$$
with the relation on the parameters:
\be
\label{relationparameters}
 \left|\begin{array}{l}
 -1+|\Dt_{20}|+b\Dt_{11}-b^2\Dt_{02}+b\Dt_{30}-b^2\Dt_{21}+b^3\Dt_{12}-b^4\Dt_{03}=0\\
 -1+\Et_{20}-b\Et_{11}+b^2\Et_{02}-b\Et_{30}+b^2\Et_{21}-b^3\Et_{12}+b^4\Et_{03}=0.
 \end{array}\right.
\ee
Let 
$$\left|\begin{array}{l}
F_1(u,\phi)= \phi-\left(\Et_{20}-b\Et_{11}\phi+b^2\Et_{02}\phi^2\right)u-b\left(-\Et_{30}+b\Et_{21}\phi-b^2\Et_{12}\phi^2+b^3\Et_{03}\phi^3\right)u^2\\
F_2(u,\phi)=1-\left(|\Dt_{20}|+b\Dt_{11}\phi-b^2\Dt_{02}\phi^2\right)u-b\left(\Dt_{30}-b\Dt_{21}\phi+b^2\Dt_{12}\phi^2-b^3\Dt_{03}\phi^3\right)u^2
\end{array}\right.
$$
then this is
\be
\label{semilinearformulation}
 \left|\begin{array}{l}
 \Lambda u=u F_2(u,\phi)\\
 \Lambda \phi+\frac{\Lambda u}{u}\phi=\gamma F_1(u,\phi).
  \end{array}\right.
 \ee
 We have from \eqref{veniovneoenonevenv}, \eqref{vnlkvenvnoenoenenv}:
 \be
 \label{liknkfones}
 \left|\begin{array}{l}
 \mathcal F_1=|\l_-|\psite\wte(c_+-c_-)u (-F_1)\\
 \mathcal F_2=\wte(c_+-c_-)|\mu_+|u(-F_2)
 \end{array}\right.
 \ee
Then $$\Lambda \phi=Z\frac{d\phi}{dZ}=Z\frac{d\phi}{du}\frac{du}{dZ}=\Lambda u \frac{d\phi}{du}$$ and hence the $\phi(u)$ renormalized quasi linear formulation
\be
\label{quaslinearformulation}
\left|\begin{array}{l}\frac{d\phi}{du}+\frac{\phi}{u}=\frac{\gamma F_1(\phi,u)}{uF_2(\phi,u)}\\
\lim_{u\to 0}\phi=0, \ \ \lim_{u\to 1}\phi=1.
\end{array}\right.
\ee

\noindent\underline{Reexpression of the nonlinear terms.} From \eqref{relationparameters}:
$$\Et_{20}=1+b(\Et_{11}+\Et_{30})-b^2(\Et_{02}+\Et_{21})+b^3\Et_{12}-b^4\Et_{03}$$
Then,
\bee
F_1(u,\phi)&=&\phi-\left(1+b(\Et_{11}+\Et_{30})-b^2(\Et_{02}+\Et_{21})+b^3\Et_{12}-b^4\Et_{03}-b\Et_{11}\phi+b^2\Et_{02}\phi^2\right)u\\
&-& b\left(-\Et_{30}+b\Et_{21}\phi-b^2\Et_{12}\phi^2+b^3\Et_{03}\phi^3\right)u^2\\
& = & \phi-u+b\left[-(\Et_{11}+\Et_{30})u+\Et_{11}\phi u +\Et_{30}u^2\right]+b^2\left[(\Et_{02}+\Et_{21})u-\Et_{02}\phi^2u-\Et_{21}\phi u^2\right]\\
& + & b^3\left[-\Et_{12}u+\Et_{12}\phi^2u^2\right]+b^4\left[\Et_{03}u-\Et_{03}\phi^3u^2\right].
\eee
Similarly,
$$|\Dt_{20}|=1-b(\Dt_{11}+\Dt_{30})+b^2(\Dt_{02}+\Dt_{21})-b^3\Dt_{12}+b^4\Dt_{03}=0$$ and 
\bee
F_2(u,\phi)&=&1-\left(1-b(\Dt_{11}+\Dt_{30})+b^2(\Dt_{02}+\Dt_{21})-b^3\Dt_{12}+b^4\Dt_{03}+b\Dt_{11}\phi-b^2\Dt_{02}\phi^2\right)u\\
&-& b\left(\Dt_{30}-b\Dt_{21}\phi+b^2\Dt_{12}\phi^2-b^3\Dt_{03}\phi^3\right)u^2\\
& = & 1-u+b\left[(\Dt_{11}+\Dt_{30})u-\Dt_{11}\phi u-\Dt_{30}u^2\right]+b^2\left[-(\Dt_{02}+\Dt_{21})u+\Dt_{02}\phi^2u+\Dt_{21}\phi u^2\right]\\
& + & b^3\left[\Dt_{12}u-\Dt_{12}\phi^2u^2\right]+b^4\left[-\Dt_{03}u+\Dt_{03}\phi^3u^2\right].
\eee

\noindent{\bf step 5} Changing variables. We change variables to make the critical points of the ode appear explicitly. Let  \be
\label{vnoivenineonvenven}
\Phit=\phi-u.
\ee
\noindent\underline{Reexpressing $F_1$.} Recall  
\bee
F_1(u,\phi)&=& \phi-u+b\left[-(\Et_{11}+\Et_{30})u+\Et_{11}\phi u +\Et_{30}u^2\right]+b^2\left[(\Et_{02}+\Et_{21})u-\Et_{02}\phi^2u-\Et_{21}\phi u^2\right]\\
& + & b^3\left[-\Et_{12}u+\Et_{12}\phi^2u^2\right]+b^4\left[\Et_{03}u-\Et_{03}\phi^3u^2\right]
\eee
then
$$
-(\Et_{11}+\Et_{30})u+\Et_{11}\phi u +\Et_{30}u^2=-(\Et_{11}+\Et_{30})u(1-u)+\Et_{11}u\Phit
$$
and
\bee
&&(\Et_{02}+\Et_{21})u-\Et_{02}\phi^2u-\Et_{21}\phi u^2=(\Et_{02}+\Et_{21})u-\Et_{02}(u+\Phit)^2u-\Et_{21}u^2(u+\Phit)\\
& = & (\Et_{02}+\Et_{21})u(1-u)(1+u)-u^2(2\Et_{02}+\Et_{21})\Phit-\Et_{02}u\Phit^2
\eee
and
\bee
&&-\Et_{12}u+\Et_{12}\phi^2u^2=-\Et_{12}u+\Et_{12}(u+\Phit)^2u^2\\
& = & -\Et_{12}u(1-u)(1+u+u^2)+2\Et_{12}u^3\Phit+\Et_{12}u^2\Phit^2
\eee
and
\bee
&&\Et_{03}u-\Et_{03}\phi^3u^2=\Et_{03}u-\Et_{03}(u+\Phit)^3u^2\\
& = & \Et_{03}u(1-u)(1+u+u^2+u^3)-\Et_{03}(3u^4\Phit+3u^3\Phit^2+u^2\Phit^3)
\eee
Thus,
\bea
\label{seconexpessiongofone}
\nonumber && F_1(u,\phi)=\Phit+b\left[-(\Et_{11}+\Et_{30})u(1-u)+\Et_{11}u\Phit \right]\\
\nonumber&+& b^2\left[(\Et_{02}+\Et_{21})u(1-u)(1+u)-u^2(2\Et_{02}+\Et_{21})\Phit-\Et_{02}u\Phit^2\right]\\
\nonumber& + & b^3\left[-\Et_{12}u(1-u)(1+u+u^2)+2\Et_{12}u^3\Phit+\Et_{12}u^2\Phit^2\right] \\
\nonumber& + & b^4\left[ \Et_{03}u(1-u)(1+u+u^2+u^3)-\Et_{03}(3u^4\Phit+3u^3\Phit^2+u^2\Phit^3)\right]\\
\nonumber& = & u(1-u)\left[-b(\Et_{11}+\Et_{30})+b^2(\Et_{02}+\Et_{21})(1+u)-b^3\Et_{12}(1+u+u^2)+b^4\Et_{03}(1+u+u^2+u^3)\right]\\
\nonumber& + & \Phit\left[1+\Et_{11}bu-(2\Et_{02}+\Et_{21})b^2u^2+2\Et_{12}b^3u^3-3\Et_{03}b^4u^4\right]\\
\nonumber& + & b\Phit^2\left[-\Et_{02}bu+\Et_{12}b^2u^2-3\Et_{03}b^3u^3\right]- b^2\Phit^3\left[\Et_{03}b^2u^2\right]\\
& = & u(1-u)bH_1(b,u)+(1+G_1(bu))\Phit+\NL_1(u,\Phit)
\eea
with
\be
\label{fneionefonone}
\left|\begin{array}{l}
H_1(b,u)=-(\Et_{11}+\Et_{30})+b(\Et_{02}+\Et_{21})(1+u)-b^2\Et_{12}(1+u+u^2)+b^3\Et_{03}(1+u+u^2+u^3)\\
G_1(x)=\Et_{11}x-(2\Et_{02}+\Et_{21})x^2+2\Et_{12}x^3-3\Et_{03}x^4\\
\NL_1(u,\Phit)= b\Phit^2\left[-\Et_{02}bu+\Et_{12}b^2u^2-3\Et_{03}b^3u^3\right]- b^2\Phit^3\left[\Et_{03}b^2u^2\right]
\end{array}\right.
\ee

\noindent\underline{Reexpressing $F_2$}. Recall  
\bee
F_2(u,\phi)& =& 1-u+b\left[(\Dt_{11}+\Dt_{30})u-\Dt_{11}\phi u-\Dt_{30}u^2\right]+b^2\left[-(\Dt_{02}+\Dt_{21})u+\Dt_{02}\phi^2u+\Dt_{21}\phi u^2\right]\\
& + & b^3\left[\Dt_{12}u-\Dt_{12}\phi^2u^2\right]+b^4\left[-\Dt_{03}u+\Dt_{03}\phi^3u^2\right]
 \eee
 then
 $$(\Dt_{11}+\Dt_{30})u-\Dt_{11}\phi u-\Dt_{30}u^2=(\Dt_{11}+\Dt_{30})u(1-u)-\Dt_{11}u\Phit$$
 and
 \bee
 &&-(\Dt_{02}+\Dt_{21})u+\Dt_{02}\phi^2u+\Dt_{21}\phi u^2=-(\Dt_{02}+\Dt_{21})u+\Dt_{02}(u+\Phit)^2u+\Dt_{21}(u+\Phit)u^2\\
 & = & -(\Dt_{02}+\Dt_{21})u(1-u)(1+u)+u^2(2\Dt_{02}+\Dt_{21})\Phit+\Dt_{02}u\Phit^2
 \eee
 and 
 \bee
 &&\Dt_{12}u-\Dt_{12}\phi^2u^2=\Dt_{12}u-\Dt_{12}(u+\Phit)^2u^2\\
 & = & \Dt_{12}u(1-u)(1+u+u^2)-2\Dt_{12}u^3\Phit-\Dt_{12}u^2\Phit^2
 \eee
 and
 \bee
 &&-\Dt_{03}u+\Dt_{03}\phi^3u^2=-\Dt_{03}u+\Dt_{03}(u+\Phit)^3u^2\\
 & = & -\Dt_{03}u(1-u)(1+u+u^2+u^3)+\Dt_{03}(3u^4\Phit+3u^3\Phit^2+u^2\Phit^3).
 \eee
 Then,
 \bea
 \label{expressionnffofw}
\nonumber && F_2(u,\phi)=1-u+b\left[(\Dt_{11}+\Dt_{30})u(1-u)-\Dt_{11}u\Phit\right]\\
\nonumber & + & b^2\left[-(\Dt_{02}+\Dt_{21})u(1-u)(1+u)+u^2(2\Dt_{02}+\Dt_{21})\Phit+\Dt_{02}u\Phit^2\right]\\
\nonumber & + & b^3\left[\Dt_{12}u(1-u)(1+u+u^2)-2\Dt_{12}u^3\Phit-\Dt_{12}u^2\Phit^2\right]\\
\nonumber & + & b^4\left[-\Dt_{03}u(1-u)(1+u+u^2+u^3)+\Dt_{03}(3u^4\Phit+3u^3\Phit^2+u^2\Phit^3)\right]\\
\nonumber & = & (1-u)\left[1+b(\Dt_{11}+\Dt_{30})u-b^2(\Dt_{02}+\Dt_{21})u(1+u)+b^3\Dt_{12}u(1+u+u^2)\right.\\
\nonumber &- & \left.b^4\Dt_{03}u(1+u+u^2+u^3)\right]\\
\nonumber & + & \Phit\left[-\Dt_{11}bu+(2\Dt_{02}+\Dt_{21})b^2u^2-2\Dt_{12}b^3u^3+3\Dt_{03}b^4u^4\right]\\
\nonumber & + & b\Phit^2\left[\Dt_{02}bu-\Dt_{12}b^2u^2+3\Dt_{03}b^3u^3\right]+b^2\Phit^3\left[\Dt_{03}b^2u^2\right]\\
& = & (1-u)\left[1+H_2(b,u)\right]+G_2(bu)\Phit+\NL_2(u,\Phit)
\eea
with
\bea
\label{eniovnevneoneneno}
\nonumber H_2(b,u)&=&b(\Dt_{11}+\Dt_{30})u-b^2(\Dt_{02}+\Dt_{21})u(1+u)+b^3\Dt_{12}u(1+u+u^2)\\
&-&  b^4\Dt_{03}u(1+u+u^2+u^3)
\eea
and
\be
\label{defnltwo}
\left|\begin{array}{l}
G_2(x)=-\Dt_{11}x+(2\Dt_{02}+\Dt_{21})x^2-2\Dt_{12}x^3+3\Dt_{03}x^4\\
\NL_2(u,\Phit)=b\Phit^2\left[\Dt_{02}bu-\Dt_{12}b^2u^2+3\Dt_{03}b^3u^3\right]+b^2\Phit^3\left[\Dt_{03}b^2u^2\right]
\end{array}\right.
\ee

\noindent{\bf step 6} Final change of variables. We now reexpress \eqref{quaslinearformulation} as
\bee
&&F_2(\phi,u)(u\phi'+\phi)=\gamma F_1\\
& \Leftrightarrow& \left[ (1-u)(1+H_2)+G_2\Phit+\NL_2\right](u\Phit'+\Phit+2u)\\
&=& \gamma\left[ u(1-u)bH_1+(1+G_1)\Phit+\NL_1\right]
\eee
We change variables  
\be
\label{vneinvenoen}
\Phit=(1-u)\Psi
\ee 
and define $$x=bu.$$

\noindent\underline{Nonlinear terms}. We rewrite from \eqref{fneionefonone}
$$
\left|\begin{array}{l}
\NL_1=bM_{11}(x)\Phit^2+b^2M_{12}(x)\Phit^3\\
M_{11}= -\Et_{02}x+\Et_{12}x^2-3\Et_{03}x^3\\
M_{12}=\Et_{03}x^2
\end{array}\right.
$$
and from \eqref{defnltwo}:
\be
\label{eniovneionveneneov}
\left|\begin{array}{l}
\NL_2=bM_{21}(x)\Phit^2+b^2M_{22}(x)\Phit^3\\
M_{21}=\Dt_{02}x-\Dt_{12}x^2+3\Dt_{03}x^3\\
M_{22}=\Dt_{03}x^2
\end{array}\right.
\ee
Then,
\bea
\label{formulanlone}
\nonumber \NL_1&=&bM_{21}(1-u)^2\Psi^2+b^2M_{12}(1-u)^3\Psi^3\\
\nonumber & = & (1-u)\left[(b-x)M_{11}(x)\Psi^2+(b^2-2bx+x^2)M_{12}\Psi^3\right]\\
&\equiv& (1-u)\NLt_1
\eea
and
\bea
\label{formulanltwo}
\nonumber \NL_2&=&bM_{21}(1-u)^2\Psi^2+b^2M_{22}(1-u)^3\Psi^3\\
\nonumber & = & (1-u)\left[(b-x)M_{21}(x)\Psi^2+(b^2-2bx+x^2)M_{22}\Psi^3\right]\\
&\equiv& (1-u)\NLt_2
\eea
\noindent\underline{$\Psi$ equation}. We compute
$$u\Phit'+\Phit+2u=u[(1-u)\Psi'-\Psi]+(1-u)\Psi+2u=u(1-u)\Psi'+(1-2u)\Psi+2u.$$
Then,
\bee
&&(1-u)(1+H_2)+G_2\Phit+\NL_2=(1-u)\left[1+H_2+G_2\Psi+\NLt_2\right]\\
&&  u(1-u)bH_1+(1+G_1)\Phit+\NL_1=(1-u)\left[ b uH_1+(1+G_1)\Psi+\NLt_1\right].
\eee
This gives the $\Psi$ equation:
\bee
&&\left[1+H_2+G_2\Psi+\NLt_2\right]\left[u(1-u)\Psi'+(1-2u)\Psi+2u\right]=\gamma\left[ xH_1+(1+G_1)\Psi+\NLt_1\right]\\
& = & u\left[\gamma bH_1+\left(\frac{\gamma}u+\frac{\gamma bG_1}{x}\right)\Psi+\frac{\gamma b\NLt_1}{x}\right].
\eee
Equivalently:
\bee
&&\left[1+H_2+G_2\Psi+\NLt_2\right]\left[u(1-u)\Psi'+(1-2u)\Psi\right]\\
& = & u\left[\gamma bH_1+\left(\frac{\gamma}u+\frac{\gamma bG_1}{x}\right)\Psi+\frac{\gamma b\NLt_1}{x}-2\left(1+H_2+G_2\Psi+\NLt_2\right)\right]
\eee
i.e.,
\bee
 &&\left[1+H_2+G_2\Psi+\NLt_2\right]u(1-u)\Psi'\\
\nonumber&+& \left[(1-2u)(1+H_2+G_2\Psi+\NLt_2)-\gamma(1+G_1)+2uG_2\right]\Psi\\
\nonumber& = &u\left[\gamma bH_1-2(1+H_2)+\frac{\gamma b\NLt_1}{x}-2\NLt_2\right].
\eee
\noindent\underline{Reordering terms.} 
We split $H_2$:
\bee
\nonumber &&H_2(b,u)=b(\Dt_{11}+\Dt_{30})u-b^2(\Dt_{02}+\Dt_{21})u(1+u)+b^3\Dt_{12}u(1+u+u^2)\\
 &-&  b^4\Dt_{03}u(1+u+u^2+u^3)=  \sum_{j=0}^3 b^jH_{2,j}(x),
\eee
and similarly
\bee
H_1(b,u)&=&-(\Et_{11}+\Et_{30})+b(\Et_{02}+\Et_{21})(1+u)-b^2\Et_{12}(1+u+u^2)+b^3\Et_{03}(1+u+u^2+u^3)\\
&=& \sum_{j=0}^3b^jH_{1,j}(bu)
\eee
with $H_{i,j}$ given by \eqref{hijformulas}. We reorder the nonlinear terms using the same rule:
\be
\label{cneoneneono}
\left|\begin{array}{l}
\NLt_1=(b-x)M_{11}(x)\Psi^2+(b^2-2bx+x^2)M_{12}\Psi^3=\sum_{j=0}^2b^j\NLt_{1j}\\
\NLt_2=(b-x)M_{21}(x)\Psi^2+(b^2-2bx+x^2)M_{22}\Psi^3=\sum_{j=0}^2b^j\NLt_{2j}
\end{array}\right.
\ee
with \eqref{nolinearterms}.
\end{proof}


\section{Bounding the Taylor series of the formal limit problem}
\label{sectionlimit}

We now start the analysis of the non linear ode \eqref{equaitoncompee} for $0<u<1$. It has a regular singular point at the origin and our first task is to estimate the growth of the Taylor coefficients of solutions' expansions at $u=0$. This will be done in two steps. First, in this section we estimate the growth of the coefficients for a formal $b=0$ limiting system, Proposition \ref{proposnvoivone}. This will make appear the function $S_\infty(d,\ell)$. Then, in section \ref{bbounds} we will obtain uniform bounds in $b$ for the Taylor coefficients associated to \eqref{equaitoncompee} {\em for frequencies $k\lesssim \frac{1}{b}$}.


\subsection{Formal limit $b=0$} Recall \eqref{equaitoncompee} and let $$\Psi(u)=\Psit(x), \ \ x=bu$$ then 
\bea
\label{equationtowrikwith}
\nonumber  &&\left[1+H_2+G_2\Psi+\NLt_2\right]\frac{x}{b}(1-\frac{x}{b})b\Psit'\\
\nonumber&+& \left[(1-\frac{2x}{b})(1+H_2+G_2\Psi+\NLt_2)-\gamma(1+G_1)+2\frac{x}{b}G_2\right]\Psit\\
\nonumber& = &\frac{x}{b}\left[\gamma bH_1-2(1+H_2)+\frac{\gamma b\NLt_1}{x}-2\NLt_2\right]\\
&\Leftrightarrow& \left[1+H_2+G_2\Psit+\NLt_2\right]x(b-x)\Psit'\\
\nonumber&+& \left[(b-2x)(1+H_2+G_2\Psit+\NLt_2)-b\gamma(1+G_1)+2xG_2\right]\Psit\\
\nonumber& = &x\left[\gamma bH_1-2(1+H_2)+\frac{\gamma b\NLt_1}{x}-2\NLt_2\right].
\eea
We introduce the parameters
\be
\label{defparameterscneoevn}
\left|\begin{array}{l}
a=\gamma b=\frac{|\l_-|}{|\mu_+|}>0\\
\nu=-\gamma b(\Dt_{11}+\Dt_{30}-\Et_{11}),
\end{array}\right.
\ee
which have a well defined limit as $b\to 0$ noted $a_\infty,\nu_\infty$,
and assume 
\be
\label{hyptnupostiif}
\nu_\infty(d,\ell)>0.
\ee
In view of the explicit formulas of Appendix \ref{apneidincoien} and remark \ref{limitnigncoie}, the {\em formal} limit $b\to 0$ is:
\bea
\label{limitingxequation}
&&\left[1+H^\infty_{20}+G^\infty_2\Psit+\NLt^\infty_{20}\right](-x^2)\Psit'\\
\nonumber&+& \left[-2x(1+H^\infty_{20}+G^\infty_2\Psit+\NLt^\infty_{20})-a_\infty(1+G^\infty_1)+2xG^\infty_2\right]\Psit\\
\nonumber& = &x\left[a_\infty H^\infty_{10}-2(1+H^\infty_{20})+\frac{a_\infty\NLt^\infty_{10}}{x}-2\NLt^\infty_{20}\right]
\eea
where we recall\footnote{Remark \ref{renkonvnvoine}} that the subscript $^\infty$ means that we compute all parameters $(\Dt_{ij},\Et_{ij})$ given by \eqref{defdtij} in their well defined limit $b=0$. 
Let us stress the fact that this is {\em not a dynamical limit}, since the change variables $x=bu$ maps the 
original flow on the set $x=0$. Our claim is that for fixed order $k$, the  Taylor coefficients associated to \eqref{limitingxequation} are the $b=0$ limit of the Taylor coefficients associated to \eqref{equaitoncompee} up to a suitable renormalization, see \eqref{limitignprocedure}.


\subsection{Boundedness of the limiting series}


The condition \eqref{hyptnupostiif} holds, by a direct examination, at $r^+$ and $\ell>d$ from \eqref{eq:algebricformulafornubis}, and at $r_*$ on a collection of subintervals of $(0,d)$ from \eqref{eq:algebricformulafornu}. In fact, \eqref{signofnu} implies that in the latter set is non-empty and for $d=2,3$
coincides with $(0,d)$.  We also note the fact that the condition \eqref{hyptnupostiif} 
is not necessary for the following arguments. We also make the following note that in the case of $r^+$, in principle, all the arguments below can be extended to the values of $\ell\in (0,d)$. In particular, from \eqref{eq:algebricformulafornubis}
the function $\nu_\infty$ is still positive there. This extension will be important for the analyticity argument in Appendix E.
\\

Our aim in this section is to prove the following bound.

\begin{proposition}[Boundedness for \eqref{limitingxequation}]
\label{proposnvoivone} 
Assume \eqref{hyptnupostiif}. Then there exists $c_{\nu_\infty,a_\infty}>0$ such that the following holds. Let $\Psit$ be the unique $\mathcal C^\infty$ local solution to \eqref{limitingxequation} on $[0,x_0]$, then the sequence $$\psit_k=\frac{\Psit^{(k)}(0)}{k!}$$ satisfies 
\be
\label{boundequenus}
|\psit_k|\leq c_{\nu,a}\frac{\Gamma(k+\nu_\infty+2)}{a_\infty^k}.
\ee
\end{proposition}

\begin{remark} The bound \eqref{boundequenus} falls within the range of nonlinear Maillet type theorems, see \cite{malgrange, ramis, sibuya}. We shall give a self contained proof which will allow us to obtain 
quantitative bounds. The latter is crucial  for future uniform $b$ independent bounds for all frequencies $k\lesssim \frac 1b$ for the full problem, see \eqref{esitmaitmitot}.
\end{remark}

\subsection{Conjugation formula}


We start by conjugating \eqref{limitingxequation} to an explicitly solvable (at the linear level) problem.

\begin{lemma}[Conjugation]
\label{lemmaconjugation}
There exist functions $\xi(x), (\mu_j(x),\nu_j(x))_{1\le j\le 4}$, holomorphic in a neighborhood of the $x=0$,
dependent on $(d,\ell,r)$,  such that the change of variables
\be
\label{condjiufgatioformula}
\left|\begin{array}{l}
\Psit(x)=M(x)\Phi(x), \ \ M(x)=e^{-\int_0^x \xi(y)dy}\\
\Theta(x)=\frac{\Phi(x)}{x}
\end{array}\right.
\ee 
maps \eqref{limitingxequation} to 
\be
\label{venonveoneonoenv}
x^2\Theta'+[a_\infty+(\nu_\infty+3)x]\Theta=\mu_0+x\left[x\sum_{j=1}^4\mu_j\Theta^j+x^2(x\Theta')\sum_{j=1}^{4}\nu_j\Theta^j\right].
\ee
\end{lemma}

\begin{proof}[Proof of Lemma \ref{lemmaconjugation}] This is an explicit computation.\\

\noindent{\bf step 1} Linear conjugation. We solve the linear problem
\be
\label{odetobesolved}
\left[x+xH^\infty_{20}\right](-x)\Psit'+ \left[-2x(1+H^\infty_{20})-a_\infty(1+G^\infty_1)+2xG^\infty_2\right]\Psit =xF.
\ee
by conjugating it back to an explicitly solvable problem:
\bee
\eqref{odetobesolved}& \Leftrightarrow& \Psit'+\frac{2x(1+H^\infty_{20})+a_\infty(1+G^\infty_1)-2xG^\infty_2}{x^2(1+H_{20})}\Psit =-\frac{xF}{x^2(1+H_{20})}\\
&\Leftrightarrow&\Psit'+\left[\frac{2}{x}+\frac{a_\infty(1+\Et^\infty_{11}x)}{x^2(1+H^\infty_{20})}+\frac{a_\infty(G^\infty_1-\Et^\infty_{11}{x})-2xG^\infty_2}{x^2(1+H^\infty_{20})}\right]\Psit=-\frac{F}{x(1+H^\infty_{20})}.
\eee
We have
\bee
&&\frac{a_\infty(1+\Et^\infty_{11}x)}{x^2(1+H^\infty_{20})}=\frac{a_\infty}{x^2}(1+\Et^\infty_{11}x)\left[1-(\Dt^\infty_{11}+\Dt^\infty_{30})x+\frac{1}{1+H^\infty_{20}}-1+(\Dt^\infty_{11}+\Dt^\infty_{30}){x}\right]\\
& = & \frac{a_\infty}{x^2}(1+\Et^\infty_{11}x)\left[\frac{1}{1+H^\infty_{20}}-1+(\Dt^\infty_{11}+\Dt^\infty_{30}){x}\right]-a\Et^\infty_{11}(\Dt^\infty_{11}+\Dt^\infty_{30})\\
& + & \frac{a_\infty}{x^2}+\frac{a_\infty[\Et^\infty_{11}-(\Dt^\infty_{11}+\Dt^\infty_{30}{)}]}{x}
\eee
Let 
\bee
\xi(x)&=&\frac{a_\infty}{x^2}(1+\Et^\infty_{11}x)\left[\frac{1}{1+H^\infty_{20}}-1+(\Dt^\infty_{11}+\Dt^\infty_{30}){x}\right]-a_\infty\Et^\infty_{11}(\Dt^\infty_{11}+\Dt^\infty_{30})\\
&+ & \frac{a_\infty(G^\infty_1-\Et^\infty_{11}{x})-2xG^\infty_2}{x^2(1+H^\infty_{20})}
\eee
Observe that
$H^\infty_{20}-(\Dt^\infty_{11}+\Dt^\infty_{30}){x}$ and $G^\infty_1-\Et^\infty_{11}{x}$ are polynomials in $x$  starting with $x^2$, while $G^\infty_2$ 
is a polynomial beginning with $x$. The function $\xi(x)$ is then holomorphic in a neighborhood of $x=0$. We have
$$\eqref{odetobesolved}\Leftrightarrow \Psit'+\left[\frac{a_\infty}{x^2}+\frac{\nu_\infty+2}{x}+\xi(x)\right]\Psit=-\frac{F}{x(1+H^\infty_{20})}.$$
Define
\be
\label{deifnitoinkernel}
M(x)=e^{-\int_0^x \xi(y)dy}
\ee
a holomorphic function in a neighborhood of $x=0$, and introduce the change of variables $$\Psit(x)=M(x)\Phi(x), \ \ G(x)=-\frac{F(x)}{M(x)(1+H^\infty_{20}(x))}.$$
We have obtained the conjugation formula:
\be
\label{newproblem}
\eqref{odetobesolved}\Leftrightarrow \Phi'+\left[\frac{a_\infty}{x^2}+\frac{\nu_\infty+2}{x}\right]\Phi=\frac{G}{x}.
\ee

\noindent{\bf step 2} Conjugation for the nonlinear problem. The equation \eqref{limitingxequation} is in the form \eqref{odetobesolved}
$$\left[x+xH^\infty_{20}\right](-x)\Psit'+ \left[-2x(1+H^\infty_{20})-a_\infty(1+G^\infty_1)+2xG^\infty_2\right]\Psit =x(F_0+F)$$ 
for the source term
$$F_0=a_\infty H^\infty_{10}-2(1+H^\infty_{20})$$ and the nonlinear term
\bee
F& = & \frac{a_\infty\NLt^\infty_{10}}{x}-2\NLt^\infty_{20}+\left[G^\infty_2\Psi+\NLt^\infty_{20}\right](2\Psi+x\Psi')
\eee
We conjugate this using \eqref{newproblem}:
$$
\left|\begin{array}{l}
\Psit(x)=M(x)\Phi(x), \ \ G(x)=-\frac{F(x)+F_0}{M(x)(1+H^\infty_{20}(x))} \\
x^2\Phi'+\left[a_\infty+(\nu_\infty+2)x\right]\Phi= xG
\end{array}\right.
$$
We express the nonlinearity in terms of $\Phi$ and obtain a representation:
$$
G=\xi_0+{\sum_{j=2}^4\xi_j\Phi^j+(x\Phi')\sum_{j=1}^3(\xit_j\Phi^j)}
$$
where $(\xi_j,\xit_j)$ are explicit holomorphic functions in a neighborhood of the origin. We have obtained the equivalent nonlinear problem:
\be
\label{newnonlinearproblem}
x^2\Phi'+\left[a_\infty+(\nu_\infty+2)x\right]\Phi=x\left[\xi_0+{\sum_{j=2}^4\xi_j\Phi^j+(x\Phi')\sum_{j=1}^3(\xit_j\Phi^j)}\right].
\ee
A direct computation shows $$\Phi(0)=0$$ and we let 
\be
\label{bvebeibebev}
\Theta=\frac{\Phi}x,
\ee 
so that:
\bee
\nonumber &&\eqref{newnonlinearproblem}\Leftrightarrow x^2\left[x\Theta'+\Theta\right]+\left[a_\infty+(\nu_\infty+2)x\right]x\Theta\\
\nonumber &=&x\left[\xi_0+{\sum_{j=2}^4\xi_jx^j\Theta^j+x(x\Theta'+\Theta)\sum_{j=1}^3(\xit_jx^j\Theta^j)}\right]\\
\nonumber &\Leftrightarrow &x^2\Theta'+[a_\infty+(\nu_\infty+3)x]\Theta=\xi_0+{\sum_{j=2}^4\xi_jx^j\Theta^j+x(x\Theta')\sum_{j=1}^3(\xit_jx^j\Theta^j)+\sum_{j=1}^3\xit_j(x\Theta)^{j+1}}\\
&\Leftrightarrow &x^2\Theta'+[a_\infty+(\nu_\infty+3)x]\Theta=\mu_0+x\left[x\sum_{j=1}^4\mu_j\Theta^j+x^2(x\Theta')\sum_{j=1}^{4}\nu_j\Theta^j\right]
\eee
where $\mu_{{j}},\nu_j$ are holomorphic functions of $x$ in a neighborhood of $0$, and \eqref{venonveoneonoenv} is proved.
\end{proof}


\subsection{The nonlinear induction relation}


The uniqueness of a local $\matchal C^\infty$ solution to \eqref{venonveoneonoenv}, and thus \eqref{limitingxequation}, follows from an elementary fixed point argument which is left to the reader. We let $$\mu_{jk}=\frac{\mu_j^{(k)}(0)}{k!}, \ \ \nu_{jk}=\frac{\nu_j^{(k)}(0)}{k!},\ \ \theta_k=\frac{\Theta^{(k)}(0)}{k!}, \ \ \phi_k=\frac{\Phi^{(k)}(0)}{k!}$$ and claim the following fundamental nonlinear bound.

\begin{lemma}[Bound on the limiting sequence]
\label{propositionboundedness}
For some large enough universal constant $c(\nu_\infty,a_\infty)>0$:
\be
\label{boundednessofthesquence}
|\phi_k|\leq c(\nu_\infty,a_\infty)\frac{\Gamma(k+\nu_\infty+2)}{a_\infty^k}, \ \ \forall k\ge 1.
\ee
{Also, we have
\bea\label{convergingsequenceforSk}
\sum_{k=0}^{+\infty}\left|\frac{a_\infty^kg_k}{\Gamma(\nu_\infty+k+3)}\right|\leq c(\nu_\infty,a_\infty)<+\infty. 
\eea}
\end{lemma}

\begin{proof}[Proof of Lemma \ref{propositionboundedness}] We compute the nonlinear induction relation and estimate the sequence using convolution estimates for the $\Gamma$ function.\\

\noindent{\bf step 1} The induction relation.  We formally expand 
\be
\label{vnieonveineonovenv}
\left|\begin{array}{l}
\Theta=\sum_{k=0}^{+\infty}\theta_kx^k\\
 G=\mu_0+x\left[x\sum_{j=1}^4\mu_j\Theta^j+x^2(x\Theta')\sum_{j=1}^4\nu_j\Theta^j\right]=\sum_{k=0}^{+\infty}g_kx^k
 \end{array}\right.
 \ee
and obtain, from \eqref{venonveoneonoenv}, the induction relation:
\bea
\label{ceineneoeoneobis}
\nonumber &&x^2\Theta'+(a_\infty+(\nu_\infty+3)x)\Theta=G\\
\nonumber &\Leftrightarrow& 
\sum_{k=1}^{+\infty}k\theta_kx^{k+1}+\sum_{k=0}^{+\infty}a_\infty\theta_k x^k+\sum_{k=0}^{+\infty}(\nu_\infty+3)\theta_kx^{k+1}=\sum_{k=0}^{+\infty}g_kx^{k}\\
\nonumber &\Leftrightarrow& \sum_{k=2}^{+\infty} (k-1)\theta_{k-1}x^{k}+\sum_{k=0}^{+\infty}a_\infty\theta_kx^k+\sum_{k=1}^{+\infty}(\nu_\infty+3)x^{k}\theta_{k-1}=\sum_{k=0}^{+\infty}g_{k}x^{k}\\
\nonumber &\Leftrightarrow&
\left|\begin{array}{l}
a_\infty\theta_0=g_0\\
(\nu_\infty+3)\theta_0+a_\infty\theta_1=g_1\\
(k+\nu_\infty+2)\theta_{k-1}+a_\infty\theta_k=g_{k}, \ \ k\ge 2
\end{array}\right.\\
&\Leftrightarrow&
\left|\begin{array}{l}
a_\infty\theta_0=g_0\\
(k+\nu_\infty+3)\theta_{k}+a_\infty\theta_{k+1}=g_{k+1}, \ \ k\ge {0}
\end{array}\right.
\eea
We now compute $g_k$ from \eqref{vnieonveineonovenv}. By Leibniz 
$$(fg)_k=\frac{1}{k!}\frac{d^k}{dx^k}(fg)(0)=\frac{1}{k{!}}\sum_{k_1+k_2=k}\frac{k!}{k_1!k_2!}f^{(k_1)}(0)g^{(k_2)}(0)=\sum_{k_1+k_2=k}f_{k_1}g_{k_2}$$ 
Therefore,
$$
(x^2\mu_j\Theta^j)_{k+1}=(\mu_j\Theta^j)_{k-1}=\sum_{k_1+\dots k_{j+1}=k-1}\mu_{jk_1}\theta_{k_2}\dots\theta_{k_{j+1}}.$$
Then $$(x\Theta')_k=k\theta_k$$ yields
$$\left(x^3(x\Theta')\nu_j\Theta^j\right)_{k+1}=((x\Theta')\nu_j\Theta^j)_{k-2}=\sum_{k_1+\dots k_{j+2}=k-2}\nu_{jk_1}\theta_{k_2}\dots\theta_{k_{j+1}}(k_{j+2}\theta_{k_{j+2}})$$
and \eqref{ceineneoeoneobis} yields the induction relation for $k\ge 5$:
\bea
\label{inductonrelation}
\nonumber &&(k+\nu_\infty+3)\theta_{k}+a_\infty\theta_{k+1}= (\mu_0)_{k+1}+\sum_{j=1}^4\sum_{k_1+\dots k_{j+1}=k-1}\mu_{jk_1}\theta_{k_2}\dots\theta_{k_{j+1}}\\
 &+& \sum_{j=1}^4\sum_{k_1+\dots k_{j+2}=k-2}\nu_{jk_1}\theta_{k_2}\dots\theta_{k_{j+1}}(k_{j+2}\theta_{k_{j+2}})
\eea

\noindent{\bf step 2} Renormalization of the sequence. We claim the bound:
\be
\label{boundednessofthesquencebis}
|\theta_k|\leq c(\nu_\infty,a_\infty)\frac{\Gamma(k+\nu_\infty+3)}{a_\infty^k}, \ \ \forall k\ge 1
\ee
and
\be
\label{boundednessofthesquencebis:bis}
\left|(k+\nu_\infty+2)\theta_{k-1}+a_\infty\theta_k\right|\leq c(\nu,a)\frac{\Gamma(k+\nu_\infty+3)}{(1+k^2)a_\infty^k}, \ \ \forall k\ge 1.
\ee
These imply, from \eqref{bvebeibebev} {and \eqref{ceineneoeoneobis}}
$$\left|\begin{array}{l}
|\phi_k| = |(x\Theta)_k|=|\theta_{k-1}|\leq c(\nu_\infty,a_\infty)\frac{\Gamma(k-1+\nu_\infty+3)}{a_\infty^k}, \\ 
{\left|\frac{a_\infty^kg_k}{\Gamma(\nu_\infty+k+3)}\right|} \leq {\frac{c(\nu_\infty,a_\infty)}{1+k^2}},
\end{array}\right.
$$
and \eqref{boundednessofthesquence} and \eqref{convergingsequenceforSk} follow. We therefore focus on the proof of \eqref{boundednessofthesquencebis}, \eqref{boundednessofthesquencebis:bis}. We start by renormalizing the sequence. Let 
\be
\label{bfeioeoneonneon}\left|\begin{array}{l}
w_k=\frac{a_\infty^k}{\Gamma(\nu_\infty+k+3)}\theta_k\\
(h_0)_k=\frac{a_\infty^k}{\Gamma(\nu_\infty+k+3)}(\mu_0)_{k+1}\\
h_k=\frac{a_\infty^k}{\Gamma(\nu_\infty+k+3)}\mu_k\\
\th_k=\frac{a_\infty^k}{\Gamma(\nu_\infty+k+3)}\nu_k.
\end{array}\right.
\ee
then \eqref{inductonrelation} becomes 
\bee
\nonumber &&\frac{\Gamma(\nu_\infty+k+4)a_\infty}{a_\infty^{k+1}}w_{k+1}+(k+\nu_\infty+3)\frac{\Gamma(\nu_\infty+k+3)}{a_\infty^{k}}w_{k}=\frac{\Gamma(\nu_\infty+k+3)}{a_\infty^k}(h_0)_k\\
&+& \frac{1}{a_\infty^{k-1}}\sum_{j=1}^4\sum_{k_1+\dots k_{j+1}=k-1}h_{jk_1}w_{k_2}\dots w_{k_{j+1}}\Pi_{i=1}^{j+1}\Gamma(\nu_\infty+k_i+3)\\
\nonumber &+& \frac{1}{a_\infty^{k-2}}\sum_{j=1}^4\sum_{k_1+\dots k_{j+2}=k-2}\th_{jk_1}w_{k_2}\dots w_{k_{j+1}}(k_{j+2}w_{k_{j+2}})\Pi_{i=1}^{j+2}\Gamma(\nu_\infty+k_i+3)\\
\nonumber &\Leftrightarrow& w_{k+1}+w_k=\frac{1}{\nu_\infty+k+3}(h_0)_k\\
& + & \frac{a_\infty}{(k+\nu_\infty+3)(k+\nu_\infty+2)}\sum_{j=1}^4\sum_{k_1+\dots k_{j+1}=k-1}h_{jk_1}w_{k_2}\dots w_{k_{j+1}}\frac{\Pi_{i=1}^{j+1}\Gamma(\nu_\infty+k_i+3)}{\Gamma(k-1+\nu_\infty+3)}\\
& + &  \frac{a_\infty^2}{\Pi_{j=1}^3(k+\nu_\infty+j)}\sum_{j=1}^4\sum_{k_1+\dots k_{j+2}=k-2}\th_{jk_1}w_{k_2}\dots w_{k_{j+1}}(k_{j+2}w_{k_{j+2}})\frac{\Pi_{i=1}^{j+2}\Gamma(\nu_\infty+k_i+3)}{\Gamma(k-2+\nu_\infty+3)}
\eee
We obtain the upper bound 
\bea
\label{upperboundsequence}
&&|w_{k+1}|<|w_k|+\frac{c_{\nu_\infty,a_\infty}}{k+1}(h_0)_k\\
\nonumber &+& \frac{c_{\nu_\infty,a_\infty}}{1+k^2}\sum_{j=1}^4\sum_{k_1+\dots k_{j+1}=k-1}|h_{jk_1}w_{k_2}\dots w_{k_{j+1}}|\frac{\Pi_{i=1}^{j+1}\Gamma(\nu_\infty+k_i+3)}{\Gamma(k-1+\nu_\infty+3)}\\
\nonumber & + &  \frac{c_{\nu_\infty,a_\infty}}{1+k^2}\sum_{j=1}^{{4}}\sum_{k_1+\dots k_{j+2}=k-2}|\th_{jk_1}w_{k_2}\dots w_{k_{j+1}}w_{k_{j+2}}|\frac{\Pi_{i=1}^{j+2}\Gamma(\nu_\infty+k_i+3)}{\Gamma(k-2+\nu_\infty+3)}.
\eea
We now proceed via a bootstrap argument and iteratively improve the control of the sequence $w_k$ from a large exponential bound to boundedness, the key being the use of suitable {\em uniform} convolution estimates. The latter are proved in Appendix \ref{appendixconvolution}.\\

\noindent{\bf step 3} Large exponential bound. Since the functions $h_0, h, \tilde h$ are holomorphic and \eqref{bfeioeoneonneon}, there exists $A=A(\nu_\infty,a_\infty)>0$ such that 
\be
\label{boundhfiheohgbis}
\forall k\ge 1, \ \ |(h_0)_k|+|h_k|+|\th_k|\leq \frac{A^k}{\Gamma(k+1)}.
\ee
 Pick $C_0\geq  C_0(\nu_\infty,a_\infty)\gg 1$ large enough, we claim 
 \be
\label{neoonenoenvbis}
|w_k|\le C_0^k .
\ee
The bound holds for $k\le 5$ by choosing $C_0$ large enough. We assume the bound for $j\le k$ and prove it for $k+1$. We have from \eqref{boundhfiheohgbis} the rough bound $$|(h_0)_k|+|h_k|+|\th_k|<C_0^k$$ provided $C_0$ has been chosen large enough. Then, from \eqref{upperboundsequence}, \eqref{convolutionbound}:
$$|w_{k+1}|<C_0^k+\frac{c_{\nu_\infty,a_\infty}C_0^k}{1+k}+\frac{c_{\nu_\infty,a_\infty}}{1+k^2}\sum_{j=1}^4C_\nu^{j+2}C_0^{k-1} $$
and 
$$\frac{|w_{k+1}|}{C_0^{k+1}}<\frac{1}{C_0}+\frac{c_{a_\infty,\nu_\infty}}{C_0(1+k)}<1$$
for $k\ge 1$ and $C_0>C_0(\nu_\infty,a_\infty)$ large enough. \eqref{neoonenoenvbis} is proved.\\

\noindent{\bf step 4} {Improvement of the} exponential bound. {Let $C_0(\nu,a)$ be a large enough constant such that \eqref{neoonenoenvbis} holds. Pick a small enough universal constant $\delta(\nu_\infty,a_\infty)>0$ such that $\de(\nu_\infty,a_\infty)\ll (C_0(\nu_\infty,a_\infty))^{-1}$}. We claim the following: assume that there exist 
\be
\label{vnnvnvoennobis}
\left|\begin{array}{l}
K_n>1\\
e^{{\delta}}{\leq} C_n\le C_0
\end{array}\right.
\ee
such that
\be
\label{nnvnevonenoebis}
\forall k\ge 1, \ \ |w_k|<K_nC_n^k,
\ee 
then there exists $K_{n+1}=K_{n+1}(K_n,C_n)>1$ such that 
\be
\label{cneivneoenvevnoebis}
\forall k\ge 1, \ \ |w_k|<K_{n+1}\left(C_{n}e^{-\delta}\right)^k.
\ee
Pick $k^*(K_n,\nu_\infty, {a}_\infty)$ large enough, {then the bound \eqref{nnvnevonenoebis} implies \eqref{cneivneoenvevnoebis} for $k\le k^*$ provided we choose $K_{n+1}(K_n,\nu_\infty, {a}_\infty)\geq K_ne^{k^*\de}$}. We now bootstrap the bound \eqref{cneivneoenvevnoebis} by induction for $k\ge k^*$, assuming it for $j\le k$ and proving it for $k+1$.\\
Assume $$k_1\leq k_2\dots \leq k_{j+1},$$ and let $m=k_1+\dots+k_{j}$ and use the rough upper bound from \eqref{boundhfiheohgbis} to estimate
\bee
&&|h_{jk_1}||w_{k_2}|\dots |w_{k_{j+1}}||\frac{\Pi_{i=1}^{j+1}\Gamma(\nu_\infty+k_i+3)}{\Gamma(k-1+\nu_\infty+3)}\\
&\lesssim &  (K_nC_n^{k_1})|w_{k_2}|\dots |w_{k_{j+1}}|\frac{\Pi_{i=1}^{j+1}\Gamma(\nu_\infty+k_i+3)}{\Gamma(k-1+\nu_\infty+3)}\\
& \lesssim & (K_nC_n^{k_1})|w_{k_2}|\dots |w_{k_{j+1}}|\frac{\Pi_{i=1}^{j}\Gamma(\nu_\infty+k_i+3)}{\Gamma(\nu_\infty+m+3)}\frac{\Gamma(\nu_\infty+m+3)\Gamma(\nu_\infty+k_{j+1}+3)}{\Gamma(\nu_\infty+k-1+3)}.
\eee 

\noindent\underline{case $m\le k_{j+1}$} In this case, {we let}   
$$m=x(k-1),\ \ k_{j+1}=(1-x)(k-1),\ \ x\le \frac 12.$$ 
For $\delta<x\le \frac 12$, from \eqref{cneovneoneo} and the monotonicity of $\Phi_{k-1}^{(2)}$: $$\Phi_{k-1}^{(2)}(x)\ge\Phi_{k-1}^{(2)}(\delta) \ge \frac 12\delta |\log \delta|$$ Then, using \eqref{nnvnevonenoebis}, {\eqref{cneiocneonone}}, \eqref{convolutionbound}:
\bee
&&\frac{c_{\nu_\infty,a_\infty}}{1+k^2} \sum_{k_1+\dots+k_{j}=m}\sum_{\delta<x\leq \frac 12} (K_nC_n^{k_1})|w_{k_2}|\dots |w_{k_{j+1}}|\\
&\times& \frac{\Pi_{i=1}^{j}\Gamma(\nu_\infty+k_i+3)}{\Gamma(\nu_\infty+m+{3})}\frac{\Gamma(\nu_\infty+m+3)(\Gamma(\nu_\infty+k_{j+1}+3)}{\Gamma(\nu_\infty+k-1+3)}\\
& \lesssim & \frac{c_{\nu_\infty,a_\infty}}{1+k^2}\sum_{k_1+\dots+k_{j}=m}K_n^{j}C^{m}_n\frac{\Pi_{i=1}^{j}\Gamma(\nu_\infty+k_i+3)}{\Gamma(\nu_\infty+m+3)}\sum_{\delta<x\leq \frac 12}K_nC^{k_{j+1}}_ne^{-\frac{\delta|\log\delta|}{4}k}\\
& \leq&  \frac{c_{\nu_\infty,a_\infty}}{1+k^2}K_n^5C_n^{k-1}e^{-\frac{\delta|\log\delta|}{8}k}\sum_{k_1+\dots+{k_{j}}=m}\frac{\Pi_{i=1}^{j}\Gamma(\nu_\infty+k_i+3)}{\Gamma(\nu_\infty+m+3)}\leq  \frac{c_{\nu_\infty,a_\infty}}{1+k^{{2}}}K_n^5\left(C_ne^{-\frac{\delta|\log\delta|}{8}}\right)^k
\eee

For $\frac{R_\nu}{k}<x<\delta$, {with $R_\nu=R_{\nu_\infty}$ introduced in Lemma \ref{lowerouvndstilignohase}, we have in view of \eqref{cneovneoneo},}
$$\Phi_k^{(2)}(x)\ge \frac{x|\log x|}{2}\ge \frac{x|\log \delta|}{2}$$ 
and thus, {since $0<\de\ll C_0^{-1}$,}
$$\sqrt{1+m}C_n^{m}e^{-k\Phi_k^{(2)}(x)}\le \sqrt{1+m}e^{m\log C_n-\frac{m|\log \delta|}{4}}\le \sqrt{1+m}e^{m\log C_0-\frac{m|\log \delta|}{4}}\leq e^{-\frac{m|\log \delta|}{8}}$$ Using the bootstrap bounds \eqref{nnvnevonenoebis} and \eqref{cneivneoenvevnoebis} for $k_{j+1}\le k-1$:
\bee
&&\frac{c_{\nu_\infty,a_\infty}}{1+k^2} \sum_{k_1+\dots+k_{j}=m}\sum_{\frac{R_\nu}{k}<x\leq \delta} (K_nC_n^{k_1})|w_{k_2}|\dots |w_{k_{j+1}}|\\
&\times& \frac{\Pi_{i=1}^{j}\Gamma(\nu_\infty+k_i+3)}{\Gamma(\nu_\infty+m+3)}\frac{\Gamma(\nu_\infty+m+3)(\Gamma(\nu_\infty+k_{j+1}+3)}{\Gamma(\nu_\infty+k-1+3)}\\
& \leq &\frac{c_{\nu_\infty,a_\infty}}{1+k^2}\sum_{k_1+\dots+k_{j}=m}K_n^{j}C^{m}_n\frac{\Pi_{i=1}^{j}\Gamma(\nu_\infty+k_i+3)}{\Gamma(\nu_\infty+m+3)}\sum_{\frac{R_\nu}{k}<x\leq \delta}K_{n+1}\left(C_{n}e^{-\delta}\right)^{k_{j+1}}\frac{e^{-\frac{m|\log \delta|}{8}}}{C_n^m}\\
& \leq& \frac{c_{\nu_\infty,a_\infty}K_n^{5}K_{n+1}}{1+k^2}\left(C_{n}e^{-\delta}\right)^{k}\sum_{k_1+\dots+{k_{j}}=m}\frac{\Pi_{i=1}^{{j}}\Gamma(\nu_\infty+k_i+3)}{\Gamma(\nu_\infty+m+3)}\\
& \leq&  \frac{c_{\nu_\infty,a_\infty}K_n^{5}K_{n+1}}{1+k^2}\left(C_{n}e^{-\delta}\right)^{k}
\eee
where we used \eqref{convolutionbound} in the last step. For $x\le \frac{R_\nu}{k}$, there are $m{\leq} xk\le R_\nu$ terms in the sum and $k_j\leq m{\leq} kx\le R_\nu$ fixed. We therefore use $\Phi_k^{(2)}(x)\ge 0$ and the bootstrap bounds \eqref{neoonenoenvbis} and \eqref{cneivneoenvevnoebis} for $k_{j+1}\le k-1$ to estimate:
\bee
&&\frac{c_{\nu,a}}{1+k^2} \sum_{k_1+\dots+k_{j}=m}\sum_{x\le\frac{R_\nu}{k}} (K_nC_n^{k_1})|w_{k_2}|\dots {|w_{k_{j+1}}|}\\
&\times& \frac{\Pi_{i=1}^{j}\Gamma(\nu_\infty+k_i+2)}{\Gamma(\nu_\infty+m+3)}\frac{\Gamma(\nu_\infty+m+3)(\Gamma(\nu_\infty+k_{j+{1}}+3)}{\Gamma(\nu_\infty+k-1+3)}\\
& \leq & \frac{c_{\nu_\infty,a_\infty}K_{n+1}}{1+k^{{2}}}\left(C_{n}e^{-\delta}\right)^{k_{j+1}}\sum_{k_1+\dots+k_{j}=m}\frac{\Pi_{i=1}^{j}\Gamma(\nu_\infty+k_i+3)}{\Gamma(\nu+m+3)}\\
& \leq & \frac{c_{\nu,a}K_{n+1}}{1+k^{{2}}}\left(C_{n}e^{-\delta}\right)^{k}.
\eee
\noindent\underline{case $ k_{j+1}\le m$} In this case, we let $\delta>0$ as in Lemma \ref{lowerouvndstilignohase}, and let $$k_{j+1}=x(k-1),\ \ m=(1-x)(k-1),\ \ x\le \frac 12.$$ Since $k_1\le k_2\le \dots k_{j+1}$, we have $$k-1=m+k_{j+1}\le (j+1)k_{j+1}$$ and we are always in the range $\delta<\frac{1}{j+2}\le x\le \frac 12$. We then use verbatim the same chain of estimates as above: from \eqref{cneovneoneo} and the monotonicity of $\Phi_{k-1}^{(2)}$: $$\Phi_{k-1}^{(2)}(x)\ge\Phi_{k-1}^{(2)}(\delta) \ge \frac 12\delta |\log \delta|$$ Therefore, using \eqref{nnvnevonenoebis}, \eqref{convolutionbound}:
\bee
&&\frac{c_{\nu_\infty,a_\infty}}{1+k^2} \sum_{k_1+\dots+k_{j}=m}\sum_{\delta<x\leq \frac 12} (K_nC_n^{k_1})|w_{k_2}|\dots |w_{k_{j+1}}||w_{k_{j+1}}|\\
&\times& \frac{\Pi_{i=1}^{j}\Gamma(\nu_\infty+k_i+3)}{\Gamma(\nu_\infty+m+3)}\frac{\Gamma(\nu_\infty+m+3)(\Gamma(\nu_\infty+k_{j+2}+3)}{\Gamma(\nu_\infty+k-1+3)}\\
& \leq & \frac{c_{\nu_\infty,a_\infty}}{1+k^2}\sum_{k_1+\dots+k_{j}=m}K_n^{j+{1}}C^{k-1}_n\frac{\Pi_{i=1}^{{j}}\Gamma(\nu_\infty+k_i+3)}{\Gamma(\nu_\infty+m+3)}\sum_{\delta<x\leq \frac 12}e^{-\frac{\delta|\log\delta|}{4}k}\\
& \leq&  \frac{c_{\nu_\infty,a_\infty}}{1+k^2}{K_n^5}C_n^{k-1}e^{-\frac{\delta|\log\delta|}{8}k}\sum_{k_1+\dots+k_{j}=m}\frac{\Pi_{i=1}^{j}\Gamma(\nu_\infty+k_i+2)}{\Gamma(\nu_\infty+m+2)}\leq  \frac{c_{\nu_\infty,a_\infty}}{1+k^{{2}}}K_n^5\left(C_ne^{-\frac{\delta|\log\delta|}{8}}\right)^k
\eee

\noindent\underline{Conclusion} The collection of above bounds inserted into \eqref{upperboundsequence}  ensures:
$$|w_{k+1}|<|w_k|+\frac{c_{\nu_\infty,a_\infty} A^k}{\Gamma(k+1)}+  \frac{c_{\nu_\infty,a_\infty}K_n^{5}K_{n+1}}{1+k^2}\left(C_{n}e^{-\delta}\right)^{k}
$$
Then, provided that $k>k^*(K_n{,\nu,a})$ has been chosen large enough:
\bee
\frac{w_{k+1}}{K_{n+1}(C_ne^{-\delta})^{k+1}}&<&\frac{K_n}{K_{n+1}C_ne^{-\delta}}+\frac{c_{\nu,a}A^k}{\Gamma(k+1)K_{n+1}(C_ne^{-\delta})^{k+1}}+\frac{c_{\nu,a}K_n^{5}}{(1+k^2)(C_ne^{-\delta})}\\
& <& {\frac{1}{2}}+\frac{c_{\nu_\infty,a_\infty}K_n^{5}}{1+k^2}<1
\eee
and \eqref{cneivneoenvevnoebis} is proved.\\

\noindent{\bf step 5} {Boundedness. The initial bound \eqref{neoonenoenvbis} allows us to apply \eqref{nnvnevonenoebis}, \eqref{cneivneoenvevnoebis} iteratively a finite number of times}\footnote{{Recall that $C_{n+1}=e^{-\de}C_n$ so that $C_n=e^{-n\de}C_0$. Thus, we need $C_n=C_0e^{-n\de}=1$, i.e. it suffices to choose the number $n$ of iterations as
$$n(\nu_\infty,a_\infty)=\frac{1}{\de(\nu_\infty,a_\infty)}\log\big(C_0(\nu_\infty,a_\infty)\big)$$
where have adjusted the choice of $\de$ to ensure that the formula provides an integer value for $n$.}} {to obtain the bound for some large enough constant $C>1$ depending on $\nu_\infty$ and $a_\infty$: 
$$\forall k\ge 1, \ \ |w_k|<C$$
and \eqref{boundednessofthesquencebis} is proved.} {As a byproduct, we obtain
\bee
&&\left|\frac{a_\infty^kg_{k+1}}{\Gamma(\nu_\infty+k+1+3)}\right| =  |w_k+w_{k+1}|< \frac{c_{\nu_\infty,a_\infty}}{k+1}(h_0)_k\\
\nonumber &+& \frac{c_{\nu_\infty,a_\infty}}{1+k^2}\sum_{j=1}^4\sum_{k_1+\dots k_{j+1}=k-1}|h_{jk_1}w_{k_2}\dots w_{k_{j+1}}|\frac{\Pi_{i=1}^{j+1}\Gamma(\nu_\infty+k_i+3)}{\Gamma(k-1+\nu_\infty+3)}\\
\nonumber & + &  \frac{c_{\nu_\infty,a_\infty}}{1+k^2}\sum_{j=1}^{4}\sum_{k_1+\dots k_{j+2}=k-2}|\th_{jk_1}w_{k_2}\dots w_{k_{j+1}}w_{k_{j+2}}|\frac{\Pi_{i=1}^{j+2}\Gamma(\nu_\infty+k_i+3)}{\Gamma(k-2+\nu_\infty+3)}\\
&<& \frac{c_{\nu_\infty,a_\infty} A^k}{\Gamma(k+1)}+  \frac{c_{\nu_\infty,a_\infty}}{1+k^2}
\eee
and  \eqref{ceineneoeoneobis} follows as well.}
\end{proof}


\subsection{Proof of Proposition \ref{proposnvoivone}}


In view of \eqref{boundednessofthesquence} and the conjugation formula \eqref{condjiufgatioformula}, the bound \eqref{boundequenus} directly follows from the following continuity lemma.

\begin{lemma}[Continuity]
\label{cneinveineonevni}
Let $M(x)$ be holomorphic in a neighborhood of $0$, then there exists $c_M>0$ such that for all functions $\Phi$ which are $\mathcal C^\infty$ at the origin, 
\be
\label{vnoenvonveneovn}
\sup_{k\ge 0}\frac{a_\infty^k|(M\Phi)_k|}{\Gamma(k+\nu_\infty+2)}\leq c_M\sup_{k\ge 0}\frac{a_\infty^k|\phi_k|}{\Gamma(k+\nu_\infty+2)}.
\ee
\end{lemma}

\begin{proof} [Proof of Lemma \ref{cneinveineonevni}] Since $M$ is holomorphic in a neighborhood of the origin 
$$|m_k|={\frac{1}{k!}}\left|\frac{d^kM}{dx^k}(0)\right|\le C^k$$ 
and then
$$\sup_{k\ge 0}\frac{a_\infty^k|m_k|}{\Gamma(k+\nu_\infty+2)}\le c_M.$$ 
We estimate using \eqref{convolutionbound}:
\bee
&&\frac{a_\infty^k|(M\Phi)_k|}{\Gamma(k+\nu+2)}=a_\infty^k\sum_{k_1+k_2=k}\frac{|m_{k_1}\phi_{k_2}|}{\Gamma(k+\nu_\infty+2)}\\
&\leq&  \left(\sup_{k\ge 0}\frac{a_\infty^k|\phi_k|}{\Gamma(k+\nu_\infty+2)}\right)\left(\sup_{k\ge 0}\frac{a_\infty^k|m_k|}{\Gamma(k+\nu_\infty+2)}\right)\sum_{k_1+k_2=k}\frac{\Gamma(k_1+\nu_\infty+2)\Gamma(k_2+\nu_\infty+2)}{\Gamma(k+\nu_\infty+2)}\\
& \leq & c_M\sup_{k\ge 0}\frac{a_\infty^k|\phi_k|}{\Gamma(k+\nu_\infty+2)}
\eee
and \eqref{vnoenvonveneovn} is proved.
\end{proof}


\section{Bounding the Taylor series for \eqref{equaitoncompee}}
\label{bbounds}

We now start the study of the full problem \eqref{equaitoncompee}. We first aim at obtaining uniform in $b$ bounds on the coefficients of the Taylor series as well as the convergence to the limiting problem as $b\to 0$. We recall the notation $$\gamma-1=K+\alpha_\gamma, \ \ K\in \Bbb N^*, \ \ 0<\alpha_\gamma<1.$$


\subsection{$b$ dependent conjugation}


We conjugate the $b$ dependent problem \eqref{equaitoncompee} to an explicitly solvable (at the linear level) $b$-dependent equation.

\begin{lemma}[Linear conjugation]
\label{vbjbvbvebi}  There exist a function $\xit_b(x)$ which is holomorphic in a $b$-independent neighborhood of $x=0$ such that the conjugation
\be
\label{defphi}
\Psi(u)=M_b(x)\Phi(u),\ \ G=\frac{F}{M_b(1+H_{20})}
\ee
with 
\be
\label{deflkenrnal}
M_b(x)=e^{\xit_b(x)}
\ee
and 
\be
\label{defnub}
\left|\begin{array}{l}
\nu_b=\nu+\xi(b)\\
\nu=-\gamma b(\Dt_{11}+\Dt_{30}-\Et_{11})
\end{array}\right.
\ee
maps 
\be
\label{odetoinvertbis}\left[1+H_{20}\right]u(1-u)\Psi'+\left[(1-2u)(1+H_{20})-\gamma(1+G_1)+2uG_2\right]\Psi={-}uF
\ee
to
\be
\label{conjuguatedflrow}
\Phi'-\left[\frac{\gamma-1}{u}+\frac{\gamma+\nu_b+1}{1-u}\right]\Phi=-\frac{G}{1-u}.
\ee
\end{lemma}

\begin{remark} We see immediately the fundamental difference between \eqref{conjuguatedflrow} and \eqref{newproblem}: for the $b$-dependent problem, the point $u=0$ is a regular singular point, while it is a singular singular point for $b=0$. 
As a result, we will see a change in the behavior of the Taylor series for large frequencies, which will be reflected in the nature  of the weight $w_{\gamma,\nu}$, see Lemma \ref{lemmanvnowvowv}.
\end{remark}

\begin{proof}[Proof of Lemma \ref{vbjbvbvebi}] This is a direct computation.\\

\noindent{\bf step 1} Linear conjugation. We rewrite \eqref{odetoinvertbis}
\be
\label{vennoenoneo}
\Psi'-\frac{\zeta'_b}{\zeta_b}\Psi=-\frac{ F}{(1-u)(1+H_{2,0})}
\ee 
with
\bee
&&\frac{\zeta'_b}{\zeta_b}=\frac{\gamma(1+G_1)}{u(1-u)(1+H_{2,0})}-\frac{2G_2}{(1-u)(1+H_{2,0})}-\frac{1-2u}{u(1-u)}\\
& = & \frac{\gamma(1+G_1)}{u(1-u)(1+H_{2,0})}-\frac{2G_2}{(1-u)(1+H_{2,0})}-\frac{1}{u}+\frac{1}{1-u}
\eee
From \eqref{fneionefonone}:
$$
\left|\begin{array}{l}
G_1(x)=\Et_{11}x+x\tilde{G}_1\\
 \tilde{G}_1=-(2\Et_{02}+\Et_{21})x+2\Et_{12}x^2-3\Et_{03}x^3
 \end{array}\right.
$$
and 
$$\frac{\gamma(1+G_1)}{u(1-u)}=\frac{\gamma(1+\Et_{11}bu+bu\tilde{G}_1)}{u(1-u)}=\frac{\gamma}{u(1-u)}+\frac{\Et_{11}b\gamma}{1-u}+\frac{\gamma b \tilde{G}_1}{1-u}$$ yields
$$\frac{\zeta'_b}{\zeta_b}=\left[\frac{\gamma}{u(1-u)}+\frac{\Et_{11}b\gamma}{1-u}\right]\frac{1}{1+H_{2,0}}+\frac{\gamma b \tilde{G}_1-2G_2}{(1-u)(1+H_{2,0})}-\frac{1}{u}+\frac{1}{1-u}
.$$
We recall that $x=bu$ and rewrite
\bee
\frac{\gamma}{u(1-u)(1+H_{2,0})}&=&\frac{\gamma}{u(1-u)}\left[1-(\Dt_{11}+\Dt_{30})bu+\left(\frac{1}{1+H_{2,0}}-1+(\Dt_{11}+\Dt_{30})x\right)\right]\\
& =& \gamma\left(\frac{1}{u}+\frac{1}{1-u}\right)-\frac{\gamma b(\Dt_{11}+\Dt_{30})}{1-u}\\
&+& \frac{\gamma b}{1-u}\left[\frac{1}{x}\left(\frac{1}{1+H_{2,0}}-1+(\Dt_{11}+\Dt_{30})x\right)\right]
\eee
and 
\bee
\frac{\Et_{11}b\gamma}{(1-u)(1+H_{2,0})}=\frac{\Et_{11}b\gamma}{1-u}+\frac{\Et_{11}b\gamma}{1-u}\left[\frac{1}{1+H_{2,0}}-1\right].
\eee
We have therefore obtained the formula:
$$
\frac{\zeta'_b}{\zeta_b}= \gamma\left(\frac{1}{u}+\frac{1}{1-u}\right)-\frac{1}{u}+\frac{1}{1-u}
-\frac{\gamma b(\Dt_{11}+\Dt_{30})}{1-u}+\frac{\Et_{11}b\gamma}{1-u}+\frac{\xi(x)}{1-u}
$$
with
\bee
\xi(x)&=&\frac{\gamma b \tilde{G}_1-2G_2}{{1+H_{2,0}}}+   \frac{\gamma b}{x}\left(\frac{1}{1+H_{2,0}}-1+(\Dt_{11}+\Dt_{30})x\right)+\gamma b\Et_{11}\left(\frac{1}{1+H_{2,0}}-1\right).
\eee
We recall from \eqref{defparameterscneoevn}
$$
\nu=-\gamma b(\Dt_{11}+\Dt_{30}-\Et_{11}),
$$ and rewrite:
\be
\label{expressionzetab}
\frac{\zeta'_b}{\zeta_b}= \frac{\gamma-1}{u}+\frac{\gamma+\nu+1}{1-u}+\frac{\xi(x)}{1-u}.
\ee

\noindent{\bf step 2} Computation of the kernel. We now use the analyticity of $\xi$ in $|x|<\frac{1}{C}$ and $\xi(0)=0$ to compute:
\bee
&&\frac{\xi(x)}{1-u}=\frac{1}{1-u}\sum_{k=1}^{+\infty} \xi_kx^k=\frac{1}{1-u}\sum_{k=1}^{+\infty} \xi_kb^ku^k=  \frac{\sum_{k=1}^{+\infty} \xi_kb^k}{1-u}-\sum_{k=1}^{+\infty} \xi_kb^k\frac{1-u^k}{1-u}\\
& = & \frac{\xi(b)}{1-u}-\sum_{k=1}^{+\infty} \xi_kb^k\sum_{j=0}^{k-1}u^j.
\eee
Let $\nu_b$ be given by \eqref{defnub}, we have obtained:
$$\frac{\zeta'_b}{\zeta_b}=\frac{\gamma-1}{u}+\frac{\gamma+\nu_b+1}{1-u}-\sum_{k=1}^{+\infty} \xi_kb^k\sum_{j=0}^{k-1}u^j.$$
We compute a primitive of the remaining term:
\bee
\sum_{k=1}^{+\infty}\xi_kb^k\sum_{j=0}^{k-1}\frac{u^{j+1}}{j+1}=\sum_{j=0}^{+\infty}\frac{u^{j+1}}{j+1}\sum_{k=j+1}^{+\infty}\xi_kb^k=\sum_{j=1}^{+\infty}\frac{u^{j}}{j}\sum_{k=j}^{+\infty}\xi_kb^k=\sum_{j=1}^{+\infty}\tilde{\xi}_{b,j}b^ju^j=\sum_{j=1}^{+\infty}\tilde{\xi}_{b,j}x^j
\eee
with $$\tilde{\xi}_{b,}j=\frac{1}{j}\sum_{k=j}^{+\infty}\xi_kb^{k-j}=\frac{1}{j}\sum_{k=0}^{+\infty}\xi_{k+j}b^k.$$ The holomorphic bound $|\xi_j|\le C^j$ ensures
\be
\label{boundxit}
|\xit_{b,j}|\leq\frac1j\sum_{k=0}^{+\infty}b^kC^{k+j}\leq \frac{C^j}{j}\frac{1}{1-bC}\leq\frac{C^j}{j}
\ee
for $0<b<b^*$ universal small enough. We have therefore obtained the formula
\be
\label{reformulationkernela}
\frac{\zeta'_b}{\zeta_b}=\frac{\gamma-1}{u}+\frac{\gamma+\nu_b+1}{1-u}+\frac{d}{du}\xit_b(x)
\ee
where 
\be
\label{vnovoeonven}
\xit_b(x)={-}\sum_{j=1}^{+\infty}\tilde{\xi}_{b,j}x^j, \ \ |\xit_{b,j}|\le C^j
\ee
 is holomorphic in a neighborhood of $x=0$ independent of $b$.\\

\noindent{\bf step 3} Conclusion. From \eqref{vennoenoneo}, \eqref{reformulationkernela}:
$$\Psi'-\left[\frac{\gamma-1}{u}+\frac{\gamma+\nu_b+1}{1-u}+{\frac{d}{du}}\xit_b(x)\right]\Psi=-\frac{ F}{(1-u)(1+H_{20})}$$ 
and \eqref{conjuguatedflrow} follows.
\end{proof}

\subsection{The $b$-dependent discrete weight}


We study the discrete weight associated to \eqref{conjuguatedflrow}.

\begin{lemma}[Properties of the weight]
\label{lemmanvnowvowv}
Let 
\be
\label{defweight}
w_{\gamma,\nu}(k)=\frac{\Gamma(\gamma-1-k)\Gamma(\nu+k+2)}{\Gamma(\gamma-1)},\ \ k\in \Bbb N.
\ee
then:\\ 
\begin{enumerate}
\item  value for $k\ge K$: $\forall j\ge 0$, 
\be
\label{cnineneonevnev:0}
w_{\gamma,\nu_b}(K+j)=(-1)^j\Gamma(\alpha_\gamma)\Gamma(1-\alpha_\gamma)\frac{\Gamma(K+j+\nu_b+2)}{\Gamma(K+\alpha_\gamma)\Gamma(j+1-\alpha_\gamma)}.
\ee
\item $\nu$ dependence:
\be
\label{cneioneineonoen}
\forall  k\ge 0, \ \ \frac{\gamma w_{\gamma,\nu_b}(k+1)}{w_{\gamma-1,\nu_b+1}(k)}=\frac{\gamma}{\gamma-2}
\ee
\item induction property:
\be
\label{inucnoitnwrmwjaggmamk}
\forall k\ge 1, \ \ (\gamma-k-2)w_{\gamma-1,\nu_b+1}(k)-(k+\nu_b+2)w_{\gamma-1,\nu_b+1}(k-1)=0.
\ee
\end{enumerate}
\end{lemma}

\begin{proof} {\eqref{cnineneonevnev:0}} directly follows from \eqref{defweight}, \eqref{formulafdebasenegatif}.
\\
We then compute for {$k\ge 0$}:
\bee
\frac{\gamma w_{\gamma,\nu_b}(k+1)}{w_{\gamma-1,\nu_b+1}(k)}=\gamma\frac{\Gamma(\gamma-1-(k+1))\Gamma(k+1+\nu_b+2)}{\Gamma(\gamma-1)\frac{\Gamma(\gamma-2-k)\Gamma(k+\nu_b+3)}{\Gamma(\gamma-2)}}{=\frac{\gamma \Gamma(\gamma-2)}{\Gamma(\gamma-1)}}=\frac{\gamma}{\gamma-2}
\eee
and \eqref{cneioneineonoen} is proved. We now turn to the induction formula: for {$k\ge 1$}, 
\bee
&&(\gamma-k-2)w_{\gamma-1,\nu_b+1}(k)-(k+\nu_b+2)w_{\gamma-1,\nu_b+1}(k-1)\\
& = & (\gamma-k-2)\frac{\Gamma(\gamma-2-k)\Gamma(k+\nu_b+1+2)}{\Gamma(\gamma-2)}\\
&-&(k+\nu_b+2)\frac{\Gamma(\gamma-2-(k-1))\Gamma(k-1+\nu_b+1+2)}{\Gamma(\gamma-2)}\\
& = & \frac{\Gamma(\gamma-k-1)\Gamma(k+\nu_b+3)}{\Gamma(\gamma-{2})}-\frac{\Gamma(\gamma-k-1)\Gamma(k+\nu_b+3)}{\Gamma(\gamma-{2})}=0.
\eee
\end{proof}


\subsection{Boundedness of the sequence}


We claim the following $b$-dependent nonlinear bound which, in a certain sense,  is a deformation of
 \eqref{boundequenus}.

\begin{proposition}[$b$-dependent boundedness]
\label{propboundbdependent}
There exist universal constants $c_{\nu,a}>0$ and $0<b^*\ll 1$ such that the following holds for all $0<b<b^*$.
Let $\Psi$ be a solutions of  \eqref{equaitoncompee} and define $$\psi_k=\frac{1}{k!}\frac{d^k\Psi}{du^k}(0),$$ then 
\be
\label{bounduniformb}
\forall \,0\le k\le K, \ \ |\psi_k|\leq c_{\nu,a}w_{\gamma,\nu_b}(k).
\ee
Moreover, let $\Psit^\infty$ be the unique local $\mathcal C^\infty$ solution to the limiting problem \eqref{limitingxequation} and 
$$\psit_k^\infty=\frac{1}{k!}\frac{d^k{\Psit}^\infty}{dx^k}(0),$$ 
then 
\be
\label{limitignprocedure}
\forall k\ge 0, \ \ \lim_{b\to 0 }\frac{\psi_k}{b^k}=\psit_k^\infty.
\ee
\end{proposition}

The rest of this section is devoted to the proof of Proposition \ref{propboundbdependent}.


\subsection{Proof of \eqref{bounduniformb} for frequencies $k\ll \gamma$}


We claim the following small frequency bound.

\begin{lemma}[Uniform small frequency bound]
\label{boundsmallk}
The limit \eqref{limitignprocedure} holds. Moreover, there exists $c_{\nu,a}>0$ such that for all $k^*\ge 0$, there exists $0<b^*(k^*)\ll1$ such that 
\be
\label{smallfrequencybound}
\forall 0<b<b^*(k^*), \ \ \forall 0\le k\le k^*, \ \ |\psi_k|\leq c_{\nu,a}w_{\gamma,{\nu_b}}(k).
\ee
\end{lemma}

\begin{proof}[Proof of Lemma \ref{boundsmallk}]
Recall $$\Psi(u)=\Psit(x), \ \ x=bu$$ so that  $$\frac{\psi_k}{b^k}=\frac{1}{k!b^k}\frac{d^k\Psi}{du^k}(0)=\frac{1}{k!}\frac{d^k\Psit}{dx^k}(0)=\psit_k$$ and \eqref{limitignprocedure} is equivalent to $$\lim_{b\to 0}\psit_k=\psit_k^\infty.$$ This follows immediately by passing to the $b\to 0$ limit for fixed $k$ in the induction relation for the $\psit_k$ from \eqref{equationtowrikwith}. The details are straightforward and left to the reader. {Similarly, note that we also have
\be\label{eq:limiteofgkonbktogkinfty}
\lim_{b\to 0}\frac{g_k}{b^k} = g_k^\infty.
\ee}

We conclude from \eqref{boundequenus} that there exists $c_{\nu,a}$ such that 
\be
\label{vneiovneneonvone}
\forall k\ge 0, \ \ |\psit^\infty_k|\le c_{\nu,a}\frac{\Gamma(\nu+k+2)}{a^k}.
\ee Pick now an arbitrary $k^*\ge 0$, then from \eqref{limitignprocedure}, \eqref{vneiovneneonvone}: 
$$\forall 0<b<b^*(k^*), \ \ \forall 0\le k\le k^*, \ \ |\psit_k|\leq 2c_{\nu,a}\frac{\Gamma(\nu+k+2)}{a^k}.$$ 
We now estimate using \eqref{lowkestimate}:
\bee
\frac{|\psi_k|}{w_{\gamma, {\nu_b}}(k)} &\leq& c_1\frac{|\psi_k|\gamma^k}{\Gamma(\nu_{{b}}+k+2)}=c_1\frac{|\psit_k|b^k\gamma^k}{\Gamma(\nu_{{b}}+k+2)}=c_1\frac{|\psit_k|a^k}{\Gamma(\nu_{{b}}+k+2)}\\
&\leq&  2c_1c_{\nu,a}{\frac{\Gamma(\nu+k+2)}{\Gamma(\nu_{{b}}+k+2)}}.
\eee
{Since $\nu_b\to\nu$ as $b\to 0$, we may choose $b^*(k^*)$ small enough so that
$$\forall 0<b<b^*(k^*), \ \ \forall 0\le k\le k^*, \ \ \frac{\Gamma(\nu+k+2)}{\Gamma(\nu_b+k+2)}\leq 2.$$ 
Then,
\bee
\frac{|\psi_k|}{w_{\gamma, {\nu_b}}(k)} \leq  4c_1c_{\nu,a}
\eee}
and \eqref{smallfrequencybound} is proved.
\end{proof}


\subsection{Nonlinear conjugation}


The proof of \eqref{bounduniformb} now requires a careful track of the $b$-dependencies  in the full problem  \eqref{equaitoncompee}. The first step is to use Lemma \ref{vbjbvbvebi} and analyze the nonlinear conjugated problem.\\

\noindent{\bf step 1} Nonlinear conjugation. Recall \eqref{equaitoncompee}:
\bea
\label{nenvneonoenvi}
\nonumber  &&\left[1+H_2+G_2\Psi+\NLt_2\right]u(1-u)\Psi'\\
\nonumber &+& \left[(1-2u)(1+H_2+G_2\Psi+\NLt_2)-\gamma(1+G_1)+2uG_2\right]\Psi\\
\nonumber & = &u\left[\gamma bH_1-2(1+H_2)+\frac{\gamma b\NLt_1}{x}-2\NLt_2\right]\\
\nonumber & \Leftrightarrow& \left[1+H_{20}\right]u(1-u)\Psi'+\left[(1-2u)(1+H_{20})-\gamma(1+G_1)+2uG_2\right]\Psi\\
& = & u\mathcal F
\eea
with recalling $\frac{1}{u}=\frac{b}{x}$: 
\bea
\label{defmahtclf}
&& \mathcal F= \gamma bH_1-2(1+H_2)\\
\nonumber &+& \frac{\gamma b\NLt_1}{x}-2\NLt_2-(1-u)\left[\sum_{j=1}^3b^jH_{2j}(x)+G_2\Psi+\NLt_2\right]\Psi'\\
\nonumber& + &\left\{-\frac{b}{x}(G_2\Psi+\NLt_2)+2(G_2\Psi+\NLt_2)-\frac{b(1-2u)}{x}\sum_{j=1}^3b^jH_{2j}(x)\right\}\Psi
\eea
From \eqref{defphi}, \eqref{conjuguatedflrow} we now obtain the nonlinear conjugated problem:
\be
\label{cneineionoeno}
\Phi'-\left[\frac{\gamma-1}{u}+\frac{\gamma+\nu_b+1}{1-u}\right]\Phi=-\frac{\mathcal G}{1-u}
\ee

\noindent{\bf step 2} Computation of $\mathcal G$. We plug \eqref{defphi} into \eqref{defmahtclf} and decompose 
\be
\label{cneineinveoneonnkenpe}
\mathcal G=\mathcal G_0+\mathcal L(\Phi)+\NL(\Phi)
\ee 
as follows.\\

\noindent\underline{Source term}. We have 
$$\mathcal G_0(b,x){=\frac{\gamma b H_1-2(1+H_2)}{M_b(1+H_{20})}}$$ 
which, from \eqref{vnovoeonven}, admits  a holomorphic expansion in a neighborhood (independent of $b$) of $x=0$  i.e., 
\be
\label{neinineonenoenv:avoidmultipledefinedlabel}
\mathcal G_{{0}}(x)=\sum_{k=0}^{+\infty}g_{0k}x^k
\ee 
for some $b$-dependent coefficients $g_{0k}$ satisfying  
\be
\label{neinineonenoenv}
|g_{0k}|\leq C^k
\ee
for some $C>0$ independent of $b$.

\noindent\underline{Small linear term}. It is given explicitly by
\bea
\label{formulalphi}
\nonumber \mathcal L(\Phi)&=&\frac{1}{M_b(1+H_{20})}\left\{{-}(1-u)\left[\sum_{j=1}^3b^jH_{2j}(x)\right](M_b(x)\Phi)'-\frac{b(1-2u)}{x}\sum_{j=1}^3b^jH_{2j}(x)M_b\Phi\right\}\\
\nonumber & = & \frac{1}{M_b(1+H_{20})}\left\{\left[{-}bM_{{b}}'(x)(1-u)\sum_{j=1}^3b^jH_{2j}(x)-\sum_{j=1}^3\left(\frac{b^{j+1}H_{2j}(x)}{x}-2b^jH_{2j}\right)\right]\Phi\right.\\
&{-}& \left.(1-u)M_b(x)\left[\sum_{j=1}^3b^jH_{2j}(x)\right]\Phi'\right\}
\eea
We rewrite $$(1-u)\Phi'=\frac{1-u}{u}u\Phi'=\left(\frac bx-1\right)u\Phi'$$ and obtain
\bee
\mathcal L(\Phi)& = & \frac{1}{M_b(1+H_{20})}\left\{\left[{-}M_{{b}}'(x)(b-x)\sum_{j=1}^3b^jH_{2j}(x)-\sum_{j=1}^3\left(\frac{b^{j+1}H_{2j}(x)}{x}-2b^jH_{2j}\right)\right]\Phi\right.\\
&{-}& \left.M_b(x)\sum_{j=1}^3\left(b^{j+1}\frac{H_{2j}(x)}{x}-b^jH_{2j}(x)\right)u\Phi'\right\}
\eee
Therefore,
$$\mathcal L(\Phi)=b\left[(bh_1(x)+xh_2(x))\Phi+(bh_3{(x)}+xh_4{(x)})u\Phi'\right]$$ where $$h_j(x)=\sum_{k=0}^{+\infty}h_{jk}x^k, \ \ |h_{jk}|\leq C^k$$ for some $C>0$ independent of b. 

\noindent\underline{Nonlinear term}. We have 
\bea
\label{formaulnonlienaterm}
&&\NL(\Phi)\\
\nonumber &=& \frac{1}{M_b(1+H_{20})}\left\{\frac{\gamma b\NLt_1}{x}-2\NLt_2-(1-u)\left[G_2\Psi+\NLt_2\right]\Psi'+(G_2\Psi+\NLt_2)\left(2-\frac bx\right)\Psi\right\}.
\eea
We rewrite
$$(1-u)\left[G_2\Psi+\NLt_2\right]\Psi'=\frac{1-u}u\left[G_2\Psi+\NLt_2\right]u\Psi'=\left(\frac{b}{x}-1\right)\left[G_2\Psi+\NLt_2\right]u\Psi'$$
and thus
 $\NL(\Phi)$ is given, structurally, by $$\NL(\Phi)=x\sum_{j=2}^4m^{(1)}_j\Phi^j+b\sum_{j=2}^4m^{(2)}_j\Phi^j+\left[x\sum_{j=1}^3m^{(3)}_j\Phi^j+b\sum_{j=1}^3m^{(4)}_j\Phi^j\right]u\Phi'$$
with $$m_j^{(\ell)}(x)=\sum_{k=0}^{+\infty}m_{jk}^{(\ell)}x^k, \ \ |m_{jk}^{(\ell)}|\le C^k.$$

\noindent\underline{Conclusion}. We have obtained the conjugated nonlinear problem \eqref{cneineionoeno} with
\bea
\label{vneivnoneoneonnee}
 \matchal G& = & \mathcal G_0+b\left[(bh_1(x)+xh_2(x))\Phi+(bh_3{(x)}+xh_4{(x)})u\Phi'\right]\\
\nonumber&+& x\sum_{j=2}^4m^{(1)}_j\Phi^j+b\sum_{j=2}^4m^{(2)}_j\Phi^j+ \left[x\sum_{j=1}^3m^{(3)}_j\Phi^j+b\sum_{j=1}^3m^{(4)}_j\Phi^j\right]u\Phi'.
\eea

\noindent{\bf step 3} Final change of variables. We now let 
\be
\label{suihfoenioeneoi}
\Phi=bu\Theta=x\Theta
\ee
 so that \eqref{cneineionoeno} becomes:
 \bea
 \label{thetaeqaiotjoihs}
 \nonumber &&\Phi'-\left[\frac{\gamma-1}{u}+\frac{\gamma+\nu_b+1}{1-u}\right]\Phi=-\frac{\mathcal G}{1-u}\\
  \nonumber&\Leftrightarrow& bu\Theta'+b\Theta -\left[\frac{\gamma-1}{u}+\frac{\gamma+\nu_b+1}{1-u}\right]bu\Theta=-\frac{\mathcal G}{1-u}\\
  \nonumber &\Leftrightarrow&\Theta'-\left[\frac{\gamma-2}{u}+\frac{\gamma+\nu_b+1}{1-u}\right]\Theta=-\frac{\mathcal G}{bu(1-u)}\\
  &\Leftrightarrow& u(1-u)\Theta' -\left[\gamma-2+(\nu_b+3)u\right]\Theta=-\frac{\mathcal G}{b}.
 \eea
 We now express $\mathcal G$ in terms of $\Theta$ and track the orders of vanishing in $x$. We compute
 $$u\Phi'=u[bu\Theta'+b\Theta]=x(u\Theta'+\Theta).$$ 
 Then,
 \bee
 && b\left[(bh_1(x)+xh_2(x))\Phi+(bh_3{(x)}+xh_4{(x)})u\Phi'\right]\\
 &=&\left[b^2h_1+bxh_2\right]x\Theta+(b^2h_3+bxh_4)[ux\Theta'+x\Theta]\\
 &=& \left[b^2x(h_1+h_3)+bx^2(h_2+h_4)\right]\Theta+\left[b^2xh_3+x^2h_4\right]u\Theta'\\
  & = & \left[b^2x\th_1+bx^2\th_2\right]\Theta+\left[b^2x\th_3+bx^2\th_4\right]u\Theta'.
   \eee
 
 For the nonlinear term:
 \bee
 && x\sum_{j=2}^4m^{(1)}_j\Phi^j+b\sum_{j=2}^4m^{(2)}_j\Phi^j+ \left[x\sum_{j=1}^3m^{(3)}_j\Phi^j+b\sum_{j=1}^3m^{(4)}_j\Phi^j\right]u\Phi'\\
 & = & x\sum_{j=2}^4x^jm^{(1)}_j\Theta^j+b\sum_{j=2}^4m^{(2)}_jx^j\Theta^j+\left[x\sum_{j=1}^3m^{(3)}_jx^j\Theta^j+b\sum_{j=1}^3m^{(4)}_jx^j\Theta^j\right](xu\Theta'+x\Theta)\\
 & = & \left[\sum_{j=2}^4x^{j+1}m^{(1)}_j\Theta^j+\sum_{j=1}^3m^{(3)}_jx^{j+2}\Theta^{j+1}+b\sum_{j=2}^4m^{(2)}_jx^j\Theta^j+b\sum_{j=1}^3m^{(4)}_jx^{j+1}\Theta^{j+1}\right]\\
 & + & \left[\sum_{j=1}^3m^{(3)}_jx^{j+2}\Theta^j+b\sum_{j={1}}^{{3}}m^{(4)}_jx^{j{+1}}\Theta^j\right](u\Theta')\\
 & = & \sum_{j={2}}^{{4}}x^{j+{1}}\mt^{(1)}_j\Theta^{{j}}+b\sum_{j=2}^4m^{(2)}_jx^j\Theta^j+  \left[\sum_{j=1}^3\mt^{(3)}_jx^{j+2}\Theta^j+b\sum_{j={1}}^{{3}}\mt^{(4)}_jx^{j{+1}}\Theta^j\right](u\Theta').
 \eee
 We now rewrite
 \bea
 \label{fialfrmaulg}
 &&\matchal G= \mathcal G_0+\left[b^2x\th_1+bx^2\th_2\right]\Theta+\left[b^2x\th_3+bx^2\th_4\right]u\Theta'\\
 \nonumber & + &  \sum_{j=2}^4x^{j+1}\mt^{(1)}_j\Theta^j+b\sum_{j=2}^4m^{(2)}_jx^j\Theta^j+  \left[\sum_{j=1}^3\mt^{(3)}_jx^{j+2}\Theta^j+b\sum_{j={1}}^{{3}}\mt^{(4)}_jx^{j{+1}}\Theta^j\right](u\Theta')
 \eea
 where, for some large enough universal constant $C=C_{\nu,a}>0$ independent of $b<b^*$ and all $k\ge 0$,
 \be
 \label{nveioneionoenv|}
{|(\mathcal{G}_0)_k|+|(\th_l)_k|+|(\mt^{(l)}_j)_k|}\le C^k b^k
 \ee


\subsection{Bounding the sequence $\theta_k$ and proof of Proposition \ref{propboundbdependent}}


We let $$\theta_k=\frac{1}{k!}\frac{d^k\Theta}{du^k}(0), \ \ g_k=\frac{1}{k!}\frac{d^k\mathcal G}{du^k}(0)$$ so that from \eqref{suihfoenioeneoi}: 
\be
\label{cneovnenvonveoinven}
\left|\begin{array}{l}\phi_0=0\\
\phi_k=b\theta_{k-1}, \ \ k\ge 1.
\end{array}\right.
\ee

\begin{lemma}[Boundedness for the $\theta_k$ sequence]
\label{lemmaboundthetak}
There exists $c_{\nu,a}>0$ and $b^*(\nu,a)$ such that for all $0<b<b^*$, for all $0\le k\le K-1$,
\be
\label{esitmaitmitot}
|\theta_k|\leq |w_{\gamma-1,\nu_b+1}(k)|
\ee
(This implies \eqref{bounduniformb}.) Moreover, 
\be
\label{boundednessgk}
\forall 0\le k\le K, \ \ |g_k|\le c_{\nu,a}\frac{|w_{\gamma,\nu_ b}(k)|}{1+k}.
\ee
\end{lemma}

\begin{proof}[Proof of Lemma \ref{lemmaboundthetak}]  This is a direct consequence of the form of $\mathcal G$ given in {\eqref{fialfrmaulg} \eqref{nveioneionoenv|}.}

\vspace{0.2cm}

\noindent{\bf step 1} Small frequency universal bound.

\begin{lemma}[Stability by multiplication]
\label{lemmamultiplication}
 Let $h(u)=\sum_{k=0}^{+\infty}b^kh_ku^k$ with the holomorphic bound $$|h_k|\leq C^k.$$ Then there exists $C_h$ and $0<b^*(C_h)\ll 1$ such that for all {$0<b<b^*(C)$ and any $0\le k^*\le K$}, 
 \be
 \label{eoneonveonoe}\max_{0\le k{\le}k^*}\frac{|(h\phi)_k|}{w_{\gamma, {\nu_b}}(k)}\leq C_h \max_{0\le k{\le}k^*}\frac{|\phi_k|}{w_{\gamma, {\nu_b}}(k)}.
 \ee
\end{lemma}

\begin{proof}[Proof of Lemma \ref{lemmamultiplication}] First observe that \eqref{vebibvebeibev} implies 
the lower bound
\be
\label{loweroubndwojdifo}
\forall 0\le k\le K, \ \ w_{\gamma,\nu_b}(k)=\frac{\Gamma(\gamma-1-k)}{\Gamma(\gamma-1)}\Gamma({\nu_b}+k+2)\ge \frac{\Gamma({\nu_b}+k+2)}{(\gamma-1)^k}
\ee
which then gives  the upper bound\footnote{{We use in particular $\nu_b>0$ and the fact that $\Gamma$  is increasing on $[2,+\infty)$.}}:  
\be
\label{nbioeoeeneo}
\frac{b^{k}C^{k}}{w_{\gamma,\nu_b}(k)}\le \frac{(bC(\gamma-1))^{k}}{\Gamma({\nu_b}+k+2)}\le \frac{(2Ca)^{k}}{\Gamma(k+{\nu_b}+2)}\leq \frac{c(C)}{1+k^2}.
\ee
for all $0\le k\le K$.
Therefore, 
\be
\label{vneoneonveonjon4oin4voe}
b^{k}\frac{|h_{k}|}{|w_{\gamma, {\nu_b}}(k)|}\le \frac{{b^kC^k}}{|w_{\gamma,\nu_b}(k)|}\le \frac{C_h}{1+k^2}.
\ee
for all $0\le k\le K$.
We then estimate using {\eqref{tobeprovoeonor}}, for $k{\le} k^*\le K$:
\bee
\left|\frac{(h\phi)_k}{w_{\gamma, {\nu_b}}(k)}\right|&=&\frac{1}{w_{\gamma, {\nu_b}}(k)}\left|\sum_{k_1+k_2=k}b^{k_1}h_{k_1}\phi_{k_2}\right|\\
&\leq& \sum_{k_1+k_2=k}b^{k_1}\frac{|h_{k_1}|}{w_{\gamma, {\nu_b}}(k_1)}\frac{|\phi_{k_2}|}{w_{\gamma, {\nu_b}}(k_2)}\frac{w_{\gamma, {\nu_b}}(k_1)w_{\gamma, {\nu_b}}(k_2)}{w_{\gamma, {\nu_b}}(k)}\\
& \leq &  C_h\max_{0\le k<k^*}\frac{|\phi_k|}{w_{\gamma, {\nu_b}}(k)}\sum_{k_1+k_2=k}\frac{w_{\gamma, {\nu_b}}(k_1)w_{\gamma, {\nu_b}}(k_2)}{w_{\gamma, {\nu_b}}(k)}\le C_h\max_{0\le k<k^*}\frac{|\phi_k|}{w_{\gamma, {\nu_b}}(k)}
\eee
and \eqref{eoneonveonoe} is proved.
\end{proof}

We conclude from \eqref{defphi}, \eqref{smallfrequencybound} that there exists $c_{\nu,a}$ such that forall $k^*>1$, {there exists $0<b^*(k^*)\ll 1$ such that for} $0<b<b^*(k^*)$ and $0\le k\le k^*$,
\be
\label{vounoeoineohik}
|\phi_k|\leq c_{\nu,a}w_{\gamma, {\nu_b}}(k).
\ee
Thus, from \eqref{cneovnenvonveoinven}:
$$
|\theta_k|=\frac 1b|\phi_{k+1}|\leq \frac{\gamma}{a}c_{\nu,a}w_{\gamma,\nu_b}(k+1).$$
{Together with \eqref{cneioneineonoen}, we deduce} the existence of $M_{\nu,a}$ such that for all $k^*\ge 1$, for all ${0<}b<b^*(k^*)$,
\be
\label{cneionienoen}
\forall 0\le k\le k^*, \ \ \frac{|\theta_{k}|}{w_{\gamma-1,\nu_b+1}(k)}\le M_{\nu,a}.
\ee

\noindent{\bf step 2} Induction relation. We now compute from \eqref{thetaeqaiotjoihs} the induction relation satisfied by derivatives at the origin. We formally expand $$\Theta=\sum_{k=0}^{+\infty}\theta_ku^k, \ \ \mathcal G=\sum_{k=0}^{+\infty}g_ku^k$$ and obtain from \eqref{thetaeqaiotjoihs}:
\bea
\label{vneineionvioenvionoe}
\nonumber &&u(1-u)\Theta' -\left[\gamma-2+(\nu_b+3)u\right]\Theta=-\frac{\mathcal G}{b}\\
\nonumber& \Leftrightarrow&\sum_{k=1}^{+\infty}(u-u^2)k\theta_ku^{k-1}-(\gamma-2)\sum_{k=0}^{+\infty}\theta_ku^k-(\nu_b+3)\sum_{k=0}^{+\infty}\theta_ku^{k+1}=-\frac{1}{b}\sum_{k=0}^{+\infty}g_ku^{k}\\
\nonumber& \Leftrightarrow& \sum_{k=1}^{+\infty}k\theta_ku^k-\sum_{k=2}^{+\infty}(k-1)\theta_{k-1}u^k-(\gamma-2)\sum_{k=0}^{+\infty}\theta_ku^k-(\nu_b+3)\sum_{k=1}^{+\infty}\theta_{k-1}u^{k}=-\frac{1}{b}\sum_{k=0}^{+\infty}g_{k}u^{k}\\
\nonumber& \Leftrightarrow&
\left|\begin{array}{l}
-(\gamma-2)\theta_0=-\frac{g_0}{b}\\
(k-\gamma+2)\theta_k-(k+\nu_b+2)\theta_{k-1}=-\frac{g_{k}}{b}, \ \ k\ge 1
\end{array}\right.\\
& \Leftrightarrow&
\left|\begin{array}{l}
\theta_0=\frac{\gamma}{a(\gamma-2)}g_0\\
(\gamma-k-2)\theta_k+(k+\nu_b+2)\theta_{k-1}=\frac{g_{k}}{b}, \ \ k\ge 1
\end{array}\right.
\eea
Let $$\zeta_k=\frac{\theta_k}{w_{\gamma-1,\nu_b+1}(k)},$$ then \eqref{inucnoitnwrmwjaggmamk},  \eqref{cneioneineonoen} yield:

\bea
\label{eboebveiboebeo}
\nonumber && (\gamma-k-2)w_{\gamma-1,\nu_b+1}(k)\zeta_k+(k+\nu_b+2)w_{\gamma-1,\nu_b+1}(k-1)\zeta_{k-1}=\frac{g_k}{b}\\
\nonumber &\Leftrightarrow& (k+\nu_b+2)w_{\gamma-1,\nu_b+1}(k-1)(\zeta_k+\zeta_{k-1})=\frac{\gamma g_k}{a}\\
&\Leftrightarrow& \zeta_k+\zeta_{k-1}=\frac{\gamma}{a(\gamma-2)}\frac{g_k}{(k+\nu_b+2)w_{\gamma,\nu_b}(k)}.
\eea

\noindent{\bf step 3} Bootstrap bound. Let $M_{\nu,a}$ be the universal constant in \eqref{cneionienoen}, we now bootstrap the bound 
\be
\label{bootbound}
\forall k^*\le k\le {K}-1, \ \ |\zeta_k|\le 2M_{\nu,a}
\ee
which, by \eqref{cneionienoen}, holds for $0\le k<k^*$ arbitrarily large and $0<b<b^*(k^*)$. We argue by induction, assuming the claim for $0\le j\le k_{\rm ind}-1$ and proving it for $k_{{\rm ind}}$. We claim the following crucial
nonlinear bound for  $g_k$
\be
\label{boundgk}
\forall 0\le k\le k_{\rm ind},  \ \ \frac{|g_k|}{|w_{\gamma,\nu_b}(k)|}\le \frac{c_{\nu,a}{M_{\nu,a}^4}}{1+k}.
\ee

\noindent\underline{Source term}. The estimates \eqref{nbioeoeeneo} and \eqref{nveioneionoenv|} yield the 
following uniform in $b,k\le K$:
$$\frac{|(\mathcal G_0)_k|}{w_{\gamma, {\nu_b}}(k)}\leq \frac{(bC)^k}{w_{\gamma, {\nu_b}}(k)}\leq\frac{ c_{\nu,a}}{1+k}.$$ 

\noindent\underline{Small linear term}. Let $1\le j\le 4$ and $j\leq{k\le k_{\rm ind}}$, then from {\eqref{cneionienoen}} \eqref{bootbound}:
\be\label{eq:firstverysimpleconsequenceofbootforzetatothetakmj}
{|\theta_{k-j}|\le M_{\nu,a}w_{\gamma-1,\nu_b+1}(k-j)}
\ee
We now observe the relation for $0\le j\le 4$ from \eqref{cneioneineonoen}:
$$\frac{b^j|w_{\gamma-1,\nu_b+1}(k-j)|}{|w_{\gamma,\nu_b}(k)|}=\frac{b^j(\gamma-2)|w_{\gamma,\nu_b}(k-j+1)|}{|w_{\gamma,\nu_b}(k)|}\le c_{\nu,a}\frac{b^{j-1}|w_{\gamma\nu_b}(k-(j-1))|}{|w_{\gamma,\nu_b}(k)|}
$$
and from \eqref{inucnoitnwrmwjaggmamk} for ${0\le j\le 4}$:
\bee
w_{\gamma,\nu_b}(k) &=&\frac{k+\nu_b+1}{\gamma-k-1}w_{\gamma,\nu_b}(k-1)=\frac{k-1+\nu_b+2}{\gamma-2-(k-1)}w_{\gamma,\nu_b}(k-1)\\
&=&w_{\gamma,\nu_b}(k-{j})\Pi_{{m}=1}^{{j}}\frac{k-{m}+\nu_b+2}{\gamma-2-{(k-m)}}
\eee
which leads to the estimate:
\be
\label{vnkndknlknvlvnlen}
\frac{|w_{\gamma,\nu_b}(k-m)|}{|w_{\gamma,\nu_b}(k)|}\leq \Pi_{j=1}^m\frac{\gamma-2-{k+j}}{|k-j+\nu_b+2|}\leq c_{\nu} \frac{\gamma^{m}}{(1+k)^m}
\ee
and then
\be
\label{kebibjbiubbi4}
\frac{b^j|w_{\gamma-1,\nu_b+1}(k-j)|}{|w_{\gamma,\nu_b}(k)|}\leq c_{\nu,a}\frac{b^{j-1}|w_{\gamma,\nu_b}(k-(j-1))|}{|w_{\gamma,\nu_b}(k)|}\leq c_{\nu,a}\frac{a^{j-1}}{(1+k)^{j-1}}.
\ee
Moreover, from  {\eqref{nbioeoeeneo}}:
\be
\label{contoneohktemrs}
\frac{b^{k}C^{k}}{|w_{\gamma-1,\nu_b+1}(k)|}\le c_{\nu,a}.
\ee
We estimate the worst term using {\eqref{nveioneionoenv|}, \eqref{eq:firstverysimpleconsequenceofbootforzetatothetakmj}}, \eqref{kebibjbiubbi4}, \eqref{contoneohktemrs}, \eqref{tobeprovoeonor}: for $k\le k_{\rm ind}$,
\bee
&&\left|\left(bx^2\th_4u\Theta'\right)_k\right|=bb^2\left|\left(u^2\th_4u\Theta'\right)_k\right|=bb^2\left|\left(\th_4u\Theta'\right)_{k-2}\right|=bb^2\left|\sum_{k_1+k_2=k-2}(\th_4)_{k_1}(k_2\theta_{k_2})\right|\\
& \le & bb^2 kc_{\nu,a}M_{\nu,a}\sum_{k_1+k_2=k-2}|w_{\gamma-1,\nu_b+1}(k_1)w_{\gamma-1,\nu_b+1}(k_2)|\leq b kc_{\nu,a}M_{\nu,a}b^2w_{\gamma-1,\nu_b+1}(k-2)\\
&\leq& \frac{c_{\nu,a}M_{\nu,a}{|w_{\gamma, \nu_b}|}}{1+k}
\eee
since $bk\le b\gamma=a.$ We estimate similarly for $k\le k_{{\rm ind}}$:
\bee
&&\left|\left(b^2x\th_{{3}}u\Theta'\right)_k\right|=b^3\left|\left(\th_{{3}}u\Theta'\right)_{k-1}\right|\leq b^3kc_{\nu,a}M_{\nu,a}|w_{\gamma-1,\nu_b+1}(k-1)|\leq b^2kc_{\nu,a}M_{\nu,a}w_{\gamma, {\nu_b}}(k)\\
& \leq & \frac{c_{\nu,a}M_{\nu,a}|w_{\gamma, {\nu_b}}(k)|}{1+k}.
\eee
Remaining linear terms do not have the $k$ loss of $(u\Phi')_k$ and are easier to estimate.\\

\noindent\underline{Nonlinear term}. The worst nonlinear term is for $1\le j\le {3}$:
\bee
\left|\left(\mt^{(3)}_jx^{j+2}\Theta^j(u\Theta')\right)_k\right|=b^{j+2}\left|\left(u^{j+2}\mt^{(3)}_j\Theta^j(u\Theta')\right)_k\right|=b^{j+2}\left|\left(\mt^{(3)}_j\Theta^j(u\Theta')\right)_{k-(j+2)}\right|.
\eee
Then for all $\ell\le k_{\rm ind}-1$ from {\eqref{nveioneionoenv|}, \eqref{eq:firstverysimpleconsequenceofbootforzetatothetakmj}}, \eqref{contoneohktemrs}, \eqref{bootbound} and {\eqref{tobeprovoeonorbibib}}:
\bee
&&\left|\left(\mt^{(3)}_j\Theta^j(u\Theta')\right)\right|_{\ell}=\left|\sum_{k_1+\dots+k_{j+2}=\ell}\mt^{(3)}_{jk_1}\theta_{k_2}\dots \theta_{k_{j+1}}(k_{j+2}\theta_{k_{j+2}})\right|\\
& \leq & {\ell} c_{\nu,a}M_{\nu,a}^{j+1}\sum_{k_1+\dots+k_{j+2}=\ell} \Pi_{i=1}^{j+2}|w_{\gamma-1,\nu_b+1}(k_i)|\le {\ell}c_{\nu,a}M_{\nu,a}^{j+1}|w_{\gamma-1,\nu_b+1}(\ell)|
\eee
Therefore, {recalling \eqref{kebibjbiubbi4}},  for $0\le k\le k_{\rm ind}$:
\bee
&&\left|\left(\mt^{(3)}_jx^{j+2}\Theta^j(u\Theta')\right)_k\right|=b^{j+2}\left|\left(\mt^{(3)}_j\Theta^j(u\Theta')\right)_{k-(j+2)}\right|\\
&\le&  {k c_{\nu,a}M_{\nu,a}^{j+1}b^{j+2}}|w_{\gamma-1,\nu_b+1}(k-(j+2))|\leq  c_{\nu,a}M_{\nu,a}^{j+1}\frac{k}{(1+k)^{j+2-1}}|w_{\gamma,\nu_b}(k)|\\
&\leq & c_{\nu,a}M_{\nu,a}^{{j+1}}\frac{|w_{\gamma,\nu_b}(k)|}{(1+k)^j}\le \frac{c_{\nu,a}M_{\nu,a}^{{4}}}{1+k}|w_{\gamma,\nu_b}(k)|
\eee
since $j\ge 1$. Similarly, for ${1}\le j\le {3}$:
\bee
&&\left|\left(b\mt^{(4)}_jx^{j{+1}}\Theta^j(u\Theta')\right)_k\right|=bb^{j{+1}}\left|\left(\mt^{(4)}_j\Theta^j(u\Theta')\right)_{k-j{-1}}\right|\\
&\le& b k c_{\nu,a}M_{\nu,a}^{j+1}b^{j{+1}}|w_{\gamma-1,\nu_b+1}(k-j{-1})|\\
& \leq &b k c_{\nu,a}M_{\nu,a}^{j+1}\frac{|w_{\gamma,\nu_b}(k)|}{(1+k)^{{j}}}\le \frac{c_{\nu,a}M_{\nu,a}^{{4}}}{1+k}|w_{\gamma,\nu_b}(k)|
\eee
since $bk\le {b}\gamma =a$ and $j\ge {1}$. The two remaining nonlinear terms in \eqref{fialfrmaulg} do not have the $k$ loss of $(u\Theta')_k$ and are thus better by a factor of $\frac{1}{1+k}$. The collection of above bounds concludes the proof of \eqref{boundgk}.\\

\noindent{\bf step 4} Closing \eqref{bootbound}. From \eqref{boundgk}, \eqref{eboebveiboebeo}, {we have for $k\leq k_{\rm ind}$}:
$$
|\zeta_k|<|\zeta_{k-1}|+\frac{c_{\nu,a}{M_{\nu,a}^{4}}}{1+k^2}.
$$
We sum over $k\in \{k^*,k_{\rm ind}\}$ and conclude from \eqref{cneionienoen}:
$$|\zeta_{k_{\rm ind}}{|}\leq |\zeta_{k^*}|+c_{\nu,a}{M_{\nu,a}^4}\sum_{k=k^*}^{k_{\rm ind}}\frac{1}{k^2}\leq  M_{\nu,a}+\frac{c_{\nu,a}{M_{\nu,a}^4}}{{k^*}}<2M_{\nu,a}$$
provided $k^*>k^*(M_{\nu,a})$ has been chosen large enough. \eqref{bootbound} is proved. This also concludes the proof of \eqref{boundednessgk} and of Lemma \ref{lemmaboundthetak}.
\end{proof}

{The proof of Proposition \ref{propboundbdependent} follows immediately from \eqref{cneovnenvonveoinven}, \eqref{esitmaitmitot}, \eqref{cneioneineonoen}, \eqref{defphi}, and Lemma \ref{lemmamultiplication}.}


\section{Quantitative study of the $\matchal C^\infty$ solution}
\label{sec:studyCinftysolution}
\label{cinftysolution}


We now turn to the qualitative of the $\matchal C^\infty$ solution of \eqref{equaitoncompee}. 
Understanding of the Taylor expansion at $u=0$ is not sufficient to analyze the solution away from $u=0$. 
Our main goal  is to show that truncating the Taylor series at $k=K$ yields the dominant terms in the solution
which, together with a remainder, can be computed and estimated thanks to an explicit integral representation.\\

We study the $\matchal C^\infty$ solution. We define the operator 
\be
\label{defopeterator}
\mathcal T(\mathcal G)=\frac{u^{\gamma-2}}{(1-u)^{\gamma+\nu_b+1}}\int_0^u\frac{(1-v)^{\gamma+\nu_b}}{v^{\gamma-1}}\frac{\matchal G}{b}dv.
\ee


\subsection{Remainder function} 


We introduce several special functions defined via the integral operator \eqref{defopeterator}. These will be fundamental  in understanding the leading order terms which appear when the Taylor expansion no longer dominates.

\begin{lemma}[Definition and properties of the first remainder function]
\label{lemmaremainder}
Let 
\bea\label{valuemzerou:0:0}
M_0(u)&=&(K+\nu_b+2)w_{\gamma-1,\nu_b+1}(K-1)\mathcal T(bu^{K}).
\eea
{Then
\bea\label{valuemzerou}
M_0(u) &=& (1+o_{K\to +\infty}(1))\Gamma(\alpha_\gamma)K^{\nu_b+4-\alpha_\gamma}\mathcal T(bu^{K}).
\eea
Moreover,} there exist universal constants $0<c_{\nu,1}<c_{\nu,2}$ such that for all $0<b<b^*(\nu)$, the following holds:\\

\noindent\underline{behavior for small u}:  for $0\le u\le  b$,
\be
\label{lowerobundzero}
c_{\nu,1}\le \frac{M_0(u)}{\Gamma(\alpha_\gamma)\Gamma(1-\alpha_\gamma)K^{\nu_b+4-\alpha_\gamma}u^{K}}\le c_{\nu,2}\ee
\noindent\underline{behavior for large u}: for $b\le u< \frac 12$:
\be
\label{lowerobundzerobis:0}
c_{\nu,1}\le \frac{M_0(u)}{\Gamma(\alpha_\gamma)\Gamma(1-\alpha_\gamma)K^{\nu_b+3}\left(\frac{u}{1-u}\right)^{K-1}u^{\alpha_\gamma}}\leq c_{\nu,2}.
\ee

\noindent\underline{control of the iterate}:  let $1\le j\le 5$, then
\be
\label{bounditerate}
\left\|\frac{\T (u^jM_0)}{M_0}\right\|_{L^\infty(u\leq \frac 12)}\le \frac{c_\nu}{b}.
\ee

\noindent\underline{control of the derivative}: 
\be
\label{vniovnioneneneo}
\forall 0\le u\le \frac 12, \ \ \frac{|uM_0'|}{M_0}\leq \frac{c_\nu}b. 
\ee 
\end{lemma}

\begin{proof}[Proof of Lemma \ref{lemmaremainder}] This follows from the explicit integral representation \eqref{defopeterator}.\\

{\noindent{\bf step 1} Proof of \eqref{valuemzerou}.} By definition:
\bee
M_0(u)&=& (K+\nu_b+2)w_{\gamma-1,\nu_b+1}(K-1)\mathcal T(bu^{K})\\
\nonumber &=& (K+\nu_b+2)\frac{\Gamma(\alpha_\gamma)\Gamma(K+\nu_b+2)}{\Gamma(\gamma-2)}\mathcal T(bu^{K})\\
\nonumber & = &  \frac{\Gamma(\alpha_\gamma)\Gamma(K+\nu_b+3)}{\Gamma(K-1+\alpha_\gamma)}\mathcal T(bu^{K})\\
\nonumber & =& (1+o_{K\to +\infty}(1))\Gamma(\alpha_\gamma)K^{\nu_b+4-\alpha_\gamma}\mathcal T(bu^{K})
\eee
where we used \eqref{aymptoticratio} in the last step.\\

\noindent{\bf step 2} Estimate for $0\le u\leq b$. For  $u\le b$, 
$${(1-u)^K}=e^{K\log (1-u)}=e^{-Ku+O(Ku^2)}=e^{O(1)}.$$ 
Then, from \eqref{valuemzerou}:
\bee
&&M_0(u)=  (1+o_{K\to +\infty}(1))\Gamma(\alpha_\gamma)K^{\nu_b+4-\alpha_\gamma}\frac{u^{\gamma-2}}{(1-u)^{\gamma+\nu_b+1}}\int_0^u\frac{(1-{v})^{\gamma+\nu_b}}{{v}^{\gamma-1}}{v}^Kdv\\
& = & e^{O(1)}\Gamma(\alpha_\gamma)K^{\nu_b+4-\alpha_\gamma}u^{K-1+\alpha_\gamma}\int_0^u\frac{dv}{v^{\alpha_\gamma}}=  e^{O_\nu(1)}\frac{\Gamma(\alpha_\gamma)}{1-\alpha_\gamma}K^{\nu_b+4-\alpha_\gamma}u^{K-1+\alpha_\gamma}u^{1-\alpha_\gamma}\\
&=& e^{O(1)}\Gamma(\alpha_\gamma)\Gamma(1-\alpha_\gamma)K^{\nu_b+4-\alpha_\gamma}u^{K}
\eee
where we used the fact that\footnote{{Indeed, we have $x\Gamma(x)=\Gamma(x+1)$ and the well known bound $0.88\leq \Gamma(x)\leq 1$ on $1\leq x\leq 2$.}} 
{$$\frac{1}{2}\leq x\Gamma(x)\leq 1\textrm{ for }0<x\leq 1,$$} 
and \eqref{lowerobundzero} is proved.\\

\noindent{\bf step 3} Estimate for $b\leq u\leq \frac 12$. First, from \eqref{valuemzerou}, \eqref{defbvebeovb}, \eqref{aymptoticratio}, we have the global bound 
\bee
M_0(u)&\leq& c_\nu \Gamma(\alpha_\gamma)K^{\nu_b+4-\alpha_\gamma}\frac{u^{\gamma-2}}{(1-u)^{\gamma+\nu_b+1}}\int_0^1\frac{(1-v)^{\gamma+\nu_b}}{v^{\gamma-1}}v^Kdv\\
& \leq &  c_\nu \Gamma(\alpha_\gamma)K^{\nu_b+4-\alpha_\gamma}\frac{u^{\gamma-2}}{(1-u)^{\gamma+\nu_b+1}}\int_0^1(1-v)^{K+1+\alpha_\gamma+\nu_b}v^{-\alpha_\gamma}dv\\
& \leq & c_\nu \Gamma(\alpha_\gamma)K^{\nu_b+4-\alpha_\gamma}\frac{u^{K-1+\alpha_\gamma}}{(1-u)^{K+1+\alpha_\gamma+\nu_b+1}}B(1-\alpha_\gamma,K+\alpha_\gamma+\nu_b+2)\\
& \le &  c_\nu \Gamma(\alpha_\gamma)K^{\nu_b+4-\alpha_\gamma}\frac{u^{K-1+\alpha_\gamma}}{(1-u)^{K+1+\alpha_\gamma+\nu_b+1}}\frac{\Gamma(1-\alpha_\gamma)\Gamma(K+\nu_b+\alpha_\gamma+2)}{\Gamma(K+\nu_b+3)}\\
& \leq & c_\nu \Gamma(\alpha_\gamma)\Gamma(1-\alpha_\gamma)K^{\nu_b+4-\alpha_\gamma}\left(\frac{u}{1-u}\right)^{K-1}\frac{u^{\alpha_\gamma}}{K^{1-\alpha_\gamma}}\\
& \leq & c_\nu \Gamma(\alpha_\gamma)\Gamma(1-\alpha_\gamma)K^{\nu_b+3}\left(\frac{u}{1-u}\right)^{K-1}u^{\alpha_\gamma}.
\eee
This gives the upper bound in {\eqref{lowerobundzerobis:0}}. For the lower bound, we use:
\bee
&&\int_0^u(1-v)^{K+1+\alpha_\gamma+\nu_b}v^{-\alpha_\gamma}dv\geq \int_0^b(1-v)^{K+1+\alpha_\gamma+\nu_b}v^{-\alpha_\gamma}dv\\
& = & \int_0^{ b}e^{\left[-(K+1+\alpha_\gamma+\nu_b)v+O(Kv^2)\right]}\frac{dv}{v^{\alpha_\gamma}}\geq c_{\nu}\int_0^{b}\frac{dv}{v^{\alpha_\gamma}}\ge \frac{c_{\nu}b^{1-\alpha_\gamma}}{1-\alpha_\gamma}\geq  \frac{c_{\nu}\Gamma(1-\alpha_\gamma)}{K^{1-\alpha_\gamma}}
\eee
Then, \eqref{valuemzerou} gives 
\bee
M_0(u)&\geq& c_\nu \Gamma(\alpha_\gamma)K^{\nu_b+4-\alpha_\gamma}\frac{u^{K-1+\alpha_\gamma}}{(1-u)^{K+1+\alpha_\gamma+\nu_b+1}}\frac{\Gamma(1-\alpha_\gamma)}{K^{1-\alpha_\gamma}}\\
& \geq & c_{\nu}\Gamma(\alpha_\gamma)\Gamma(1-\alpha_\gamma)K^{\nu_b+3}\left(\frac{u}{1-u}\right)^{K-1}u^{\alpha_\gamma}
\eee
and {\eqref{lowerobundzerobis:0}} is proved.\\

\noindent{\bf step 4} Control of the iterate. For $u\leq b$, {and since $j\geq 1$}, we estimate from \eqref{lowerobundzero}:
\bee
\frac{\T(u^jM_0)}{M_0}&\leq& \frac{c_\nu}{u^K}\frac{u^{\gamma-2}}{(1-u)^{\gamma+\nu_b+1}}\int_0^u\frac{(1-v)^{\gamma+\nu_b}}{v^{\gamma-1}}\frac{v^{K+j}}{b}dv\leq  \frac{c_\nu}{u^{{1-\alpha_\gamma}}}\int_0^u\frac{{v}^{j-\alpha_\gamma}dv}{b}\\
& \leq & \frac{c_\nu u^{{j}}}{b(j+1-\alpha_\gamma)}\le \frac{c_\nu b^{{j}}}{b(2-\alpha_\gamma)}\le c_\nu
\eee
For $b\le u\le \frac 12$, we estimate {from \eqref{lowerobundzero}}:
\bee
&&\int_0^b\frac{(1-v)^{\gamma+\nu_b}}{v^{\gamma-1}}v^jM_0(v)dv\le \int_0^b\Gamma(\alpha_\gamma)\Gamma(1-\alpha_\gamma)K^{\nu_b+4-\alpha_\gamma}\frac{v^{K+j}}{v^{\gamma-1}}dv\\
&\leq&  \Gamma(\alpha_\gamma)\Gamma(1-\alpha_\gamma)K^{\nu_b+4-\alpha_\gamma}\frac{b^{j+1-\alpha_\gamma}}{j+1-\alpha_\gamma}\leq  b^j\Gamma(\alpha_\gamma)\Gamma(1-\alpha_\gamma)K^{\nu_b+3}
\eee
and {from \eqref{lowerobundzerobis:0}}
\bee
&&\int_b^u\frac{(1-v)^{\gamma+\nu_b}}{v^{\gamma-1}}v^jM_0(v)dv\leq c_{\nu}\int_b^u\frac{(1-v)^{\gamma+\nu_b}}{v^{\gamma-1}}\Gamma(\alpha_\gamma)\Gamma(1-\alpha_\gamma)K^{\nu_b+3}\left(\frac{v}{1-v}\right)^{K-1}v^{\alpha_\gamma}v^jdv\\
& \leq & c_\nu\Gamma(\alpha_\gamma)\Gamma(1-\alpha_\gamma)K^{\nu_b+3}\int_b^u v^{j-1}dv\leq c_\nu\Gamma(\alpha_\gamma)\Gamma(1-\alpha_\gamma)K^{\nu_b+3}.
\eee
Thus,
\bee
\frac{\T(u^jM_0)}{M_0}&\leq& \frac{c_\nu}{\Gamma(\alpha_\gamma)\Gamma(1-\alpha_\gamma)K^{\nu_b+3}\left(\frac{u}{1-u}\right)^{K-1}u^{\alpha_\gamma}}\times \frac{u^{\gamma-2}}{(1-u)^{\gamma+\nu_b+1}}\frac{1}{b}\Gamma(\alpha_\gamma)\Gamma(1-\alpha_\gamma)K^{\nu_b+3}\\
&\leq&   \frac{c_\nu}{b}
\eee
and \eqref{bounditerate} is proved.\\

\noindent{\bf step 5} Control of the derivative. From {\eqref{defopeterator}}:
\be
\label{vnoneneoneonve}
u\left[\T(\mathcal G)\right]'=\frac{1}{1-u}\left[\frac{\mathcal G}{b}+[(\gamma-2)+(\nu_b+3)u]\mathcal T(\mathcal G)\right]
\ee
which yields, recalling \eqref{valuemzerou}:
\bea
\label{formualderninenoen}
\nonumber uM_0'& = &(1+o_{K\to +\infty}(1))\Gamma(\alpha_\gamma)K^{\nu_b+4-\alpha_\gamma}u\left[\mathcal T(bu^{K})\right]'\\
\nonumber & = & \frac{(1+o_{K\to +\infty}(1))\Gamma(\alpha_\gamma)K^{\nu_b+4-\alpha_\gamma}}{1-u}\left[\frac{bu^K}{b}+[(\gamma-2)+(\nu_b+3)u]\mathcal T(bu^k)\right]\\
& = & \frac{(1+o_{K\to +\infty}(1))\Gamma(\alpha_\gamma)K^{\nu_b+4-\alpha_\gamma}}{1-u}u^K+{\frac{[(\gamma-2)+(\nu_b+3)u]M_0}{1-u}}.
\eea
From \eqref{lowerobundzero}, \eqref{formualderninenoen}, we estimate for $u\le b$:
\bee
\frac{|uM_0'|}{M_0}&\leq &c_\nu\frac{\Gamma(\alpha_\gamma)K^{\nu_b+4-\alpha_\gamma}u^K}{\Gamma(\alpha_\gamma)\Gamma(1-\alpha_\gamma)K^{\nu_b+4-\alpha_\gamma}u^{K}}+\frac{c_\nu}{b}\leq \frac{c_\nu}{b}
\eee
and for $b\le u\leq \frac 12$ from {\eqref{lowerobundzerobis:0}, \eqref{formualderninenoen}}:
$$\frac{|uM_0'|}{M_0}\leq \frac{\Gamma(\alpha_\gamma)K^{\nu_b+4-\alpha_\gamma}u^K}{\Gamma(\alpha_\gamma)\Gamma(1-\alpha_\gamma)K^{\nu_b+3}\left(\frac{u}{1-u}\right)^{K-1}u^{\alpha_\gamma}}+\frac{c_\nu}{b}\leq \frac{c_\nu}{b}
$$
and \eqref{vniovnioneneneo} is proved.
\end{proof}

We now establish additional estimates for the remainder functions.

\begin{lemma}[Holomorphic representation and bounds]
Let $j\ge 0$ and 
\be
\label{Pdefmj}
M_j(u)=(K+j+\nu_b+2)(-1)^{j}w_{\gamma-1,\nu_b+1}(K-1+j)\mathcal T(bu^{K+j}),
\ee
then we have the following convergent series representation for $|u|<1$:
\be
\label{formualmj}
M_j(u)=\sum_{m=j+1}^{+\infty}(-1)^{m}w_{\gamma-1,\nu_b+1}(K-1+m)u^{K-1+m}.
\ee
Moreover, let $1\le j\le {5}$, then there exist universal constants {$0<c_{\nu,1}<c_{\nu,2}$} such that:\\ \noindent\underline{behavior for small u}:  for $0\le u\le  b$,
\be
\label{lowerobundzerobis}
{c_{\nu,1}}\le \frac{M_j(u)}{\Gamma(\alpha_\gamma)\Gamma(1-\alpha_\gamma)K^{\nu_b+j+4-\alpha_\gamma}u^{K+j}}\le {c_{\nu,2}}\ee
\noindent\underline{behavior for large u}: for $b\le u< \frac 12$:
\be
\label{lowerobundzerobisbis}
{c_{\nu,1}}\le \frac{M_j(u)}{\Gamma(\alpha_\gamma)\Gamma(1-\alpha_\gamma)K^{\nu_b+3}\left(\frac{u}{1-u}\right)^{K-1}u^{\alpha_\gamma}}\leq {c_{\nu,2}}.
\ee
\end{lemma}

\begin{proof}
\noindent{\bf step 1} Holomorphic representation. Given a $\mathcal C^\infty$ function with $\mathcal G=O_{u\to 0}(u^K)$, $\Theta=\mathcal T(\mathcal G)$ satisfies the linear equation
\be
\label{vneoinvenvoe}
u(1-u)\Theta' -\left[\gamma-2+(\nu_b+3)u\right]\Theta=\frac{\mathcal G}{b}
\ee
which can formally solved via a series representation:
$$
\left|\begin{array}{l}
\mathcal G=\sum_{k=0}^{+\infty}g_k u^k\\
\Theta=\sum_{k={0}}^{+\infty}\theta_k u^k
\end{array}\right.
$$
i.e.,
\bee
&&u(1-u)\left(\sum_{k=0}^{\infty}\theta_ku^k\right)'-(\gamma-2+(\nu_b+3)u)\sum_{k=0}^{+\infty}\theta_ku^k\\
&=& \sum_{k=1}^{+\infty}(k\theta_ku^k-k\theta_ku^{k+1})-\sum_{k=0}^{+\infty}(\gamma-2)\theta_ku^k-(\nu_b+3)\sum_{k=0}^{+\infty}\theta_ku^{k+1}\\
& = & -(\gamma-2)\theta_0+\sum_{k=1}^{+\infty}(k-\gamma+2)\theta_ku^k-\sum_{k=1}^{+\infty}(k+\nu_b+2)\theta_{k-1}u^{k}\\
& = &  -(\gamma-2)\theta_0+\sum_{k=1}^{+\infty}\left[(k-\gamma+2)\theta_k-(k+\nu_b+2)\theta_{k-1}\right]u^k
\eee
with the induction relation
$$
\left|\begin{array}{l}
\theta_0=-\frac{g_0}{b(\gamma-2)}\\
(\gamma-k-2)\theta_k+(k+\nu_b+2)\theta_{k-1}=-\frac{g_k}{b}, \ \ k\ge 1.
\end{array}\right.
$$
Let 
$$\zeta_k=\frac{\theta_k}{w_{\gamma-1,\nu_b+1}(k)},$$ 
then {as in} \eqref{eboebveiboebeo}:
$$\zeta_k+\zeta_{k-1}=-\frac{\gamma}{a(\gamma-2)}\frac{g_k}{(k+\nu_b+2)w_{\gamma,\nu_b}(k)}, \ \ k\ge 1
$$
which yields 
\bee
\zeta_k=-\frac{\gamma}{a(\gamma-2)}(-1)^k\sum_{j=0}^k\frac{(-1)^jg_j}{(j+\nu_b+2)w_{\gamma,\nu_b}(j)}, \ \  k\ge 1
\eee
and thus
\be
\label{enkvnevenenoenoivnedldnl}
\left|\begin{array}{l}
\theta_k=(-1)^kw_{\gamma-1,\nu_b+1}(k) S_k, \ \ k\ge 0\\
S_k=-\frac{\gamma}{a(\gamma-2)}\sum_{j=0}^k\frac{(-1)^jg_j}{(j+\nu_b+2)w_{\gamma,\nu_b}(j)}.
\end{array}\right.
\ee
Given $j\ge 0$ and $$\mathcal G=u^{K+j}$$ this yields:
$$\Theta_k=\left|\begin{array}{l} 0\ \ \mbox{for}\ \ k\le K+j-1\\
(-1)^kw_{\gamma-1,\nu_b+1}(k) S_{K+j}\ \ \mbox{for}\ \ k\ge K+j
\end{array}\right.
$$ Therefore, the representation is a normally convergent series\footnote{{From \eqref{cnineneonevnev:0} 
\eqref{cneioneineonoen} and \eqref{aymptoticratio}, we have 
$$w_{\gamma-1,\nu_b+1}(k)=O(k^{K+\nu_b+3+\alpha_\gamma})\textrm{ as }k\to +\infty$$
which implies that the series converges for $|u|<1$.}} for $|u|<1$. Using \eqref{cneioneineonoen}:
\bee
\mathcal T(u^{K+j})&=&S_{K+j}\sum_{k=K+j}^{+\infty}(-1)^kw_{\gamma-1,\nu_b+1}(k) u^k\\
&=& -\frac{\gamma}{a(\gamma-2)}\frac{(-1)^{K+j}}{(K+j+\nu_b+2)w_{\gamma,\nu_b}(K+j)}\sum_{k=K+j}^{+\infty}(-1)^kw_{\gamma-1,\nu_b+1}(k) u^k\\
& = & \frac{1}{b}\frac{(-1)^{K-1+j}}{(K+j+\nu_b+2)w_{\gamma-1,\nu_b+1}(K+j-1)}\sum_{k=K+j}^{+\infty}(-1)^kw_{\gamma-1,\nu_b+1}(k) u^k
\eee
gives 
$$M_j(u)=(-1)^{K-1}\sum_{k=K+j}^{+\infty}(-1)^{k}w_{\gamma-1,\nu_b+1}(k)u^k=\sum_{m=j+1}^{+\infty}(-1)^{m}w_{\gamma-1,\nu_b+1}(K-1+m)u^{K-1+m}$$
and \eqref{formualmj} is proved. We now assume $j\ge 1$.\\

\noindent{\bf step 2} Estimate for $0\le u\leq b$. {From \eqref{cnineneonevnev:0} 
\eqref{cneioneineonoen} and \eqref{aymptoticratio}, we have}
\bea
\label{nekvneneonene}
\nonumber&&(-1)^{j}w_{\gamma-1,\nu_b+1}(K-1+j)=\Gamma(\alpha_\gamma)\Gamma(1-\alpha_\gamma)\frac{\Gamma(K+j+\nu_b+2)}{\Gamma(K-1+\alpha_\gamma)\Gamma(j+1-\alpha_\gamma)}\\
\nonumber & = & \frac{\Gamma(\alpha_\gamma)\Gamma(1-\alpha_\gamma)}{\Gamma(j+1-\alpha_\gamma)}\left[1+o_{b\to 0}(1)\right]K^{\nu_b+j+3-\alpha_\gamma}\\
&=& e^{O(1)}\Gamma(\alpha_\gamma)\Gamma(1-\alpha_\gamma)K^{\nu_b+j+3-\alpha_\gamma}
\eea
where we used {$j\leq 5\ll \gamma$ in the second step and} $j\ge 1$ in the last step.

For  $u\le b$, $${(1-u)^K}=e^{K\log (1-u)}=e^{-Ku+O(Ku^2)}=e^{O(1)}$$ and therefore:
\bee
M_j(u)&=&  {e^{O(1)}(K+j+\nu_b+2)\Gamma(\alpha_\gamma)\Gamma(1-\alpha_\gamma)K^{\nu_b+j+3-\alpha_\gamma}}\\
&\times & \frac{u^{\gamma-2}}{(1-u)^{\gamma+\nu_b+1}}\int_0^u\frac{(1-{v})^{\gamma+\nu_b}}{{v}^{\gamma-1}}{v}^{K+j}dv\\
& = & e^{O(1)}\Gamma(\alpha_\gamma)\Gamma(1-\alpha_\gamma)K^{\nu_b+j+4-\alpha_\gamma}u^{K-1+\alpha_\gamma}\int_0^u\frac{v^jdv}{v^{\alpha_\gamma}}\\
&=& e^{O(1)}\Gamma(\alpha_\gamma)\Gamma(1-\alpha_\gamma)K^{\nu_b+j+4-\alpha_\gamma}u^{K+j}.
\eee

\noindent{\bf step 3} Estimate for $b\leq u\leq \frac 12$. We have the global bound from  {\eqref{nekvneneonene}, \eqref{defbvebeovb}, \eqref{aymptoticratio}}:
\bee
&&M_j(u)\leq c_\nu \Gamma(\alpha_\gamma)\Gamma(1-\alpha_\gamma)K^{\nu_b+j+4-\alpha_\gamma}\frac{u^{\gamma-2}}{(1-u)^{\gamma+\nu_b+1}}\int_0^1\frac{(1-v)^{\gamma+\nu_b}}{v^{\gamma-1}}v^{K+j}dv\\
& \leq &  c_\nu \Gamma(\alpha_\gamma)\Gamma(1-\alpha_\gamma)K^{\nu_b+j+4-\alpha_\gamma}\frac{u^{\gamma-2}}{(1-u)^{\gamma+\nu_b+1}}\int_0^1(1-v)^{K+1+\alpha_\gamma+\nu_b}v^{j-\alpha_\gamma}dv\\
& \leq & c_\nu \Gamma(\alpha_\gamma)\Gamma(1-\alpha_\gamma)K^{\nu_b+j+4-\alpha_\gamma}\frac{u^{K-1+\alpha_\gamma}}{(1-u)^{K+1+\alpha_\gamma+\nu_b+1}}B(j+1-\alpha_\gamma,K+\alpha_\gamma+\nu_b+2)\\
& \le &  c_\nu\Gamma(\alpha_\gamma)\Gamma(1-\alpha_\gamma)K^{\nu_b+j+4-\alpha_\gamma}\frac{u^{K-1+\alpha_\gamma}}{(1-u)^{K+1+\alpha_\gamma+\nu_b+1}}\frac{\Gamma(j+1-\alpha_\gamma)\Gamma(K+\nu_b+\alpha_\gamma+2)}{\Gamma(K+j+\alpha_\gamma+\nu_b+3-\alpha_\gamma)}\\
& \leq & c_\nu \Gamma(\alpha_\gamma)\Gamma(1-\alpha_\gamma)K^{\nu_b+j+4-\alpha_\gamma}\left(\frac{u}{1-u}\right)^{K-1}\frac{u^{\alpha_\gamma}}{K^{j+1-\alpha_\gamma}}\\
& \leq & c_\nu \Gamma(\alpha_\gamma)\Gamma(1-\alpha_\gamma)K^{\nu_b+3}\left(\frac{u}{1-u}\right)^{K-1}u^{\alpha_\gamma}
\eee
which yields the upper bound in \eqref{lowerobundzerobis}. For the lower bound, we use:
\bee
&&\int_0^u(1-v)^{K+1+\alpha_\gamma+\nu_b}v^{j-\alpha_\gamma}dv\geq \int_0^b(1-v)^{K+1+\alpha_\gamma+\nu_b}v^{j-\alpha_\gamma}dv\\
& = & \int_0^{ b}e^{\left[-(K+1+\alpha_\gamma+\nu_b)v+O(Kv^2)\right]}\frac{v^jdv}{v^{\alpha_\gamma}}\geq c_{\nu}\int_0^{b}\frac{v^jdv}{v^{\alpha_\gamma}}\ge c_{\nu}b^{j+1-\alpha_\gamma}
\eee
{ \eqref{nekvneneonene}} then gives the lower bound:
\bee
M_{{j}}(u)&\geq& c_\nu \Gamma(\alpha_\gamma)\Gamma(1-\alpha_\gamma)K^{\nu_b+j+4-\alpha_\gamma}\frac{u^{K-1+\alpha_\gamma}}{(1-u)^{K+1+\alpha_\gamma+\nu_b+1}}b^{j+1-\alpha_\gamma}\\
& \geq & c_{\nu}\Gamma(\alpha_\gamma)\Gamma(1-\alpha_\gamma)K^{\nu_b+3}\left(\frac{u}{1-u}\right)^{K-1}u^{\alpha_\gamma}
\eee
and \eqref{lowerobundzerobis} is proved.
\end{proof}


\subsection{Fixed point formulation of the $\mathcal C^\infty$ solution}


For a given function $F$ with sufficient regularity at the origin, we denote
\be
\label{definitonioperatorramineder}
r_F(u)=F(u)-\sum_{k=0}^{K-1}f_ku^k, \ \ f_k=\frac{f^{(k)}(0)}{k!}.
\ee

\begin{lemma}[Fixed point formulation for the $\mathcal C^\infty$ solution]
\label{fromulationcingty}
Let $\matchal G$ be given by \eqref{fialfrmaulg} and consider the decomposition 
\be
\label{cnekoneonveonoevn}
\left|\begin{array}{l}
\mathcal G=\sum_{k=0}^{K-1}g_k u^k+r_{\matchal G},\\
\Theta=\sum_{k={0}}^{K-1}\theta_k u^k+r_\Theta,
\end{array}\right.
\ee
where
\be
\label{enkvnevenenoenoivne}
\left|\begin{array}{l}
\theta_k=(-1)^kw_{\gamma-1,\nu_b+1}(k) S_k, \ \ 0\le k\le K-1,\\
S_k=\frac{\gamma}{a(\gamma-2)}\sum_{j=0}^k\frac{(-1)^jg_j}{(j+\nu_b+2)w_{\gamma,\nu_b}(j)},
\end{array}\right.
\ee
then the unique solution to the fixed point problem
\be
\label{fineioneoineeogn}
r_\Theta=(-1)^{K-1}S_{K-1}M_0(u)-\T(r_\mathcal G).
\ee
generates the unique solution $\Theta$ to \eqref{thetaeqaiotjoihs} which satisfies:
\be
\label{neoneneonovenoev}
\forall k\ge 0, \ \ \lim_{u\downarrow 0} \frac{\Theta^{(k)}(u)}{k!}=\theta_k
\ee
where $(\theta_k)_{k\ge 0}$ is computed by induction from \eqref{vneineionvioenvionoe}.
\end{lemma}

\begin{proof}[Proof of Lemma \ref{fromulationcingty}]
Recall {\eqref{thetaeqaiotjoihs}}:
 $$u(1-u)\Theta' -\left[\gamma-2+(\nu_b+3)u\right]\Theta=-\frac{\mathcal G}{b}.$$
We let 
$$
\left|\begin{array}{l}
\mathcal G=\sum_{k=0}^{K-1}g_k u^k+r_{\matchal G}\\
\Theta=\sum_{k=1}^{K-1}\theta_k u^k+r_\Theta
\end{array}\right.
$$
\noindent{\bf step 1} Polynomial cancellations. We compute:
\bee
&&u(1-u)\left(\sum_{k=0}^{K-1}\theta_ku^k\right)'-(\gamma-2+(\nu_b+3)u)\sum_{k=0}^{K-1}\theta_ku^k\\
&=& \sum_{k=1}^{K-1}(k\theta_ku^k-k\theta_ku^{k+1})-\sum_{k=0}^{K-1}(\gamma-2)\theta_ku^k-(\nu_b+3)\sum_{k=0}^{K-1}\theta_ku^{k+1}\\
& = & -(\gamma-2)\theta_0+\sum_{k=1}^{K-1}(k-\gamma+2)\theta_ku^k-\sum_{k=1}^{K}(k+\nu_b+2)\theta_{k-1}u^{k}\\
& = &  -(\gamma-2)\theta_0+\sum_{k=1}^{K-1}\left[(k-\gamma+2)\theta_k-(k+\nu_b+2)\theta_{k-1}\right]u^k-(K+\nu_b+2)\theta_{K-1}u^{K}
\eee
Therefore,
\bee
&&u(1-u)\left(\sum_{k=0}^{K-1}\theta_ku^k\right)'-(\gamma-2+(\nu_b+3)u)\sum_{k=0}^{K-1}\theta_ku^k+\sum_{k=0}^{K-1}\frac{g_k}{b}u^k\\
& = & \sum_{k=1}^{K-1}\left[(k-\gamma+2)\theta_k-(k+\nu_b+2)\theta_{k-1}+\frac{g_k}{b}\right]u^k\\
&-& (\gamma-2)\theta_0+\frac{g_0}{b}-(K+\nu_b+2)\theta_{K-1}u^{K}
\eee
which yields the induction relation
$$
\left|\begin{array}{l}
\theta_0=\frac{g_0}{b(\gamma-2)}\\
(\gamma-k-2)\theta_k+(k+\nu_b+2)\theta_{k-1}=\frac{g_k}{b}, \ \ 1\le k\le K-1.
\end{array}\right.
$$
Let $$\zeta_k=\frac{\theta_k}{w_{\gamma-1,\nu_b+1}(k)},$$ then equivalently from \eqref{eboebveiboebeo}:
$$\zeta_k+\zeta_{k-1}=\frac{\gamma}{a(\gamma-2)}\frac{g_k}{(k+\nu_b+2)w_{\gamma,\nu_b}(k)}, \ \ 1\le k\le K-1
$$
which yields 
\bee
\zeta_k=\frac{\gamma}{a(\gamma-2)}(-1)^k\sum_{j=0}^k\frac{(-1)^jg_j}{(j+\nu_b+2)w_{\gamma,\nu_b}(j)}, \ \ 0\le k\le K-1
\eee
and \eqref{enkvnevenenoenoivne} follows.\\

\noindent{\bf step 2} Equation for the remainder. We conclude:
\bee
&&u(1-u)\Theta' -\left[\gamma-2+(\nu_b+3)u\right]\Theta=-\frac{\mathcal G}{b}\\
&\Leftrightarrow &-(K+\nu_b+2)\theta_{K-1}u^{K}+u(1-u)r_\Theta ' -\left[\gamma-2+(\nu_b+3)u\right]r_\Theta=-\frac{r_\mathcal G}{b}\\
&\Leftrightarrow& r_\Theta'-\frac{\gamma-2+(\nu_b+3)u}{u(1-u)}r_\Theta=\frac{1}{u(1-u)}\left[(K+\nu_b+2)\theta_{K-1}u^{K}-\frac{r_\mathcal G}{b}\right]\\
&\Leftrightarrow& r_\Theta'-\left[\frac{\gamma-2}{u}+\frac{\gamma+\nu_b+1}{1-u}\right]r_\Theta=\frac{1}{u(1-u)}\left[(K+\nu_b+2)\theta_{K-1}u^{K}-\frac{r_\mathcal G}{b}\right]\\
&\Leftrightarrow&\frac{d}{du}\left(\frac{(1-u)^{\gamma+\nu_b+1}}{u^{\gamma-2}}r_\Theta\right)'=\frac{(1-u)^{\gamma+\nu_b+1}}{u^{\gamma-2}}\frac{1}{u(1-u)}\left[(K+\nu_b+2)\theta_{K-1}u^{K}-\frac{r_\mathcal G}{b}\right]\\
\eee
and thus, any $\mathcal C^\infty$  solution must be the unique solution to the fixed point equation:
$$
r_\Theta=\frac{u^{\gamma-2}}{(1-u)^{\gamma+\nu_b+1}}\int_0^u\frac{(1-{v})^{\gamma+\nu_b}}{{v}^{\gamma-1}}\left[(K+\nu_b+2)\theta_{K-1}v^{K}-\frac{r_\mathcal G}{b}\right]dv
$$
with \eqref{neoneneonovenoev} forced by the Taylor expansion\footnote{The statement on existence and uniqueness of the fixed point, the fact that the corresponding solution is smooth, and the fact that \eqref{neoneneonovenoev} holds has in fact already been proved in a more general case in Lemma \ref{prop:behavioroftheflownearP2:smooth}.}. 
We now recall {\eqref{valuemzerou:0:0}}:
\bee
M_0(u)&=& (K+\nu_b+2)w_{\gamma-1,\nu_b+1}(K-1)\mathcal T(bu^{K})\\
&=&(K+\nu_b+2)w_{\gamma-1,\nu_b+1}(K-1)\frac{u^{\gamma-2}}{(1-u)^{\gamma+\nu_b+1}}\int_0^u\frac{(1-{v})^{\gamma+\nu_b}}{{v}^{\gamma-1}}v^{K}dv
\eee
and thus
$$r_\Theta=\frac{\theta_{K-1}}{w_{\gamma-1,\nu_b+1}(K-1)}M_0(u)-\T(r_\mathcal G)=(-1)^{K-1}S_{K-1}M_0(u)-\T(r_\mathcal G),$$ this is \eqref{fineioneoineeogn}. 
\end{proof}


\subsection{Convergence of the leading order Taylor coefficient}


The truncation of the Taylor series produces the leading order term, provided the last Taylor coefficient is non zero. This is 
the $S_\infty(d,\ell)\neq 0$ condition.

\begin{lemma}
\label{lenkenopoeje}
We have
\bea\label{behaviorSKasBconvergesto0}
S_{K-1}=(1+o_{b\to 0}(1))S_\infty
\eea
where $S_\infty$ is given by 
\be\label{Sinf}
S_\infty(d,\ell) :=  \frac{1}{a}\sum_{j=0}^{+\infty}\frac{(-1)^ja^jg_j^\infty}{\Gamma(\nu+j+3)}
\ee
and where $g_j^\infty$ corresponds to the limiting problem \eqref{limitingxequation} and is given by  \eqref{vnieonveineonovenv}. 
\end{lemma}

{\begin{proof}[Proof of Lemma \ref{lenkenopoeje}] Since we have from \eqref{ceineneoeoneobis}
\bee
\sum_{j=0}^{+\infty}\left|\frac{(-1)^ja^jg_j^\infty}{\Gamma(\nu+j+3)}\right| \leq c_\nu\left(\sum_{j=0}^{+\infty}\frac{1}{1+j^2}\right)\leq c_\nu<+\infty,
\eee 
and from \eqref{boundednessgk}, 
\bee
\sum_{j=0}^{+\infty}\left|\frac{(-1)^jg_j}{(j+\nu_b+2)w_{\gamma,\nu_b}(j)}\right| \leq c_\nu\left(\sum_{j=0}^{+\infty}\frac{1}{1+j^2}\right)\leq c_\nu<+\infty,
\eee 
its suffices to prove the convergence term by term as $b\to 0$. Now, recall \eqref{eq:limiteofgkonbktogkinfty},
\bee
\lim_{b\to 0}\frac{g_j}{b^j} = g_j^\infty.
\eee
Since from the definition of $w_{\gamma,\nu_b}(j)$, 
\bee
(j+\nu_b+2)w_{\gamma,\nu_b}(j) &=& \Gamma(\nu_b+j+3)\frac{\Gamma(\gamma-1-j)}{\Gamma(\gamma-1)}= \frac{\Gamma(\nu+j+3)}{\gamma^j}(1+o_{b\to 0}(1))\\
&=& \Gamma(\nu+j+3)\frac{b^j}{a^j}(1+o_{b\to 0}(1))
\eee
we obtain 
\bee
\lim_{b\to 0}(j+\nu_b+2)\frac{w_{\gamma,\nu_b}(j)}{b^j} &=& \frac{\Gamma(\nu+j+3)}{a^j}
\eee
from which we deduce the convergence term by term when $b\to 0$, as desired. 
\end{proof}}


\subsection{The $\Thetam$ leading order term}


We may now extract the leading order terms in $\Theta$. From {\eqref{cnekoneonveonoevn}, \eqref{enkvnevenenoenoivne}, \eqref{fineioneoineeogn}} {and \eqref{behaviorSKasBconvergesto0}}, 
\bea
\label{expressionnonlnieanrtmer}
\nonumber \Theta(u) &=&\sum_{k=0}^{K-2}\theta_k u^k+\theta_{K-1}u^{K-1}+(-1)^{K-1}S_{K-1}M_0(u)-\T(r_\mathcal G)\\
\nonumber & = & \sum_{k=0}^{K-2}\theta_k u^k+(-1)^{K-1}S_{K-1}\left[w_{\gamma-1,\nu_b+1}(K-1)u^{K-1}+M_0(u)\right]-\T(r_\mathcal G)\\
& = & \sum_{k=0}^{K-2}\theta_k u^k+(-1)^{K-1}S_{\infty}\left[1+o_{b\to 0}(1)\right]\Thetam(u)-\T(r_\mathcal G)
\eea
and from \eqref{formualmj}
\bea
\label{deftehtemain}
\nonumber \Thetam(u)&=&w_{\gamma-1,\nu_b+1}(K-1)u^{K-1}+M_0(u)\\
\nonumber &=& w_{\gamma-1,\nu_b+1}(K-1)u^{K-1}+\sum_{m=1}^{+\infty}(-1)^{m}w_{\gamma-1,\nu_b+1}(K-1+m)u^{K-1+m}\\
& = &\sum_{j=0}^{+\infty}(-1)^{j}w_{\gamma-1,\nu_b+1}(K-1+j)u^{K-1+j}\nonumber\\&=&{\Gamma(\alpha_\gamma)\Gamma(1-\alpha_\gamma)\sum_{j=0}^{+\infty}\frac{\Gamma(K+j+\nu_b+2)}{\Gamma(K-1+\alpha_\gamma)\Gamma(j+1-\alpha_\gamma)}u^{K-1+j}.}
\eea
Equivalently, from \eqref{nekvneneonene} {and \eqref{formualmj}}:
\bea
\label{cneiovnenvenenne}
\nonumber &&\Thetam{(u)} =\Gamma(\alpha_\gamma)\Gamma(1-\alpha_\gamma)\frac{\Gamma(K+\nu_b+2)}{\Gamma(K-1+\alpha_\gamma)\Gamma(1-\alpha_\gamma)}u^{K-1}\\
\nonumber &+& \Gamma(\alpha_\gamma)\Gamma(1-\alpha_\gamma)\frac{\Gamma(K+\nu_b+3)}{\Gamma(K-1+\alpha_\gamma)\Gamma(2-\alpha_\gamma)}u^K+M_1(u)\\
\nonumber & = & \Gamma(\alpha_\gamma)\frac{\Gamma(K+\nu_b+2)}{\Gamma(K-1+\alpha_\gamma)}u^{K-1}\left\{1+\frac{K+\nu_b+2}{\Gamma(2-\alpha_\gamma)}\Gamma(1-\alpha_\gamma)u\right\}+M_1(u)\\
& = & \Gamma(\alpha_\gamma)\Gamma(1-\alpha_\gamma)K^{\nu_b+3-\alpha_\gamma}u^{K-1}\\
\nonumber &\times & \left\{\left[1+o_{b\to 0}(1)\right]\left[\frac{1}{\Gamma(1-\alpha_\gamma)}+\frac{K+\nu_b+2}{\Gamma(2-\alpha_\gamma)}u\right]+(u K)^{\alpha_\gamma}\frac{M_1(u)}{u^{K-1}\Gamma(\alpha_\gamma)\Gamma(1-\alpha_\gamma)K^{\nu_b+3}u^{\alpha_\gamma}}\right\}.
\eea
We now turn to the study of $\Thetam$.

\begin{lemma}[Properties of $\Thetam$]
\label{lemmanekvneneoneon}
The function $\Thetam$ is {positive and strictly increasing} on $[0,\frac 12]$. Moreover, pick universal constants $\frac{1}{\delta},\Theta^*\gg1$ {then for all} $0<b<b^*(\Theta^*,\delta)$, the unique solution to 
\be
\label{neknvonenneudatsgeeaa}
\Thetam(u^*(\alpha_\gamma))=\Theta^*,  \ \ u^*(\alpha_\gamma)\in \left(0,\frac 12\right)
\ee 
satisfies the following bounds:\\

\underline{first integer boundary layer}.  {If $\alpha_\gamma$ is such that}
 $${\alpha_\gamma}=\frac{K^{\nu_b+3}(\sigma b)^{K-1}}{\Theta^*}, \ \textrm{{ with} }\delta<\sigma<\frac{1}{\delta}$$ then 
\be
\label{valueapprocimate}
\Gamma(\alpha_\gamma)\Gamma(1-\alpha_\gamma)K^{\nu_b+3-\alpha_\gamma}(u^*(\alpha_\gamma))^{K-1}=\Theta^*e^{O_{\nu}(\sigma)}
\ee
and
\be
\label{firsteataimteboundary}
u^*({\alpha_\gamma})=\sigma b(1+o_{b\to 0}(1)).
\ee

\noindent\underline{second {integer} boundary layer}.  {If $\alpha_\gamma$ is such that}
$${\alpha_\gamma}=1-\frac{K^{\nu_b+3}(\sigma b)^{K-1}}{\Theta^*}, \ \textrm{ {with} } \delta<\sigma<\frac 1{\delta},$$ then 
 \eqref{firsteataimteboundary} holds {and} 
 {\be
\label{valueapprocimate:2ndcase}
\Gamma(\alpha_\gamma)\Gamma(1-\alpha_\gamma)K^{\nu_b+3-\alpha_\gamma}\frac{K+\nu_b+2}{\Gamma(2-\alpha_\gamma)}(u^*(\alpha_\gamma))^{K}=\Theta^*e^{O_{\nu}(\sigma)}.
\ee}

\noindent\underline{away from the integer boundary layer $\alpha_\gamma$}.   {If $\alpha_\gamma$ is such that}
$$\frac{K^{\nu_b+3}\left(\frac{b}{\delta}\right)^{K-1}}{\Theta^*}<\alpha_\gamma<1-\frac{K^{\nu_b+3}\left(\frac{b}{\delta}\right)^{K-1}}{\Theta^*}$$
then 
\be
\label{jvheoeonenoenenv}
\Gamma(\alpha_\gamma)\Gamma(1-\alpha_\gamma)K^{\nu_b+3}\left(\frac{u^*(\alpha_\gamma)}{1-u^*(\alpha_\gamma)}\right)^{K-1}(u^*(\alpha_\gamma))^{\alpha_\gamma}=\Theta^*e^{O_\nu(1)}
\ee and 
\be\label{vnieneoneonevnonen}
\frac{b}{2\delta}<u^*(\alpha_\gamma)<\frac 12.
\ee
\end{lemma}

\begin{proof}[Proof of Lemma \ref{lemmanekvneneoneon}] The strict monotonicity and positivity follows from \eqref{deftehtemain}, {since $\Theta_{\rm main}$ is given by a series with positive coefficients}. From {\eqref{cneiovnenvenenne}} \eqref{lowerobundzerobisbis}, 
$${\Theta_{\rm main}\left(\frac 12\right)\ge} M_1\left(\frac 12\right)\ge c_\nu \Gamma(\alpha_\gamma)\Gamma(1-\alpha_\gamma)K^{\nu_b+3} {\ge c_\nu K^3}\gg \Theta^*
$$ which together with $\Thetam(0)=0$ ensures that that there is a unique solution of the equation
$$\Thetam(u^*(\alpha_\gamma))=\Theta^*, \ \ 0<u^*(\alpha_\gamma)<\frac 12.$$ We now aim at estimating the size of the unique solution $u^*(\alpha_\gamma)$ in various regimes of the parameter $\alpha_\gamma$.\\

\noindent{\bf step 1} First boundary layer. {We introduce the following function
\bee
M^*(u):=\Gamma(\alpha_\gamma)\Gamma(1-\alpha_\gamma)K^{\nu_b+3-\alpha_\gamma}u^{K-1}.
\eee}
{Assume that }
$${\alpha_\gamma}=\frac{K^{\nu_b+3}(\sigma b)^{K-1}}{\Theta^*}\ \ \textrm{ for some }\delta<\sigma<\frac 1{\delta}.$$ 
Then $\alpha_\gamma|\log b|\ll 1$ for $0<b<b^*(\delta)$ and we compute:
\bea
\label{formalruwhwoew}
u&=&\left(\frac{M^*{(u)}}{K^{\nu_b+3-\alpha_\gamma}\Gamma(\alpha_\gamma)\Gamma(1-\alpha_\gamma)}\right)^{\frac{1}{K-1}}=\left(\frac{\alpha_\gamma M^*{(u)}}{K^{\nu_b+3}(1+O(\alpha_\gamma|\log b|))}\right)^{\frac{1}{K-1}}\\
\nonumber&=&  \left(\frac{K^{\nu_b+3}(\sigma b)^{K-1}}{\Theta^*}\frac{ M^*{(u)}}{K^{\nu_b+3}(1+o_{b\to0}(1))}\right)^{\frac{1}{K-1}}=  (1+o_{b\to0}(1))\left(\frac{M^*{(u)}}{\Theta^*}\right)^{\frac 1{K-1}}\sigma b.
\eea

{Next, we bound from below and above the function $\Thetam(u)/M^*(u)$ for $u$ in the interval $b\de\leq u\leq b/\de$. First, in view of  \eqref{cneiovnenvenenne}, we have
\bea\label{eq:addedlabelasitisusefulforthereader}
\nonumber\frac{\Thetam(u)}{M^*(u)} & = &  \left[1+o_{b\to 0}(1)\right]\left[\frac{1}{\Gamma(1-\alpha_\gamma)}+\frac{K+\nu_b+2}{\Gamma(2-\alpha_\gamma)}u\right]\\
&&+(u K)^{\alpha_\gamma}\frac{M_1(u)}{u^{K-1}\Gamma(\alpha_\gamma)\Gamma(1-\alpha_\gamma)K^{\nu_b+3}u^{\alpha_\gamma}}.
\eea
Since all three terms in \eqref{eq:addedlabelasitisusefulforthereader} are positive and in view of the range of $\alpha_\gamma$, we deduce
\bee
\frac{\Thetam(u)}{M^*(u)} &\geq& \left[1+o_{b\to 0}(1)\right]\frac{1}{\Gamma(1-\alpha_\gamma)}\geq \frac{1}{2}.
\eee} 
{From \eqref{lowerobundzerobisbis} with $j=1$, we estimate for $b\leq u\leq \frac{b}{\delta}$, using also the fact that} $|\log (1-u)|\le u$:
\bea
\label{eoijvenoieneon}
\nonumber (u K)^{\alpha_\gamma}\frac{M_1(u)}{{M^*(u)}}&\le &\frac{c_{\nu}\Gamma(\alpha_\gamma)\Gamma(1-\alpha_\gamma)K^{\nu_b+3}\left(\frac{u}{1-u}\right)^{K-1}u^{\alpha_\gamma}}{u^{K-1}\Gamma(\alpha_\gamma)\Gamma(1-\alpha_\gamma)K^{\nu_b+3}u^{\alpha_\gamma}}\\
& \leq & c_\nu e^{(K-1)\frac{b}{{\delta}}}\leq  {c_\nu} e^{\frac{c_\nu}{{\delta}}}.
\eea
{Using \eqref{lowerobundzerobis}, the estimate also holds for $\delta b\leq u\leq b$}. We conclude that for {$\delta b\leq u\leq \frac{b}{\delta}$}:
\bee
{\frac{1}{2}\,\frac{M^*(u)}{\Theta_*}}\leq \frac{\Thetam(u)}{{\Theta_*}}\leq e^{{\frac{c_\nu}{\delta}}}\,{\frac{M^*(u)}{\Theta_*}}.
\eee
{Since $M^*((0,+\infty))=(0,+\infty)$, there exists $0<u_1<u_2<+\infty$ such that
$$
\frac 12 \frac{M^*(u_2)}{\Theta^*}\ge 1,\qquad e^{\frac{c_\nu}\delta}\frac{M^*(u_1)}{\Theta^*}\le 1.
$$
Note also that for any $u$ such that $M^*(u)=O_\nu(\delta^{-1})\Theta_*$, we have  from \eqref{formalruwhwoew}
\bee
u= (1+o_{b\to0}(1))\left(\frac{M^*(u)}{\Theta^*}\right)^{\frac 1{K-1}}\sigma b =  \sigma b\left(e^{O_\nu(\delta^{-1})}\right)^{\frac{1}{K-1}}=\sigma b(1+o_{b\to 0}(1))
\eee
so that $u$ belongs to the range $b\leq u\leq \frac{b}{\delta}$. This is in particular true for $u_1$ and $u_2$, and therefore
\bee
\frac{\Thetam(u_2)}{\Theta_*}\geq 1, \ \ \frac{\Thetam(u_1)}{\Theta_*}\leq 1, \ \ b\delta\leq u_1<u_2\leq \frac{b}{\delta}
\eee
so by the mean value theorem, $u^*(\alpha_\gamma)$ belongs to the range $\delta b\leq u\leq b/\delta$, and satisfies
\bee
u^*(\alpha_\gamma)=\sigma b(1+o_{b\to 0}(1))
\eee
and} \eqref{valueapprocimate}, \eqref{firsteataimteboundary} are proved.\\

\noindent{\bf step 2} Second boundary layer.  {We introduce the following function
\bee
M^{**}(u):=\Gamma(\alpha_\gamma)\Gamma(1-\alpha_\gamma)K^{\nu_b+3-\alpha_\gamma}u^{K-1}\frac{K+\nu_b+2}{\Gamma(2-\alpha_\gamma)}u.
\eee}
{Assume that we have}
$${\alpha_\gamma}=1-\frac{K^{\nu_b+3}(\sigma b)^{K-1}}{\Theta^*}\ \ \textrm{{ for some }} \delta<\sigma<\frac 1{\delta}$$ 
and compute:
\bea
\label{vnekvnenevonevnovenone}
\nonumber u&=& \left(\frac{M^{{**}}{(u)}}{\Gamma(\alpha_\gamma)\Gamma(1-\alpha_\gamma)\frac{K+\nu_b+2}{{\Gamma(2-\alpha_\gamma)}}K^{\nu_b+3-\alpha_\gamma}}\right)^{\frac{1}{K}}\\
\nonumber& = & \left(\frac{M^{{**}}{(u)}(1-\alpha_\gamma)(1+o_{b\to 0}(1))}{K^{\nu_b+3}}\right)^{\frac{1}{K}}=\left(\frac{M^{{**}}{(u)}}{\Theta^*}\right)^{\frac{1}{K}}(\sigma b)^{{\frac{K-1}{K}}}(1+o_{b\to 0}(1))\\
&=&\left(\frac{M^{{**}}{(u)}}{\Theta^*}\right)^{\frac{1}{K}} \sigma b(1+o_{b\to 0}(1))
\eea

{Next, we bound from below and above the function $\Thetam(u)/M^{**}(u)$ for $u$ in the interval $b\de\leq u\leq b/\de$. First, in view of  \eqref{cneiovnenvenenne}, we have
{\bea\label{eq:addedlabelasitisusefulforthereader:bis}
\nonumber\frac{\Thetam(u)}{M^{**}(u)} & = &  \left[1+o_{b\to 0}(1)\right]\left[1+\frac{\Gamma(2-\alpha_\gamma)}{\Gamma(1-\alpha_\gamma)(K+\nu_b+2)u}\right]\\
&&+\frac{\Gamma(2-\alpha_\gamma)}{(K+\nu_b+2)u}(u K)^{\alpha_\gamma}\frac{M_1(u)}{u^{K-1}\Gamma(\alpha_\gamma)\Gamma(1-\alpha_\gamma)K^{\nu_b+3}u^{\alpha_\gamma}}.
\eea}
Since all three terms in \eqref{eq:addedlabelasitisusefulforthereader:bis} are positive, we deduce
\bee
\frac{\Thetam(u)}{M^{**}(u)} &\geq& \left[1+o_{b\to 0}(1)\right]\geq \frac{1}{2}.
\eee} 

For {$\delta b\le u\le b/\delta$}, we have, recalling \eqref{eoijvenoieneon}:
\bee
{M_1(u)} \le c_\nu\Gamma(\alpha_\gamma)\Gamma(1-\alpha_\gamma)K^{\nu_b+3-\alpha_\gamma}u^{K-1}{e^{\frac{c_\nu}{\delta}}}
\eee 
{and, since 
\bee
\frac{\Gamma(2-\alpha_\gamma)}{(K+\nu_b+2)u} &\leq& \frac{c_\nu}{bu}\leq\frac{c_\nu}{\delta}
\eee}
we have 
{\bee
\frac{1}{2}\,\frac{M^{**}(u)}{\Theta_*}\leq \frac{\Thetam(u)}{\Theta_*}\leq e^{{\frac{c_\nu}{\delta}}}\,\frac{M^{**}(u)}{\Theta_*}
\eee
We may conclude, as in step 1, that $u^*(\alpha_\gamma)$ belongs to the range $\delta b\leq u\leq b/\delta$,} and {\eqref{firsteataimteboundary} \eqref{valueapprocimate:2ndcase} hold}.\\

\noindent{\bf step 3} Away from the integer boundary {layer}. {We now}  rewrite \eqref{cneiovnenvenenne} {using \eqref{aymptoticratio}}:
\bee
\nonumber &&\Thetam{(u)} = \Gamma(\alpha_\gamma)\frac{\Gamma(K+\nu_b+2)}{\Gamma(K-1+\alpha_\gamma)}u^{K-1}\left\{1+\frac{K+\nu_b+2}{\Gamma(2-\alpha_\gamma)}\Gamma(1-\alpha_\gamma)u\right\}+M_1(u)\\
& = &\Gamma(\alpha_\gamma)\Gamma(1-\alpha_\gamma)K^{\nu_b+3}\left(\frac{u}{1-u}\right)^{K-1}u^{\alpha_\gamma}\left[\frac{M_1(u)}{\Gamma(\alpha_\gamma)\Gamma(1-\alpha_\gamma){K^{\nu_b+3}}\left(\frac{u}{1-u}\right)^{K-1}u^{\alpha_\gamma}}+\right.\\
& +& \left. \frac{(1+o_{b\to 0}(1))(1-u)^{K-1}}{(uK)^{\alpha_\gamma}\Gamma(1-\alpha_\gamma)}+\frac{(1+o_{b\to 0}(1))uK(1-u)^{K-1}}{(uK)^{\alpha_\gamma}\Gamma({2-}\alpha_\gamma)}\right].
\eee
For $u=\sigma b$, $\sigma \ge \frac{1}{\delta}$, we have using the bound $\log (1-u)\le -u$:
$${\frac{(uK)^{-\alpha_\gamma}(1-u)^{K-1}}{\Gamma(1-\alpha_\gamma)}+\frac{(uK)^{1-\alpha_\gamma}(1-u)^{K-1}}{\Gamma(2-\alpha_\gamma)}}\leq c_\nu\sigma^{1-\alpha_\gamma}e^{-(K-1)\sigma b}\leq c_\nu \sigma^{1-\alpha_\gamma}e^{-c_\nu\sigma}=o_{\delta\to 0}{(1)}$$
We then obtain the representation formula in this zone
\bee
&&\Thetam{(u)}\\
& =& \Gamma(\alpha_\gamma)\Gamma(1-\alpha_\gamma)K^{\nu_b+3}\left(\frac{u}{1-u}\right)^{K-1}u^{\alpha_\gamma}\left[\frac{M_1(u)}{\Gamma(\alpha_\gamma)\Gamma(1-\alpha_\gamma){K^{\nu_b+3}}\left(\frac{u}{1-u}\right)^{K-1}u^{\alpha_\gamma}}+o_{\delta\to 0}(1)\right].
\eee
We conclude from \eqref{lowerobundzerobisbis}:
$$\Thetam(u)=\Theta^* \Leftrightarrow\Gamma(\alpha_\gamma)\Gamma(1-\alpha_\gamma)K^{\nu_b+3}\left(\frac{u}{1-u}\right)^{K-1}u^{\alpha_\gamma}=\Theta^*e^{O_\nu(1)}
 $$
 and \eqref{jvheoeonenoenenv} follows. \\
 {Also,  we have
 \bee
 \frac{\Theta_*}{\Gamma(\alpha_\gamma)\Gamma(1-\alpha_\gamma)K^{\nu_b+3}}=\frac{\Theta_*\alpha_\gamma}{(\Gamma(1-\alpha_\gamma))^2K^{\nu_b+3}}=\frac{\Theta_*(1-\alpha_\gamma)}{\Gamma(\alpha_\gamma)\Gamma(2-\alpha_\gamma)K^{\nu_b+3}}
 \eee
 Thus, in view of the range of $\alpha_\gamma$, we infer
 \bee
 \frac{\Theta_*}{\Gamma(\alpha_\gamma)\Gamma(1-\alpha_\gamma)K^{\nu_b+3}}\geq \frac{1}{\Gamma\left(\frac{1}{2}\right)\max\left(\Gamma\left(\frac{1}{2}\right), \Gamma\left(\frac{3}{2}\right)\right)}\left(\frac{b}{\de}\right)^{K-1}
 \eee 
 so that 
 \bee
 \left(\frac{u}{1-u}\right)^{K-1}u^{\alpha_\gamma}\geq \left(\frac{b}{\de}\right)^{K-1}e^{O_\nu(1)}
 \eee
 and \eqref{vnieneoneonevnonen} easily follows.}
\end{proof}


\subsection{Bilinear estimate for $\mathcal T$}


We now develop the set of nonlinear estimates to control the fixed point equation \eqref{fineioneoineeogn}.

\begin{lemma}[Pointwise bilinear estimate]
\label{lemmabilinear}
Let $$F=\sum_{k=0}^{K-1}f_ku^k+r_F, \ \ G=\sum_{k=0}^{K-1}g_ku^k+r_G.$$ 
and $$A_F=\sup_{0\le k\le K-1}\frac{|f_k|}{|w_{\gamma-1,\nu_b+1}(k)|}, \ \ A_G=\sup_{0\le k\le K-1}\frac{|g_k|}{|w_{\gamma-1,\nu_b+1}(k)|},$$ then we have the following pointwise bounds for $0\leq u\leq \frac 12$:\\
\begin{enumerate}
\item{Bound for $\T r$}. Let $1\le j\le 5$, then
\bea
\label{pointwiseboudnproftu}
&&\left|\frac{\mathcal T(r_{u^jFG})}{M_0}(u)\right|\leq   c_{\nu,a}A_FA_G\left[1+u^{K+j-2}\Gamma(\alpha_\gamma)K^{\nu_b+3-\alpha_\gamma}\right]\\
\nonumber & + &  \frac{c_{\nu,a}}{b}\left(1+u^{K-1}\Gamma(\alpha_\gamma)K^{\nu_b+3-\alpha_\gamma}\right)\left[A_F\left\|\frac{r_G}{M_0}\right\|_{L^\infty(v\le u)}+A_G\left\|\frac{r_F}{M_0}\right\|_{L^\infty(v\le u)}\right]\\
\nonumber &+&\frac{c_{\nu,a}}{b} \left\|\frac{r_F}{M_0}\right\|_{L^\infty(v\le u)}\left\|\frac{r_G}{M_0}\right\|_{L^\infty(v\le u)}\|M_0\|_{L^\infty(v\le u)}.
\eea

\item{Bound for $r$}. Let $0\le j\le 5$,
\bea
\label{pointwiseboudnproftubis}
\nonumber&&\left|\frac{r_{u^jFG}}{M_0}(u)\right|\leq   c_{\nu,a}A_FA_G\left[1 {+u^{K-2}\Gamma(\alpha_\gamma)K^{\nu_b+2-\alpha_\gamma}}+u^{K-1}\Gamma(\alpha_\gamma)K^{\nu_b+3-\alpha_\gamma}\right]\\
\nonumber & + &  c_{\nu,a}\left(1+u^{K-1}\Gamma(\alpha_\gamma)K^{\nu_b+3-\alpha_\gamma}\right)\left[A_F\left\|\frac{r_G}{M_0}\right\|_{L^\infty(v\le u)}+A_G\left\|\frac{r_F}{M_0}\right\|_{L^\infty(v\le u)}\right]\\
 &+&c_{\nu,a} \left\|\frac{r_F}{M_0}\right\|_{L^\infty(v\le u)}\left\|\frac{r_G}{M_0}\right\|_{L^\infty(v\le u)}\|M_0\|_{L^\infty(v\le u)}.
\eea
\end{enumerate}
\end{lemma}

\begin{proof}[Proof of Lemma \ref{lemmabilinear}] We compute
\bee
u^jFG&=&\sum_{k=0}^{2(K-1)}\left[\sum_{k_1+k_2=k{, 0\le k_1,k_2\le K-1}}f_{k_1}g_{k_2}\right]u^{k+j}+ u^jr_G\left(\sum_{k=0}^{K-1}f_ku^k\right)\\
&&+u^jr_F\left(\sum_{k=0}^{K-1}g_ku^k\right)+u^jr_Fr_G.
\eee
Then, recalling \eqref{definitonioperatorramineder}:
$$r_{u^jFG}=\sum_{k=K-j}^{2(K-1)}\left[\sum_{k_1+k_2=k,0\le k_1,k_2\le K-1}f_{k_1}g_{k_2}\right]u^{k+j}+u^j\left\{r_G\left(\sum_{k=0}^{K-1}f_ku^k\right)+r_F\left(\sum_{k=0}^{K-1}g_ku^k\right)+r_Fr_G\right\}.$$

We now apply $\mathcal T$ and estimate all the terms in the corresponding identity.\\

\noindent{\bf step 1} Polynomial terms. We compute from {\eqref{valuemzerou:0:0}}:
\bee
&&\frac{1}{M_0}\T\left[\left(\sum_{k=K-j}^{2(K-1)}\sum_{k_1+k_2=k,0\le k_1,k_2\le K-1}f_{k_1}g_{k_2}\right)u^{k+j}\right]\\
&=& \sum_{k=K-j}^{2(K-1)}\left[\sum_{k_1+k_2=k,0\le k_1,k_2\le K-1}f_{k_1}g_{k_2}\right]\frac{\T(u^{k+j})}{(K+\nu_b+2)w_{\gamma-1,\nu_b+1}(K-1)\mathcal T(bu^{K})}.
\eee
We now observe the bound for $m\ge 0$ and $0\le u<1$:
\be
\label{cneioneoneonevneo}
0\le \frac{\T(u^{K+m})}{\T(u^K)}=\frac{\int_0^u\frac{(1-{v})^{\gamma+\nu_b}}{{v}^{\gamma-1}}v^{K+m}dv}{\int_0^u\frac{(1-{v})^{\gamma+\nu_b}v^K}{{v}^{\gamma-1}}dv}\le u^m
\ee
which gives the estimate
\bee
&&\left|\frac{1}{M_0}\T\left(\sum_{k=K-j}^{2(K-1)}\sum_{k_1+k_2=k,0\le k_1,k_2\le K-1}f_{k_1}g_{k_2}\right)u^{k+j}\right|\\
& \leq & \left(\sup_{0\le k\le K-1}\frac{|f_k|}{|w_{\gamma-1,\nu_b+1}(k)|}\right)\left(\sup_{0\le k\le K-1}\frac{|g_k|}{|w_{\gamma-1,\nu_b+1}(k){|}}\right)\frac 1{b(K+\nu_b+2)w_{\gamma-1,\nu_b+1}(K-1)}\\
&\times &  \sum_{k=K-j}^{2(K-1)}\left[\sum_{k_1+k_2=k,0\le k_1,k_2\le K-1}|w_{\gamma-1,\nu_b+1}(k_1)||w_{\gamma-1,\nu_b+1}(k_2)|\right]u^{k+j-K}\\
& \leq & \frac{A_FA_G}{w_{\gamma-1,\nu_b+1}(K-1)} \sum_{k=K-j}^{2(K-1)}\left[\sum_{k_1+k_2=k,0\le k_1,k_2\le K-1}|w_{\gamma-1,\nu_b+1}(k_1)||w_{\gamma-1,\nu_b+1}(k_2)|\right]u^{k+j-K}.
\eee

\noindent\underline{case $K-j\le k\le K-1$}: we use the convolution bound \eqref{tobeprovoeonor} which gives the estimate
\bee
&&\sum_{k=K-j}^{K-1}\left[\sum_{k_1+k_2=k,0\le k_1,k_2\le K-1}|w_{\gamma-1,\nu_b+1}(k_1)||w_{\gamma-1,\nu_b+1}(k_2)|\right]u^{k+j-K}\\
& \leq & c_{\nu,a}\sum_{k=K-j}^{K-1}w_{\gamma-1,\nu_b+1}(k)u^{k+j-K}=c_{\nu,a}\sum_{m=0}^{j-1}{w_{\gamma-1,\nu_b+1}(K-1-(j-1-m))}u^{m}.
\eee
We now recall {\eqref{cneioneineonoen}} \eqref{vnkndknlknvlvnlen} which ensures that for $m\le j-1\le 5$:
$$\frac{w_{\gamma-1,\nu_b+1}(K-1-(j-1-m))}{w_{{\gamma-1},{\nu_b+1}}(K-1)}={\frac{w_{\gamma,\nu_b}(K-(j-1-m))}{w_{\gamma,\nu_b}(K)}}\leq c_{\nu} \left(\frac{\gamma}{1+{K}}\right)^{j-1-m}\leq c_\nu$$
and leads to the bound:
\bea
\label{bpotjtpjhtopjhpjh}
\nonumber &&\frac{1}{w_{\gamma-1,\nu_b+1}(K-1)} \sum_{k=K-j}^{K-1}\left[\sum_{k_1+k_2=k,0\le k_1,k_2\le K-1}|w_{\gamma-1,\nu_b+1}(k_1)||w_{\gamma-1,\nu_b+1}(k_2)|\right]u^{k+j-K}\\
& \leq & c_{\nu,a}\sum_{m=0}^{j-1}u^j\leq c_{\nu,a}.
\eea

\noindent\underline{case $k\ge K$}: we use {\eqref{cneioneineonoen} and} the truncated convolution bound \eqref{truncatedconvolutionestimate}:
\bea
\label{ndknkbnklbnlbndbd}
\nonumber &&\frac{1}{w_{\gamma-1,\nu_b+1}(K-1)}\sum_{k=K}^{2(K-1)}\left[\sum_{k_1+k_2=k,0\le k_1,k_2\le K-1}|w_{\gamma-1,\nu_b+1}(k_1)||w_{\gamma-1,\nu_b+1}(k_2)|\right]u^{k+j-K}\\
\nonumber & \leq & \frac{c_{\nu,a}}{w_{\gamma-1,\nu_b+1}(K-1)} \sum_{k=K}^{2(K-1)}w_{\gamma-1,\nu_b+1}(k-(K-1))w_{\gamma-1,\nu_b+1}(K-1)u^{k+j-K}\\
& \leq & c_{\nu,a}u^j\sum_{m=0}^{K-2}w_{\gamma-1,\nu_b+1}(m+1)u^m\leq  c_{\nu,a}u^{j-1}\sum_{k=1}^{K-1}w_{\gamma-1,\nu_b+1}(k)u^k.
\eea
We now claim the uniform bound for $u\le \frac 12$:
\be
\label{estraminderwrror}
\sum_{k=0}^{K-2}w_{\gamma-1,\nu_b+1}(k)u^k\le c_{\nu,a}
\ee
which is proved below. Since from \eqref{aymptoticratio} 
\bea
\label{cneoneneoen}
\nonumber w_{\gamma-1,\nu_b+1}(K-1)&=&\frac{\Gamma(\alpha_\gamma)\Gamma(K-1+\nu_b+1+2)}{\Gamma(\gamma-2)}=\frac{\Gamma(\alpha_\gamma)\Gamma(K+\nu_b+2)}{\Gamma(K-1+\alpha_\gamma)}\\
&=& \left[1+o_{K\to _\infty}(1)\right]\Gamma(\alpha_\gamma)K^{\nu_b+3-\alpha_\gamma},
\eea
we conclude:
\be
\label{vklsnbkcbueghioeo}
u^{j-1}\sum_{k=1}^{K-1}w_{\gamma-1,\nu_b+1}(k)u^k\leq c_{\nu,a}u^{j-1}\left[1+u^{K-1}\Gamma(\alpha_\gamma)K^{\nu_b+3-\alpha_\gamma}\right],
\ee
and the collection of the above bounds yields, {using also $j\geq 1$},
\bea
\label{enoenoenenoenvo}
\nonumber &&\left|\frac{1}{M_0}\T\left(\sum_{k=K-j}^{2(K-1)}\sum_{k_1+k_2=k,0\le k_1,k_2\le K-1}f_{k_1}g_{k_2}\right)u^{k+j}\right|\\
\nonumber &\leq &  \frac{A_FA_G}{w_{\gamma-1,\nu_b+1}(K-1)} \sum_{k=K-j}^{2(K-1)}\left[\sum_{k_1+k_2=k,0\le k_1,k_2\le K-1}|w_{\gamma-1,\nu_b+1}(k_1)||w_{\gamma-1,\nu_b+1}(k_2)|\right]u^{k+j-K}\\
&\leq & c_{\nu,a}A_FA_G\left[1+u^{K+j-2}\Gamma(\alpha_\gamma)K^{\nu_b+3-\alpha_\gamma}\right].
\eea

\noindent{\em Proof of \eqref{estraminderwrror}}. From {\eqref{cneioneineonoen} and} \eqref{summationphi}, for some large enough $K_\nu$:
$$\sum_{k=0}^{K-K_\nu}w_{\gamma-1,\nu_b+1}(k)u^k\le \sum_{k=0}^{K-K_\nu}w_{\gamma-1,\nu_b+1}(k){=(\gamma-2)\sum_{k=1}^{K+1-K_\nu}w_{\gamma,\nu_b}(k)}\le c_\nu$$ 
and from {\eqref{cneioneineonoen} and \eqref{weightnenoe}}:
\bee
\sum_{k=K-K_\nu+1}^{K-2}w_{\gamma-1,\nu_b+1}(k)u^k &=& {\sum_{k=K-K_\nu+1}^{K-2}(\gamma-2)w_{\gamma,\nu_b}(k+1)u^k}\\
&\le& c_\nu\sum_{k=K-K_\nu+1}^{K-2}{(\gamma-2)}\Gamma(\gamma-{2}-k)\gamma^{\nu_b+2-(\gamma-{2}-k)}\frac{1}{2^k}\le \frac{K^{c_\nu}}{2^K}\le c_\nu
\eee
and \eqref{estraminderwrror} is proved.\\

\noindent{\bf step 2} Cross terms. We estimate pointwise from \eqref{bounditerate}, \eqref{estraminderwrror}, \eqref{cneoneneoen}:
\bee
&&\left|\frac{\T\left(u^jr_G(\sum_{k=0}^{K-1}f_ku^k)\right)}{M_0}\right|\leq A_F\left(\sup_{v\le u}\sum_{k=0}^{K-1}w_{\gamma-1,\nu_b+1}(k)v^k\right)\left\|\frac{r_G}{M_0}\right\|_{L^\infty(v\le u)}\frac{\T(u^jM_0)}{M_0}\\
& \leq & \frac{c_{\nu,a}}{b}A_F\left(1+u^{K-1}\Gamma(\alpha_\gamma)K^{\nu_b+3-\alpha_\gamma}\right)\left\|\frac{r_G}{M_0}\right\|_{L^\infty(v\le u)}
\eee
and similarly for the other cross term.\\

\noindent{\bf step 3} Nonlinear term. We estimate direcrtly from \eqref{bounditerate}:
\bee
\left|\frac{\T(u^jr_Fr_G)}{M_0}(u)\right|&\le& \left\|\frac{r_F}{M_0}\right\|_{L^\infty(v\le u)}\left\|\frac{r_G}{M_0}\right\|_{L^\infty(v\le u)}\|M_0\|_{L^\infty(v\le u)}\frac{\T(u^jM_0)}{M_0}\\
&\le&  \frac{c_{\nu,a}}{b} \left\|\frac{r_F}{M_0}\right\|_{L^\infty(v\le u)}\left\|\frac{r_G}{M_0}\right\|_{L^\infty(v\le u)}\|M_0\|_{L^\infty(v\le u)}.
\eee

\noindent{\bf step 4} Estimate for $r$. We now revisit the above estimates allowing for $j=0$. We start with the polynomial term. We {have}:
\bee
&&\left|\sum_{k=K-j}^{2(K-1)}\left[\sum_{k_1+k_2=k,0\le k_1,k_2\le K-1}f_{k_1}g_{k_2}\right]u^{k+j}\right|\\
&\leq &c_{\nu,a} A_F A_G u^K\sum_{k=K-j}^{2(K-1)}\left[\sum_{k_1+k_2=k,0\le k_1,k_2\le K-1}|w_{\gamma-1,\nu_b+1}(k_1)||w_{\gamma-1,\nu_b+1}(k_2)|\right]u^{k+j-K}.
\eee
For $j\ge {0}$ and $K-j\le k\le K-1$, we use \eqref{bpotjtpjhtopjhpjh} {and \eqref{cneoneneoen}} which imply:
\bee
&&u^K\sum_{k=K-j}^{K-1}\left[\sum_{k_1+k_2=k,0\le k_1,k_2\le K-1}|w_{\gamma-1,\nu_b+1}(k_1)||w_{\gamma-1,\nu_b+1}(k_2)|\right]u^{k+j-K}\\
&\le&  c_{\nu,a}w_{\gamma-1,\nu_b+1}(K-1)u^K\leq  c_{\nu,a}\Gamma(\alpha_\gamma)K^{\nu_{{b}}+3-\alpha_\gamma}u^K.
\eee
For $K\le k\le2(K-1)$ and $j\ge 0$, we argue differently, depending on u.\\

\noindent\underline{case $u\le b$}. We recall \eqref{ndknkbnklbnlbndbd} which yields:
\bee
&&u^K\sum_{k=K}^{2(K-1)}\left[\sum_{k_1+k_2=k,0\le k_1,k_2\le K-1}|w_{\gamma-1,\nu_b+1}(k_1)||w_{\gamma-1,\nu_b+1}(k_2)|\right]u^{k+j-K}\\
&\leq & c_{\nu,a}w_{\gamma-1,\nu_b+1}(K-1)u^Ku^j\sum_{m=0}^{K-2}w_{\gamma-1,\nu_b+1}(m+1)u^m.
\eee
We now have the rough bound from \eqref{inucnoitnwrmwjaggmamk}:
$$w_{\gamma-1,\nu_b+1}(m+1)=\frac{m+\nu_b+3}{{\gamma}-m-3}w_{\gamma-1,{\nu_b+1}}(m)=\frac{m+\nu_b+3}{K+\alpha_\gamma-m-{2}}w_{\gamma-1, {\nu_b+1}}(m)\le \frac{c_\nu}{b}w_{\gamma-1, {\nu_b+1}}(m).$$
{for $m\leq K-3$}. Then, from \eqref{estraminderwrror} {and \eqref{cneoneneoen}}:
\bee
&&u^K\sum_{k=K}^{2(K-1)}\left[\sum_{k_1+k_2=k,0\le k_1,k_2\le K-1}|w_{\gamma-1,\nu_b+1}(k_1)||w_{\gamma-1,\nu_b+1}(k_2)|\right]u^{k+j-K}\\
&\leq & \frac{c_{\nu,a}\Gamma(\alpha_\gamma)K^{\nu_b+3-\alpha_\gamma}}{b}u^{K+j}\sum_{m=0}^{K-{3}}w_{\gamma-1,\nu_b+1}(m)u^m\\
&& {+c_{\nu,a}\Gamma(\alpha_\gamma)K^{\nu_b+3-\alpha_\gamma}u^Ku^jw_{\gamma-1,\nu_b+1}(K-1)u^{K-2}}\\
&\le& \frac{c_{\nu,a}\Gamma(\alpha_\gamma)K^{\nu_b+3-\alpha_\gamma}}{b}u^{K+j} {+c_{\nu,a}\Gamma(\alpha_\gamma)K^{\nu_b+3-\alpha_\gamma}u^Ku^j\Gamma(\alpha_\gamma)K^{\nu_b+3-\alpha_\gamma}u^{K-2}}
\eee
which yields the bound:
\bee
&&\left|\sum_{k=K-j}^{2(K-1)}\left[\sum_{k_1+k_2=k,0\le k_1,k_2\le K-1}f_{k_1}g_{k_2}\right]u^{k+j}\right|\\
&\leq& \frac{c_{\nu,a}}{b} A_F A_G \Gamma(\alpha_\gamma)K^{\nu_{{b}}+3-\alpha_\gamma}u^K {+A_FA_Gc_{\nu,a}\Gamma(\alpha_\gamma)K^{\nu_b+3-\alpha_\gamma}u^K\Gamma(\alpha_\gamma)K^{\nu_b+3-\alpha_\gamma}u^{K-2}}.
\eee
Therefore, from \eqref{lowerobundzero} for $u\le b$:
\bee
&&\frac{1}{M_0}\left|\sum_{k=K-j}^{2(K-1)}\left[\sum_{k_1+k_2=k,0\le k_1,k_2\le K-1}f_{k_1}g_{k_2}\right]u^{k+j}\right|\\
&\leq&  \frac{{\frac{c_{\nu,a}}{b}}A_FA_G \Gamma(\alpha_\gamma)K^{\nu_{{b}}+3-\alpha_\gamma}u^K {+A_FA_Gc_{\nu,a}\Gamma(\alpha_\gamma)K^{\nu_b+3-\alpha_\gamma}u^K\Gamma(\alpha_\gamma)K^{\nu_b+3-\alpha_\gamma}u^{K-2}}}{\Gamma(\alpha_\gamma)\Gamma(1-\alpha_\gamma)K^{\nu_b+4-\alpha_\gamma}u^{K}}\\
& \leq &c_{\nu}A_FA_G {+c_{\nu}A_FA_G\Gamma(\alpha_\gamma)K^{\nu_b+2-\alpha_\gamma}u^{K-2}}\leq {c_{\nu}A_FA_G\big[1+u^{K-2}\Gamma(\alpha_\gamma)K^{\nu_b+2-\alpha_\gamma}\big]}.
\eee

\noindent{\underline{case $u\ge b$}}. We recall {\eqref{ndknkbnklbnlbndbd} \eqref{cneoneneoen}} \eqref{vklsnbkcbueghioeo} which yield
\bee
&&u^K\sum_{k=K}^{2(K-1)}\left[\sum_{k_1+k_2=k,0\le k_1,k_2\le K-1}|w_{\gamma-1,\nu_b+1}(k_1)||w_{\gamma-1,\nu_b+1}(k_2)|\right]u^{k+j-K}\\
&\leq& c_{\nu,a}u^{j-1}\Gamma(\alpha_\gamma)K^{\nu_b+3-\alpha_\gamma}u^{K}u^{j-1}\left[1+u^{K-1}\Gamma(\alpha_\gamma)K^{\nu_b+3-\alpha_\gamma}\right]
\eee
and thus, for $b\le u\le \frac 12$, from {\eqref{lowerobundzerobis:0}}:
\bee
&&\frac{1}{M_0}\left|\sum_{k=K-j}^{2(K-1)}\left[\sum_{k_1+k_2=k,0\le k_1,k_2\le K-1}f_{k_1}g_{k_2}\right]u^{k+j}\right|\\
&\leq & \frac{A_FA_G\Gamma(\alpha_\gamma)K^{\nu_b+3-\alpha_\gamma}u^{K}u^{j-1}\left[1+u^{K-1}\Gamma(\alpha_\gamma)K^{\nu_b+3-\alpha_\gamma}\right]}{\Gamma(\alpha_\gamma)\Gamma(1-\alpha_\gamma)K^{\nu_b+3}\left(\frac{u}{1-u}\right)^{K-1}u^{\alpha_\gamma}}\\
& \leq & c_{\nu}A_FA_G\left[1+u^{K-1}\Gamma(\alpha_\gamma)K^{\nu_b+3-\alpha_\gamma}\right].
\eee
We estimate the cross term from \eqref{estraminderwrror}, \eqref{cneoneneoen}
\bee
&&\left|\frac{u^jr_G(\sum_{k=0}^{K-1}f_ku^k)}{M_0}\right|\leq {A_F}\left(\sup_{v\le u}\sum_{k=0}^{K-1}w_{\gamma-1,\nu_b+1}(k)v^k\right)\left\|\frac{r_G}{M_0}\right\|_{L^\infty(v\le u)}\\
& \leq & c_{\nu,a}A_F\left(1+u^{K-1}\Gamma(\alpha_\gamma)K^{\nu_b+3-\alpha_\gamma}\right)\left\|\frac{r_G}{M_0}\right\|_{L^\infty(v\le u)}
\eee
and the nonlinear term 
\bee
\left|\frac{u^jr_Fr_G}{M_0}(u)\right|&\le& \left\|\frac{r_F}{M_0}\right\|_{L^\infty(v\le u)}\left\|\frac{r_G}{M_0}\right\|_{L^\infty(v\le u)}\|M_0\|_{L^\infty(v\le u)}.
\eee
\end{proof}


\subsection{Controlling the final remainder}


We are now in position to prove the exit condition for the $\mathcal C^\infty$ solution for a large enough range of parameters.

\begin{lemma}[Uniform control of the final remainder]
\label{propositionfundamental}
Pick universal constants $\frac{1}{\delta},\Theta^*\gg 1$, then for all $0<b<b^*(\delta,\Theta^*)\ll1$ small enough, the following holds. Let 
\be
\label{boudnaryleyeralphag}
\frac{K^{\nu_b+3}\left(b\delta\right)^{K-1}}{\Theta^*}<\alpha_\gamma<1-\frac{K^{\nu_b+3}\left(b\delta\right)^{K-1}}{\Theta^*}
\ee
and  let $u^*(\alpha_\gamma)$ be the solution to \eqref{neknvonenneudatsgeeaa} described by Lemma \ref{lemmanekvneneoneon}. Then the $\mathcal C^\infty$ solution to \eqref{fineioneoineeogn} satisfies the bound:
\be
\label{neionenoenven}
\frac{|r_{\mathcal G}|}{M_0}+\frac{|\T r_{\mathcal G}|}{M_0}<\sqrt{b}.
\ee
\end{lemma}

\begin{proof}[Proof of Lemma \ref{propositionfundamental}]  We recall the fixed point formulation \eqref{expressionnonlnieanrtmer} of the $\mathcal C^\infty$ solution and the expression \eqref{fialfrmaulg} for the nonlinear term:
\bee
 &&\matchal G= \mathcal G_0+\left[b^2x\th_1+bx^2\th_2\right]\Theta+\left[b^2x\th_3+bx^2\th_4\right]u\Theta'\\
 \nonumber & + &  \sum_{j=2}^4x^{j+1}\mt^{(1)}_j\Theta^j+b\sum_{j=2}^4m^{(2)}_jx^j\Theta^j+  \left[\sum_{j=1}^3\mt^{(3)}_jx^{j+2}\Theta^j+b\sum_{j={1}}^{{3}}\mt^{(4)}_jx^{j{+1}}\Theta^j\right](u\Theta').
 \eee
 We now bootstrap the bound
 \be
 \label{bootstrap}
 \frac{|r_\Theta(u)|}{M_0}<\Theta^*.
 \ee
 {Let us first check that \eqref{bootstrap} holds for $u$ small enough. From  \eqref{expressionnonlnieanrtmer}, we have
\bee
 \frac{|r_\Theta(u)|}{M_0} &\leq& c_{\nu,a}+ \frac{|\TT(r_{\mathcal{G}})(u)|}{M_0}\leq c_{\nu,a}+ \left(\sup_{v\leq u}\frac{|\mathcal{G}(v)|}{v^K}\right)\frac{|\TT(u^K)|}{M_0}\\
 &\leq& c_{\nu,a}+ \left(\sup_{v\leq u}\frac{|\mathcal{G}(v)|}{v^K}\right)\frac{1}{b(K+\nu_b+2)w_{\gamma-1,\nu_b+1}(K-1)}  
\eee
where we have used \eqref{valuemzerou:0:0} in the last inequality. Also, we have, using \eqref{boundednessgk} and \eqref{cneioneineonoen},
\bee
\lim_{u\to 0}\sup_{v\leq u}\frac{|\mathcal{G}(v)|}{v^K} = |g_K| \leq c_{\nu,a}\frac{|w_{\gamma,\nu_ b}(K)|}{1+K}\leq c_{\nu,a}\frac{|w_{\gamma-1,\nu_ b+1}(K-1)|}{(1+K)(\gamma-2)}.
\eee
We infer for $u$ small enough that
\bee
\frac{|r_\Theta(u)|}{M_0} \leq c_{\nu,a}<\Theta_*
\eee
so that \eqref{bootstrap} indeed holds for $u$ small enough}. We therefore work on the interval ${u}\in[0,u_{\rm boot}]$ with ${0<}u_{\rm boot}\le u^*(\alpha\gamma)$ where \eqref{bootstrap} holds, and aim at improving \eqref{bootstrap}.\\

\noindent{\bf step 1} Uniform bounds for $0\le u\le u^*(\alpha_\gamma)$. By definition \eqref{deftehtemain}: $$\Thetam(u)=w_{\gamma-1,\nu_b+1}(K-1)u^{K-1}+M_0(u)$$ and hence, since $\Thetam$ is non decreasing: 
\be
\label{uniformbound}
\forall u\in [0,u^*(\alpha_\gamma)], \ \ 0\le M_0(u)\le \Thetam(u)\le \Thetam(u^*(\alpha_\gamma))=\Theta^*.
\ee
Observe that in the regime \eqref{valueapprocimate}:
$$
\Gamma(\alpha_\gamma)\Gamma(1-\alpha_\gamma)K^{\nu_b+3-\alpha_\gamma}(u^*(\alpha_\gamma))^{K-1}=C_{\Theta^*,\delta},$$
{ in the regime \eqref{valueapprocimate:2ndcase}, recalling \eqref{firsteataimteboundary}:
\bee
\Gamma(\alpha_\gamma)\Gamma(1-\alpha_\gamma)K^{\nu_b+3-\alpha_\gamma}(u^*(\alpha_\gamma))^{K-1}\leq \frac{\Gamma(2-\alpha_\gamma)\Theta^*e^{O_{\nu}(\sigma)}}{u^*(\alpha_\gamma)(K+\nu_b+2)}\leq \frac{\Theta^*e^{O_{\nu}(\sigma)}}{\de}\leq C_{\Theta^*,\delta},
\eee}
and in the regime \eqref{jvheoeonenoenenv}, recalling \eqref{vnieneoneonevnonen}:
$$\Gamma(\alpha_\gamma)\Gamma(1-\alpha_\gamma)K^{\nu_b+3-\alpha_\gamma}(u^*(\alpha_\gamma))^{K-1}\leq  \frac{(1-u^*(\alpha_\gamma))^{K-1}}{(Ku^*(\alpha_\gamma))^{\alpha_\gamma}}\Theta^*e^{O_\nu(1)}\leq C_{\Theta^*,\delta}
$$
In {all three} cases 
\be
\label{noennoeoenvoen}
\forall u\in[0,u^*(\alpha_\gamma)], \ \  \Gamma(\alpha_\gamma)\Gamma(1-\alpha_\gamma)K^{\nu_b+3-\alpha_\gamma}u^{K-1}\le C_{\Theta^*,\delta}.
\ee
{Also, we have, using \eqref{noennoeoenvoen},
\bee
\Gamma(\alpha_\gamma)\Gamma(1-\alpha_\gamma)K^{\nu_b+2-\alpha_\gamma}u^{K-2}\le \Gamma(\alpha_\gamma)\Gamma(1-\alpha_\gamma)K^{\nu_b+2-\alpha_\gamma}(u^*(\alpha_\gamma))^{K-2}\le \frac{C_{\Theta^*,\delta}}{Ku^*(\alpha_\gamma)}
\eee
and thus, using \eqref{firsteataimteboundary} \eqref{vnieneoneonevnonen}, we deduce
\be
\label{noennoeoenvoen:0:0}
\forall u\in[0,u^*(\alpha_\gamma)], \ \  \Gamma(\alpha_\gamma)\Gamma(1-\alpha_\gamma)K^{\nu_b+2-\alpha_\gamma}u^{K-2}\le C_{\Theta^*,\delta}.
\ee
}
We now {decompose $r_{\mathcal{G}}$ according to the decomposition \eqref{fialfrmaulg} of $\mathcal{G}$} and estimate all the terms.\\

\noindent{\bf step 2} Source term. For any holomorphic function $H(x)$, we have from {\eqref{definitonioperatorramineder}}:
$$|r_{H}|=\left|\sum_{k=K}^{+\infty}h_ku^{k}\right|\leq \sum_{k=K}^{+\infty}(bC_hu)^{k}\le (bC_h)^Ku^K.$$ We therefore estimate for $u\le b$ from \eqref{lowerobundzero}:
$$\frac{|r_{H}|}{M_0}\leq  \frac{(bC_H)^Ku^K}{\Gamma(\alpha_\gamma)\Gamma(1-\alpha_\gamma)K^{\nu_b+4-\alpha_\gamma}u^{K}}\le    \frac{(bC_H)^K}{\Gamma(\alpha_\gamma)\Gamma(1-\alpha_\gamma)K^{\nu_b+3}}$$
and for $b\le u\le \frac 12$ from {\eqref{lowerobundzerobis:0}}:
\bee
\frac{|r_H|}{M_0}\le \frac{(bC_H)^Ku^K}{\Gamma(\alpha_\gamma)\Gamma(1-\alpha_\gamma)K^{\nu_b+3}\left(\frac{u}{1-u}\right)^{K-1}u^{\alpha_\gamma}}\leq \frac{(bC_H)^K}{\Gamma(\alpha_\gamma)\Gamma(1-\alpha_\gamma)K^{\nu_b+3}}
\eee
It implies the rough bound
\be
\label{vebivebibebeibv}
\left\|\frac{r_{H}}{M_0}\right\|_{L^\infty(0\le u\le \frac 12)}\leq b^4.
\ee 
In view of \eqref{nveioneionoenv|} this bound can be applied to the source term $\matchal G_0$.\\

\noindent{\bf step 3} Derivative term. Let $$\Theta=\sum_{k=0}^{K-1}\theta_ku^k+r_\Theta,$$ then $$u\Theta'=\sum_{k=1}^{K-1}k\theta_ku^k+ur'_\Theta.$$ Therefore, 
\be
\label{firstrkatjoijit}
\left|\begin{array}{l}
(u\Theta')_k=k\theta_k, \ \  0\le k\le K-1\\
r_{u\Theta'}=ur_\Theta'.
\end{array}\right.
\ee
Moreover, from \eqref{fineioneoineeogn}:
$$ur'_\Theta=(-1)^{K-1}S_{K-1}uM'_0(u)-u\left[\T(r_\mathcal G)\right]'$$ 
We recall from {\eqref{vnoneneoneonve}}:
$$u\left[\T(\mathcal G)\right]'=\frac{1}{1-u}\left[\frac{\mathcal G}{b}+[(\gamma-2)+(\nu_b+3)u]\mathcal T(\mathcal G)\right]$$
which yields 
$$
r_{u\Theta'}=(-1)^{K-1}S_{K-1}uM'_0(u)-\frac{1}{1-u}\left[\frac{r_\mathcal G}{b}+[(\gamma-2)+(\nu_b+3)u]\mathcal T(r_\mathcal G)\right]
$$ 
We obtain the estimate,  using \eqref{vniovnioneneneo} {and \eqref{esitmaitmitot}}:
\be
\label{estiamtieboudnary}
\left|\begin{array}{l}
\frac{|r_{u\Theta'}|}{M_0}\leq \frac{c_\nu}{b}\left[1+\frac{|r_\mathcal G|}{M_0}+\frac{|\T(r_\mathcal G)|}{M_0}\right]\\
\sup_{0\le k\le K-1}\frac{|(u\Theta')_k|}{w_{\gamma-1,\nu_b+1}(k)}\leq \frac{c_{\nu}}{b}.
\end{array}\right.
\ee

\noindent{\bf step 4} Linear term. Recall from \eqref{eoneonveonoe} that for a holomorphic function: $$\forall 0\le k\le K-1, \ \ |(H \Theta)_k|\le c_H w_{\gamma-1,\nu_b+1}(k).$$
\noindent\underline{no derivative term}. We apply \eqref{pointwiseboudnproftu} with the bounds \eqref{vebivebibebeibv}, \eqref{uniformbound}, \eqref{noennoeoenvoen} {\eqref{noennoeoenvoen:0:0} and \eqref{esitmaitmitot}} to derive for $ 1\le j\le 5$:
\bee
&&\left|\frac{\mathcal T(r_{u^jH\Theta})}{M_0}(u)\right|\leq   c_{\nu,a}A_\Theta A_H\left[1+u^{K+j-2}\Gamma(\alpha_\gamma)K^{\nu_b+3-\alpha_\gamma}\right]\\
\nonumber & + &  \frac{c_{\nu,a}}{b}\left(1+u^{K-1}\Gamma(\alpha_\gamma)K^{\nu_b+3-\alpha_\gamma}\right)\left[A_\Theta\left\|\frac{r_H}{M_0}\right\|_{L^\infty(v\le u)}+A_H\left\|\frac{r_\Theta}{M_0}\right\|_{L^\infty(v\le u)}\right]\\
\nonumber &+&\frac{c_{\nu,a}}{b} \left\|\frac{r_\Theta}{M_0}\right\|_{L^\infty(v\le u)}\left\|\frac{r_H}{M_0}\right\|_{L^\infty(v\le u)}\|M_0\|_{L^\infty(v\le u)}\leq c_{\nu,a}+\frac{C_{\Theta^*,\delta}}{b}
\eee
where we used the bootstrap bound \eqref{bootstrap} in the last step.
Now, writing $$\left|\begin{array}{l}
b^2x\th_1\Theta =b^3u(\th_1\Theta)\\
bx^2\th_2\Theta=b^3u^2(\th_2\Theta)
\end{array}\right.
$$
gives
$$\left|\frac{\T r_{\left[b^2x\th_1+bx^2\th_2\right]\Theta}}{M_0}\right|\leq b^3\left[c_{\nu,a}+\frac{C_{\Theta^*,\delta}}{b}\right]\le b^2C_{\Theta^*,\delta}.$$
Similarly, from \eqref{pointwiseboudnproftubis}, \eqref{bootstrap}, {\eqref{vebivebibebeibv}, \eqref{uniformbound}, \eqref{noennoeoenvoen} \eqref{noennoeoenvoen:0:0} and \eqref{esitmaitmitot}}:
$$\left|\frac{r_{u^jH\Theta}}{M_0}(u)\right|\leq   C_{\Theta^*,\delta}$$
Therefore,
$$\left|\frac{r_{\left[b^2x\th_1+bx^2\th_2\right]\Theta}}{M_0}\right|\le b^3C_{\Theta^*,\delta}.$$

\noindent\underline{derivative term}. We first use \eqref{estiamtieboudnary}, \eqref{pointwiseboudnproftubis} to estimate for a holomorphic function $H$:
\bee
&&\left|\frac{r_{u^jH(u\Theta')}}{M_0}(u)\right|\leq   \frac{c_{\nu,a}}{b}\left[1{+u^{K-1}\Gamma(\alpha_\gamma)K^{\nu_b+3-\alpha_\gamma}
+u^{K-2}\Gamma(\alpha_\gamma)K^{\nu_b+2-\alpha_\gamma}}\right]\\
\nonumber & + &  c_{\nu,a}\left(1+u^{K-1}\Gamma(\alpha_\gamma)K^{\nu_b+3-\alpha_\gamma}\right)\left[1+\left\|\frac{r_{u\Theta'}}{M_0}\right\|_{L^\infty(v\le u)}\right]\\
\nonumber &+&C_{\Theta^*,\delta} \left\|\frac{r_{u\Theta'}}{M_0}\right\|_{L^\infty(v\le u)}\leq  C_{\Theta^*,\delta}\left[\frac{1}{b}+\left\|\frac{r_{u\Theta'}}{M_0}\right\|_{L^\infty(v\le u)}\right]
\eee
and from {\eqref{estiamtieboudnary}}, \eqref{pointwiseboudnproftu}:
$$\left|\frac{\mathcal T(r_{u^jH(u\Theta')})}{M_0}(u)\right|\leq   C_{\Theta^*,\delta}\left[\frac{1}{b^2}+\frac{1}{b}\left\|\frac{r_{u\Theta'}}{M_0}\right\|_{L^\infty(v\le u)}\right].$$
We conclude, using \eqref{estiamtieboudnary}:
\bee
&&\left|\frac{\T r_{\left[b^2x\th_1+bx^2\th_2\right]u\Theta'}}{M_0}\right|\leq b^3C_{\Theta^*,\delta}\left[\frac{1}{b^2}+\frac{1}{b}\left\|\frac{r_{u\Theta'}}{M_0}\right\|_{L^\infty(v\le u)}\right]\leq  C_{\Theta^*,\delta}\left[b+b^2\left\|\frac{r_{u\Theta'}}{M_0}\right\|_{L^\infty(v\le u)}\right]\\
& \leq & C_{\Theta^*,\delta}\left[b+b\left\|\frac{r_\mathcal G}{M_0}\right\|_{L^\infty(v\le u)}+b\left\|\frac{\mathcal T(r_\mathcal G)}{M_0}\right\|_{L^\infty(v\le u)}\right]
\eee
and
\bee
&&\left|\frac{r_{\left[b^2x\th_1+bx^2\th_2\right]u\Theta'}}{M_0}\right|\le b^3C_{\Theta^*,\delta}\left[\frac{1}{b}+\left\|\frac{r_{u\Theta'}}{M_0}\right\|_{L^\infty(v\le u)}\right]\\
& \leq & C_{\Theta^*,\delta}\left[b^2+b^2\left\|\frac{r_\mathcal G}{M_0}\right\|_{L^\infty(v\le u)}+b^2\left\|\frac{\mathcal T(r_\mathcal G)}{M_0}\right\|_{L^\infty(v\le u)}\right]
\eee

\noindent{\bf step 5} Nonlinear term.\\

\noindent\underline{no derivative term}. First, {we have in view of \eqref{esitmaitmitot}}  \eqref{tobeprovoeonorbibib} that for $1\le m\le 5$ and $0\le k\le K-1$: 
\bee
|(\Theta^m)_k|&\le& {\sum_{k_1+\cdots k_m=k}|\Theta_{k_1}|\cdots |\Theta_{k_m}|\le c_{\nu,a}\sum_{k_1+\cdots k_m=k}w_{\gamma-1, \nu_b+1}(k_1)\cdots w_{\gamma-1, \nu_b+1}(k_m)}\\
&\le& c_{\nu,a}w_{\gamma-1,\nu_b+1}{(k)}
\eee
We then estimate, {using} \eqref{pointwiseboudnproftubis} {iteratively in $m$ }for $2\le m\le 5$, $0\le j\le 5$, {and using also \eqref{bootstrap}, \eqref{vebivebibebeibv}, \eqref{uniformbound}, \eqref{noennoeoenvoen} \eqref{noennoeoenvoen:0:0}}:
\bee
\left|\frac{r_{u^j\Theta^m}}{M_0}(u)\right|&\leq&   c_{\nu,a}{A_{u^j}\left(A_{\Theta}+\left\|\frac{r_\Theta}{M_0}\right\|_{L^\infty(v\le u)}\right)\left(\sum_{\ell=1}^mA_{\Theta^\ell}+\left(A_{\Theta}+\left\|\frac{r_\Theta}{M_0}\right\|_{L^\infty(v\le u)}\right)^{m-2}\right)}\\
&&\times\left[1+u^{K-1}\Gamma(\alpha_\gamma)K^{\nu_b+3-\alpha_\gamma}{+u^{K-2}\Gamma(\alpha_\gamma)K^{\nu_b+2-\alpha_\gamma}}\right]^{{m-1}}\leq C_{\Theta^*,\delta}
\eee
Similarly,  for a holomorphic function $H$:
\be
\label{neevnnel;m;ldeoen}
\left|\frac{r_{u^jH\Theta^m}}{M_0}(u)\right|   {\leq C_{\Theta^*,\delta}}
\ee
which implies that
$$
 \left|\frac{r_{\sum_{j=2}^4x^{j+1}\mt^{(1)}_j\Theta^j+b\sum_{j=2}^4m^{(2)}_jx^j\Theta^j}}{M_0}\right|\le C_{\Theta^*,\delta}b^3.$$
Similarly, for $1\le j\le 5$ from {\eqref{bootstrap}} \eqref{vebivebibebeibv}, \eqref{uniformbound}, \eqref{noennoeoenvoen}:
$$
\left|\frac{\mathcal T(r_{u^jH\Theta^j})}{M_0}(u)\right|\le \frac{C_{\Theta^*,\delta}}{b}
$$
gives
\bee
 \left|\frac{\T r_{\sum_{j=2}^4x^{j+1}\mt^{(1)}_j\Theta^j+b\sum_{j=2}^4m^{(2)}_jx^j\Theta^j}}{M_0}\right|\le b^3\frac{C_{\Theta^*,\delta}}{b}\le C_{\Theta^*,\delta}b^2.
 \eee
 
 \noindent\underline{derivative term}. We estimate from \eqref{neevnnel;m;ldeoen}, \eqref{pointwiseboudnproftubis}, \eqref{estiamtieboudnary}, {\eqref{uniformbound}, \eqref{noennoeoenvoen} \eqref{noennoeoenvoen:0:0}}:
 \bee
&& \left|\frac{r_{u^jH\Theta^m(u\Theta')}}{M_0}(u)\right|\leq C_{\Theta^*,\delta}\left[\frac{1}{b}+\left\|\frac{r_{u\Theta'}}{M_0}\right\|_{L^\infty(v\le u)}\right]\\
& \leq & C_{\Theta^*,\delta}\left[\frac{1}{b}+\frac{1}{b}\left({\left\|\frac{r_\mathcal G}{M_0}\right\|_{L^\infty(v\le u)}+\left\|\frac{\mathcal T(r_\mathcal G)}{M_0}\right\|_{L^\infty(v\le u)}}\right)\right].
 \eee
 Similarly,
 \bee
 && \left|\frac{\T r_{u^jH\Theta^m(u\Theta')}}{M_0}(u)\right|\leq \frac{C_{\Theta^*,\delta}}{b}\left[\frac{1}{b}+\left\|\frac{r_{u\Theta'}}{M_0}\right\|_{L^\infty(v\le u)}\right]\\
 & \leq &  \frac{C_{\Theta^*,\delta}}{b^2}\left[1+{\left\|\frac{r_\mathcal G}{M_0}\right\|_{L^\infty(v\le u)}+\left\|\frac{\mathcal T(r_\mathcal G)}{M_0}\right\|_{L^\infty(v\le u)}}\right].
 \eee
 Therefore,
 \bee
 &&\frac{|r_{ \left[\sum_{j=1}^3\mt^{(3)}_jx^{j+2}\Theta^j+b\sum_{j={1}}^{{3}}\mt^{(4)}_jx^{j{+1}}\Theta^j\right](u\Theta')}|}{M_0}\\
 &\leq&  b^3C_{\Theta^*,\delta}\left[\frac{1}{b}+\frac{1}{b}\left({\left\|\frac{r_\mathcal G}{M_0}\right\|_{L^\infty(v\le u)}+\left\|\frac{\mathcal T(r_\mathcal G)}{M_0}\right\|_{L^\infty(v\le u)}}\right)\right]\\
 & \leq & C_{\Theta^*,\delta}b^2\left[1+{\left\|\frac{r_\mathcal G}{M_0}\right\|_{L^\infty(v\le u)}+\left\|\frac{\mathcal T(r_\mathcal G)}{M_0}\right\|_{L^\infty(v\le u)}}\right]
 \eee
and
 \bee
 &&\frac{|\T r_{ \left[\sum_{j=1}^3\mt^{(3)}_jx^{j+2}\Theta^j+b\sum_{j={1}}^{{3}}\mt^{(4)}_jx^{j{+1}}\Theta^j\right](u\Theta')}|}{M_0}\\
 &\leq&  b^3\frac{C_{\Theta^*,\delta}}{b^{{2}}}\left[1{+\left\|\frac{r_\mathcal G}{M_0}\right\|_{L^\infty(v\le u)}+\left\|\frac{\mathcal T(r_\mathcal G)}{M_0}\right\|_{L^\infty(v\le u)}}\right]\\
 & \leq & bC_{\Theta^*,\delta}\left[1+{\left\|\frac{r_\mathcal G}{M_0}\right\|_{L^\infty(v\le u)}+\left\|\frac{\mathcal T(r_\mathcal G)}{M_0}\right\|_{L^\infty(v\le u)}}\right].
 \eee
 
\noindent{\bf step 6} Conclusion. The collection of the above bounds yields:
\bee
\frac{|r_{\mathcal G}|}{M_0}\le C_{\Theta^*,\delta}{b^2\left[1+\left\|\frac{r_\mathcal G}{M_0}\right\|_{L^\infty(v\le u)}+\left\|\frac{\mathcal T(r_\mathcal G)}{M_0}\right\|_{L^\infty(v\le u)}\right]}
\eee
and
\bee
\frac{|\T r_{\mathcal G}|}{M_0}\leq bC_{\Theta^*,\delta}\left[1+{\left\|\frac{r_\mathcal G}{M_0}\right\|_{L^\infty(v\le u)}+\left\|\frac{\mathcal T(r_\mathcal G)}{M_0}\right\|_{L^\infty(v\le u)}}\right]
\eee
which imply for $0<b<b^*(\Theta*,\delta)$ small enough $$\frac{|r_{\mathcal G}|}{M_0}+\frac{|\T r_{\mathcal G}|}{M_0}<\sqrt{b}$$ which, reinserted into \eqref{fineioneoineeogn}, yields:
$$\frac{|r_\Theta|}{M_0}\le c_{\nu,a}+\sqrt{b}<\frac{\Theta^*}{2},$$ and \eqref{bootstrap}, \eqref{neionenoenven} are proved.
\end{proof}


\subsection{{Exit on the left of $P_2$}}


We are now in position to establish the fundamental exit property of the $\mathcal C^\infty$ solution to the left of $P_2$. We assume without loss of generality that 
\be
\label{neinenvoneonve}
S_\infty(d,\ell)>0
\ee and need only to reverse the parity of $K$ in the following Lemma if $S_\infty(d,\ell)<0$.

\begin{lemma}[Exit on the left]
\label{lemmaexitleft}
Pick universal constants $\frac{1}{\delta},\Theta^*\gg1 $ large enough as in Lemma \ref{propositionfundamental}, then  for all $0<b<b^*(\Theta^*,\delta)$ small enough and $\alpha_\gamma$ in the range \eqref{boudnaryleyeralphag}, the $\mathcal C^\infty$ solution exits on the left of $P_2$ { at $u=U_*$, where $0<U_*<\frac{3}{4}$}, {by crossing $\Delta_2=0$} for $K$ odd, and {by crossing $\Delta_1=0$} for $K$ even.
\end{lemma}

\begin{proof}[Proof of Lemma \ref{lemmaexitleft}]  
\noindent{\bf step 1} Reaching $\Theta^*$. Recall {\eqref{expressionnonlnieanrtmer}}
$$
\Theta(u) =\sum_{k=0}^{K-2}\theta_k u^k+(-1)^{K-1}S_{\infty}\left[1+o_{b\to 0}(1)\right]\Thetam(u)-\T(r_\mathcal G)
$$ 
then, provided $\Theta^*$ has been chosen large enough, we conclude from {\eqref{estraminderwrror},} \eqref{neionenoenven}, \eqref{uniformbound} that for $0<b<b^*(\Theta^*,\delta)$ small enough, for all $u\in [0,u^*(\alpha_\gamma)]$:
\be\label{cenoenneeo:bbb}
|\Theta(u)-(-1)^{K-1}S_{\infty}\Thetam(u)|\leq \frac{S_\infty\Theta^*}{10}.
\ee
Therefore,
\be
\label{cenoenneeo}
{\frac{\Theta^*}{2}\leq \frac{(-1)^{K-1}\Theta(u^*(\alpha_\gamma))}{S_\infty}\le 2\Theta^*.}
\ee

\noindent{\bf step 2} Computation of $\Delta_1,\Delta_2$. 
We now unfold our changes of variables and show that, depending on the sign of $(-1)^{K-1}S_\infty$,  
we must have passed through either the green or the red curves. {To this end, we examine $\Delta_1$ and $\Delta_2$ under the assumption 
\bea\label{eq:aprioriestimateonPsiwhichisusedtoestimateNL1and2andholdsintheend}
|u|\leq \frac{3}{4}, \ \ |\Psi(u)|\leq 2\Theta_*.
\eea}
From {\eqref{suihfoenioeneoi}, \eqref{defphi}, \eqref{vneinvenoen}} 
$$\left|\begin{array}{l}
\Phit=(1-u)\Psi\\
\Psi(u)=M_b(x)\Phi(u)\\
\Phi=bu\Theta\\
\end{array}\right.
$$
where $M_b$ is bounded and given by \eqref{deflkenrnal}. Moreover, from {\eqref{defgaonetwo}, \eqref{defttoneone}, \eqref{liknkfones},  \eqref{seconexpessiongofone}, \eqref{nioneinevioohve}}:
\bee
&& \Delta_1-c_-\Delta_2=\mathcal G_2=b^2\mathcal F_1=-b^2|\l_-|\psite\wte(c_+-c_-)uF_1\\
 & = & -b^2|\l_-|\psite\wte(c_+-c_-)u\left[u(1-u)bH_1(b,u)+(1+G_1(bu))\Phit+\NL_1(u,\Phit)\right]\\
 & = & -{b^2\wte^2\et_{20}}{(1+O(b))}u\left[u(1-u)bH_1(b,u)+(1+G_1(bu))bu(1-u)M_b\Theta+\NL_1(u,\Phit)\right]\\
 & = & -{b^3\wte^2\et_{20}}{(1+O(b))}u^2(1-u)\left[H_1(b,u)+(1+G_1(bu))M_b\Theta+\frac{\NL_1(u,\Phit)}{bu(1-u)}\right]
 \eee
 and from \eqref{defgaone}, \eqref{lineftwogone}, {\eqref{liknkfones},} \eqref{expressionnffofw}, \eqref{nioneinevioohve}:
\bee
&&-\Delta_1+c_+\Delta_2= \mathcal G_1=-b^2\mathcal F_2=b^2\wte(c_+-c_-)|\mu_+|uF_2\\
& = & b^2\wte(c_+-c_-)|\mu_+|u\left[(1-u)\left[1+H_2(b,u)\right]+G_2(bu)\Phit+\NL_2(u,\Phit)\right]\\
& = & {b^2|\dt_{20}|\wte^2}{(1+O(b))}u(1-u)\left[1+H_2(b,u)+G_2(bu)M_b\Phi+\frac{\NL_2(u,\Phit)}{1-u}\right]
\eee
{From \eqref{deflkenrnal}, for $x=bu$ and $u\leq 1$,
\bea\label{eq:controlofMbforulessthan1}
M_b(x)=1+O(b),
\eea
and from \eqref{fneionefonone}, \eqref{eniovnevneoneneno}, \eqref{defnltwo}} {and \eqref{eq:aprioriestimateonPsiwhichisusedtoestimateNL1and2andholdsintheend}}
$$\left|\begin{array}{l}
H_1(b,u)=-(\Et_{11}+\Et_{30})+O(b)\\
G_1=b\Et_{11}u+O(b^2)\\
\NL_1=O({b^2u})\\
H_2(b,u)=O(b)\\
G_2(x)=O(b)\\
\NL_2=O({b^2u})
\end{array}\right.
$$
It implies that, {as long as $\Theta(u)\neq 0$, we have}:
$$\left|\begin{array}{l}
\Delta_1-c_-\Delta_2= -\Theta(u)bu\left\{\wte^2{\et_{20}}b^2u(1-u)\left(1{+O\left(\frac{1}{\Theta(u)}\right)}+O(b)\right)\right\}\\
-\Delta_1+c_+\Delta_2={|\dt_{20}|}\wte^2{b^2}u(1-u)(1+O(b))
\end{array}\right.
$$
i.e.,
$$\left|\begin{array}{l}
\Delta_1=\frac1{c_+-c_-}\wte^2{b^2}u(1-u)\left(1{+O\left(\frac{1}{\Theta(u)}+b\right)}\right)\left[-c_+\et_{20}\Theta(u) bu{+}c_-|\dt_{20}|\right]\\[2mm]
\Delta_2={\frac{1}{c_+-c_-}}\wte^2{b^2}u(1-u)\left(1{+O\left(\frac{1}{\Theta(u)}+b\right)}\right)\left[{-}\et_{20}\Theta(u) bu+|\dt_{20}|\right]
\end{array}\right.
$$
or, equivalently,
\be
\label{cnonwinwnwno}
\left|\begin{array}{l}
\Delta_1={-}\frac1{c_+-c_-}\wte^2{b^2}u(1-u)\left(1{+O\left(\frac{1}{\Theta(u)}+b\right)}\right)\left[{c_+}\et_{20}\Phi(u)+|c_-||\dt_{20}|\right]\\[2mm]
\Delta_2={\frac{1}{c_+-c_-}}\wte^2{b^2}u(1-u)\left(1{+O\left(\frac{1}{\Theta(u)}+b\right)}\right)\left[{-}\et_{20}\Phi(u)+|\dt_{20}|\right].
\end{array}\right.
\ee
{Note that we have used the signs, valid for all $d\geq 2$ and $0\leq \ell\leq d$:
$$\dt_{20}<0, \ \ \et_{20}>0,$$
see \eqref{signofdt20andet20}.}\\

\noindent{\bf step 3} Touching {$\Delta_1=0$ or $\Delta_2=0$}. At $u=u^*(\alpha_\gamma)$, 
{by \eqref{cenoenneeo}} we have
$$\Phi(u^*)=bu^*{\Theta(u^*), \ \ \frac{(-1)^{K-1}\Theta(u^*)}{S_\infty}\geq \frac{\Theta^*}{2}}.$$ 
We now claim that there exists 
$$u^*\le U^*\le \frac34, \ \ \Phi(U^*)={(-1)^{K-1}}\Theta^*$$ 
Then, from \eqref{cnonwinwnwno}, {since $\dt_{20}\neq 0$ and $\et_{20}> 0$ by \eqref{signofdt20andet20}, and since $\Phi(0)=0$,} we must have crossed $\Delta_1=0$ or $\Delta_2=0$, depending on {wether $K$ is even or odd}.\\
{Assume $K$ odd, so that from \eqref{cenoenneeo} $\Theta(u_*)>0$. The case $K$ even can be treated similarly.} Since $u\le \frac 34$ and $M_b(u)=1+{O(b)}$:
$$ \Psi(u^*)=bu^*\Theta(u^*){(1+O(b))}, \ \ \Psi(U^*)=\Theta^*.$$ 
{In view of \eqref{equaitoncompee}}
\bee
 &&\left[1+H_2+G_2\Psi+\NLt_2\right]u(1-u)\Psi'\\
\nonumber&+& \left[(1-2u)(1+H_2+G_2\Psi+\NLt_2)-\gamma(1+G_1)+2uG_2\right]\Psi\\
\nonumber& = &u\left[\gamma bH_1-2(1+H_2)+\frac{\gamma b\NLt_1}{x}-2\NLt_2\right].
\eee
Then, on the interval $u\in[u^*, {U^*}]$ with $U^*\le \frac 34$ and 
\be
\label{botboudnpsishuert}
\frac{{|S_{\infty}|}bu^{{*}}\Theta^*}{{4}}<\Psi(u)\le 2\Theta^*
\ee 
we have for $b<b^*(\Theta^*)$ from {\eqref{decopmositionhtwo}}, \eqref{cneoneneono}, \eqref{defnltwo}, {\eqref{fneionefonone}}:
$$\left|\begin{array}{l}
\left|H_1(b,u)+(\Et_{11}+\Et_{30})\right|+|H_2|+|G_2|+|G_1|\le C_{\Theta^*}b\\
|\NLt_2|+|\frac{\NLt_1}{x}|\le C_{\Theta^*}b\\
\end{array}\right.
$$
We insert this into \eqref{equaitoncompee} and conclude from \eqref{botboudnpsishuert}, provided $\Theta^*>0$ has been chosen large enough,
\bea\label{eq:ODEinequalityforPsithatshowsinparitcularnochangeofsign}
u\Psi'\geq \frac{\gamma}{4}\Psi.
\eea
Therefore,
$$\Psi(u)\ge \Psi(u^*)\left(\frac{u}{u^*}\right)^{\frac{\gamma}{4}}\geq \frac{{|S_\infty|}b{u^*}\Theta^*}{{4}}\left(\frac{u}{u^*}\right)^{\frac{\gamma}{4}}\ge \Theta^*$$ 
for 
$$\left(\frac{u}{u^*}\right)^{\frac{\gamma}{4}}\ge {\frac{4}{|S_\infty|bu^*\Theta^*}}, \ \ u\ge u^*\left({\frac{4}{|S_\infty|bu^*\Theta^*}}\right)^{\frac{4}{{\gamma}}}=u^*\left(1+O({b|\log b|})\right)$$ 
{where we have used the fact that $u^*\geq b\delta$ in view of \eqref{firsteataimteboundary} \eqref{vnieneoneonevnonen}}.
Since $u^*\le \frac 12$, we established that the contact happens before $u=\frac 34$.
\end{proof}


\subsection{{Exit on the right}}


{Below, we obtain an analog of Lemma \ref{lemmaexitleft} on the right, albeit in a significantly more restricted range of $\alpha_\gamma$ in $(0,1)$.} We assume \eqref{neinenvoneonve}.

\begin{lemma}[Exit on the right]
\label{lemmaexitright}
{Pick universal constants $\frac{1}{\delta},\Theta^*\gg1 $ large enough as in Lemma \ref{propositionfundamental}, then  for all $0<b<b^*(\Theta^*,\delta)$ small enough and $\alpha_\gamma$ given by 
\be
\label{boudnaryleyeralphagright}
\alpha_\gamma=\frac{K^{\nu_b+3}\left(b\sqrt{\delta}\right)^{K-1}}{\Theta^*}
\ee
or 
\be
\label{boudnaryleyeralphagrightbis}
\alpha_\gamma=1-\frac{K^{\nu_b+3}\left(b\sqrt{\delta}\right)^{K-1}}{\Theta^*},
\ee
the $\mathcal C^\infty$ solution exits on the right of $P_2$ at $u=U_*$, where $0<U_*<\frac{3}{4}$, by crossing $\Delta_1=0$ in the case \eqref{boudnaryleyeralphagright} and by crossing $\Delta_2=0$ in the case \eqref{boudnaryleyeralphagrightbis}.}
\end{lemma}

\begin{proof}[Proof of Lemma \ref{lemmaexitright}] For $\alpha_\gamma$ given by \eqref{boudnaryleyeralphagright} or \eqref{boudnaryleyeralphagrightbis} we only need to consider $u$ in the the range $-b\leq u\leq 0$. In that range, the proof of  Lemma \ref{lemmaexitright} follows very closely the one of Lemma \ref{lemmaexitleft} for the case $0\leq u\leq b$. The $\matchal C^\infty$ regularity at the right of $P_2$ all the way to $u=0$ together with the property \eqref{neoneneonovenoev} follow again from the explicit integral representation of the remainder function. We focus on the exit behavior.\\

{\noindent{\bf step 1} Bounds on $M_j$. For $-b\leq u\leq b$, we have 
$$(1-u)^K=e^{K\log(1-u)}=e^{O(1)}$$
so that the cases $0\leq u\leq b$ and $-b\leq u\leq 0$ can be treated similarly in  the definition of $\TT$ and, as a consequence, 
in $M_0$ and $M_j$. In particular, the proof of \eqref{lowerobundzero}, \eqref{bounditerate}, \eqref{vniovnioneneneo} and \eqref{lowerobundzerobis} obtained for $0\leq u\leq b$ immediately extends to the case $-b\le u\le  0$, i.e., we have for $-b\le u\le  0$
\be
\label{lowerobundzero:right}
c_{\nu,1}\le \frac{M_0(u)}{\Gamma(\alpha_\gamma)\Gamma(1-\alpha_\gamma)K^{\nu_b+4-\alpha_\gamma}u^{K}}\le c_{\nu,2},
\ee
for $1\le j\le 5$, 
\be
\label{bounditerate:right}
\left\|\frac{\T (u^jM_0)}{M_0}\right\|_{L^\infty(-b\le u\le  0)}\le \frac{c_\nu}{b},
\ee
\be
\label{vniovnioneneneo:right}
\forall -b\le u\le  0, \ \ \frac{|uM_0'|}{M_0}\leq \frac{c_\nu}b. 
\ee 
and for  $1\le j\le 5$ and $-b\le u\le  0$
\be
\label{lowerobundzerobis:right}
c_{\nu,1}\le \frac{M_j(u)}{\Gamma(\alpha_\gamma)\Gamma(1-\alpha_\gamma)K^{\nu_b+j+4-\alpha_\gamma}u^{K+j}}\le c_{\nu,2}.
\ee}

{\noindent{\bf step 2} Estimate on $\Thetam$. Recall from \eqref{cneiovnenvenenne}
\bee
&&\Thetam(u) =  \Gamma(\alpha_\gamma)\Gamma(1-\alpha_\gamma)K^{\nu_b+3-\alpha_\gamma}u^{K-1}\\
\nonumber &\times & \left\{\left[1+o_{b\to 0}(1)\right]\left[\frac{1}{\Gamma(1-\alpha_\gamma)}+\frac{K+\nu_b+2}{\Gamma(2-\alpha_\gamma)}u\right]+(u K)^{\alpha_\gamma}\frac{M_1(u)}{u^{K-1}\Gamma(\alpha_\gamma)\Gamma(1-\alpha_\gamma)K^{\nu_b+3}u^{\alpha_\gamma}}\right\}.
\eee
In view of \eqref{lowerobundzerobis:right}, we deduce for $-b\de^{\frac{1}{3}}\leq u\leq -b\delta$, 
\bea\label{cneiovnenvenenne:right}
\Thetam(u) & = & \Gamma(\alpha_\gamma)\Gamma(1-\alpha_\gamma)K^{\nu_b+3-\alpha_\gamma}u^{K-1}\\
\nonumber &\times & \left\{\left[1+o_{b\to 0}(1)\right]\left[\frac{1}{\Gamma(1-\alpha_\gamma)}+\frac{K+\nu_b+2}{\Gamma(2-\alpha_\gamma)}u\right]+O\left(c_\nu\de^{\frac{2}{3}}\right)\right\}.
\eea}

\noindent{\bf step 3} Boundary layer. In view of \eqref{cneiovnenvenenne:right}, we easily obtain the following analog of the first two cases of Lemma \ref{lemmanekvneneoneon} for $\alpha_\gamma$ given by  \eqref{boudnaryleyeralphagright} or \eqref{boudnaryleyeralphagrightbis}. Pick universal constants $\frac{1}{\delta},\Theta^*\gg1$ then  for all $0<b<b^*(\Theta^*,\delta)$, we have:\\
\noindent\underline{first layer}: if $\alpha_\gamma$ is given by  \eqref{boudnaryleyeralphagright}, then there exist a  solution to 
\be
\label{neknvonenneudatsgeeaa:right}
\Thetam(u^*(\alpha_\gamma))=(-1)^{K-1}\Theta^*,  
\ee 
satisfying the following bounds
\be
\label{valueapprocimate:right}
\Gamma(\alpha_\gamma)\Gamma(1-\alpha_\gamma)K^{\nu_b+3-\alpha_\gamma}(u^*(\alpha_\gamma))^{K-1}=\Theta^*e^{O_{\nu}(\delta^{-1})}
\ee
and
\be
\label{firsteataimteboundary:right}
u^*(\alpha_\gamma)=-\sqrt{\delta} b(1+o_{b\to 0}(1)).
\ee
\noindent\underline{second layer}: if $\alpha_\gamma$ is given by  \eqref{boudnaryleyeralphagrightbis}, then there exist a  solution to 
\be
\label{neknvonenneudatsgeeaa:right:bis}
\Thetam(u^*(\alpha_\gamma))=(-1)^{K}\Theta^*,  
\ee 
satisfying  \eqref{firsteataimteboundary:right} and 
 \be
\label{valueapprocimate:2ndcase:right}
\Gamma(\alpha_\gamma)\Gamma(1-\alpha_\gamma)K^{\nu_b+3-\alpha_\gamma}\frac{K+\nu_b+2}{\Gamma(2-\alpha_\gamma)}(u^*(\alpha_\gamma))^{K}=\Theta^*e^{O_{\nu}(\delta^{-1})}.
\ee
 
{\noindent{\bf step 4} Estimate on $r_\mathcal G$. In view of \eqref{lowerobundzero:right} and \eqref{firsteataimteboundary:right}, we have for $u^*(\alpha_\gamma)$ defined in step 3
\be\label{uniformbound:right}
\forall u\in [u^*(\alpha_\gamma),0], \ \ 0\le |M_0(u)|\le c_\nu |M_0(u^*(\alpha_\gamma))|\leq \frac{c_\nu\Theta^*}{\sqrt{\delta}}=C_{\Theta^*,\delta}.
\ee
Also, proceeding as in the proof of \eqref{noennoeoenvoen} and \eqref{noennoeoenvoen:0:0}, we obtain the following analogs
\be
\label{noennoeoenvoen:right}
\forall u\in[u^*(\alpha_\gamma),0], \ \  \Gamma(\alpha_\gamma)\Gamma(1-\alpha_\gamma)K^{\nu_b+3-\alpha_\gamma}|u|^{K-1}\le C_{\Theta^*,\delta}.
\ee
and
\be
\label{noennoeoenvoen:0:0:right}
\forall u\in[u^*(\alpha_\gamma),0], \ \  \Gamma(\alpha_\gamma)\Gamma(1-\alpha_\gamma)K^{\nu_b+2-\alpha_\gamma}|u|^{K-2}\le C_{\Theta^*,\delta}.
\ee}
The estimate \eqref{neionenoenven} holds for $-b\leq u\leq 0$, i.e. for  $-b\leq u\leq 0$, and we now claim:
 \be
\label{neionenoenven:right}
\frac{|r_{\mathcal G}|}{M_0}+\frac{|\T r_{\mathcal G}|}{M_0}<\sqrt{b}.
\ee
Indeed, the proof of \eqref{neionenoenven} for the range $0\leq u\leq b$ does not use the sign of $u$, and all estimates hold by replacing everywhere $u$ with $|u|$. Using also \eqref{uniformbound:right} \eqref{noennoeoenvoen:right} \eqref{noennoeoenvoen:0:0:right}, the proof immediately extends to the range $-b\leq u\leq 0$. Thus, \eqref{neionenoenven:right} holds for $-b\leq u\leq 0$.\\

\noindent{\bf step 6} Conclusion. Recall \eqref{expressionnonlnieanrtmer}
$$
\Theta(u) =\sum_{k=0}^{K-2}\theta_k u^k+(-1)^{K-1}S_{\infty}\left[1+o_{b\to 0}(1)\right]\Thetam(u)-\T(r_\mathcal G)
$$ 
then, provided $\Theta^*$ has been chosen large enough, we conclude from \eqref{estraminderwrror} (with $u$ replaced by $|u|$), \eqref{neionenoenven:right}, \eqref{uniformbound:right} that for $0<b<b^*(\Theta^*,\delta)$ small enough, for all $u\in [u^*(\alpha_\gamma),0]$:
$$|\Theta(u)-(-1)^{K-1}S_{\infty}\Thetam(u)|\leq \frac{S_\infty\Theta^*}{10}.$$ 
Therefore, if $\alpha_\gamma$ is given by  \eqref{boudnaryleyeralphagright}
\be
\label{cenoenneeo:right}
\frac{\Theta^*}{2}\leq \frac{\Theta(u^*(\alpha_\gamma))}{S_\infty}\le 2\Theta^*,
\ee
and, if $\alpha_\gamma$ is given by  \eqref{boudnaryleyeralphagrightbis},
\be
\label{cenoenneeo:bis:right}
\frac{\Theta^*}{2}\leq \frac{-\Theta(u^*(\alpha_\gamma))}{S_\infty}\le 2\Theta^*.
\ee
We now note that \eqref{cnonwinwnwno} holds independently of the sign of $u\in (-1,1)$. If $\alpha_\gamma$ is given by  \eqref{boudnaryleyeralphagright}, then, by \eqref{cenoenneeo:right},
$$\Phi(u^*)=bu^*\Theta(u^*), \ \ \frac{\Theta(u^*)}{S_\infty}\geq \frac{\Theta^*}{2},$$ 
 and, if $\alpha_\gamma$ is given by  \eqref{boudnaryleyeralphagrightbis}, then, by \eqref{cenoenneeo:bis:right},
$$\Phi(u^*)=bu^*\Theta(u^*), \ \ \frac{-\Theta(u^*)}{S_\infty}\geq \frac{\Theta^*}{2}.$$ 
The rest of the argument of step 3 in the proof of Lemma \ref{lemmaexitleft} extends to the case $u<0$. Therefore,\\
\noindent\underline{first layer}: if $\alpha_\gamma$ is given by  \eqref{boudnaryleyeralphagright}, there exists $U^*(\alpha_\gamma)$ such that
$$-\frac34\le U^*\le u^*, \ \ \Phi(U^*)=-\Theta^*,$$
and the smooth solution crosses $\Delta_1=0$ for $u\ge -3/4$.\\
\noindent\underline{second layer}: if $\alpha_\gamma$  is given by  \eqref{boudnaryleyeralphagrightbis}, there exists $U^*(\alpha_\gamma)$ such that
$$-\frac34\le U^*\le u^*, \ \ \Phi(U^*)=\Theta^*,$$
and the smooth solution crosses $\Delta_2=0$ for $u\ge -3/4$.
This concludes the proof of Lemma \ref{lemmaexitright}.
\end{proof}


\subsection{Proof of Theorem \ref{thmmain}}


We are now in position to conclude the proof of Theorem \ref{thmmain}.\\

\noindent{\bf step 1} Continuous deformation of $\alpha_\gamma$. Let $K$ be even and large enough, we claim that there exists $\alpha^K_\gamma$ in $(0,1)$ such that the $\mathcal C^\infty$ solution 
$\Phi[K, \alpha^K_\gamma](u)$ coincides with the unique $P_6-P_2$ solution (to the right of $P_2$)
and  exits to the left of $P_2$ by crossing $\Delta_1=0$ before reaching $P_5$ (and, as a result, extends to $P_4$.).\\

\noindent Indeed, let $\frac{1}{\delta},\Theta^*\gg1 $ large enough and $0<b<b^*(\Theta^*,\delta)$ small enough as in Lemma \ref{lemmaexitleft}, and $\alpha_\gamma$ in the range \eqref{boudnaryleyeralphag}. Assume also that $K$ is even. Then, in view of Lemma \ref{lemmaexitleft}, 
 the $\mathcal C^\infty$ solution exits on the left of $P_2$ by crossing $\Delta_1=0$ before $u=\frac{3}{4}$. Consider then the $\mathcal C^\infty$ solution $\Phi[K, \alpha_\gamma](u)$ on the right of $P_2$. Let also $\Phi^{(rad)}[K, \alpha_\gamma](u)$ denote the unique solution, constructed earlier, which corresponds to the $P_2-P_6$ trajectory. 
 Then, let
\bea
F(\alpha_\gamma) &:=& \Phi[K, \alpha_\gamma]\left(-\frac{3}{4}\right) - \Phi^{(rad)}[K, \alpha_\gamma]\left(-\frac{3}{4}\right). 
\eea
Then, for $\alpha_\gamma$ given by \eqref{boudnaryleyeralphagright}, since $\Phi^{(rad)}[K, \alpha_\gamma](u)$ is located between the $\Delta_1=0$ and $\Delta_2=0$ curves for any $u<0$, and, since $\Phi[K, \alpha_\gamma](u)$ has crossed $\Delta_1=0$ before $u=-3/4$ and cannot cross $\Delta_1=0$ twice, we deduce
\bea
F\left(\frac{K^{\nu_b+3}\left(b\sqrt{\delta}\right)^{K-1}}{\Theta^*}\right) >0.
\eea
Also, for $\alpha_\gamma$ given by \eqref{boudnaryleyeralphagrightbis}, since $\Phi^{(rad)}[K, \alpha_\gamma](u)$ lives between $\Delta_1=0$ and $\Delta_2=0$ curves for any $u<0$, and, since $\Phi[K, \alpha_\gamma](u)$ has crossed $\Delta_2=0$ before $u=-3/4$ and cannot cross $\Delta_2=0$ twice, we deduce
\bea
F\left(1-\frac{K^{\nu_b+3}\left(b\sqrt{\delta}\right)^{K-1}}{\Theta^*}\right) <0.
\eea
Continuous dependence of the ode on the parameter $\alpha_\gamma\in (0,1)$ implies the continuity of $F$. We then 
infer by the mean value theorem the existence of  $\alpha^K_\gamma$ such that
\bea
F\left(\alpha^K_\gamma\right)=0, \ \ \alpha^K_\gamma\in \left(\frac{K^{\nu_b+3}\left(b\sqrt{\delta}\right)^{K-1}}{\Theta^*},  1-\frac{K^{\nu_b+3}\left(b\sqrt{\delta}\right)^{K-1}}{\Theta^*}\right).
\eea
Then, by the uniqueness of solutions to the ode at $u=3/4$, $\Phi[K, \alpha^K_\gamma]$ must coincide with $\Phi^{(rad)}[K, \alpha^K_\gamma]$. 

To the left of $P_2$, with $\alpha^K_\gamma$ in the range \eqref{boudnaryleyeralphag}, we have that $\Phi[K, \alpha^K_\gamma]$ crosses $\Delta_1=0$ before reaching $P_{\hskip -.1pc\peye}$. 

Thus, we have obtained for any even $K$ large enough the existence of $\alpha^K_\gamma$ in $(0,1)$ such that the smooth profile  $\Phi[K, \alpha^K_\gamma](u)$ coincides with the $P_2-P_6$ solution to the right of $P_2$ and  
exits on the left of $P_2$ by crossing $\Delta_1=0$ before reaching $P5$. The constructed solution is $\mathcal C^\infty$ to the right and the left of $P_2$, with derivatives satisfying \eqref{neoneneonovenoev} on both sides.\\

\noindent{\bf step 2} Conclusion. Since the curve crosses $\Delta_1=0$ for $\sigma_2<\sigma<\sigma_5$, it is attracted to $P_4$ by Lemma \ref{lemmaconnection}. It remains to show that in the $(\sigma(x),w(x))$ parametrization, the point $P_2$ is reached in finite time. Indeed, $$\lim_{\Sigma \to 0}\frac{W}{\Sigma}=c_-$$ and thus, from \eqref{expressionfodetijbis}, \eqref{realtionslopeweignefuncitons}, \eqref{calculparametres}:
$$\frac{d\Sigma}{dx}=\frac{d\sigma}{dx}=-\frac{\Delta_2}{\Delta}=-\frac{(c_2c_-+c_4)\Sigma(1+o(1))}{-2\sigma_2(1+c_-)\Sigma(1+o(1))}=\frac{\l_+}{2\sigma_2(1+c_-)}(1+o(1))>0$$
which proves the claim. The resulting $\mathcal C^\infty$ solution corresponds to the global $P_6-P_2-P_4$ trajectory.


\subsection{Uniform control of the $P_2-P_{\hskip -.1pc\peye}$ separatrix}


We conclude this section with a simple uniform estimate on the separatrix $P_2-P_{\hskip -.1pc\peye}$ which will be required in the proof of the exterior positivity property, see Lemma \ref{lemma:finallemmaonpositivityofthequadraticforms:regionwgeqw2}. 

\begin{lemma}[Computing $\Theta_S$]
\label{lemmaseprpatira}
There exist universal constant $C,b^*$ such that for all $0<b<b^*$, the unique separatrix curve in the eye $P_{\hskip -.1pc\peye}-P_2$ satisfies in the $\Theta_S$ variable: 
\be
\label{estimateseprpapt}
\forall 0\le u\leq \frac 12,\ \ \left|\Theta_S+\Et_{11}+\Et_{30}+\frac 2a\right|<Cb.
\ee
\end{lemma}

\begin{proof}[Proof of Lemma \ref{lemmaseprpatira}] This follows from the fixed point representation of the separatrix.\\

\noindent{\bf step 1} Bound on the separatrix. We first claim for the separatrix
\be
\label{neoinvienoenv}
|\Theta_S|\lesssim 1
\ee
uniformly as $b\to 0$.  Indeed, there holds from \eqref{thetaeqaiotjoihs}:
\bea
\label{eibeviebivbei}
&&\Theta'-\left[\frac{\gamma-2}{u}+\frac{\gamma+\nu_b+1}{1-u}\right]\Theta=-\frac{\mathcal G}{bu(1-u)}\\
\nonumber &\Leftrightarrow& \left[\frac{(1-u)^{\gamma+\nu_b+1}}{u^{\gamma-2}}\Theta\right]'=-\frac{\mathcal G(1-u)^{\nu_b+\gamma+1}}{bu(1-u) u^{\gamma-2}}
\eea
and the separatrix is the unique solution which reaches $u=1$: 
\be
\label{vneioneoenvnen}
\Theta_S=\frac{u^{\gamma-2}}{(1-u)^{\gamma+\nu_b+1}}\int_u^1\frac{\matchal G}{b\sigma(1-\sigma)}\frac{(1-\sigma)^{\gamma+\nu_b+1}}{\sigma^{\gamma-2}}d\sigma.
\ee 
Let the operator $$T(\mathcal G)(u)=\frac{u^{\gamma-2}}{(1-u)^{\gamma+\nu_b+1}}\int_u^1\frac{\matchal G}{b\sigma(1-\sigma)}\frac{(1-\sigma)^{\gamma+\nu_b+1}}{\sigma^{\gamma-2}}d\sigma,$$ 
we claim, recalling  $a=\gamma b$,
\be
\label{neionenveioneoi}
\|T(\mathcal G)\|_{L^\infty[0,1]}\lesssim_a \|\mathcal G\|_{L^\infty([0,1])}
\ee
Assume \eqref{neionenveioneoi}, then the source term is given by \eqref{cneineinveoneonnkenpe}:
$$\left|\begin{array}{l}
\mathcal G=\mathcal G_0+\mathcal L(\Phi)+{\rm NL(\Phi)}\\
\Phi=bu\Theta\\
\mathcal G_0=\frac{\gamma bH_1-2(1+H_2)}{M_b(1+H_{20})}
\end{array}\right.
$$
which is just bounded and \eqref{neoinvienoenv} follows by an elementary fixed point argument.\\
\noindent{\em Proof of \eqref{neionenveioneoi}}.  We estimate for $\frac 12\le u\leq 1$:
\bee
&&\frac{u^{\gamma-2}}{(1-u)^{\gamma+\nu_b+1}}\int_u^1\frac{1}{b\sigma(1-\sigma)}\frac{(1-\sigma)^{\gamma+\nu_b+1}}{\sigma^{\gamma-2}}d\sigma\\& \leq & \frac{u^{\gamma-2}}{(1-u)^{\gamma+\nu_b+1}}\frac{1}{b u^{\gamma-1}}\int_u^1(1-\sigma)^{\gamma+\nu_b}d\sigma\leq \frac{2}{b(\ga+\nu_b)}\leq\frac{2}{a}
\eee
and for $0<u<\frac 12$:
\bee
&&\frac{u^{\gamma-2}}{(1-u)^{\gamma+\nu_b+1}}\int_u^1\frac{1}{b\sigma(1-\sigma)}\frac{(1-\sigma)^{\gamma+\nu_b+1}}{\sigma^{\gamma-2}}d\sigma\\
&\leq& \frac{u^{\gamma-2}}{(1-u)^{\gamma+\nu_b+1}}\frac{(1-u)^{\gamma+\nu_b}}{b}\int_u^1\frac{d\sigma}{\sigma^{\gamma-1}}\leq \frac{2}{b(\ga-2)}\leq\frac{3}{a}
\eee
and  \eqref{neionenveioneoi} follows.\\

\noindent{\bf step 2} Next term. We now extract the main term.  Since $H_2=O(b)$ we have
 $$\mathcal G_0=\frac{\gamma bH_1-2(1+H_2)}{M_b(1+H_{20})}=g_0+O(b), \ \ \left|\begin{array}{l}
g_0=-a(\Et_{11}+\Et_{30})-2\\
a=\gamma b=\frac{|\l_-|}{|\mu_+|}.
\end{array}\right.
$$
 Let $$\Theta=\frac{g_0}{a} +\Thetat$$ then from \eqref{eibeviebivbei}:
$$
\Thetat'-\left[\frac{\gamma-2}{u}+\frac{\gamma+\nu_b+1}{1-u}\right]\Thetat=-\frac{\tilde{\mathcal G}}{bu(1-u)}
$$
with
\bee
\tilde{\mathcal G}=\matchal G-bu(1-u)\left(\frac{\gamma-2}{u}+\frac{\gamma+\nu_b+1}{1-u}\right)\frac{g_0}{a}=  \matchal G-b\gamma \frac{g_0}{a}+O(b)=O(b)
\eee
and hence \eqref{neionenveioneoi} ensures $\|\Thetat\|_{L^\infty([0,1])}\lesssim b$
which is \eqref{estimateseprpapt}.
 \end{proof}


\section{Interior positivity}
\label{sec:posotivitynecessaryforcontrollinearizationinsidelightcone}


This section is devoted to the proof of Lemma \ref{lemmainside}. This is the first positivity property at the heart of the control of the linearized operator in \cite{MRRSnls}. In this section we fix  $(d,\ell,r)$ and assume that we are in the range \eqref{lemmaisolated} for some sufficiently small $0<\e(d,\ell)\ll1$. We let $(\sigma(x),w(x))$ be the $P_6-P_2$ trajectory given in Lemma \ref{lememarompfoe}, and aim at proving the positivity property \eqref{coercivityquadrcouplinginside} in the region $\sigma\ge \sigma_2$, where we recall the definition \eqref{difinitotniof}.\\

The proof relies on convexity properties of the $P_6-P_2$ curve. We will need to compute various numerical constants depending on $(d,\ell,r)$. Their signs will be essential and checked at $r=\re(d,\ell)$ only, using \eqref{suaempeonpoengopngoe} and $\e(d,\ell)$ small enough. We will work in the variables $$W=w-w_2,\ \ \Sigma=\sigma-\sigma_2, \ \ \Phi=\frac{W}{\Sigma}$$ and use the algebra which follows from \eqref{expressionfodetij}:
\be
\label{expressioninsignaphi}
\left|\begin{array}{l}
\Delta_1=\Sigma\left[c_1\Phi+c_3+[d_{20}\Phi^2+d_{11}\Phi+d_{02}]\Sigma+[\Phi^3-d\Phi]\Sigma^2\right]\\[2mm]
\Delta_2=\Sigma\left[c_2\Phi+c_4+[e_{20}\Phi^2+e_{11}\Phi+e_{02}]\Sigma+[e_{21}\Phi^2-1]\Sigma^2\right]\\[2mm]
\Delta=-\Sigma\left[2\sigma_2(1+\Phi)+\Sigma(1-\Phi^2)\right]
\end{array}\right.
\ee


\subsection{Sharp bound on the slope function $\Phi$}


We recall the notations $$w_e=\frac{\ell(r-1)}{d}, \ \ w_2=1-\sigma_2.$$

\begin{lemma}[Sharp bound on the slope function]
\label{lemmaimprovedslope}
Assume \eqref{suaempeonpoengopngoe}. Let 
\be
\label{definiotinalphsatar}
\left|\begin{array}{l}
\alpha(r)=\frac{|c_-|}{w_2-w_e}\\
\alpha^\infty=\alpha(\re)=\left|\begin{array}{l} \frac{2d(\ell+\sqrt{d})}{d(\sqrt{d}-1)+\ell(\sqrt{d}+1)}\ \ \mbox{for}\ \ r=r^*\\
\frac{d(1+\sqrt{\ell})^2}{d\sqrt{\ell}+\ell}\ \ \mbox{for}\ \ r=r_+
\end{array}\right.
\end{array}\right.
\ee
then for $\Sigma\ge 0$: 
\be
\label{boundpthilde}
\frac{c_-}{1+\alpha(r)\Sigma}\leq\Phi<0.
\ee
\end{lemma}

\begin{remark} The bound \eqref{boundpthilde}
 is saturated at $P_2$.
 \end{remark}

\begin{proof}[Proof of Lemma \ref{lemmaimprovedslope}] 
We compute
\bea
\label{coiceoennenve} 
\nonumber w_2(r^*)-w_e(r^*)& = & 1-\frac{\sqrt{d}}{\ell+\sqrt{d}}-\frac{\ell\left(\frac{d+\ell}{\ell+\sqrt{d}}-1\right)}{d}=\frac{\ell}{\ell+\sqrt{d}}-\frac{\ell(\sqrt{d}-1)}{\sqrt{d}(\ell+\sqrt{d})}\\
& = & \frac{\ell}{\sqrt{d}(\ell+\sqrt{d})}
\eea
and
\bea
\label{coiceoennenvebis}
\nonumber w_2(r_+)-w_e(r_+)&=&\frac{\sqrt{\ell}}{1+\sqrt{\ell}}-\frac{\ell(d-1)}{d(1+\sqrt{\ell})^2}=\frac{\sqrt{\ell}d(1+\sqrt{\ell})-\ell(d-1)}{d(1+\sqrt{\ell})^2}\\
&=& \frac{d\sqrt{\ell}+\ell}{d(1+\sqrt{\ell})^2}.
\eea
It implies that $w_2>w_e$ for r close enough to $\re(\ell)$. 
For $0<\alpha<\alpha^\infty$ let us consider the function $$\Phit=(1+\alpha \Sigma)\Phi=\Phi+\alpha W$$ then
\be
\label{cennenevonveoevn}
\left|\begin{array}{l}
\Phit(0)=c_- \\
 \Phit(+\infty)=-\alpha(w_2-w_e)
 \end{array}\right.
 \ee
From \eqref{definiotinalphsatar} and \eqref{calculparametres}, \eqref{calculparametresbis}
which imply that $c_-<0$,
$$ c_-<-\alpha(w_2-w_e)\Leftrightarrow \alpha<\frac{|c_-|}{w_2-w_e}=\alpha(r).
$$
We compute:
$$\frac{dW}{d\Sigma}=\Sigma\frac{d\Phi}{d\Sigma}+\Phi=\frac{\Delta_1}{\Delta_2}\Leftrightarrow \frac{d\Phi}{d\Sigma}=\frac{\Delta_1-\Phi \Delta_2}{\Sigma\Delta_2}$$
Thus,
\bee
\frac{d\Phit}{d\Sigma}&=&\frac{d\Phi}{d\Sigma}+\alpha\frac{dW}{d\Sigma}=\frac{\Delta_1-\Phi \Delta_2}{\Sigma\Delta_2}+\alpha\frac{\Delta_1}{\Delta_2}=\frac{\Delta_1-\Phi\Delta_2+\alpha\Sigma\Delta_1}{\Sigma\Delta_2}\\
& = & \frac{(1+\alpha\Sigma)^2\Delta_1-\tilde{\Phi}\Delta_2}{\Sigma(1+\alpha\Sigma)\Delta_2}.
\eee

\noindent{\bf step 1} Repulsivity. We assume that there exists a point $\Sigma$ where $\tilde{\Phi}=c_-$ with $\Sigma>0$. First, we claim 
that
\be
\label{einonenneonbeo}
(1+\alpha\Sigma)^2\Delta_1-\tilde{\Phi}\Delta_2>0.
\ee 
Assuming that, since $\Delta_2<0$ in this zone, we obtain $$\frac{d\Phit}{d\Sigma}<0.$$  Since the $\Phit(\infty)>c_-$,
this leads to a contradiction and implies that
 $\frac{c_-}{1+\alpha \Sigma}<\Phi<0$. \eqref{boundpthilde} follows by passing to the limit $\alpha \to \alpha(r)$.\\
\noindent{\em Proof of \eqref{einonenneonbeo}}. 
\bee
T&=&\frac{(1+\alpha\Sigma)^2\Delta_1-c_-\Delta_2}{\Sigma}=(1+\alpha\Sigma)^2\left[c_1c_-+c_3+[d_{20}c_-^2+d_{11}c_-+d_{02}]\Sigma+[c_-^3-dc_-]\Sigma^2\right]\\
& - &c_-\left[c_2c_-+c_4+[e_{20}c_-^2+e_{11}c_-+e_{02}]\Sigma+[e_{21}c_-^2-1]\Sigma^2 \right]
\eee
From  \eqref{equationcminus}, $c_1c_-+c_3-c_-(c_2c_-+c_4)=0$. At $\re(d,\ell)$ we also have 
the additional cancellation $(c_2c_-+c_4)=\l_+=0$. Therefore,
\bee
&&\frac{T}{\Sigma}=(1+\alpha\Sigma)^2\left[d_{20}c_-^2+d_{11}c_-+d_{02}+(c_-^3-dc_-)\Sigma\right]\\
&-& c_-\left[e_{20}c_-^2+e_{11}c_-+e_{02}+(e_{21}c_-^2-1)\Sigma \right]\\
& =& (1+2\alpha\Sigma+\alpha^2\Sigma^2)\left[d_{20}c_-^2+d_{11}c_-+d_{02}+(c_-^3-dc_-)\Sigma\right]\\
&-& c_-\left[e_{20}c_-^2+e_{11}c_-+e_{02}+(e_{21}c_-^2-1)\Sigma \right]=  \sum_{i=0}^3\At^\infty_i(1+o_{b\to 0}(1))\Sigma^i
\eee
where $\tilde{A}^\infty_i$ corresponds to the limiting values at $r=\re(d,\ell)$. We claim 
\be
\label{vneionveionoev}
\tilde{A}^\infty_i>0
\ee
which concludes the proof of \eqref{einonenneonbeo} for $r$ close enough to $r^*(d,\ell)$. \\

\noindent{\bf step 2} $\At^\infty_0$. We compute:
\bee
\At^\infty_0&=& d^\infty_{20}(c^\infty_-)^2+d^\infty_{11}c^\infty_-+d^\infty_{02}-c^\infty_-(e^\infty_{20}(c^\infty_-)^2+e^\infty_{11}c^\infty_-+e^\infty_{02})\\
&=&- e^\infty_{20}(c^\infty_-)^3+(c^\infty_-)^2(d^\infty_{20}-e^\infty_{11})+(d^\infty_{11}-e^\infty_{02})c^\infty_-+d^\infty_{02}=\tilde{e}^\infty_{20}>0
\eee
from  \eqref{signofdt20andet20}, \eqref{signofdt20andet20bis}.\\

\noindent{\bf step 3} $\At^\infty_3$. We compute:
$$\At^\infty_3=(\alpha^\infty)^2\left[(c^\infty_-)^3-dc^\infty_-\right]=-c^\infty_-(\alpha^\infty)^2(d-(c^\infty_-)^2).$$
\noindent\underline{case $r=r^*$}:
\bee
d-(c^\infty_-)^2&=&d-\frac{4\ell^2 d}{\left[d(\sqrt{d}-1)+\ell(\sqrt{d}+1)\right]^2}\\
& = & \frac{d}{\left[d(\sqrt{d}-1)+\ell(\sqrt{d}+1)\right]^2}\left(\left[d(\sqrt{d}-1)+\ell(\sqrt{d}+1)\right]^2-4\ell^2\right)\\
& =& \frac{d\left[d(\sqrt{d}-1)+\ell(\sqrt{d}+1)-2\ell\right]\left[d(\sqrt{d}-1)+\ell(\sqrt{d}+1)+2\ell\right]}{\left[d(\sqrt{d}-1)+\ell(\sqrt{d}+1)\right]^2}\\
& = & \frac{d\left[d(\sqrt{d}-1)+\ell(\sqrt{d}-1)\right]\left[d(\sqrt{d}-1)+\ell(\sqrt{d}+3)\right]}{\left[d(\sqrt{d}-1)+\ell(\sqrt{d}+1)\right]^2}>0
\eee
\noindent\underline{case $r=r_+$}:
$$\At_3=-c^\infty_-(\alpha^\infty)^2(d-(c^\infty_-)^2)=(\alpha^\infty)^2(d-1).$$
Hence $\At^\infty_3>0$.\\

\noindent{\bf step 4} $\At^\infty_1$. We compute:
\bee 
\At^\infty_1& = & (c^\infty_-)^3-dc^\infty_--c^\infty_-(e^\infty_{21}(c^\infty_-)^2-1)+2\alpha^\infty(d^\infty_{20}(c^\infty_-)^2+d^\infty_{11}c^\infty_-+d^\infty_{02})\\
& = & -(d-1)c^\infty_--(c^\infty_-)^3\frac{d-1}{\ell}+2\alpha^\infty(d^\infty_{20}(c^\infty_-)^2+d^\infty_{11}c^\infty_-+d^\infty_{02}).
\eee
\noindent\underline{case $r=r^*$}. 
\bee
&&d^\infty_{20}c^\infty_-+d^\infty_{11}=\frac{\sqrt{d}+d-\ell}{\ell+\sqrt{d}}\frac{2\ell\sqrt{d}}{d(\sqrt{d}-1)+\ell(\sqrt{d}+1)}-\frac{2d\sqrt{d}}{\ell+\sqrt{d}}\\
& = & \frac{2\sqrt{d}}{(\ell+\sqrt{d})(d(\sqrt{d}-1)+\ell(\sqrt{d}+1))}\left\{-d\left[d(\sqrt{d}-1)+\ell(\sqrt{d}+1)\right]+\ell(\sqrt{d}+d-\ell)\right\}\\
& = &  \frac{2\sqrt{d}}{(\ell+\sqrt{d})(d(\sqrt{d}-1)+\ell(\sqrt{d}+1))}\left[-d^2(\sqrt{d}-1)-\ell\sqrt{d}(d-1)-\ell^2\right]<0
\eee
It implies
\bee
&&d^\infty_{20}(c^\infty_-)^2+d^\infty_{11}c^\infty_-+d^\infty_{02}=c^\infty_-(d^\infty_{20}c^\infty_-+d^\infty_{11})-\frac{\ell\sqrt{d}}{\ell+\sqrt{d}}\\
& = & -\frac{2\ell\sqrt{d}}{d(\sqrt{d}-1)+\ell(\sqrt{d}+1)}\frac{2\sqrt{d}}{(\ell+\sqrt{d})[d(\sqrt{d}-1)+\ell(\sqrt{d}+1)]}\left[-d^2(\sqrt{d}-1)-\ell\sqrt{d}(d-1)-\ell^2\right]\\
& - & \frac{\ell\sqrt{d}}{\ell+\sqrt{d}}\\
& = & \frac{\ell\sqrt{d}}{\left[d(\sqrt{d}-1)+\ell(\sqrt{d}+1)\right]^2(\ell+\sqrt{d})}\Big\{4\sqrt{d}\left[d^2(\sqrt{d}-1)+\ell\sqrt{d}(d-1)+\ell^2\right]\\
& - & \left[d(\sqrt{d}-1)+\ell(\sqrt{d}+1)\right]^2\Big\}
\eee
and 
\bee
&&4\sqrt{d}\left[d^2(\sqrt{d}-1)+\ell\sqrt{d}(d-1)+\ell^2\right]-\left[d(\sqrt{d}-1)+\ell(\sqrt{d}+1)\right]^2\\
& = & d^2(\sqrt{d}-1)[4\sqrt{d}-(\sqrt{d}-1)]+\ell\left[4d(d-1)-2d(d-1)\right]+\ell^2\left[4\sqrt{d}-(d+2\sqrt{d}+1)\right]\\
& = & d^2(\sqrt{d}-1)(3\sqrt{d}+1)+2d(d-1)\ell-\ell^2(d+1-2\sqrt{d})
\eee
Therefore,
\bea
\label{nbeivbebebebnebnebnvo}
&&d^\infty_{20}(c^\infty_-)^2+d^\infty_{11}c^\infty_-+d^\infty_{02}\\
\nonumber & =& \frac{\ell\sqrt{d}[d^2(\sqrt{d}-1)(3\sqrt{d}+1)+2d(d-1)\ell-\ell^2(d+1-2\sqrt{d})]}{\left[d(\sqrt{d}-1)+\ell(\sqrt{d}+1)\right]^2(\ell+\sqrt{d})}.
\eea
In the numerator, the { polynomial in $\ell$  is increasing on $0<\ell<d$; it is  $>0$ at $\ell=0$, so that it is strictly positive on  $0<\ell<d$. Therefore,}
$$d^\infty_{20}(c^\infty_-)^2+d^\infty_{11}c^\infty_-+d^\infty_{02}>0.$$ Hence $\At_1>0$ near $r^*$.\\

\noindent\underline{case $r=r_+$}. We compute
\bea
\label{neioneioneonoenv}
&&d^\infty_{11}-d^\infty_{20}-d^\infty_{02}=\frac{-2d}{1+\sqrt{\ell}}-\frac{\ell-\sqrt{\ell}-d-1}{(1+\sqrt{\ell})^2}+\frac{\ell+d\sqrt{\ell}}{(1+\sqrt{\ell})^2}\\
\nonumber & = & \frac{-2d(1+\sqrt{\ell})+\ell+d\sqrt{\ell}-\ell+\sqrt{\ell}+d+1}{(1+\sqrt{\ell})^2}=\frac{\sqrt{\ell}(-d+1)-d+1}{(1+\sqrt{\ell})^2}= \frac{-(d-1)}{1+\sqrt{\ell}}
\eea
and hence
\bee 
\At^\infty_1& = & -(d-1)c^\infty_--(c^\infty_-)^3\frac{d-1}{\ell}+2\alpha^\infty(d^\infty_{20}(c^\infty_-)^2+d^\infty_{11}c^\infty_-+d^\infty_{02})\\
& = & d-1+\frac{d-1}{\ell}+2\alpha(d^\infty_{20}-d^\infty_{11}+d^\infty_{02})>0.
\eee
\noindent{\bf step 5} $\At_2$. We compute
\bee
\At^\infty_2& = & 2\alpha^\infty\left[(c^\infty_-)^3-dc^\infty_-\right]+(\alpha^\infty)^2(d^\infty_{20}(c^\infty_-)^2+d^\infty_{11}c^\infty_-+d^\infty_{02})\\
& = & -2c^\infty_-\alpha^\infty(d-(c^\infty_-)^2)+(\alpha^\infty)^2(d^\infty_{20}(c^\infty_-)^2+d^\infty_{11}c^\infty_-+d^\infty_{02})>0
\eee 
Since $-1\leq c^\infty_-<0$, $\alpha^\infty>0$ and $d^\infty_{20}(c^\infty_-)^2+d^\infty_{11}c^\infty_-+d^\infty_{02}>0$, we infer $\At_2>0$.\\
This concludes the proof of \eqref{einonenneonbeo}.
\end{proof}


\subsection{Formula for $F$}

 
 The function $F$ given by \eqref{difinitotniof} has a special structure.
 
\begin{lemma}[Formula for F]
\label{lemmaformulaf}
Assume \eqref{suaempeonpoengopngoe}. Then,
\be
\label{formularuvoe}
F{=}-\frac{(d-1)\sigma}{\ell\Delta}(w-w_-)(w-w_+)
\ee
where $(w_-,w_+)$ are the $w$-coordinates of $P_2,P_3$. 
\end{lemma}

\begin{proof}[Proof of Lemma \ref{lemmaformulaf}]
 We compute
\bea
\label{bebebebebif}
\nonumber F& = & \sigma+\sigma' =\frac{1}{\Delta}\left[\sigma\Delta -\Delta_2\right]\\
\nonumber&=& \frac{\sigma}{\ell\Delta}\left[\ell(w^2-2w+1-\sigma^2)-(\ell+d-1)w^2+(\ell+d+\ell r-r)w-\ell r+\ell \sigma^2\right]\\
\nonumber& = & -\frac{\sigma}{\ell\Delta}\left[(d-1)w^2-(d-r+\ell(r-1))w+\ell(r-1)\right]\\
& = & -(d-1)\frac{\sigma}{\ell\Delta}(w-w_-)(w-w_+)
\eea
where $(w_-,w_+)$ are the ordinates of $P_2,P_3$ from \eqref{defonfeoneo}.
\end{proof}


\subsection{Positivity to the right of $P_2$}


We claim the following {\em fundamental} lower bound at the right of $P_2$.

\begin{lemma}[Positivity to the right of $P_2$]
\label{keylemmapositivity}
Assume \eqref{suaempeonpoengopngoe}. Then, for $\Sigma\ge 0$:
\be
\label{funamdnanteo}
1-w-w'-F\ge c>0.
\ee
\end{lemma}

\begin{proof}[Proof of Lemma \ref{keylemmapositivity}] The key here is the sharp bound \eqref{boundpthilde}.\\

\noindent{\bf step 1} Value at $P_2$. We compute at $P_2$:
\bee
F &=& \sigma+\sigma'= \sigma -\frac{\Delta_2}{\Delta}\\
&=& \sigma_2+{\Sigma}-\frac{\Sigma\left[c_2\Phi+{c_4}+[e_{20}\Phi^2+e_{11}\Phi+e_{02}]\Sigma+[e_{21}\Phi^2-1]\Sigma^2\right]}{-\Sigma\left[2\sigma_2(1+\Phi)+\Sigma(1-\Phi^2)\right]}
\eee 
and therefore since $\Phi(P_2)=c_-$ and $-1<c_-<0$, $c_2<c_4<0$ from \eqref{tionveiogbngo3o}:
$$F(P_2)=\sigma_2+\frac{c_2c_{-}+c_4}{2\sigma_2(1+c_-)}=\frac{c_-(c_2-c_4)}{2\sigma_2(1+c_-)}>0$$ 
and 
$$1-w_2-F(P_2)=-\frac{\l_+}{2\sigma_2(1+c_-)}>0.$$ 
Then,
$$w'=-\frac{\Delta_1}{\Delta}=\frac{c_1c_-+{c_3}}{2\sigma_2(1+c_-)}+O(\Sigma)=\frac{c_-\l_+}{2\sigma_2(1+c_-)}+O(\Sigma)$$ 
and 
\be
\label{cenioneineonv}
(1-w-w'-F)(P_2)=-\frac{\l_+}{2\sigma_2(1+c_-)}-\frac{c_-\l_+}{2\sigma_2(1+c_-)}=-\frac{\l_+}{2\sigma_2}>0.
\ee

\noindent{\bf step 2} Bad set and no contact condition. We now study the {\it null} equation:
\bee
&&\Delta(1-w-w'-F)=0\Leftrightarrow \Delta\left[1-w+\frac{\Delta_1}{\Delta}-\sigma+\frac{\Delta_2}{\Delta}\right]=0\\
&\Leftrightarrow &\left[(1-w)^2-\sigma^2\right](1-w-\sigma)+\Delta_1+\Delta_2=0\\
&\Leftrightarrow &(1-w)^3-(1-w)^2\sigma-(1-w)\sigma^2+\sigma^3\\
& + & w(1-w)(r-w)-d(w-w_e)\sigma^2+\frac{\sigma}{\ell}\Big[(\ell+d-1)w^2-w(\ell+d+\ell r-r)+\ell r-\ell \sigma^2\Big]=0\\
& \Leftrightarrow & (1-w)\left[(1-w)^2+w(r-w)\right]+\left[\frac{(d+\ell-1)w^2}{\ell}-(1-w)^2+r-\frac w\ell(\ell+d+\ell r-r)\right]\sigma\\ &&-\left[d(w-w_e)+1-w\right]\sigma^2=0\\
& \Leftrightarrow &\left[d(w-w_e)+1-w\right]\sigma^2-\left[\frac{(d-1)w^2+(\ell+r-d-\ell r) w {+} \ell(r-1)}{\ell}\right]\sigma\\ &&-(1-w)\left[(1-w)^2+w(r-w)\right]=0.
\eee
To the right of $P_2$, between the red and green curves, we have $w_e<w<w_2<1$. As a result, the coefficients $\alpha(w), \gamma(w)$ of the above quadratic equation $$
\alpha(w) \sigma^2+\beta(w) \sigma -\gamma(w) =0
$$ are positive and therefore there is exactly one positive root $\sigma_0(w)$. The dependence of $\sigma_0(w)$ on $w$ is continuous
and $\sigma_0(w_2)=\sigma_2$. In the $(\Sigma, W)$ plane the null set is represented by the continuous null curve 
$(\Sigma_{{0}}(W), W)$ with $W\in [{w_e-w_2},0]$.\\
We now compute the derivative at $P_2$ recalling \eqref{defvalueci}:
\bee
&&(1-w-\sigma_0(w))\Delta+\Delta_1(w,\sigma_0(w))+\Delta_2(w,\sigma_{{0}}(w))=0\Rightarrow c_1+\sigma_0'c_3+c_2+\sigma_0'c_4=0\\
&\Rightarrow & \sigma_0'{(w_2)}=-\frac{c_1+c_2}{c_3+c_4}\Rightarrow  w_0'(\sigma_{{2}})=-\frac{c_3+c_4}{c_1+c_2}
\eee
and hence the local representation in the $(\Sigma,W)$ plane of the null set: $$W_0(\Sigma)=-\frac{c_3+c_4}{c_1+c_2}\Sigma+O(\Sigma^2).$$ On the other hand, from \eqref{equationcminus}, \eqref{calculparametres}:
\bee
&&-\frac{c_3+c_4}{c_1+c_2}<c_-\Leftrightarrow c_3+c_4<-(c_1+c_2)c_-\Leftrightarrow c_1c_-+c_3+c_2c_-+c_4<0\\
& \Leftrightarrow & (1+c_-)\l_+<0.
\eee 
Therefore, the curve $\Phi=\frac{c_-}{1+\alpha(r)\Sigma}$, which is equivalent to $W=\frac{c_-\Sigma}{1+\alpha(r)\Sigma}$, 
lies for $\Sigma>0$ small, strictly above the null curve $(\Sigma_0(W), W)$.\\ 
We now claim that, if we can show that 
\be
\label{nontouchingcondition}
\forall \Sigma>0, \ \ \left(1-w-w'-F\right)\left(\Sigma, \frac{c_-}{1+\alpha(r) \Sigma}\right)>0,
\ee
it will imply that $\left(1-w-w'-F\right)>0$ $\forall \Sigma\ge 0$ along the solution curve. We argue by contradiction. First,
using \eqref{cenioneineonv}, we could find a positive value $\Sigma_0$ such that 
$\left(1-w-w'-F\right)(\Sigma_0, W(\Sigma_0))=0$. Therefore, the point $(\Sigma_0, W(\Sigma_0))$ belongs to the null set.
Since along the solution curve $w_e<w(\sigma)<w_2$ for any $\sigma>\sigma_2$, the point $(\Sigma_0, W(\Sigma_0))$ must 
lie on the null curve $(\Sigma_0(W),W)$. That is $\Sigma_0(W(\Sigma_0))=\Sigma_0$. We now follow this curve from the 
value $W(\Sigma_0)$ to $W=0$. For $W$ sufficiently small the null curve lies below the curve $W=\frac{c_-\Sigma}{1+\alpha(r)\Sigma}$, while at the point $(\Sigma_0, W(\Sigma_0))$, which belongs to the solution curve, it must be above 
$W=\frac{c_-\Sigma}{1+\alpha(r)\Sigma}$. Therefore, the curve $W=\frac{c_-\Sigma}{1+\alpha(r)\Sigma}$ and the null
curve must intersect, which is impossible in view of the claim \eqref{nontouchingcondition}.\\
{From now on, we focus on the proof of \eqref{nontouchingcondition}.}\\

\noindent{\bf step 3} Computation of $1-w-w'$. We  compute in $(W,\Sigma)$ coordinates:
\bee
&&\Delta(1-w-w')=\Delta(1-w)+\Delta_1=-(\sigma_2-W)\Sigma\left[2\sigma_2(1+\Phi)+\Sigma(1-\Phi^2)\right]\\
&+&\Sigma\left[c_1\Phi+c_3+[d_{20}\Phi^2+d_{11}\Phi+d_{02}]\Sigma+[\Phi^3-d\Phi]\Sigma^2\right]\\
& =& \Sigma\Big\{-\sigma_2\left[2\sigma_2(1+\Phi)+\Sigma(1-\Phi^2)\right]+\Phi\Sigma\left[2\sigma_2(1+\Phi)+\Sigma(1-\Phi^2)\right]\\
& + & c_1\Phi+c_3+[d_{20}\Phi^2+d_{11}\Phi+d_{02}]\Sigma+[\Phi^3-d\Phi]\Sigma^2\Big\}\\
& = & \Sigma\left\{B_0+B_1\Sigma+B_2\Sigma^2\right\}
\eee
with
$$\left|\begin{array}{l}
B_0=c_1\Phi+c_3-2\sigma_2^2(1+\Phi)\\
B_1=d_{20}\Phi^2+d_{11}\Phi+d_{02}-\sigma_2(1-\Phi^2)+2\sigma_2\Phi(1+\Phi)\\
B_2=\Phi^3-d\Phi+\Phi(1-\Phi^2)=-(d-1)\Phi
\end{array}\right.
$$
Let $$K(\Sigma,\Phi)=B_0(\Phi)+B_1(\Phi)\Sigma+B_2(\Phi)\Sigma^2,$$ then we have obtained:
$$1-w-w'=\frac{\Delta(1-w)+\Delta_1}{\Delta}=\frac{\Sigma K(\Sigma,\Phi)}{\Delta}.
$$
To ease the notations, we let $\alpha=\alpha(r)$ and compute
\bee
&&K\left(\Sigma,\frac{c_-}{1+\alpha \Sigma}\right)=  \frac{c_1c_-}{1+\alpha \Sigma}+c_3-2\sigma_2^2\left(1+\frac{c_-}{1+\alpha\Sigma}\right)\\
& + & \left[(d_{20}+3\sigma_2)\left(\frac{c_-}{1+\alpha\Sigma}\right)^2+(d_{11}+2\sigma_2)\left(\frac{c_-}{1+\alpha\Sigma}\right)+d_{02}-\sigma_2\right]\Sigma\\
& - & (d-1)\frac{c_-}{1+\alpha\Sigma}\Sigma^2\\
& =& \frac{1}{(1+\alpha \Sigma)^2}\Big \{(c_1c_--2\sigma_2^2c_-)(1+\alpha\Sigma)+(c_3-2\sigma_2^2)(1+\alpha\Sigma)^2\\
& + & (d_{20}+3\sigma_2)c_-^2\Sigma+c_-(d_{11}+2\sigma_2)\Sigma(1+\alpha \Sigma)+(d_{02}-\sigma_2)\Sigma(1+\alpha\Sigma)^2\\
& - & (d-1)c_-\Sigma^2(1+\alpha\Sigma)\Big\}=  \frac{\sum_{i=0}^3\Bt_i\Sigma^i}{(1+\alpha \Sigma)^2}
\eee
which yields the formula:
\be
\label{tobeusedlater}
(1-w-w')\left(\Sigma,\frac{c_-}{1+\alpha\Sigma}\right)=\frac{\Sigma(\sum_{i=0}^3\Bt_i\Sigma^i)}{\Delta(1+\alpha \Sigma)^2}.
\ee
{We now compute the coefficients $\tilde{B}_i$. For $\tilde{B}_0$, we need to keep the exact structure at $r$,} {but for $\tilde{B}_{1,2,3}$ is suffices to compute the approximate values at $r=\re(d,\ell)$ which is done below.}\\ 

\noindent{\bf step 4} $\Bt_0$. We compute for all $r$:
\bee
\Bt_0& =& c_1c_--2\sigma_2^2c_-+c_3-2\sigma_2^2=c_-\l_+-2\sigma_2^2(1+c_-).
\eee

\noindent{\bf step 5} $\Bt^\infty_1$. We compute
\bee
\Bt^\infty_1& =&(c^\infty_1c^\infty_--2(\sigma^\infty_2)^2c^\infty_-)\alpha^\infty+2\alpha^\infty(c^\infty_3-2(\sigma^\infty_2)^2)+(c^\infty_-)^2(d^\infty_{20}+3\sigma^\infty_2)\\
& + & c^\infty_-(d^\infty_{11}+2\sigma^\infty_2)+d^\infty_{02}-\sigma^\infty_2.
\eee
\noindent\underline{case $r=r^*$}. We compute using $c^\infty_1c^\infty_-+c^\infty_3=0$:
\bee
&&c^\infty_1c^\infty_--2(\sigma^\infty_2)^2c^\infty_-+2(c^\infty_3-2(\sigma^\infty_2)^2)=c^\infty_3-2(\sigma^\infty_2)^2(2+c^\infty_-)\\
&=& c^\infty_3-2(\sigma^\infty_2)^2-2(\sigma^\infty_2)^2(1+c^\infty_-)\\
& = & -\frac{2d(\ell+1)}{(\ell+\sqrt{d})^2}-2\sigma_2^2(1+c_-)=-\frac{2d(\ell+1)}{(\ell+\sqrt{d})^2}-\frac{2d}{(\ell+\sqrt{d})^2}\frac{(\sqrt{d}-1)(d-\ell)}{d(\sqrt{d}-1)+\ell(\sqrt{d}+1)}\\
& = & -\frac{2d}{(\ell+\sqrt{d})^2\left[d(\sqrt{d}-1)+\ell(\sqrt{d}+1)\right]}\left\{(\ell+1)[d(\sqrt{d}-1)+\ell(\sqrt{d}+1)]+(d-\ell)(\sqrt{d}-1)\right\}\\
& = & -\frac{2d}{(\ell+\sqrt{d})^2\left[d(\sqrt{d}-1)+\ell(\sqrt{d}+1)\right]}\left\{2d(\sqrt{d}-1)+(d\sqrt{d}-d+2)\ell+(\sqrt{d}+1)\ell^2\right\}
\eee
and hence
\bee
&&(c^\infty_1c^\infty_--2(\sigma^\infty_2)^2c^\infty_-)\alpha^\infty+2\alpha^\infty(c^\infty_3-2(\sigma^\infty_2)^2)\\
&=&  -\frac{2d}{(\ell+\sqrt{d})^2\left[d(\sqrt{d}-1)+\ell(\sqrt{d}+1)\right]}\left\{2d(\sqrt{d}-1)+(d\sqrt{d}-d+2)\ell+(\sqrt{d}+1)\ell^2\right\}\\
&\times & \frac{2d(\ell+\sqrt{d})}{d(\sqrt{d}-1)+\ell(\sqrt{d}+1)}=  -\frac{4d^2\left\{2d(\sqrt{d}-1)+(d\sqrt{d}-d+2)\ell+(\sqrt{d}+1)\ell^2\right\}}{(\ell+\sqrt{d})\left[d(\sqrt{d}-1)+\ell(\sqrt{d}+1)\right]^2}.
\eee
Then,
\bee
&&3(c^\infty_-)^2\sigma^\infty_2+2\sigma^\infty_2c^\infty_--\sigma^\infty_2=\sigma^\infty_2(3(c^\infty_-)^2+2c^\infty_--1)=\sigma^\infty_2(1+c^\infty_-)(3c^\infty_--1)\\
& = &- \frac{\sqrt{d}}{\ell+\sqrt{d}}\frac{(\sqrt{d}-1)(d-\ell)}{d(\sqrt{d}-1)+\ell(\sqrt{d}+1)}  \left(1+\frac{6\ell\sqrt{d}}{d(\sqrt{d}-1)+\ell(\sqrt{d}+1)}\right)\\
& = & -\frac{\sqrt{d}(\sqrt{d}-1)(d-\ell)\left[d(\sqrt{d}-1)+\ell(7\sqrt{d}+1)\right]}{(\ell+\sqrt{d})\left[d(\sqrt{d}-1)+\ell(\sqrt{d}+1)\right]^2}.
\eee
Now, recalling \eqref{nbeivbebebebnebnebnvo}:
$$\Bt^\infty_1=-\frac{P_{\Bt_1}(\ell)}{(\ell+\sqrt{d})\left[d(\sqrt{d}-1)+\ell(\sqrt{d}+1)\right]^2}$$ with
\bee
&&P_{\Bt_1}(\ell)=4d^2\left\{2d(\sqrt{d}-1)+(d\sqrt{d}-d+2)\ell+(\sqrt{d}+1)\ell^2\right\}\\
&+& \sqrt{d}(\sqrt{d}-1)(d-\ell)\left[d(\sqrt{d}-1)+\ell(7\sqrt{d}+1)\right]\\
& - & \ell\sqrt{d}[d^2(\sqrt{d}-1)(3\sqrt{d}+1)+2d(d-1)\ell-\ell^2(d+1-2\sqrt{d})]\\
& = & \sqrt{d}(\sqrt{d}-1)(d-\ell)\left[d(\sqrt{d}-1)+\ell(7\sqrt{d}+1)\right]\\
& + & 8d^3(\sqrt{d}-1)+\ell\left[d^3\sqrt{d}{-2}d^3+d^2\sqrt{d}+8d^2\right]+\ell^2\left[2d^2\sqrt{d}+4d^2+2d\sqrt{d}\right]+\ell^3\sqrt{d}(d+1-2\sqrt{d})\\
&>& 0
\eee
and  thus $\Bt^\infty_1<0$. We can fully expand:
$$(d-\ell)\left[d(\sqrt{d}-1)+\ell(7\sqrt{d}+1)\right]=d^2(\sqrt{d}-1)+(6d\sqrt{d}+2d)\ell-(7\sqrt{d}+1)\ell^2
$$
and 
\bee
&&P_{\Bt_1}(\ell)=8d^3(\sqrt{d}-1)+\sqrt{d}(\sqrt{d}-1)d^2(\sqrt{d}-1)\\
& + & \left[d^3\sqrt{d}{-2}d^3+d^2\sqrt{d}+8d^2+\sqrt{d}(\sqrt{d}-1)(6d\sqrt{d}+2d)\right]\ell\\
& + & \left[2d^2\sqrt{d}+4d^2+2d\sqrt{d}- \sqrt{d}(\sqrt{d}-1)(7\sqrt{d}+1)\right]\ell^2\\
& + & \sqrt{d}(d+1-2\sqrt{d})\ell^3\\
& = & 9d^3\sqrt{d}-10d^3+d^2\sqrt{d}+  \left[d^3\sqrt{d}{-2}d^3+{7}d^2\sqrt{d}+4d^2-2d\sqrt{d}\right]\ell\\
& + & \left[2d^2\sqrt{d}+4d^2-5d\sqrt{d}+6d+\sqrt{d}\right]\ell^2+\left[d\sqrt{d}-2d+\sqrt{d}\right]\ell^3.
\eee

\noindent\underline{case $r=r_+$}.  We compute using $c^\infty_1c^\infty_-+c^\infty_3=0$:
\bee
&&c^\infty_1c^\infty_--2(\sigma^\infty_2)^2c^\infty_-+2(c^\infty_3-2(\sigma^\infty_2)^2)=c^\infty_3-2(\sigma^\infty_2)^2(2+c^\infty_-)\\
&=&c^\infty_3-2(\sigma^\infty_2)^2-2(\sigma^\infty_2)^2(1+c^\infty_-)\\
& = & c^\infty_3-2(\sigma^\infty_2)^2=-\frac{2\sqrt{\ell}(d+\sqrt{\ell})}{(1+\sqrt{\ell})^3}-\frac{2}{(1+\sqrt{\ell})^2}=-\frac{2[\sqrt{\ell}(d+\sqrt{\ell})+1+\sqrt{\ell}]}{(1+\sqrt{\ell})^3}\\
& = & \frac{-2\ell-2(d+1)\sqrt{\ell}-2}{(1+\sqrt{\ell})^3}
\eee
and hence
\bee
&&(c^\infty_1c^\infty_--2(\sigma^\infty_2)^2c^\infty_-)\alpha^\infty+2\alpha^\infty(c^\infty_3-2(\sigma^\infty_2)^2)=\frac{-2\ell-2(d+1)\sqrt{\ell}-2}{(1+\sqrt{\ell})^3}\left(\frac{d(1+\sqrt{\ell})^2}{d\sqrt{\ell}+\ell}\right)\\
& = & \frac{-2d\ell-2d(d+1)\sqrt{\ell}-2d}{(1+\sqrt{\ell})(d\sqrt{\ell}+\ell)}.
\eee
Then 
$$
3(c^\infty_-)^2\sigma^\infty_2+2\sigma^\infty_2c^\infty_--\sigma^\infty_2=\sigma^\infty_2(3(c^\infty_-)^2+2c^\infty_--1)=0.$$ Finally recalling \eqref{neioneioneonoenv}:
$$
(c^\infty_-)^2d^\infty_{20}+c^\infty_-d^\infty_{11}+d^\infty_{02}=d^\infty_{20}-d^\infty_{11}+d^\infty_{02}= \frac{d-1}{1+\sqrt{\ell}}
$$ and hence
\bee
\tilde{B^\infty_1}&=&\frac{-2d\ell-2d(d+1)\sqrt{\ell}-2d}{(1+\sqrt{\ell})(d\sqrt{\ell}+\ell)}+\frac{d-1}{1+\sqrt{\ell}}=\frac{-2d\ell-2d(d+1)\sqrt{\ell}-2d+(d-1)(d\sqrt{\ell}+\ell)}{(1+\sqrt{\ell})(d\sqrt{\ell}+\ell)}\\
& = & -\frac{(d+1)\ell+(d^2+3d)\sqrt{\ell}+2d}{(1+\sqrt{\ell})(d\sqrt{\ell}+\ell)}<0.
\eee

\noindent{\bf step 6} $\Bt_2$.\\
\noindent\underline{case $r=r^*$}.  We compute at $r=r^*(d,\ell)$:
\bee
&&\Bt^\infty_2 = (c^\infty_3-2(\sigma^\infty_2)^2)(\alpha^\infty)^2+c^\infty_-\alpha^\infty(d^\infty_{11}+2\sigma^\infty_2)+2\alpha^\infty(d^\infty_{02}-\sigma^\infty_2)-(d-1)c^\infty_-\\
& = & \left(\frac{-2\ell d}{(\ell+\sqrt{d})^2}-\frac{2d}{(\ell+\sqrt{d})^2}\right)\alpha^2+c_-\alpha\left[-\frac{2d\sqrt{d}}{\ell+\sqrt{d}}+\frac{2\sqrt{d}}{\ell+\sqrt{d}}\right]+  2\alpha\left[-\frac{(\ell+1)\sqrt{d}}{\ell+\sqrt{d}}\right]-(d-1)c_-\\
& =& -\frac{(2d+2\ell d)4d^2}{\left[d(\sqrt{d}-1)+\ell(\sqrt{d}+1)\right]^2}-\frac{2\alpha\sqrt{d}}{\ell+\sqrt{d}}\left[\ell+1+c_-(d-1)\right]-(d-1)c_-\\
& = &  -\frac{(2d+2\ell d)4d^2}{\left[d(\sqrt{d}-1)+\ell(\sqrt{d}+1)\right]^2}-\frac{4d\sqrt{d}}{d(\sqrt{d}-1)+\ell(\sqrt{d}+1)}(\ell+1+c_-(d-1))-(d-1)c_-\\
& = & -\frac{(2d+2\ell d)4d^2}{\left[d(\sqrt{d}-1)+\ell(\sqrt{d}+1)\right]^2}-\frac{4d\sqrt{d}}{d(\sqrt{d}-1)+\ell(\sqrt{d}+1)}\left(\ell+1-\frac{2\ell(d-1)\sqrt{d}}{d(\sqrt{d}-1)+\ell(\sqrt{d}+1)}\right)\\
&+& \frac{2\ell\sqrt{d}(d-1)}{d(\sqrt{d}-1)+\ell(\sqrt{d}+1)}=-  \frac{P_{\Bt_2}(\ell)}{\left[d(\sqrt{d}-1)+\ell(\sqrt{d}+1)\right]^2}
\eee
with 
\bee
&&P_{\Bt_2}(\ell)=8d^3(d+\ell)+4d\sqrt{d}\left[(\ell+1)\left[d(\sqrt{d}-1)+\ell(\sqrt{d}+1)\right]-2\ell(d-1)\sqrt{d}\right]\\
& - & 2\ell\sqrt{d}(d-1)\left[d(\sqrt{d}-1)+\ell(\sqrt{d}+1)\right]\\
& = & 8d^4+4d^2(d-\sqrt{d})+\left[2d^3-2d^2\sqrt{d}+14d^2+{2}d\sqrt{d}\right]\ell+\left[2d^2+2d\sqrt{d}+2d+2\sqrt{d}\right]\ell^2\\
&>&0
\eee
which implies $P_{\Bt_2}(\ell)>0$ and $\Bt^\infty_2<0$.\\

\noindent\underline{case $r=r_+$}. We compute at $r=r_+(\ell)$:
\bee
&&\Bt^\infty_2 = (c^\infty_3-2(\sigma^\infty_2)^2)(\alpha^\infty)^2+c^\infty_-\alpha^\infty(d^\infty_{11}+2\sigma^\infty_2)+2\alpha^\infty(d^\infty_{02}-\sigma^\infty_2)-(d-1)c^\infty_-\\
& = & (c^\infty_3-2(\sigma^\infty_2)^2)(\alpha^\infty)^2+(2d^\infty_{02}-d^\infty_{11}-4\sigma^\infty_2)\alpha^\infty+d-1\\
& = & \left[-\frac{2\sqrt{\ell}(d+\sqrt{\ell})}{(1+\sqrt{\ell})^3}-\frac{2}{(1+\sqrt{\ell})^2}\right]\left(\frac{d(1+\sqrt{\ell})^2}{d\sqrt{\ell}+\ell}\right)^2\\
&+&\left[2\left(\frac{-\ell-d\sqrt{\ell}}{(1+\sqrt{\ell})^2}\right)+\frac{2d-4}{1+\sqrt{\ell}}\right]\frac{d(1+\sqrt{\ell})^2}{d\sqrt{\ell}+\ell}+d-1\\
& = & -\frac{d^2(1+\sqrt{\ell})}{(d\sqrt{\ell}+\ell)^2}\left[2\sqrt{\ell}(d+\sqrt{\ell})+2(1+\sqrt{\ell})\right]+\frac{d}{d\sqrt{\ell}+\ell}\left[-2\ell-2d\sqrt{\ell}+(2d-4)(1+\sqrt{\ell})\right]+d-1\\
& = & -\frac{d^2(1+\sqrt{\ell})[2\ell+(2d+2)\sqrt{\ell}+2]}{(d\sqrt{\ell}+\ell)^2}+\frac{d(-2\ell-4\sqrt{\ell}+2d-4)}{d\sqrt{\ell}+\ell}+d-1\\
& = & -\frac{P_{\Bt_2}(\ell)}{(d\sqrt{\ell}+\ell)^2}
\eee
with
\bee
&&P_{\Bt_2}(\ell)=d^2(1+\sqrt{\ell})[2\ell+(2d+2)\sqrt{\ell}+2]+d(2\ell+4\sqrt{\ell}-2d+4)(d\sqrt{\ell}+\ell)-(d-1)(d\sqrt{\ell}+\ell)^2\\
& = & (d+1)\ell^2+(2d^2+6d)\ell\sqrt{\ell}+(d^3+7d^2+4d)\ell+(8d^2)\sqrt{\ell}+2d^2
\eee
and hence $P_{\Bt_2}(\ell)>0$ and $\Bt^\infty_2<0$.\\

\noindent{\bf step 7} $\Bt_3$.\\
\noindent\underline{case $r=r^*$}. We compute at $r=r^*(d,\ell)$:
\bee
\Bt_3& = & (\alpha^\infty)^2(d^\infty_{02}-\sigma^\infty_2)-(d-1)c^\infty_-\alpha^\infty=\alpha^\infty\left[\frac{2d(\ell+\sqrt{d})}{d(\sqrt{d}-1)+\ell(\sqrt{d}+1)}\left(-\frac{\ell\sqrt{d}}{\ell+\sqrt{d}}-\frac{\sqrt{d}}{\ell+\sqrt{d}}\right)\right.\\
& + & \left.(d-1)\frac{2\ell\sqrt{d}}{d(\sqrt{d}-1)+\ell(\sqrt{d}+1)}\right]=  -\frac{2\alpha\sqrt{d}}{d(\sqrt{d}-1)+\ell(\sqrt{d}+1)}\left[d(\ell+1)-(d-1)\ell\right]\\
& = & -\frac{2\alpha(d+\ell)\sqrt{d}}{d(\sqrt{d}-1)+\ell(\sqrt{d}+1)}<0.
\eee
\noindent\underline{case $r=r_+$}. We compute at $r=r_+(\ell)$:
\bee
\Bt_3& = & (\alpha^\infty)^2(d^\infty_{02}-\sigma^\infty_2)-(d-1)c^\infty_-\alpha^\infty\\
&=& \alpha\left[\left(\frac{d(1+\sqrt{\ell})^2}{d\sqrt{\ell}+\ell}\right)\left(\frac{-\ell-d\sqrt{\ell}}{(1+\sqrt{\ell})^2}-\frac1{1+\sqrt{\ell}}\right)+d-1\right]\\
& =&\alpha\left\{\frac{d(1+\sqrt{\ell})^2}{d\sqrt{\ell}+\ell}\left(\frac{-\ell-d\sqrt{\ell}-1-\sqrt{\ell}}{(1+\sqrt{\ell})^2}\right)+d-1\right\}\\
& = & \frac{\alpha}{d\sqrt{\ell}+\ell}\left[d\left(-\ell-(d+1)\sqrt{\ell}-1\right)+(d-1)(d\sqrt{\ell}+\ell)\right]\\
& = & \frac{\alpha}{d\sqrt{\ell}+\ell}[-\ell-2d\sqrt{\ell}-d]=-\frac{d(1+\sqrt{\ell})^2(\ell+2d\sqrt{\ell}+d)}{(d\sqrt{\ell}+\ell)^2}<0
\eee

\noindent{\bf step 7} Computation of $F$. 
\bee
&&F=\sigma+\sigma'=\sigma-\frac{\Delta_2}{\Delta}=\frac{\sigma\Delta-\Delta_2}{\Delta}\\
& = & -\frac{\Sigma}{\Delta}\left\{(\sigma_2+\Sigma)[2\sigma_2(1+\Phi)+\Sigma(1-\Phi^2)]+c_2\Phi+c_4+[e_{20}\Phi^2+e_{11}\Phi+e_{02}]\Sigma+[e_{21}\Phi^2-1]\Sigma^2\right\}\\
& = & -\frac{\Sigma}{\Delta}\Big\{2\sigma_2^2(1+\Phi)+c_2\Phi+c_4\\
& + & \left[2\sigma_2(1+\Phi)+\sigma_2(1-\Phi^2)+e_{20}\Phi^2+e_{11}\Phi+e_{02}\right]\Sigma+\left[1-\Phi^2+e_{21}\Phi^2-1\right]\Sigma^2\Big\}\\
& = & \frac{\Sigma\Big\{(2\sigma_2^2+c_2)\Phi+\left[(2\sigma_2+e_{11})\Phi+(e_{20}-\sigma_2)\Phi^2\right]\Sigma+(e_{21}-1)\Phi^2\Sigma^2\Big\}}{(-\Delta)}
\eee
where we used that {for all $r$, $$c_4=-2\sigma_2^2, \ \ e_{02}=-3\sigma_2.$$} 
We therefore obtain
\bea
\label{venvoenennenveneoF}
\nonumber &&F\left(\Sigma,\frac{c_-}{1+\alpha\Sigma}\right)\\
\nonumber &=& \frac{\Sigma}{-\Delta}\Bigg\{(2\sigma_2^2+c_2)\frac{c_-}{1+\alpha\Sigma}+\left[(2\sigma_2+e_{11})\frac{c_-}{1+\alpha\Sigma}+(e_{20}-\sigma_2)\left(\frac{c_-}{1+\alpha\Sigma}\right)^2\right]\Sigma\\
\nonumber &+&(e_{21}-1)\left(\frac{c_-}{1+\alpha\Sigma}\right)^2\Sigma^2\Bigg\}\\
\nonumber & = &\frac{\Sigma}{-\Delta(1+\alpha\Sigma)^2}\Big\{(2\sigma_2^2+c_2)c_-(1+\alpha\Sigma)+(2\sigma_2+e_{11})c_-\Sigma(1+\alpha\Sigma)\nonumber\\&+&{(e_{20}-\sigma_2)(c_-)^2\Sigma+(e_{21}-1)(c_-)^2\Sigma^2}\Big\}
 =  \frac{\Sigma\sum_{i=0}^2A_i\Sigma^i}{(-\Delta)(1+\alpha \Sigma)^2}.
\eea
{As for the $\tilde{B}_i$, we compute $A_0$ for all r but $A_{1,2}$ at $r=r^*(d,\ell)$ only.}\\
 
\noindent{\bf step 8} $A_0$. We compute:
\bee
&&A_0=(2\sigma_2^2+c_2)c_-.
\eee
\noindent{\bf step 9} $A^\infty_1$. We compute 
\bee
&&A^\infty_1=\alpha^\infty(2(\sigma^\infty_2)^2+c^\infty_2)c^\infty_-+(2\sigma^\infty_2+e^\infty_{11})c^\infty_-+(e^\infty_{20}-\sigma^\infty_2)(c^\infty_-)^2
\eee
\noindent\underline{case $r=r^*(d,\ell)$}. First,
\bee
&&2\sigma^\infty_2+e^\infty_{11}=\frac{2\sqrt{d}}{\ell+\sqrt{d}}-\frac{d(\sqrt{d}-1)+\ell(1+\sqrt{d})}{\ell(\ell+\sqrt{d})}=  \frac{2\ell\sqrt{d}-d(\sqrt{d}-1)-\ell(1+\sqrt{d})}{\ell(\ell+\sqrt{d})}\\
& = & \frac{-d(\sqrt{d}-1)+\ell(\sqrt{d}-1)}{\ell(\ell+\sqrt{d})}=-\frac{(\sqrt{d}-1)(d-\ell)}{\ell(\ell+\sqrt{d})}
\eee
and 
\bee
&&\alpha^\infty(2(\sigma^\infty_2)^2+c^\infty_2)\\
&=&\frac{2d(\ell+\sqrt{d})}{d(\sqrt{d}-1)+\ell(\sqrt{d}+1)}\left\{\frac{2d}{(\ell+\sqrt{d})^2}-\frac{\sqrt{d}}{\ell(\ell+\sqrt{d})^2}\left[d(\sqrt{d}-1)+\ell(\sqrt{d}+1)\right]\right\}\\
& = & \frac{2d\sqrt{d}}{\ell(\ell+\sqrt{d})\left[d(\sqrt{d}-1)+\ell(\sqrt{d}+1)\right]}\left[2\ell \sqrt{d}-\left(d(\sqrt{d}-1)+\ell(\sqrt{d}+1)\right)\right]\\
& = & -\frac{2d\sqrt{d}(\sqrt{d}-1)(d-\ell)}{\ell(\ell+\sqrt{d})\left[d(\sqrt{d}-1)+\ell(\sqrt{d}+1)\right]}
\eee
As a result,
\bee
&&\alpha^\infty(2(\sigma^\infty_2)^2+c^\infty_2)c^\infty_-+(2\sigma^\infty_2+e^\infty_{11})c^\infty_-\\
& = & \left[ -\frac{2d\sqrt{d}(\sqrt{d}-1)(d-\ell)}{\ell(\ell+\sqrt{d})\left[d(\sqrt{d}-1)+\ell(\sqrt{d}+1)\right]}-\frac{(\sqrt{d}-1)(d-\ell)}{\ell(\ell+\sqrt{d})}\right]c_-\\
& = & \left[\frac{2\ell\sqrt{d}}{d(\sqrt{d}-1)+\ell(\sqrt{d}+1)}\right]\frac{(\sqrt{d}-1)(d-\ell)[2d\sqrt{d}+d(\sqrt{d}-1)+\ell(\sqrt{d}+1)]}{\ell(\ell+\sqrt{d})\left[d(\sqrt{d}-1)+\ell(\sqrt{d}+1)\right]}\\
& = & \frac{2\sqrt{d}(\sqrt{d}-1)(d-\ell)[3d\sqrt{d}-d+\ell(\sqrt{d}+1)]}{(\ell+\sqrt{d})\left[d(\sqrt{d}-1)+\ell(\sqrt{d}+1)\right]^2}.
\eee
We expand
\bee
&&(d-\ell)[3d\sqrt{d}-d+\ell(\sqrt{d}+1)]=d^2(3\sqrt{d}-1)+\ell(d\sqrt{d}+d-3d\sqrt{d} {+d})-(\sqrt{d}+1)\ell^2\\
& = & d^2(3\sqrt{d}-1)+(-2d\sqrt{d}+{2}d)\ell-(\sqrt{d}+1)\ell^2
\eee
and obtain
\bee
&&\alpha^\infty(2(\sigma^\infty_2)^2+c^\infty_2)c^\infty_-+(2\sigma^\infty_2+e^\infty_{11})c^\infty_-\\
&=& \frac{2\sqrt{d}(\sqrt{d}-1)\left[d^2(3\sqrt{d}-1)+(-2d\sqrt{d}+{2}d)\ell-(\sqrt{d}+1)\ell^2\right]}{(\ell+\sqrt{d})\left[d(\sqrt{d}-1)+\ell(\sqrt{d}+1)\right]^2}\\
& = & \frac{6d^3\sqrt{d}-8d^3+2d^2\sqrt{d}+\left[-4d^2\sqrt{d}+{8}d^2-{4}d\sqrt{d}\right]\ell-2\sqrt{d}(d-1)\ell^2}{(\ell+\sqrt{d})\left[d(\sqrt{d}-1)+\ell(\sqrt{d}+1)\right]^2}.
\eee 
We  now add
\bee
&&(e^\infty_{20}-\sigma^\infty_2)(c^\infty_-)^2=\left(\frac{\sqrt{d}(d-1+\ell)}{\ell(\ell+\sqrt{d})}-\frac{\sqrt{d}}{\ell+\sqrt{d}}\right)c_-^2=\frac{\sqrt{d}(d-1)}{\ell(\ell+\sqrt{d})}\left(\frac{2\ell\sqrt{d}}{d(\sqrt{d}-1)+\ell(\sqrt{d}+1)}\right)^2\\
& = & \frac{4\ell d\sqrt{d}(d-1)}{(\ell+\sqrt{d})\left[d(\sqrt{d}-1)+\ell(\sqrt{d}+1)\right]^2}
\eee
and get the formula:
$$A^\infty_1=\frac{P_{A_1}}{(\ell+\sqrt{d})\left[d(\sqrt{d}-1)+\ell(\sqrt{d}+1)\right]^2}$$ with 
\bee
P_{A_1}&=&6d^3\sqrt{d}-8d^3+2d^2\sqrt{d}+\left[-4d^2\sqrt{d}+{8}d^2-{4}d\sqrt{d}+4d^2\sqrt{d}-4d{\sqrt{d}}\right]\ell-2\sqrt{d}(d-1)\ell^2\\
& = & 6d^3\sqrt{d}-8d^3+2d^2\sqrt{d}+\left[{8}d^2-{8}d\sqrt{d}\right]\ell-2\sqrt{d}(d-1)\ell^2
\eee

\noindent\underline{case $r=r_+(d,\ell)$}. Recall 
$$A^\infty_1=\alpha^\infty(2(\sigma^\infty_2)^2+c^\infty_2)c^\infty_-+(2\sigma^\infty_2+e^\infty_{11})c^\infty_-+(e^\infty_{20}-\sigma^\infty_2)(c^\infty_-)^2.
$$
Observe
$$2\sigma^\infty_2+e^\infty_{11}=\frac{2}{1+\sqrt{\ell}}-\frac{2}{1+\sqrt{\ell}}=0
$$
and 
$$2(\sigma^\infty_2)^2+c^\infty_2=\frac{2}{(1+\sqrt{\ell})^2}-\frac{2}{(1+\sqrt{\ell})^2}=0$$
and hence
$$A^\infty_1=e^\infty_{20}-\sigma^\infty_2=\frac{\ell+d-1}{\ell(1+\sqrt{\ell})}-\frac{1}{1+\sqrt{\ell}}=\frac{d-1}{\ell(1+\sqrt{\ell})}.$$

\noindent{\bf step 10} $A^\infty_2$. We have
$$A^\infty_2=\alpha^\infty(2\sigma^\infty_2+e^\infty_{11})c^\infty_-+(e^\infty_{21}-1)(c^\infty_-)^2.$$
\noindent\underline{case $r=r^*(d,\ell)$}. First,
$$(e^\infty_{21}-1)(c^\infty_-)^2=\left(\frac{\ell+d-1}{\ell}-1\right)\left(\frac{2\ell\sqrt{d}}{d(\sqrt{d}-1)+\ell(\sqrt{d}+1)}\right)^2
=\frac{4\ell d(d-1)}{[d(\sqrt{d}-1)+\ell(\sqrt{d}+1)]^2}$$
and
\bee
&&\alpha^\infty(2\sigma^\infty_2+e^\infty_{11})c^\infty_-\\
&=&\frac{2d(\ell+\sqrt{d})}{d(\sqrt{d}-1)+\ell(\sqrt{d}+1)}\left[-\frac{(\sqrt{d}-1)(d-\ell)}{\ell(\ell+\sqrt{d})}\right]\left[-\frac{2\ell\sqrt{d}}{d(\sqrt{d}-1)+\ell(\sqrt{d}+1)}\right]\\
& = & \frac{4d\sqrt{d}(\sqrt{d}-1)(d-\ell)}{[d(\sqrt{d}-1)+\ell(\sqrt{d}+1)]^2}
\eee
which implies 
$$A^\infty_2=\frac{P_{A_2}}{\left[d(\sqrt{d}-1)+\ell(\sqrt{d}+1)\right]^2}$$
with 
$$
P_{A_2}=4\ell d(d-1)+4d\sqrt{d}(\sqrt{d}-1)(d-\ell)=4d^2(d-\sqrt{d})+\ell(4d\sqrt{d}-4d).
$$
\noindent\underline{case $r=r_+(d,\ell)$}. Recall 
$$A_2=\alpha(2\sigma_2+e_{11})c_-+(e_{21}-1)c_-^2$$
and hence
$$A^\infty_2=\alpha^\infty\left[\frac{2}{1+\sqrt{\ell}}-\frac{2}{1+\sqrt{\ell}}\right]+\frac{\ell+d-1}{\ell}-1=\frac{d-1}{\ell}.$$

\noindent{\bf step 11} Conclusion. We are now in position to prove \eqref{nontouchingcondition}. From \eqref{venvoenennenveneoF}, \eqref{tobeusedlater}:
\bee
(1-w-w'-F)\left(\Sigma,\frac{c_-}{1+\alpha\Sigma}\right)&=&\frac{\Sigma(\sum_{i=0}^3\Bt_i\Sigma^i)}{\Delta(1+\alpha \Sigma)^2}-\frac{\Sigma\sum_{i=0}^2A_i\Sigma^i}{(-\Delta)(1+\alpha \Sigma)^2}\\
&=&  \frac{\Sigma}{\Delta(1+\alpha \Sigma)^2}\sum_{i=0^3}(\Bt_i+A_i)\Sigma^i
\eee
with $A_3=0.$ We claim for $0<\re(d,\ell)-r\ll 1$:
\be
\label{vneonvenenvepmemepmeo}
\Bt_i+\At_i<0, \ \ 0\le i\le 3
\ee
Since $\Delta<0$ at the right of $P_2$ where $\Sigma>0$, \eqref{vneonvenenvepmemepmeo} gives
$$(1-w-w'-F)\left(\Sigma,\frac{c_-}{1+\alpha\Sigma}\right)=\frac{\Sigma}{\Delta(1+\alpha \Sigma)^2}\sum_{i=0}^3(\Bt_i+A_i)\Sigma^i>0$$ and \eqref{nontouchingcondition} is proved.\\

\noindent{\bf step 12} Proof of \eqref{vneonvenenvepmemepmeo} for $\re=r^*$.\\
\noindent\underline{$\Bt_0+A_0$}. We compute using $c_4=-2\sigma_2^2$ for all r
$$A_0=(2\sigma_2^2+c_2)c_-=2\sigma^2_2(1+c_-)+c_2c_-+c_4=2\sigma_2^2(1+c_-)+\l_+$$ 
and hence
$$A_0+\Bt_0=2\sigma_2^2(1+c_-)+\l_++c_-\l_+-2\sigma_2^2(1+c_-)=(1+c_-)\l_+<0.$$
\noindent\underline{$\Bt^\infty_1+A^\infty_1$}. We compute at $r^*$
\bee
\Bt^\infty_1+A^\infty_1&=&-\frac{P_{\Bt_1}(\ell)}{(\ell+\sqrt{d})\left[d(\sqrt{d}-1)+\ell(\sqrt{d}+1)\right]^2}{+}\frac{P_{A_1}}{(\ell+\sqrt{d})\left[d(\sqrt{d}-1)+\ell(\sqrt{d}+1)\right]^2}\\
& = & -\frac{P_{\Bt_1}-P_{A_1}}{(\ell+\sqrt{d})\left[d(\sqrt{d}-1)+\ell(\sqrt{d}+1)\right]^2}
\eee
with
\bee
&&P_{\Bt_1}-P_{A_1}=9d^3\sqrt{d}-10d^3+d^2\sqrt{d}+  \left[d^3\sqrt{d}{-2}d^3+{7}d^2\sqrt{d}+4d^2-2d\sqrt{d}\right]\ell\\
& + & \left[2d^2\sqrt{d}+4d^2-5d\sqrt{d}+6d+\sqrt{d}\right]\ell^2+\left[d\sqrt{d}-2d+\sqrt{d}\right]\ell^3\\
& - & \left[6d^3\sqrt{d}-8d^3+2d^2\sqrt{d}+\left[{8}d^2-{8}d\sqrt{d}\right]\ell-2\sqrt{d}(d-1)\ell^2\right]\\
& = & 3d^3\sqrt{d}-2d^3-d^2\sqrt{d}+\left[d^3\sqrt{d}-2d^3+{7}d^2\sqrt{d}-{4}d^2+{6}d{\sqrt{d}}\right]\ell\\
& + & \left[2d^2\sqrt{d}+4d^2-3d\sqrt{d}+6d-\sqrt{d}\right]\ell^2+\left[d\sqrt{d}-2d+\sqrt{d}\right]\ell^3\\
&>0&
\eee
for $d\ge 2$ and therefore $\Bt^\infty_1+A^\infty_1<0.$\\
\noindent\underline{$\Bt^\infty_2+A^\infty_2$}. We compute:
$$\Bt_2+A_2=-\frac{P_{\Bt_2}-P_{A_2}}{\left[d(\sqrt{d}-1)+(\sqrt{d}+1)\right]^2}$$
with 
\bee
&&P_{\Bt_2}-P_{A_2}\\
&=&8d^4+4d^2(d-\sqrt{d})+\left[2d^3-2d^2\sqrt{d}+14d^2+2d\sqrt{d}\right]\ell+\left[2d^2+2d\sqrt{d}+2d+2\sqrt{d}\right]\ell^2\\
& - & \left[4d^2(d-\sqrt{d})+\ell(4d\sqrt{d}-4d)\right]\\
& = & 8d^4+\left[2d^3-2d^2\sqrt{d}+14d^2-2d\sqrt{d}+4d\right]\ell\\
&+&\left[2d^2{+}2d\sqrt{d}+{2}d+{2}\sqrt{d}\right]\ell^2{> 0}
\eee
{for $d\geq 2$}, which implies  $\Bt_2+A_2<0$.\\
\noindent\underline{$\Bt^\infty_3+A^\infty_3$}. We have $\Bt_3+A_3=\Bt_3<0$.\\
This concludes the proof of \eqref{vneonvenenvepmemepmeo} for $\re=r^*$.\\

\noindent{\bf step 13} Proof of \eqref{vneonvenenvepmemepmeo} for $\re=r_+$.\\
\noindent\underline{$\Bt_0+A_0$}. We have verbatim as above $$A_0+\Bt_0=\l_+<0.$$
\noindent\underline{$\Bt^\infty_1+A^\infty_1$}. We compute:
\bee
&&\Bt^\infty_1+A^\infty_1=-\frac{(d+1)\ell+(d^2+3d)\sqrt{\ell}+2d}{(1+\sqrt{\ell})(d\sqrt{\ell}+\ell)}+\frac{d-1}{\ell(1+\sqrt{\ell})}\\
& = & -\frac{\ell\left[(d+1)\ell+(d^2+3d)\sqrt{\ell}+2d\right]-(d-1)(d\sqrt{\ell}+\ell)}{\ell(1+\sqrt{\ell})(d\sqrt{\ell}+\ell)}=-\frac{Q_1(\ell)}{\ell(1+\sqrt{\ell})(d\sqrt{\ell}+\ell)}
\eee
with 
$$Q_1(\ell)=(d+1)\ell^2+(d^2+3d)\ell\sqrt{\ell}+(d+1)\ell-d(d-1)\sqrt{\ell}>0$$ for $\ell>d$, $d=2,3$.

\noindent\underline{$\Bt^\infty_2+A^\infty_2$}. We compute:
\bee
&&\Bt^\infty_2+A^\infty_2\\
&=&-\frac{(d+1)\ell^2+(2d^2+6d)\ell\sqrt{\ell}+(d^3+7d^2+4d)\ell+(8d^2)\sqrt{\ell}+2d^2}{(d\sqrt{\ell}+\ell)^2}+\frac{d-1}{\ell}\\
& = & -\frac{Q_2(\ell)}{\ell(d\sqrt{\ell}+\ell)^2}
\eee
with 
\bee
&&Q_2=(d+1)\ell^3+(2d^2+6d)\ell^2\sqrt{\ell}+(d^3+7d^2+4d)\ell^2+(8d^2)\ell\sqrt{\ell}+2d^2\ell\\
&-& (d-1)(d^2\ell+2d\ell\sqrt{\ell}+\ell^2)\\
& = & (d+1)\ell^3+(2d^2+6d)\ell^2\sqrt{\ell}+(d^3+7d^2+3d+1)\ell^2+(6d^2+2d)\ell\sqrt{\ell}\\
& + & (3d^2-d^3)\ell>0.
\eee

\noindent\underline{$\Bt^\infty_3+A^\infty_3$}. We have $\Bt^\infty_3+A^\infty_3=\Bt^\infty_3<0$.\\
This concludes the proof of \eqref{vneonvenenvepmemepmeo} for $\re=r_+$.
\end{proof}


\subsection{Proof of Lemma \ref{lemmainside}}


We are now in position to finish the proof of Lemma \ref{lemmainside}. We need to show \eqref{coercivityquadrcouplinginside}. To the right of $P_2$, $w'>0$ and $F>0$ from \eqref{formularuvoe}. Therefore, 
the first statement in \eqref{coercivityquadrcouplinginside} follows from \eqref{funamdnanteo}. For the second, to the right of $P_2$ we have $\sigma\ge \sigma_2$ and $\Delta<0$, i.e., $\sigma>1-w$. These, together with \eqref{funamdnanteo},
imply  $$1-w-w'-\frac{1-w}{\sigma}F\ge 1-w-w'-F\ge c>0$$ The third statement follows from $F>0$ at the right of $P_2$. \eqref{coercivityquadrcouplinginside} is proved.


\section{Exterior positivity}
\label{sec:posotivitynecessaryforcontrollinearizationoutsidelightcone}


We now turn to the proof of Lemma \ref{repsusoutides}. We will establish the following result which gives a 
precise range of validity of \eqref{propertyptobeproved} and enlarges the set of admissible parameters. Let 
\be
\label{conditionnls}
\ell_2(d)=d-2\sqrt{d} \ \ \mbox{for}\ \  d\ge 5.
\ee 
The following implies Lemma \ref{repsusoutides}.

\begin{lemma}[Necessary/sufficient conditions for \eqref{propertyptobeproved}]\label{lemma:posotivitynecessaryforcontrollinearization}
Assume $d\ge 3$ and  \eqref{suaempeonpoengopngoe}. For $\ell>d$, $\re=r^*(d,\ell)$, let $\ell_2(d)$ be given by  \eqref{conditionnls}, and assume moreover that $(\ell,d)$ satisfy: 
\be
\label{assumtionneoneo}
\left|\begin{array}{ll}
\sqrt{3}<\ell & \mbox{for}\ \ d=3\\
0.11<\ell<\ell_2(5)=0.53 & \mbox{for}\ \ d=5 \ \ \mbox{i.e.,}\ \  9\le p \le 37,   \\
0.2<\ell< \ell_2(6)=1.10 & \mbox{for}\ \ d=6 \ \ \mbox{i.e.,}\ \   5\le p\le 21, \\
0.3<\ell<\ell_2(7)=1.71 & \mbox{for}\ \ d=7  \ \ \mbox{i.e.,}\ \ 4\le p\le 14,\\
0.45<\ell<\ell_2(8)=2.34 & \mbox{for}\ \ d=8 \ \ \mbox{i.e.,}\ \ 3\le p\le 9,\\
0.65<\ell<\ell_2(9)=3 & \mbox{for}\ \  d=9 \ \ \mbox{i.e.,}\ \ 3\le p\le 7;
\end{array}\right.
\ee
Then there exists $\e(d,\ell)$ such that for all 
\be
\label{neoinvieonioeneonenv}
\re(d,\ell)-\e(d,\ell)<r<\re(d,\ell)
\ee
 and for any  $P_2-P_4$ trajectory with $c_-$ slope at $P_2$ as in Lemma \ref{lemmaconnection}, there exists $c_r>0$ such that 
\be
\label{propertyptobeprovedbis} 
\forall 0<\sigma\le \sigma_2, \ \ \left|\begin{array}{l} (1-w-w')^2-F^2>c_r\\ 1-w- w'>c_r.
\end{array}\right.
\ee
\end{lemma}

\begin{remark} We note that 
\be
\label{nenneonvoneo}
\ell<\ell_2(d)\Leftrightarrow r^*(d,\ell)>2
\ee
The latter is  a fundamental property for the study of the defocusing (NLS) problem, \cite{MRRSnls}. 
The lower bounds in \eqref{assumtionneoneo} are actually given by the condition 
\be
\label{vnioneneone}
\ell>\ell_1(d)
\ee
where
\bea
\label{defellone}
&&\ell_1(d)\\
\nonumber &=& \frac{-(d\sqrt{d}-d+5\sqrt{d}-1)+\sqrt{(d\sqrt{d}-d+5\sqrt{d}-1)^2+4(d\sqrt{d}-d-2\sqrt{d})(\sqrt{d}+1)}}{2(\sqrt{d}+1)},
\eea
see \eqref{cenonvenvoineionoenv}.
We will see below that \eqref{vnioneneone} is a necessary condition for \eqref{propertyptobeprovedbis} to hold near $r^*(d,\ell)$.
\end{remark}

\begin{remark} 
\label{rem:sign}
From  \eqref{formularuvoe}, we have that along  the solution curve for $\sigma<\sigma_2$
\bee
F<0\textrm{ for }w< w_2, \ \ \ \ \ \ F>0\textrm{ for }w> w_2.
\eee
We will first treat the case $w<w_2$, in which case it suffices to prove $1-w-w'+F>0$, and then the case $w>w_2$, in which case it suffices to prove $1-w-w'-F>0$.
\end{remark}


\subsection{Bound on the slope for the $P_2-P_4$ separatrix}


The $P_2-P_4$ separatrix\footnote{We will sometime denote by a subscript $S$, e.g. $\Phi_S$}, i.e., the unique solution connecting $P_2$ and $P_4$ with the slope $c_+$ at $P_2$ described in (3) of Lemma \ref{curvespraitapjoi04ptwo}, furnishes a natural lower bound for the solution curves of Theorem \ref{thmmain}. We start with a {\em rough} estimate on its slope. 

\begin{lemma}[Upper bound on the slope for the $P_2-P_4$ separatrix]
\label{uperobundpareerjo}
Under the assumptions of Lemma \ref{lemma:posotivitynecessaryforcontrollinearization}, the $P_2-P_4$ separatrix given in (3) of 
Lemma \ref{curvespraitapjoi04ptwo} satisfies 
\be
\label{loweroubndphiplus}
0<\Phi_S=\frac{W}{\Sigma}\le c_+.
\ee
\end{lemma}

\begin{proof}[Proof of Lemma \ref{uperobundpareerjo}]
 We have $\sigma<\sigma_2$, i.e., $\Sigma<0$ and $W=w-w_2<0$. Therefore, $$\Phi=\frac{W}{\Sigma}>0.$$
 
 \noindent{\bf step 1} Setting up. We have 
$$
\Phi_S(P_2)=c_+, \ \ \Phi_S(P_4)=\frac{-w_2}{-\sigma_2}=\frac{1-\sigma_2}{\sigma_2}
$$
and hence at $r=\re$:
$$
\Phi_S(P^\infty_2)-\Phi_S(P_4^\infty)=\left|\begin{array}{l}
\ell-\frac{1-\sigma_2}{\sigma_2}=\ell+1-\frac{\ell+\sqrt{d}}{\sqrt{d}}=\frac{\ell(\sqrt{d}-1)}{\sqrt{d}}>0\ \ \mbox{for}\ \ r=r^*\\
\frac{\sqrt{\ell}(d+\sqrt{\ell})}{1+\sqrt{\ell}}-\sqrt{\ell}=\frac{(d-1)\sqrt{\ell}}{1+\sqrt{\ell}}>0 \ \  \mbox{for}\ \ r=r_+
\end{array}\right.
$$
which ensures
\be
\label{neionveinenvo}
\Phi_S(P_2)-\Phi_S(P_4)>0.
\ee
We now recall 
$$\frac{d\Phi_S}{d\Sigma}=\frac{\Delta_1-\Phi_S \Delta_2}{\Sigma\Delta_2}$$ and study the sign at a possible point of contact on the curve with the value $$\Phi_S(\Sigma)=c_+, \ \ \Sigma<0.$$ From \eqref{expressioninsignaphi}, at the point of contact:
\bee
\frac{\Delta_1-c_+ \Delta_2}{\Sigma}& = & c_1c_++c_3-c_+(c_2c_++c_4)+  \left[d_{20}c_+^2+d_{11}c_++d_{02}-c_+(e_{20}c_+^2+e_{11}c_++e_{02})\right]\Sigma\\
& + & \left[c_+^3-dc_+-c_+(e_{21}c_+^2-1)\right]\Sigma^2=  A_0+A_1\Sigma+A_2\Sigma^2.
\eee
We have $A_0=0$ from \eqref{equationcminus}. We now compute all coefficients at $\re(d,\ell)$, and the associated non degeneracy claim will follow for $r$ close enough to $\re(d,\ell)$.\\

\noindent{\bf step 2} Sign of $A^\infty_2$.\\
\noindent\underline{case $r=r^*$}. We compute:
\bee
A^\infty_2&=&\ell^3-d\ell-\ell\left(\frac{\ell+d-1}{\ell}\ell^2-1\right)=\ell^3-d\ell-\ell^2(\ell+d-1)+\ell\\
& = & \ell\left[-d-\ell(d-1)+1\right]=-\ell(d-1)(1+\ell)<0.
\eee
\noindent\underline{case $r=r_+$}. We compute:
\bee
&&A^\infty_2=(c^\infty_+)^3-dc^\infty_+-c^\infty_+(e^\infty_{21}(c^\infty_+)^2-1)=c^\infty_+\left[(1-e^\infty_{21})(c^\infty_+)^2-(d-1)\right]\\
&= & c^\infty_+\left[\left(1-\frac{\ell+d-1}{\ell}\right)\frac{\ell(d+\sqrt{\ell})^2}{(1+\sqrt{\ell})^2}-(d-1)\right]=-(d-1)c^\infty_+\left[\frac{(d+\sqrt{\ell})^2}{(1+\sqrt{\ell})^2}+1\right]<0.
\eee

\noindent{\bf step 3} Sign of $A^\infty_1$.\\
\noindent\underline{case $r=r^*$}. We have
\bee
A^\infty_1& = & d^\infty_{20}(c^\infty_+)^2+d^\infty_{11}c^\infty_++d^\infty_{02}-c^\infty_+(e^\infty_{20}(c^\infty_+)^2+e^\infty_{11}c^\infty_++e^\infty_{02})
\eee
We compute:
\bee
&&e^\infty_{20}(c^\infty_+)^2+e^\infty_{11}c^\infty_++e^\infty_{02}=\frac{\sqrt{d}(d-1+\ell)}{\ell(\ell+\sqrt{d})}\ell^2-\frac{d(\sqrt{d}-1)+\ell(1+\sqrt{d})}{\ell(\ell+\sqrt{d})}\ell-\frac{3\sqrt{d}}{\ell+\sqrt{d}}\\
& = & \frac{\ell\sqrt{d}(d-1+\ell)-d\sqrt{d}+d-\ell(1+\sqrt{d})-3\sqrt{d}}{\ell+\sqrt{d}}\\
& = & \frac{-d\sqrt{d}+d-3\sqrt{d}+\ell(d\sqrt{d}-\sqrt{d}-1-\sqrt{d})+\ell^2\sqrt{d}}{\ell+\sqrt{d}}\\
& = & \frac{-d\sqrt{d}+d-3\sqrt{d}+\ell(d\sqrt{d}-2\sqrt{d}-1)+\ell^2\sqrt{d}}{\ell+\sqrt{d}}
\eee
and 
\bee
&&d^\infty_{20}(c^\infty_+)^2+d^\infty_{11}c^\infty_++d^\infty_{02}=\frac{-\sqrt{d}-(d-\ell)}{\ell+\sqrt{d}}\ell^2-\frac{2d\sqrt{d}}{\ell+\sqrt{d}}\ell-\frac{\ell\sqrt{d}}{\ell+\sqrt{d}}\\
& = & -\frac{\ell}{\ell+\sqrt{d}}\left[\ell(\sqrt{d}+d-\ell)+2d\sqrt{d}+\sqrt{d}\right]\\
& =& -\frac{\ell}{\ell+\sqrt{d}}\left[-\ell^2+\ell(\sqrt{d}+d)+2d\sqrt{d}+\sqrt{d}\right].
\eee
Therefore,
\bee
&&(\ell+\sqrt{d})A^\infty_1=\ell\left[\ell^2-\ell(\sqrt{d}+d)-2d\sqrt{d}-\sqrt{d}\right]\\
&-&\ell\left[-d\sqrt{d}+d-3\sqrt{d}+\ell(d\sqrt{d}-2\sqrt{d}-1)+\ell^2\sqrt{d}\right]\\
& = & \ell\left[\ell^2(1-\sqrt{d})-\ell(\sqrt{d}+d+d\sqrt{d}-2\sqrt{d}-1)-2d\sqrt{d}-\sqrt{d}+d\sqrt{d}-d+3\sqrt{d}\right]\\
& =&- \ell\left[(\sqrt{d}-1)\ell^2+(d\sqrt{d}+d-\sqrt{d}-1)\ell+d\sqrt{d}+d-2\sqrt{d}\right]<0.
\eee

\noindent\underline{case $r=r_+$}. We have
\bee
A^\infty_1& = & d^\infty_{20}(c^\infty_+)^2+d^\infty_{11}c^\infty_++d^\infty_{02}-c^\infty_+(e^\infty_{20}(c^\infty_+)^2+e^\infty_{11}c^\infty_++e^\infty_{02})
\eee
We compute:
\bee
&&e^\infty_{20}(c^\infty_+)^2+e^\infty_{11}c^\infty_++e^\infty_{02}=\frac{\ell+d-1}{\ell(1+\sqrt{\ell})}\left(\frac{\sqrt{\ell}(d+\sqrt{\ell})}{1+\sqrt{\ell}}\right)^2-\frac{2}{1+\sqrt{\ell}}\left(\frac{\sqrt{\ell}(d+\sqrt{\ell})}{1+\sqrt{\ell}}\right)-\frac{3}{1+\sqrt{\ell}}\\
& = & \frac{(\ell+d-1)(d+\sqrt{\ell})^2-2\sqrt{\ell}(d+\sqrt{\ell})(1+\sqrt{\ell})-3(1+\sqrt{\ell})^2}{(1+\sqrt{\ell})^3
}\\
& = & \frac{(\ell+d-1)(d^2+2d\sqrt{\ell}+\ell)-2\sqrt{\ell}(d+(d+1)\sqrt{\ell}+\ell)-3(1+2\sqrt{\ell}+\ell)}{(1+\sqrt{\ell})^3}\\
& = & \frac{\ell^2+(2d-2)\ell\sqrt{\ell}+(d^2-d-6)\ell+(2d^2-4d-6)\sqrt{\ell}+d^3-d^2-3}{(1+\sqrt{\ell})^3}=\frac{Q_1(\ell)}{(1+\sqrt{\ell})^3}
\eee
and 
\bee
&&d^\infty_{20}(c^\infty_+)^2+d^\infty_{11}c^\infty_++d^\infty_{02}\\
&=& \frac{\ell-\sqrt{\ell}-d-1}{(1+\sqrt{\ell})^2}\left(\frac{\sqrt{\ell}(d+\sqrt{\ell})}{1+\sqrt{\ell}}\right)^2-\frac{2d}{1+\sqrt{\ell}}\left(\frac{\sqrt{\ell}(d+\sqrt{\ell})}{1+\sqrt{\ell}}\right)-\frac{\ell+d\sqrt{\ell}}{(1+\sqrt{\ell})^2}\\
& = & \frac{\ell(\ell-\sqrt{\ell}-d-1)(d+\sqrt{\ell})^2-2d\sqrt{\ell}(d+\sqrt{\ell})(1+\sqrt{\ell})^2-(\ell+d\sqrt{\ell})(1+\sqrt{\ell})^2}{(1+\sqrt{\ell})^4}
\eee
and 
\bee
&&Q_2(\ell)=\ell(\ell-\sqrt{\ell}-d-1)(d+\sqrt{\ell})^2-2d\sqrt{\ell}(d+\sqrt{\ell})(1+\sqrt{\ell})^2-(\ell+d\sqrt{\ell})(1+\sqrt{\ell})^2\\
& = & [\ell^2-\ell\sqrt{\ell}-(d+1)\ell](d^2+2d\sqrt{\ell}+\ell)-(2d^2\sqrt{\ell}+2d\ell)(1+2\sqrt{\ell}+\ell)-(\ell+d\sqrt{\ell})(1+2\sqrt{\ell}+\ell)\\
& = & \ell^3+(2d-1)\ell^2\sqrt{\ell}+(d^2-5d-2)\ell^2+(-5d^2-7d-2)\ell\sqrt{\ell}\\
&+& (-d^3-5d^2-4d-1)\ell+(-2d^2-d)\sqrt{\ell}.
\eee
Hence
$$
A^\infty_1=\frac{Q_2(\ell)}{(1+\sqrt{\ell})^4}-\frac{\sqrt{\ell}(d+\sqrt{\ell})}{1+\sqrt{\ell}}\frac{Q_1(\ell)}{(1+\sqrt{\ell})^3}=-\frac{Q_3}{(1+\sqrt{\ell})^4}
$$
with 
\bee
&&Q_3=\sqrt{\ell}(d+\sqrt{\ell})Q_1-Q_2\\
&=&(d\sqrt{\ell}+\ell)\Big\{\ell^2+(2d-2)\ell\sqrt{\ell}+(d^2-d-6)\ell+(2d^2-4d-6)\sqrt{\ell}+d^3-d^2-3\Big\}\\
&-& \ell^3-(2d-1)\ell^2\sqrt{\ell}-(d^2-5d-2)\ell^2+(5d^2+7d+2)\ell\sqrt{\ell}\\
&+& (d^3+5d^2+4d+1)\ell+(2d^2+d)\sqrt{\ell}\\
& = & (d-1)\ell^2\sqrt{\ell}+(2d^2+2d-4)\ell^2+(d^3+6d^2-3d-4)\ell\sqrt{\ell}+(4d^3-2d-2)\ell\\
&+& (d^4-d^3+2d^2-2d)\sqrt{\ell}>0
\eee
and hence $A^\infty_1<0$.\\

\noindent\underline{Derivative at $P_2$}. Near $P_2$ we have the Taylor expansion:
\be
\label{neniovinveoiederivative}
\frac{d\Phi}{d\Sigma}=\frac{A_1\Sigma+A_2\Sigma^2}{\Sigma\left[c_2\Phi+c_4+O(\Sigma)\right]}\to \frac{A_1}{c_2c_++c_4}=\frac{A_1}{\l_-}>0.
\ee
\noindent\underline{Conclusion}. At the point of contact, we therefore obtain:
$$\frac{d\Phi}{d\Sigma}=\frac{\Delta_1-c_+\Delta_2}{\Sigma\Delta_2}=\frac{\Sigma(A_1+A_2\Sigma)}{\Delta_2}=-\frac{|\Sigma|(-|A_1|+|A_2||\Sigma|)}{\Delta_2}=\frac{|\Sigma|(|A_1|-|A_2||\Sigma|)}{\Delta_2}.$$ Let $$\Sigmat=-\Sigma, \ \ \Phit(\Sigmat)=\Phi(\Sigma),$$ then $$\frac{d\Phit}{d\Sigmat}=-\frac{\Sigmat(|A_1|-|A_2|\Sigmat)}{\Delta_2}.$$ Assume now that there exists $0<\Sigmat^*<\sigma_2$ with $\Phit(\Sigmat^*)>c_+$, then from \eqref{neionveinenvo}, \eqref{neniovinveoiederivative}, the curve is strictly below $c_+$ for $0<\Sigmat\ll 1$ and for $|\Sigmat-\sigma_2|\ll1 $. Therefore, there must exist $\Sigmat_1<\Sigmat^*<\Sigmat_2$ with $$\left|\begin{array}{l}
\Phit(\Sigmat_1)=\Phit(\Sigmat_2)=c_+\\
\Phit'(\Sigmat_1)\ge 0\\
\Phit'(\Sigmat_2)\le 0
\end{array}\right.
$$ 
Since $\Delta_2>0$ in this zone: 
$$|A_1|-|A_2|\Sigmat_1 { \le } 0, \ \ |A_1|-|A_2|\Sigmat_2\ge0$$ 
which forces 
$$\Sigmat_2\leq \frac{|A_1|}{|A_2|}\leq \Sigmat_1,$$ 
a contradiction.
\end{proof}


\subsection{Study of $F$ on the $P_2-P_4$ separatrix}


\begin{lemma}[Value at $P_2$]
\label{vlaueatotwtocplus}
Under the assumptions of Lemma \ref{lemma:posotivitynecessaryforcontrollinearization}, for the $P_4-P_2$ separatrix and 
for $\re(\ell)-\e(d,\ell)<r<\re(d,\ell)$: 
\be
\label{eoneoneneoneno}
F_S(P_2)<0\\
\ee
Moreover, in the case $\ell<d$, we have:
\bea
\label{veniovneinenepwotoiviory}
&&\nonumber \left[\exists 0<\e(d,\ell)\ll 1 \ \ \mbox{such that}\ \  \forall 0<\e<\e(d,\ell), \ \ (1-w-w'+F)(P_2)>0\right]\\
& \Leftrightarrow& \ \ \ell>\ell_1(d)
\eea
with $\ell_1$ given by \eqref{defellone}.
\end{lemma}

\begin{remark} 
The necessary admissible range near $r^*$ is therefore 
$$\ell_1(d)<\ell<\ell_2(d)=d-2\sqrt{d}.$$ 
We compute numerically:
\be
\label{cenonvenvoineionoenv}
\left|\begin{array}{ll}
\ell_1(5)={0.1023}, &  \ell_2(5)={0.5279},\\
\ell_1(6)=0.1845, & \ell_2(6)=1.101,\\
\ell_1(7)=0.2525, & \ell_2(7)=1.7085,\\
\ell_1(8)={0.3098}, & \ell_2(8)=2.3431,\\
\ell_1(9)={0.3589}, & \ell_2(9)=3.
\end{array}\right.
\ee
\end{remark}

\begin{proof}[Proof of Lemma \ref{vlaueatotwtocplus}] We compute all coefficients at $\re(d,\ell)$ and argue by continuity for $r$ close enough to $\re(d,\ell)$. Near $P_2$ from \eqref{expressioninsignaphi}, \eqref{equationcminus}, \eqref{realtionslopeweignefuncitons}:
\bee
F_S&=&\sigma_2-\frac{\Sigma\left[c_2\Phi+c_4+[e_{20}\Phi^2+e_{11}\Phi+e_{02}]\Sigma+[e_{21}\Phi^2-1]\Sigma^2\right]}{-\Sigma\left[2\sigma_2(1+\Phi)+\Sigma(1-\Phi^2)\right]}+O(\Sigma)\\
& = & \sigma_2+\frac{c_2c_{\pm}+c_4}{2\sigma_2(1+c_\pm)}+O(\Sigma)=\sigma_2+\frac{\l_-}{2\sigma_2(1+c_+)}+O(\Sigma)
\eee
and 
\bee
w'&=&-\frac{\Delta_1}{\Delta}=-\frac{c_1\Phi+c_3+[d_{20}\Phi^2+d_{11}\Phi+d_{02}]\Sigma+[\Phi^3-d\Phi]\Sigma^2}{-\left[2\sigma_2(1+\Phi)+\Sigma(1-\Phi^2)\right]}=\frac{c_1c_++c_3}{2\sigma_2(1+c_+)}+O(\Sigma)\\
& = & \frac{c_+\l_-}{2\sigma_2(1+c_+)}+O(\Sigma)
\eee
and we compute these quantities at $P_2$ and $r=\re$.\\

\noindent\underline{case $\re=r^*$}. 
\bee
F_S&=&\frac{2(\sigma_2^\infty)^2(1+\ell)+\l_-}{2\sigma_2(1+\ell)}\\
& = & \frac{1}{2\sigma^\infty_2(1+\ell)(\ell+\sqrt{d})^2}\left[2d(1+\ell)-\left[d(d-\sqrt{d})+2d+\ell(d+\sqrt{d})\right]\right]\\
&=&-\frac{(d-\ell)(d-\sqrt{d})}{2\sqrt{d}(1+\ell)(\ell+\sqrt{d})}<0
\eee
and 
\bea
\label{epocneoneone}
\nonumber &&1-w-w'+F_S\\
\nonumber&=&\sigma^\infty_2-\frac{\ell\l_-}{2\sigma^\infty_2(1+\ell)}+\sigma^\infty_2+\frac{\l_-}{2\sigma^\infty_2(1+\ell)}=  2\left[\sigma^\infty_2+\frac{\l_-}{2\sigma^\infty_2(1+\ell)}\right]-\frac{(1+\ell)\l_-}{2\sigma^\infty_2(1+\ell)}\\
\nonumber & = & -\frac{2(d-\ell)(d-\sqrt{d})}{2\sqrt{d}(1+\ell)(\ell+\sqrt{d})}+\frac{1+\ell}{\frac{2\sqrt{d}(1+\ell)}{\ell+\sqrt{d}}}\frac{d(d-\sqrt{d})+2d+\ell(d+\sqrt{d})}{(\ell+\sqrt{d})^2}\\
\nonumber & = & \frac{1}{2\sqrt{d}(1+\ell)(\ell+\sqrt{d})}\left[-2(d-\ell)(d-\sqrt{d})+(1+\ell)\left(d(d-\sqrt{d})+2d+\ell(d+\sqrt{d})\right)\right]\\
& = & \frac{P_d(\ell)}{2\sqrt{d}(1+\ell)(\ell+\sqrt{d})}
\eea
with
\bea
\label{defpld}
\nonumber &&P_d(\ell)=-2(d-\ell)(d-\sqrt{d})+(1+\ell)\left(d(d-\sqrt{d})+2d+\ell(d+\sqrt{d})\right)\\
\nonumber& = & -2d(d-\sqrt{d})+d(d-\sqrt{d})+2d\\
\nonumber& + & \ell\left[2(d-\sqrt{d})+d(d-\sqrt{d})+2d+d+\sqrt{d}\right]+\ell^2(d+\sqrt{d})\\
\nonumber&= & -d^2+d\sqrt{d}+2d+\left[d^2-d\sqrt{d}+5d-\sqrt{d}\right]\ell+(d+\sqrt{d})\ell^2\\
& = & \sqrt{d}\left\{-d\sqrt{d}+d+2\sqrt{d}+\left[d\sqrt{d}-d+5\sqrt{d}-1\right]\ell+(\sqrt{d}+1)\ell^2\right\}
\eea
Observe that 
\bee
P_d(1)&=&-2(d-1)(d-\sqrt{d})+2\left(d(d-\sqrt{d})+2d+(d+\sqrt{d})\right)\\
& = & 2(d-\sqrt{d})+2\left(2d+(d+\sqrt{d})\right)>0
\eee
\and 
$$ P_d(0)=-2d(d-\sqrt{d})+d(d-\sqrt{d})+2d<0
$$
which implies the condition $$\ell>\ell_1(d),  \ \ \ell_1(d)<1$$ with $\ell_1$ given by \eqref{defellone}.\\

\noindent\underline{case $\re=r_+$}. We compute
$$
F_S=\sigma_2+\frac{c_2c_++c_4}{2\sigma_2(1+c_+)}=\frac{-c_4(1+c_+)+c_2c_++c_4}{2\sigma_2(1+c_+)}=\frac{c_+(c_2-c_4)}{2\sigma_2(1+c_+)}<0
$$
from \eqref{tionveiogbngo3o}, and $F(P_2^\infty)=0$ from \eqref{calculparametresbis}. Then at $\re$ and $P_2^\infty$:
$$1-w-w'+F_S=\sigma_2-\frac{c_+\l_-}{2\sigma_2(1+c_+)}=\frac{2\sigma_2^2(1+c_+)-c_+\l_-}{2\sigma_2(1+c_+)}=-\frac{(1+c_+)\l_-}{2\sigma_2(1+c_+)}>0$$
\end{proof}


\subsection{Positivity in the region where $w\leq w_2$}


We now consider {\em any solution curve} $P_2-P_4$ with the $c_-$ slope at $P_2$ 
given in Lemma \ref{lemmaconnection}. In view of the phase portrait of figure \ref{fig:signofDeltasinphaseportrait}, there exists a unique $0<\sigma_1<\sigma_2$ such that we have 
\be
\label{vneoneoineinoenenoentt}
\left|\begin{array}{l}
w> w_2\ \ \mbox{for} \ \ \sigma_1<\sigma<\sigma_2\\
w< w_2\ \ \mbox{for} \ \ 0<\sigma<\sigma_1
\end{array}\right.
\ee

\begin{lemma}[Positivity in the region where  $\sigma<\sigma_1$ and $w< w_2$] 
\label{lemmaconditional}
Under the assumptions of Lemma \ref{lemma:posotivitynecessaryforcontrollinearization}, any $P_2-P_4$ curve with $c_-$ slope at $P_2$ satisfies  \eqref{propertyptobeprovedbis} in the region $w< w_2$ and $0<\sigma<\sigma_1$, 
\end{lemma}

\begin{proof}[Proof of Lemma \ref{lemmaconditional}] 
Using \eqref{formularuvoe}, we have along the solution curve in the region $w<w_2$:
$$F<0, \ \ w<w_2<1\ \ \mbox{and}\ \ \Delta>0.$$
Note also that the solution curve has the slope $c_-$ at $P_2$ and can not intersect the separatrix curve (strictly) 
between $P_2$ and $P_4$. Therefore, it must lie above the $P_2$-$P_4$ separatrix. As a consequence its function $\Phi$ also   satisfies \eqref{loweroubndphiplus}:
\bee
0<\Phi\le c_+,
\eee
where the lower bound follows from the fact that we consider the region $w\leq w_2$. Thus we focus below on the region 
\bee
0<\sigma<\sigma_1, \ \ \ \ w<w_2, \ \ \ \ 0<\Phi\le c_+.
\eee

\noindent{\bf step 1} Study of $1-w-w'+F$. 
\begin{remark}
We will consider the expression for $1-w-w'+F$
$$
1-w-w'+F=(1-w+\sigma)+\frac{\Delta_1-\Delta_2}{\Delta}
$$
 not just as a function on the solution
curve but more generally as a function of $\sigma$ and $w$.
\end{remark}

We compute
\bea
\label{eninenenenvnoe}
\nonumber &&(1-w-w')+F=-\frac{(d-1)\sigma}{\ell\Delta}(w-w_-)(w-w_+)+(1-w+\frac{\Delta_1}{\Delta})\\
\nonumber &=& \frac{1}{\ell\Delta}\Big\{-(d-1)\sigma(w-w_-)(w-w_+)\\
\nonumber &+& \ell\Big[(1-w)[(1-w)^2-\sigma^2]+w(1-w)(r-w)-d(w-w_e)\sigma^2\Big]\Big\}\\
\nonumber &= & -\frac{1}{\ell\Delta}\Big\{\ell\left[1-w+d(w-w_e)\right]\sigma^2+(d-1)(w_--w)(w_+-w)\sigma-\ell(1-w)[(1-w)^2+w(r-w)]\Big\}\\
\nonumber &=& -\frac{1}{\ell\Delta}\Big\{\ell(d-1)(w-w_*)\sigma^2+(d-1)(w_--w)(w_+-w)\sigma-\ell(1-w)[(1-w)^2+w(r-w)]\Big\}\\
\nonumber & = & -\frac{1}{\ell\Delta}\Big\{\ell(d-1)(w-w_*)\sigma^2+(d-1)(w_--w)(w_+-w)\sigma-\ell(1-w)[1+(r-2)w]\Big\}\\
& = & -\frac{P_\sigma(w)}{\ell \Delta}
\eea
with 
\be
\label{formulapsgina}
P_\sigma(w)=[(d-1)\sigma+\ell(r-2)]w^2+a_1(\sigma)w+a_2(\sigma)
\ee
and 
\be
\label{defwstart}
w_*=\frac{dw_e-1}{d-1}=\frac{\ell(r-1)-1}{d-1}.
\ee

We study the roots of $P_\sigma(w)$ for $0<w<w_2=w_-<w_+<1<r$ which ensures: 
\be
\label{cniovnoneovn}
\left|\begin{array}{l}
(1-w)[(1-w)^2+w(r-w)]>0\\
(w_+-w)(w_--w)>0
\end{array}\right.
\ee
Then, at $r^*(d,\ell)$:
\bee
&&(d-1)(w^\infty_2-(w^*)^\infty)=(d-1)w^\infty_2+1-\ell(r^*-1)=(d-1)(1-\sigma^\infty_2)+1-\ell\left(\frac{d+\ell}{\ell+\sqrt{d}}-1\right)\\
&= & d-(d-1)\frac{\sqrt{d}}{\ell+\sqrt{d}}-\frac{\ell(d-\sqrt{d})}{\ell+\sqrt{d}}=\frac{d(\ell+\sqrt{d})-d\sqrt{d}+\sqrt{d}-\ell d+\ell\sqrt{d}}{\ell+\sqrt{d}}=\frac{\sqrt{d}(\ell+1)}{\ell+\sqrt{d}}>0
\eee
and at $r_+(d,\ell)$:
\bee
&&(d-1)(w^\infty_2-(w_*)^\infty)=\frac{(d-1)\sqrt{\ell}}{1+\sqrt{\ell}}+1-\frac{\ell(d-1)}{(1+\sqrt{\ell})^2}\\
& = & \frac{d\sqrt{\ell}+1}{1+\sqrt{\ell}}-\frac{\ell(d-1)}{(1+\sqrt{\ell})^2}=\frac{\ell+(d+1)\sqrt{\ell}+1}{(1+\sqrt{\ell})^2}>0
\eee
and hence for $r$ close enough to $\re$: $$w_2>w_*.$$
 
\noindent{\noindent{\bf step 2} Root for $w_2\le w\le 1$. In this regime, we have from \eqref{cniovnoneovn},   since $w\geq w_2>w^*$, only one positive root
\bea
\label{defsigmaoneomega}
&&\sigma=\sigma_1(w)= \frac{1}{2\ell(d-1)(w-w_*)}\Big\{-(d-1)(w_--w)(w_+-w)\\
\nonumber &+& \sqrt{\left[(d-1)(w_--w)(w_+-w)\right]^2+4\ell(1-w)[(1-w)^2+{w}(r-w)]\ell(d-1)(w-w_*)}\Big\}.
\eea
The points $P_2=(\sigma_2,w_2)$, $P_3=(\sigma_+,w_+)$ and $P_1=(0,1)$ are among the roots of $P_\sigma(w)$, and hence the curve $\sigma_1(w)$ must pass through these points and connects continuously 
$(\sigma_2,w_2)$ to $(0,1)$. }\\

\noindent{\bf step 3} Positivity {on} $\Phi=c_+$. Let the line 
\be
\label{defline}
(D)=\{W=c_+ \Sigma, \ \ 0\le \sigma\le \sigma_2\}
\ee
we claim that under the assumptions of Lemma \ref{lemma:posotivitynecessaryforcontrollinearization}:
\be
\label{tobeproveddd}
1-w-w'+F>0\ \ \mbox{on}\ \ (D)
\ee
Indeed, we compute:
\bee
&&\Delta(1-w-w'+F)=(1-w+\sigma)\Delta+\Delta_1-\Delta_2\\
&=& (1-w_2-W+\sigma_2+\Sigma)\left\{-\Sigma\left[2\sigma_2(1+\Phi)+\Sigma(1-\Phi^2)\right]\right\}\\
&+& \Sigma\left[c_1\Phi+c_3+[d_{20}\Phi^2+d_{11}\Phi+d_{02}]\Sigma+[\Phi^3-d\Phi]\Sigma^2\right]\\
&-& \Sigma\left[c_2\Phi+c_4+[e_{20}\Phi^2+e_{11}\Phi+e_{02}]\Sigma+[e_{21}\Phi^2-1]\Sigma^2\right]\\
& = & -\Sigma G(\Sigma,\Phi)
\eee
with
\bee
&&G(\Sigma,\Phi)=[2\sigma_2+\Sigma(1-\Phi)][2\sigma_2(1+\Phi)+\Sigma(1-\Phi^2)]\\
&-& c_1\Phi-c_3-[d_{20}\Phi^2+d_{11}\Phi+d_{02}]\Sigma-[\Phi^3-d\Phi]\Sigma^2\\
&+&c_2\Phi+c_4+[e_{20}\Phi^2+e_{11}\Phi+e_{02}]\Sigma+[e_{21}\Phi^2-1]\Sigma^2\\
& = & 4\sigma_2^2(1+\Phi)+\Sigma[2\sigma_2(1-\Phi^2)+2\sigma_2(1-\Phi^2)]+\Sigma^2(1-\Phi)^2(1+\Phi)\\
&-& c_1\Phi-c_3-[d_{20}\Phi^2+d_{11}\Phi+d_{02}]\Sigma-[\Phi^3-d\Phi]\Sigma^2\\
& + & c_2\Phi+c_4+[e_{20}\Phi^2+e_{11}\Phi+e_{02}]\Sigma+[e_{21}\Phi^2-1]\Sigma^2\\
& = & 4\sigma_2^2(1+\Phi)+c_2\Phi+c_4-c_1\Phi-c_3\\
&+& [4\sigma_2(1-\Phi^2)+e_{20}\Phi^2+e_{11}\Phi+e_{02}-d_{20}\Phi^2-d_{11}\Phi-d_{02}]\Sigma\\
& + & [(1-\Phi)^2(1+\Phi)+e_{21}\Phi^2-1-\Phi^3+d\Phi]\Sigma^2\\
& = & A_0(\Phi)+A_1(\Phi)\Sigma+A_2(\Phi)\Sigma^2.
\eee

and we now distinguish $\re=r^*$ and $\re=r_+$.\\

\noindent{\bf step 4} Proof of \eqref{defline} for $\re=r^*$. We compute the $A_i$ on $(D)$ at $r=r^*(d,\ell)$ i.e., $\Phi=c_+=\ell$ and evaluate the obtained sign of $G$.

\noindent\underline{Computation of $A^\infty_0$}. We have from \eqref{epocneoneone}, \eqref{veniovneinenepwotoiviory}:
$$
A^\infty_0 =  4(\sigma^\infty_2)^2(1+\ell)+\l^\infty_--\ell\l^\infty_-=2\sigma^\infty_2(1+\ell)\frac{P_d(\ell)}{2\sqrt{d}(1+\ell)(\ell+\sqrt{d})}=\frac{P_d(\ell)}{(\ell+\sqrt{d})^2}>0.
$$

\noindent\underline{Computation of $A^\infty_1$}. We have $$A^\infty_1=A_{11}+A_{12}$$ with
\bee
A_{11}& = & 4\sigma^\infty_2(1-\ell^2)+e^\infty_{20}\ell^2+e^\infty_{11}\ell+e^\infty_{02}\\
& = & \frac{4\sqrt{d}}{\ell+\sqrt{d}}(1-\ell^2)+\frac{\sqrt{d}(d-1+\ell)}{\ell(\ell+\sqrt{d})}\ell^2-\frac{d(\sqrt{d}-1)+\ell(1+\sqrt{d})}{\ell(\ell+\sqrt{d})}\ell-\frac{3\sqrt{d}}{\ell+\sqrt{d}}\\
& =& \frac{(4\sqrt{d}-4\ell^2\sqrt{d}+\ell\sqrt{d}(d-1+\ell)-d(\sqrt{d}-1)-\ell(1+\sqrt{d})-3\sqrt{d}}{\ell+\sqrt{d}}\\
& = & \frac{\sqrt{d}-d\sqrt{d}+d+\ell^2(-4\sqrt{d}+\sqrt{d})+\ell(\sqrt{d}(d-1)-1-\sqrt{d})}{\ell+\sqrt{d}}\\
& =& -\frac{3\sqrt{d}\ell^2-[(d-2)\sqrt{d}-1]\ell+\sqrt{d}(d-\sqrt{d}-1)}{\ell+\sqrt{d}}
\eee
and 
\bee
A_{12}& = & -d^\infty_{20}\ell^2-d^\infty_{11}\ell-d^\infty_{02}=\left(\frac{\sqrt{d}+(d-\ell)}{\ell+\sqrt{d}}\right)\ell^2+\left(\frac{2d\sqrt{d}}{\ell+\sqrt{d}}\right)\ell+\frac{\ell\sqrt{d}}{\ell+\sqrt{d}}\\
& = & \frac{\ell(2d\sqrt{d}+\sqrt{d})+(d-\ell+\sqrt{d})\ell^2}{\ell+\sqrt{d}}
\eee
which implies 
\bee
A^\infty_1& = & \frac{-3\sqrt{d}\ell^2+[(d-2)\sqrt{d}-1]\ell-\sqrt{d}(d-\sqrt{d}-1)+\ell(2d\sqrt{d}+\sqrt{d})+(d-\ell+\sqrt{d})\ell^2}{\ell+\sqrt{d}}\\
& = & \frac{(d-2\sqrt{d}-\ell)\ell^2+(3d\sqrt{d}-\sqrt{d}-1)\ell-\sqrt{d}(d-\sqrt{d}-1)}{\ell+\sqrt{d}}\\
&=&-\frac{Q_d(\ell)}{\ell+\sqrt{d}}
\eee
with
$$Q_d(\ell)=\sqrt{d}(d-\sqrt{d}-1)-(3d\sqrt{d}-\sqrt{d}-1)\ell-(d-2\sqrt{d}-\ell)\ell^2.$$

\noindent\underline{Computation of $A_2$}. We have
\bea
\label{eninveioneo}
\nonumber A_2(\Phi)& =& (1-\Phi)^2(1+\Phi)+e_{21}\Phi^2-1-\Phi^3+d\Phi\\
\nonumber&=& (1-\Phi)(1-\Phi^2)+e_{21}\Phi^2-1-\Phi^3+d\Phi\\
\nonumber& = & 1-\Phi^2-\Phi+\Phi^3+e_{21}\Phi^2-1-\Phi^3+d\Phi\\
\nonumber& = & \Phi\left[(d-1)+(e_{21}-1)\Phi\right]=\Phi\left[d-1+\left(\frac{\ell+d-1}{\ell}-1\right)\Phi\right]\\
& = & \frac{(d-1)\Phi(\ell+\Phi)}{\ell}
\eea
which implies $$A^\infty_2=2\ell(d-1).$$

\noindent\underline{Discriminant} Recall $$G(\Sigma,\ell)=A_0+A_1\Sigma+A_2\Sigma^2,$$ we compute the discriminant at the critical value:
$$
{\rm Discr}_d(\ell)=(A^\infty_1)^2-4A^\infty_0A^\infty_2=\frac{Q_d^2}{(\ell+\sqrt{d})^2}-\frac{8\ell(d-1)P_d}{(\ell+\sqrt{d})^2}=\frac{Q_d^2-8\ell(d-1)P_d}{(\ell+\sqrt{d})^2}.$$ 
We collect the values
\be
\label{colelcitonvalues}
\left|\begin{array}{l}
P_d(\ell)=-d^2+d\sqrt{d}+2d+\left[d^2-d\sqrt{d}+5d-\sqrt{d}\right]\ell+(d+\sqrt{d})\ell^2\\
Q_d(\ell)=\sqrt{d}(d-\sqrt{d}-1)-(3d\sqrt{d}-\sqrt{d}-1)\ell-(d-2\sqrt{d}-\ell)\ell^2\\
{\rm test}(d,\ell)=Q_d^2-8\ell(d-1)P_d
\end{array}\right.
\ee
and evaluate numerically:\\
\noindent\underline{for $5\le d\le 12$}: we find ${\rm Discr}_d(\ell)<0$ in the range \eqref{assumtionneoneo}. It fails numerically for $d\ge 13$.\\

\noindent\underline{for $d=3$}, $\ell_1(3)<0$ and for $\sqrt{3}=1.73<\ell<3$, ${\rm Discr}_3(\ell)<0$.\\

\noindent\underline{Conclusion}. Since $A^\infty_2>0$, we conclude $G(\Sigma,\Phi)= A_0(\Phi)+A_1(\Phi)\Sigma+A_2(\Phi)\Sigma^2>0$ on $(D)$ given by \eqref{defline}. Therefore,
$$1-w-w'+F=-\frac{\Sigma G(\Sigma,\Phi)}{\Delta}>0$$ 
on $(D)$, and \eqref{tobeproveddd} is proved.\\

\noindent{\bf step 5} Proof of \eqref{defline} for $\re=r_+$.\\
\noindent\underline{Computation of $A_0$}. We have since $\Phi=c_+$ and recalling \eqref{equationcminus}, \eqref{realtionslopeweignefuncitons}:
\bee
A_0 &=&  4\sigma_2^2(1+c_+)+c_2c_++c_4-c_1c_+-c_3=-2c_4(1+c_+)+c_2c_++c_4-c_+\l_-\\
& = & -c_4-c_2c_+-2(c_4-c_2)c_+- c_+\l_-=-(1+c_+)\l_--2(c_4-c_2)c_+
\eee
and hence
$$A_0^\infty=-(1+c^\infty_+)\l^\infty_->0.$$
We compute explicitly
\bee
A_0&=&\left(1+\frac{\sqrt{\ell}(d+\sqrt{\ell})}{1+\sqrt{\ell}}\right)\frac{2(\ell+(d+1)\sqrt{\ell}+1}{(1+\sqrt{\ell})^3}\\
& = & \frac{2(\ell+(d+1)\sqrt{\ell}+1)^2}{(1+\sqrt{\ell})^4}
\eee

\noindent\underline{Computation of $A^\infty_1$}. We have $$A^\infty_1=A_{11}+A_{12}$$ with
\bee
&&A_{11}=  4\sigma^\infty_2(1-(c^\infty_+)^2)+e^\infty_{20}(c^\infty_+)^2+e^\infty_{11} c^\infty_++e^\infty_{02}\\
& = & \frac{4}{1+\sqrt{\ell}}\left[1-\frac{\ell(d+\sqrt{\ell})^2}{(1+\sqrt{\ell})^2}\right]+\frac{\ell+d-1}{\ell(1+\sqrt{\ell})}\frac{\ell(d+\sqrt{\ell})^2}{(1+\sqrt{\ell})^2}-\frac{2}{1+\sqrt{\ell}}\frac{\sqrt{\ell}(d+\sqrt{\ell})}{1+\sqrt{\ell}}-\frac{3}{1+\sqrt{\ell}}\\
& = & \frac{Q_1(\ell)}{(1+\sqrt{\ell})^3}\\
\eee
with
\bee
Q_1& = & 4\left[1+2\sqrt{\ell}+\ell-\ell(d^2+2d\sqrt{\ell}+\ell)\right]+(\ell+d-1)(d^2+2d\sqrt{\ell}+\ell)\\
&-&2(\ell+d\sqrt{\ell})(1+\sqrt{\ell})-3(1+2\sqrt{\ell}+\ell)\\
& = & \ell^2(-3)+\ell\sqrt{\ell}(-6d-2)+\ell(-3d^2-d-2)+\sqrt{\ell}(2d^2-4d+2)+d^3-d^2+1.
\eee
Then 
\bee
&&A_{12}=  -d^\infty_{20}(c^\infty_+)^2-d^\infty_{11}c^\infty_+-d^\infty_{02}\\
&=&-\frac{\ell-\sqrt{\ell}-d-1}{(1+\sqrt{\ell})^2}\frac{\ell(d+\sqrt{\ell})^2}{(1+\sqrt{\ell})^2}+\frac{2d}{1+\sqrt{\ell}}\frac{\sqrt{\ell}(d+\sqrt{\ell})}{1+\sqrt{\ell}}+\frac{\ell+d\sqrt{\ell}}{(1+\sqrt{\ell})^2}=  \frac{Q_2(\ell)}{(1+\sqrt{\ell})^4}
\eee
with
\bee
&&Q_2=(d^2+2d\sqrt{\ell}+\ell)(-\ell^2+\ell\sqrt{\ell}+(d+1)\ell)+(2d\ell+2d^2\sqrt{\ell})(1+2\sqrt{\ell}+\ell)\\
&+&(\ell+d\sqrt{\ell})(1+2\sqrt{\ell}+\ell)\\
& = & -\ell^3+(-2d+1)\ell^2\sqrt{\ell}+\ell^2(-d^2+5d+2)+\ell\sqrt{\ell}(5d^2+7d+2)\\
& + & \ell(d^3+5d^2+4d+1)+\sqrt{\ell}(d+2d^2).
\eee
Hence $$A^\infty_1=\frac{Q_3}{(1+\sqrt{\ell})^4}$$ with 
\bee
Q_3& = & \left[ \ell^2(-3)+\ell\sqrt{\ell}(-6d-2)+\ell(-3d^2-d-2)+\sqrt{\ell}(2d^2-4d+2)+d^3-d^2+1\right](1+\sqrt{\ell})\\
&-&\ell^3+(-2d+1)\ell^2\sqrt{\ell}+\ell^2(-d^2+5d+2)+\ell\sqrt{\ell}(5d^2+7d+2)\\
& + & \ell(d^3+5d^2+4d+1)+\sqrt{\ell}(d+2d^2)\\
& = & -\ell^3-(2d+2)\ell^2\sqrt{\ell}-(d^2+d+3)\ell^2+(2d^2-2)\ell\sqrt{\ell}\\
&+& (d^3+4d^2-d+1)\ell+(d^3+3d^2-3d+3)\sqrt{\ell}+d^3-d^2+1
\eee

\noindent\underline{Computation of $A_2$}. We have $A_2(\Phi)>0$ from \eqref{eninveioneo} and explicitely
\bee
A^\infty_2=\frac{(d-1)c_+(\ell+c_+)}{\ell}=\frac{d-1}{\ell}\frac{\sqrt{\ell}(d+\sqrt{\ell})}{1+\sqrt{\ell}}\left[\ell+\frac{\sqrt{\ell}(d+\sqrt{\ell})}{1+\sqrt{\ell}}\right]=\frac{(d-1)(d+\sqrt{\ell})(\ell+2\sqrt{\ell}+d)}{(1+\sqrt{\ell})^2}
\eee

\noindent\underline{Conclusion}. We are in the case $d=3$. We numerically evaluate $Q_3(\ell)$ and obtain $Q_3(\ell)<0$ for $\ell\ge 4$ in which case $A_1^\infty<0$ and since $\Sigma<0$, $G(\Sigma,\Phi)= A_0(\Phi)+A_1(\Phi)\Sigma+A_2(\Phi)\Sigma^2>0$. For $3<\ell<4$, we form the discriminant
\bee
{\rm Discr}&=&(A^\infty)_1^2-4A^\infty_0A^\infty_2=\frac{Q_3^2}{(1+\sqrt{\ell})^8}-\frac{8(\ell+(d+1)\sqrt{\ell}+1)^2}{(1+\sqrt{\ell})^4}\frac{(d-1)(d+\sqrt{\ell})(\ell+2\sqrt{\ell}+d)}{(1+\sqrt{\ell})^2}\\
& = & \frac{Q_4}{(1+\sqrt{\ell})^8}
\eee
with
$$Q_4=Q_3^2-8(d-1)(1+\sqrt{\ell})^2(\ell+(d+1)\sqrt{\ell}+1)^2(d+\sqrt{\ell})(\ell+2\sqrt{\ell}+d)$$
and numerically evaluate $Q_4(\ell)<0$ for $3<\ell<4$, and hence $G(\Sigma,\Phi)>0$. Hence
$$1-w-w'+F=-\frac{\Sigma G(\Sigma,\Phi)}{\Delta}>0$$ 
on $(D)$, and \eqref{tobeproveddd} is proved.\\

\noindent{\bf step 4} Proof of  \eqref{propertyptobeprovedbis}. {Observe that 
$$P_\sigma(w_2)=\ell(d-1)(w_2-w_*)\sigma^2-\ell(1-w_2)[1+(r-2)w_2]$$ is a second order polynomial in $\sigma$ with $P_{\sigma_2}(w_2)=0$, positive highest order coeffieicent and such that at $r^*(d,\ell)$:
$$P_0(w_2)=-\ell(1-w_2)[1+(r-2)w_2]$$ with 
\bee
&&1+(r^*(d,\ell)-2)w_2=1+\left(\frac{d+\ell}{\ell+\sqrt{d}}-2\right)\left(1-\frac{\sqrt{d}}{\ell+\sqrt{d}}\right)=1+\frac{\ell(d-\ell-2\sqrt{d})}{(\ell+\sqrt{d})^2}\\
& = & \frac{\ell^2+2\ell\sqrt{d}+d+\ell(d-\ell-2\sqrt{d})}{(\ell+\sqrt{d})^2}=\frac{\ell(d+1)}{(\ell+\sqrt{d})^2}
\eee
and
\bee
&&1+(r_+(d,\ell)-2)w_2=1+\left(1+\frac{d-1}{(1+\sqrt{\ell})^2}-2\right)\frac{\sqrt{\ell}}{1+\sqrt{\ell}}\\
& = & 1-\frac{\sqrt{\ell}}{1+\sqrt{\ell}}\left[1-\frac{d-1}{(1+\sqrt{\ell})^2}\right]=1-\frac{\sqrt{\ell}[(1+\sqrt{\ell})^2-(d-1)]}{(1+\sqrt{\ell})^3}\\
& = & \frac{1+(d+1)\sqrt{\ell}+\ell}{(1+\sqrt{\ell})^3}>0.
\eee
Hence $P_0(w_2)<0$ implies that
$$P_\sigma(w_2)<0\ \ \mbox{for}\ \ 0\leq \sigma<\sigma_2.$$}

Let $w_3(\sigma)$ parametrize the line $(w_3(\sigma),\sigma)\in(D)$, then we also have $P_\sigma(w_3(\sigma))<0$ from \eqref{tobeproveddd}. We now distinguish three cases:\\

\noindent\underline{case $(d-1)\sigma+\ell(r-2)>0$}. 
 Since $P_\sigma(w)$ is second order in $w$ with a positive highest order coefficient,  
 $$P_\sigma(w)<0\ \ \mbox{for}\ \ w_3(\sigma)\le w\le w_2$$ which, together with the 
 fact that \eqref{loweroubndphiplus} holds on the solution curve, implies that
 $$1-w-w'+F>0$$ 
 along the trajectory. The function $1-w-w'+F$ is strictly positive at $P_2$, converges to 1 as $\sigma\to 0$ and can not vanish. It implies $$1-w-w'+F\ge c>0.$$ Since, by the Remark \ref{rem:sign},  $F\le 0$ at the left of $P_2$ on the solution curve, \eqref{propertyptobeprovedbis} follows.\\
 
 \noindent\underline{case $(d-1)\sigma+\ell(r-2)< 0$}.
  {In this case, for each $0\le\sigma<\sigma_2$, $P_\sigma(w_2)<0$. On the other hand, we also have the 
 curve $\sigma_1(w)$ which, as we vary $w\in [w_2,1]$ connects $\sigma=\sigma_2$ and $\sigma=0$, and on which 
 $P_\sigma(w)=0$. Since the point $(0,1)$ belongs to this curve and lies above the line $w=w_2$, the whole curve $\sigma_1(w)$
 must lie above the line $w=w_2$. In particular, for each $0\le\sigma<\sigma_2$ we can find a (possibly non-unique) value
 $w_1(\sigma)$ such that $P_\sigma(w_1)=0$ and  $w_1(\sigma)> w_2$.
 This implies $P_\sigma(w)<0$ for $w<w_2$, thus, in particular, on the solution curve, and the conclusion follows as above.}\\
 
{ \noindent\underline{case $(d-1)\sigma+\ell(r-2)= 0$}.   Since $P_\sigma(w)$ is first order in $w$, this implies
 $$P_\sigma(w)<0\ \ \mbox{for}\ \ w_3(\sigma)\le w\le w_2$$ 
 and we conclude as above.}
  \end{proof}


\subsection{Positivity in the region where $w\geq w_2$}


 We now are in position to conclude the proof of  \eqref{propertyptobeprovedbis} and of Lemma \ref{lemma:posotivitynecessaryforcontrollinearization}. To this end, it suffices to prove the following lemma concerning the region $w\ge w_2$. 
 
 \begin{lemma}[Positivity in the region where  $\sigma_1\le \sigma\le \sigma_2$ and $ w\ge w_2$]
 \label{lemma:finallemmaonpositivityofthequadraticforms:regionwgeqw2}
Under the assumptions of Lemma \ref{lemma:posotivitynecessaryforcontrollinearization}, any $P_2-P_4$ curve with $c_-$ slope at $P_2$ satisfies  \eqref{propertyptobeprovedbis} in the region $\sigma_1\le \sigma\le \sigma_2$ and $ w\ge w_2$. 
\end{lemma}
 
 \begin{proof}
For $w\leq w_2$, $F<0$ and the conclusion follows from $1-w-w'+F>0$ which has already been established. 
Also, for $w\geq w_2$, in the region $\Delta_1\geq 0$, we have 
 \bee
\Delta(1-w-w' {-}F) &=& (1-w{-}\sigma)\Delta+\Delta_1{+}\Delta_2 >0
\eee
since $\Delta_2>0$, $\Delta>0$ and $1-w{-}\sigma>0$ on the solution curve for $\Sigma<0$. Thus, it suffices to consider the region $\Sigma<0$ before the solution curve crosses the middle root of $\Delta_1=0$. Note in particular that this region is included in  
 \bee
\{\sigma_5-\sigma_2<\Sigma<0\}\cap\{\Delta_2>0\}\cap\{\Delta_1<0\}\cap\{1-w -\sigma>0\}.
 \eee
 
 \noindent{\bf step 1} Reduction to the control of the $P_2-P_{\hskip -.1pc\peye}$ separatrix. 
  We compute
 \bee
&&\Delta(1-w-w' {-}F)=(1-w{-}\sigma)\Delta+\Delta_1{+}\Delta_2\\
&=& (1-w_2-W{-}\sigma_2{-}\Sigma)\left\{-\Sigma\left[2\sigma_2(1+\Phi)+\Sigma(1-\Phi^2)\right]\right\}\\
&+& \Sigma\left[c_1\Phi+c_3+[d_{20}\Phi^2+d_{11}\Phi+d_{02}]\Sigma+[\Phi^3-d\Phi]\Sigma^2\right]\\
&{+}& \Sigma\left[c_2\Phi+c_4+[e_{20}\Phi^2+e_{11}\Phi+e_{02}]\Sigma+[e_{21}\Phi^2-1]\Sigma^2\right]\\
& = & -\Sigma H(\Sigma,\Phi)
\eee
with
\bee
&&H(\Sigma,\Phi)={-\Sigma(1+\Phi)}[2\sigma_2(1+\Phi)+\Sigma(1-\Phi^2)]\\
&-& c_1\Phi-c_3-[d_{20}\Phi^2+d_{11}\Phi+d_{02}]\Sigma-[\Phi^3-d\Phi]\Sigma^2\\
&{-}&c_2\Phi{-}c_4{-}[e_{20}\Phi^2+e_{11}\Phi+e_{02}]\Sigma{-}[e_{21}\Phi^2-1]\Sigma^2\\
& = & {-}c_2\Phi{-}c_4-c_1\Phi-c_3\\
&+& [{-2}\sigma_2{(1+\Phi)^2}{-}e_{20}\Phi^2{-}e_{11}\Phi{-}e_{02}-d_{20}\Phi^2-d_{11}\Phi-d_{02}]\Sigma\\
& + & [{-}(1-\Phi)(1+\Phi)^{{2}}{-}e_{21}\Phi^2{+1}-\Phi^3+d\Phi]\Sigma^2.
\eee

We introduce the notation
\bee
\Phit &:=& \Phi-c_-.
\eee
We infer, using \eqref{equationcminus}, \eqref{realtionslopeweignefuncitons}:
\bee
H(\Sigma,\Phi) & = & -(c_2+c_1)\Phit -c_2c_- - c_4-c_1c_- -c_3\\
&+& [-2\sigma_2-e_{02}-d_{02} +(-4\sigma_2-e_{11}-d_{11})\Phi +(-2\sigma_2-e_{20}-d_{20})\Phi^2]\Sigma\\
& + & [ (d-1)\Phi  +(1-e_{21})\Phi^2 ]\Sigma^2\\
& = & -\l_+(1+c_-) -(c_2+c_1)\Phit \\
&+& [-2\sigma_2-e_{02}-d_{02} +(-4\sigma_2-e_{11}-d_{11})c_- +(-2\sigma_2-e_{20}-d_{20})c_-^2\\
&&+(-4\sigma_2-e_{11}-d_{11})\Phit +2c_-(-2\sigma_2-e_{20}-d_{20})\Phit+(-2\sigma_2-e_{20}-d_{20})\Phit^2]\Sigma\\
& + & [ (d-1)c_- +(1-e_{21})c_-^2 + (d-1)\Phit   +2c_-(1-e_{21})\Phit+(1-e_{21})\Phit^2 ]\Sigma^2\\
&=& F_0(\Sigma)+F_1(\Sigma)\Phit+F_2(\Sigma)\Phit^2
\eee
where
\bee
F_0(\Sigma) &:=& -\l_+(1+c_-) + [-2\sigma_2-e_{02}-d_{02} +(-4\sigma_2-e_{11}-d_{11})c_- +(-2\sigma_2-e_{20}-d_{20})c_-^2]\Sigma\\
&& + [ (d-1)c_- +(1-e_{21})c_-^2]\Sigma^2,
\eee
\bee
F_1(\Sigma) &:=& -(c_2+c_1)+[(-4\sigma_2-e_{11}-d_{11})+2c_-(-2\sigma_2-e_{20}-d_{20})]\Sigma\\
&& +[d-1   +2c_-(1-e_{21})]\Sigma^2,
\eee
\bee
F_2(\Sigma) &:=& (-2\sigma_2-e_{20}-d_{20})\Sigma+(1-e_{21})\Sigma^2.
\eee

\noindent\underline{Sign of $F_1$}. Since $-(c_2+c_1)>0$ and does not degenerate as $r\to \re$, we infer
\bee
F_1(\Sigma)=|c_1+c_2|+O(\Sigma)>0.
\eee
\noindent\underline{Sign of $F_2$}. We have  
\bea
F_2(\Sigma) = (2\sigma_2+e_{20}+d_{20})|\Sigma|(1+O(\Sigma))
\eea
In the case $\re=r^*$, we have
\bea
\label{nveoennonoene}
&& 2\sigma_2+e_{20}+d_{20} = \frac{2\sqrt{d}}{\ell+\sqrt{d}}+\frac{\sqrt{d}(d-1+\ell)}{\ell(\ell+\sqrt{d})} +\frac{-\sqrt{d}-(d-\ell)}{\ell+\sqrt{d}}\\
\nonumber &=& \frac{2\sqrt{d}\ell+\sqrt{d}(d-1+\ell)-\sqrt{d}\ell-\ell(d-\ell)}{\ell(\ell+\sqrt{d})}= \frac{\ell^2 +(2\sqrt{d}-d)\ell+\sqrt{d}(d-1)}{\ell(\ell+\sqrt{d})}
\eea
The discriminant of the second order polynomial $\ell^2 +(2\sqrt{d}-d)\ell+\sqrt{d}(d-1)$ is given by
\bee
(2\sqrt{d}-d)^2-4\sqrt{d}(d-1) &=& \sqrt{d}\Big(\sqrt{d}(\sqrt{d}-2)^2-4(d-1)\Big)\\
&=& \sqrt{d}\Big(d\sqrt{d}-8d+4\sqrt{d}+4\Big)
\eee
which is negative for $3\leq d\leq 9$, so that the second order polynomial in $\ell$ has the sign of the main term. Therefore,
 \bee
 \ell^2 +(2\sqrt{d}-d)\ell+\sqrt{d}(d-1)>0
 \eee
 which implies $F_2(\Sigma)>0$ in the case $\re=r^*$.\\
 In the case $\re=r_+$, we have
 \bee
&& 2\sigma_2+e_{20}+d_{20} = \frac{\ell^2+2\ell^{\frac{3}{2}}+2\ell-\sqrt{\ell}+d(1+\sqrt{\ell}-\ell)-1}{\ell(1+\sqrt{\ell})^2}>0
\eee
in the case $d=3$, $\ell>d$, which implies $F_2(\Sigma)>0$ in that case as well.\\

 \noindent\underline{Computation of the roots}. We have $F_0(\Sigma)=|\l_+||1+c_-|+O(\Sigma)=O(|r-r^*(\ell)|)$. Thus, the discriminant 
\bee
F_1(\Sigma)^2-4F_2(\Sigma)F_0(\Sigma) =|c_2+c_1| +O(r-r^*(\ell))>0.
\eee
Using also $F_2(\Sigma)>0$ and $F_1(\Sigma)>0$, the roots are given by 
\bee
\Phit_\pm &=& \frac{-|F_1(\Sigma)| \pm \sqrt{F_1(\Sigma)^2-4|F_2(\Sigma)|F_0(\Sigma)}}{2|F_2(\Sigma)|}.
\eee
We rewrite
\be
\label{rofftioegopih}
\Phit_+ = -\frac{2F_0(\Sigma)}{|F_1(\Sigma)| + \sqrt{F_1(\Sigma)^2-4|F_2(\Sigma)|F_0(\Sigma)}}.
\ee
\noindent\underline{Conclusion}. Since $F_2(\Sigma)>0$ and $\Phit_-\ll -1$, we have $1-w-w' -F>0$ if and only if
\bee
\Phit >\Phit_+.
\eee
Since the curve connects $P_2$ to $P_4$ with $c_-$ curve, the solution must lie strictly below the $P2-P_{\hskip -.1pc\peye}$ separatrix and hence $\Phi>\Phi_S$ and the conclusion follows from the lower bound along the separatrix curve:
\be
\label{neoieioeivnoeopeoe}
\Phi_S>c_-+\Phit_+\ , \ \ \forall u\in \left[0, \frac{3}{4}\right].
\ee
Assuming \eqref{neoieioeivnoeopeoe}, we have  
$1-w-w' -F>0$, and \eqref{propertyptobeprovedbis} is proved.\\

\noindent{\bf step 2} Proof of \eqref{neoieioeivnoeopeoe}. We reexpress \eqref{neoieioeivnoeopeoe} using the renormalization \eqref{changevariables}, $u\in [0,1]$. Let 
\be
\label{constantstobecopmuted}
\left|\begin{array}{l}
f_{01}=-2\sigma_2-e_{02}-d_{02} +(-4\sigma_2-e_{11}-d_{11})c_- +(-2\sigma_2-e_{20}-d_{20})c_-^2\\
f_{10}=-(c_2+c_1)\\
B_0=-\frac{b(1+c_-)}{c_1+c_2}\\
B_1=\frac{c_+-c_-}{\dt_{20}}\left[-\frac{\et_{20}}{\l_-}+\frac{f_{01}}{f_{10}}\right]\\
\end{array}\right.
\ee
We claim that there exist $C(d,\ell)>0$ and $B_3(d,\ell),B_4(d,\ell)$ given by \eqref{vnvemopempeioneneo} such that for all $0<b<b^*(d,\ell)\ll1 $ small enough, $\forall u\in [0,1]$, for  $\re=r^*$:
\be
\label{vnineoneonve:rstar}
\left|\Phi_S-\Phi_+-\left[B_0+bB_1u\right]\right|\le Cb^2
\ee
and  for $\re=r_+$:
\be
\label{vnineoneonve}
\left|\Phi_S-\Phi_+-\left[B_0+bB_1u+b^2(B_3u+B_4u^2)\right]\right|\le Cb^3
\ee
Then, \eqref{neoieioeivnoeopeoe} follows from the statement: $\exists c(d,\ell)>0$ , $\exists 0<b^*(d,\ell)\ll1$ such that for  $\re=r^*$:
\be
\label{tobecheceked:rstar}
\forall 0<b<b^*, \ \ \forall u\in \left[0,\frac{3}{4}\right], \ \ B_0+bB_1u>cb.
\ee
and  for $\re=r_+$:
\be
\label{tobecheceked}
\forall 0<b<b^*, \ \ \forall u\in \left[0,\frac{3}{4}\right], \ \ B_0+bB_1u+b^2(B_3u+B_4u^2)>cb^2.
\ee
The proof of \eqref{vnineoneonve} is detailed in Appendix \ref{aioenoenoneieon} together with the explicit computation of the constants $B_0,B_1,B_3,B_4$ given by \eqref{vnvemopempeioneneo} which allows us to conclude the proof of \eqref{tobecheceked}.\\
\noindent\underline{case $\re=r^*(d,\ell)$}. We have $$B_0=-\frac{b(1+c_-)}{c_1+c_2}=\frac{b|1+c^\infty_-|}{|c^\infty_1+c^\infty_2|}(1+O(b)).$$
Let $$\beta=2\sigma^\infty_2+e^\infty_{02}+d^\infty_{02} +(4\sigma^\infty_2+e^\infty_{11}+d^\infty_{11})c^\infty_- +(2\sigma^\infty_2+e^\infty_{20}+d^\infty_{20})(c^\infty_-)^2>0.$$ The inequality above follows by a direct check. Then, uniformly in $b$ small enough and $u\in [0,1]$:
\bee
&&B_0+bB_1u=b\left[\frac{1+c^\infty_-}{-(c_1^\infty+c_2^\infty)}+  \frac{c_+^\infty-c_-^\infty}{\dt_{20}^\infty}\left(-\frac{\et_{20}^\infty}{\l_-^\infty}+\frac{\beta }{c_1^\infty+c_2^\infty}\right)u\right]+O(b^2)\\
& = & b\left[\frac{|1+c^\infty_-|}{|c_1^\infty+c_2^\infty|}+  \frac{c_+^\infty-c_-^\infty}{|\dt_{20}^\infty|}\left(-\frac{\et_{20}^\infty}{|\l_-^\infty|}+\frac{\beta }{|c_1^\infty+c_2^\infty|}|\right)u\right]+O(b^2)\\
& = &  \frac{b}{|c_1^\infty+c_2^\infty|}\left[|1+c^\infty_-|+  \frac{c_+^\infty-c_-^\infty}{|\dt_{20}^\infty|}\left(|\beta|-\frac{\et_{20}^\infty|c_1^\infty+c_2^\infty|}{|\l_-^\infty|}\right)u\right]+O(b^2).
\eee
We compute:
\bee
 \frac{(c^\infty_+-c^\infty_-)}{|\dt^\infty_{20}|}\left[|\b| - \frac{|\et^\infty_{20}||c^\infty_1+c^\infty_2|}{|\l^\infty_-|}\right] &=& -\frac{(\sqrt{d}-1)^2(d-\ell)(\ell+2\sqrt{d}+d)}{2\sqrt{d}(d+\ell)(d(\sqrt{d}-1)+\ell(\sqrt{d}+1))}
\eee
and 
\bee
1+c^\infty_- &=& \frac{(\sqrt{d}-1)(d-\ell)}{d(\sqrt{d}-1)+\ell(\sqrt{d}+1)}
\eee
so that 
$$|1+c^\infty_-|+  \frac{c_+^\infty-c_-^\infty}{|\dt_{20}^\infty|}\left(|\beta|-\frac{\et_{20}^\infty|c_1^\infty+c_2^\infty|}{|\l_-^\infty|}\right)u =\frac{(\sqrt{d}-1)(d-\ell)}{2\sqrt{d}(d+\ell)(d(\sqrt{d}-1)+\ell(\sqrt{d}+1))}Q(d,\ell,u)
$$
with $$Q(d, \ell, u) = 2\sqrt{d}(d+\ell) - (\sqrt{d}-1)(\ell+2\sqrt{d}+d)u.$$
We have that $Q(d, \ell, u)$  is strictly positive on $0\leq u< u(\ell, d)$ where
\bee
u(\ell, d) &:=& \frac{2\sqrt{d}(d+\ell)}{(\sqrt{d}-1)(\ell+2\sqrt{d}+d)}.
\eee
Also, since $\ell\to  u(\ell, d)$ is increasing for $\ell\geq 0$, we have 
\bee
u(\ell, d) \geq u(0, d) = \frac{2d^{\frac{3}{2}}}{(\sqrt{d}-1)(2\sqrt{d}+d)}\textrm{ for }\ell\geq 0.
\eee
Now, $d\to u(0, d)$ attains its minimum for $d=16$ and we find
\bee
\frac{2d^{\frac{3}{2}}}{(\sqrt{d}-1)(2\sqrt{d}+d)}\geq \frac{16}{9}>1 \textrm{ for }d\geq 1.
\eee
Therefore, $u(\ell, d)>1$ for all $d\geq 1$ and thus $\inf_{u\in [0,1]}Q(d,\ell)(u)>c(d,\ell)>0$ 
which ensures $$B_0+bB_1u\ge \frac{b}{|c_1^\infty+c_2^\infty|}\left[c(d,\ell)+O(b)\right]>0$$ and \eqref{tobecheceked:rstar} is proved.\\

\noindent\underline{case $\re=r_+(d,\ell)$}. This case is degenerate since an explicit computation shows that
$$\left|\begin{array}{l}
B_0(b=0)=\frac{dB_0}{db}_{|b=0}=0\\
B_1(b=0)=0
\end{array}\right.
$$
and, as a result, we need to consider the order $b^2$. We compute, using Appendix \ref{aioenoenoneieon}:
$$
B_0+bB_1u+b^2(B_3u+B_4u^2) = b^2F(\ell, d, u)+O(b^3)
$$
with
\bee
&&F(\ell,d,u) = \frac{\sqrt{d-1}\big(1+\sqrt{\ell}\big)^5(1+\ell)}{2\ell^{\frac{3}{4}}\Big(1+(1+d)\sqrt{\ell}+\ell\Big)^2}(1-u)+ \frac{3\sqrt{d-1}\big(1+\sqrt{\ell}\big)^5(1+\ell)}{\ell^{\frac{3}{4}}\Big(1+(1+d)\sqrt{\ell}+\ell\Big)^2}u(1-u)
\eee
Then \eqref{tobecheceked} follows from
\be
\label{neioneneonvn}
\forall \ell>0, \ \ \forall d>1, \ \ \inf_{u\in [0,\frac{3}{4}]} F(\ell,d,u)>0.
\ee
which is immediate from the above formula for $F(\ell, d,u)$.
\end{proof}


\begin{appendix}



\section{Facts related to the $\Gamma$ function}


We collect various classical facts about the $\Gamma$ function.\\

\noindent\underline{Euler $\beta$ function}. 
\be
\label{defbvebeovb}
B(x,y)=\int_0^1u^{x-1}(1-u)^{y-1}du=\int_0^{+\infty}\frac{Y^{x-1}}{(1+Y)^{x+y}}dY=\frac{\Gamma(x)\Gamma(y)}{\Gamma(x+y)}.
\ee 
\bea
\label{ceoveonve}
\nonumber \int_0^{+\infty}\frac{dz}{(1+z)^{K+2}z^{\alpha_\gamma}}dz&=&\int_0^{+\infty}\frac{z^{-\alpha_\gamma}}{(1+z)^{K+2}}dz=B(1-\alpha_\gamma,K+1+\alpha_\gamma)\\
&=& \frac{\Gamma(1-\alpha_\gamma)\Gamma(K+1+\alpha_\gamma)}{\Gamma(K+2)}.
\eea

\noindent\underline{Recurrence formulas}. We compute for $k_1\leq k_2$:
$$
\Pi_{j=k_1}^{k_2}(\gamma-j)=\Pi_{j=k_1}^{k_2}(\alpha_\gamma+K+1-j)=\Pi_{\ell=K+1-k_2}^{K+1-k_1}(\alpha_j+\ell)=\frac{\Pi_{\ell=1}^{K+1-k_1}(\alpha_j+\ell)}{\Pi_{\ell=1}^{K-k_2}(\alpha_j+\ell)}.$$
We now recall 
\be
\label{neineneon}
\left|\begin{array}{l}
\Gamma(x+1)=x\Gamma(x), \ \ x\in \Bbb C\backslash {{\Bbb N}}_-\\
\ \Gamma(k+1)=k!
\end{array}\right.
\ee from which
\be
\label{cineoieneneo}
\Pi_{\ell=1}^{m}(x+\ell)=\frac{\Gamma(x+m+1)}{\Gamma(x+1)}, \ \ x\in \Bbb C\backslash {{\Bbb N}}_-
\ee
 which yields
\be
\label{pnwioqpjpne}
\Pi_{j=k_1}^{k_2}(\gamma-j)=\frac{\Gamma(\alpha_\gamma+K+2-k_1)}{\Gamma(\alpha_\gamma+K-k_2+1)}=\frac{\Gamma(\gamma+1-k_1)}{\Gamma(\gamma-k_2)}
\ee Therefore,
\be
\label{nevnonee}
\Pi_{j=0}^{K-1}(\gamma-j-2)=\Pi_{j=2}^{K+1}(\gamma-j)=\frac{\Gamma(\alpha_\gamma+K)}{\Gamma(\alpha_\gamma)}.
\ee 

\noindent\underline{Asymptotics}. We recall Stirling's formula
\be
\label{striling}
\Gamma(x+1)=(1+o_{x\to +\infty}(1))\left(\frac{x}{e}\right)^x\sqrt{2\pi x}.
\ee
Let $$|x|\lesssim 1\ll \gamma,$$ this yields:
\bea
\label{aymptoticratio}
\nonumber &&\frac{\Gamma(\gamma+x+1)}{\Gamma(\gamma+1)}=(1+o_{\gamma\to +\infty}(1))\frac{\left(\frac{\gamma+x}{e}\right)^{\gamma+x}\sqrt{2\pi (\gamma+x)}}{\left(\frac{\gamma}{e}\right)^\gamma\sqrt{2\pi \gamma}}\\
\nonumber & = & (1+o_{\gamma\to +\infty}(1))\frac{1}{e^{x}}e^{(\gamma+x)\left[\log \gamma +\frac{x}{\gamma}+O\left(\frac{1}{\gamma^2}\right)\right]-\gamma\log \gamma}\\
& = & (1+o_{\gamma\to +\infty}(1))\frac{1}{e^{x}}e^{x\log \gamma+x+O_x\left(\frac{1}{\gamma}\right)}=(1+o_{\gamma\to +\infty}(1))\gamma^x.
\eea
Moreover,
\bea
\label{econasympottic}
\nonumber &&\Gamma(\gamma+1+x)=(1+o_{\gamma\to +\infty}(1))\left(\frac{\gamma+x}{e}\right)^{\gamma+x}\sqrt{2\pi (\gamma+x)}\\
\nonumber&=&(1+o_{\gamma\to +\infty}(1))\frac{\sqrt{2\pi\gamma}}{e^{\gamma+x}}e^{(\gamma+x)\left[\log \gamma+\frac{x}{\gamma}+O_x\left(\frac{1}{\gamma^2}\right)\right]}\\
\nonumber& = & (1+o_{\gamma\to +\infty}(1))\frac{\sqrt{2\pi\gamma}}{e^{\gamma+x}}e^{\gamma\log \gamma+x+x\log \gamma+O_x\left(\frac{1}{\gamma}\right)}\\
& = & \sqrt{2\pi}(1+o_{\gamma\to +\infty}(1))\frac{\gamma^{\gamma+x+\frac 12}}{e^{\gamma}}.
\eea

\noindent\underline{Value on $\Bbb R\backslash \Bbb N_-$}.  

\begin{lemma}[Value of $\Gamma(x)$ for $x\in \Bbb R\backslash \Bbb N_-$]
Let $$x=-K_x+\alpha_x, \ \ K_x\in \Bbb N^*, \ \ 0<\alpha_x<1$$ then
\be
\label{formulafdebasenegatif}
\Gamma(x)=(-1)^{K_x}\frac{\Gamma(\alpha_x)\Gamma(1-\alpha_x)}{\Gamma(1-x)}.
\ee
\end{lemma}

\begin{proof} By definition
$$\Gamma(x)=\frac{\Gamma(x+1)}{x}=\frac{\Gamma(x+J+1)}{\Pi_{j=0}^J(x+j)}$$ and thus, with $J=K_x-1$ and using \eqref{cineoieneneo},
\bee
\Gamma(x)&=&\frac{\Gamma(-K_x+\alpha_x+K_x-1+1)}{\Pi_{j=0}^{K_x-1}\Gamma(\alpha_x-K_x+j)}=\frac{\Gamma(\alpha_x)}{\Pi_{m=1}^{K_x}(\alpha_x-m)}=(-1)^{K_x}\frac{\Gamma(\alpha_x)}{\Pi_{m=1}^{K_x}(-\alpha_x+m)}\\
& = & (-1)^{K_x}\frac{\Gamma(\alpha_x)}{\frac{\Gamma(-\alpha_x+K_x+1)}{\Gamma(-\alpha_x+1)}}=(-1)^{K_x}\frac{\Gamma(\alpha_x)\Gamma(1-\alpha_x)}{\Gamma(1-x)}.
\eee
\end{proof}


\section{Uniform convolution bounds for $b=0$}
\label{appendixconvolution}

\begin{lemma}[Lower bound on the Stirling phase]
\label{lowerouvndstilignohase}
Let $\nu>0$. {Pick a large enough universal constant  $R>R_\nu\gg 1$}. Then there exist $c_\nu>1$ and $k^*=k^*(\nu,R)\gg1 $ such that for all $k>k^*$, for all $$k_1=kx, \ \ k_2=k(1-x), \ \ 0\le x\le \frac 12,$$ 
there holds the bound 
\be
\label{cneiocneonone}
\frac{\Gamma(k_1+\nu+2)\Gamma(k_2+\nu+2)}{\Gamma(k+\nu+2)}\le c_\nu\sqrt{1+k_1}e^{-k\left[\Phi^{(2)}_k(x)\right]}
\ee
 where the phase function satisfies 
 $$\Phi_k^{(2)}(x)\ge 0, \ \ (\Phi_k^{(2)})'(x)\ge 0.$$  

Moreover, for {$\frac{R}{k}<x\leq\frac{1}{2}$}, 
\be
\label{cneovneoneo}
\Phi_k^{(2)}(x)\ge \frac 12x|\log x|.
\ee
\end{lemma}

\begin{proof}[Proof of Lemma \ref{lowerouvndstilignohase}]  We assume $k\ge k^*(R,
\nu)$ large enough and  apply Stirling's formula
$$\Gamma(x+1)=(1+o_{x\to +\infty}(1))\left(\frac{x}{e}\right)^x\sqrt{2\pi x}$$
to upper bound
\bee
&&\frac{\Gamma(k_1+\nu+2)\Gamma(k_2+\nu+2)}{\Gamma(k+\nu+2)}\\
& \lesssim_\nu &\frac{\left(\frac{k_1+\nu+1}{e}\right)^{k_1+\nu+1}\sqrt{k_1+\nu+1}\left(\frac{k_2+\nu+1}{e}\right)^{k_2+\nu+1}\sqrt{k_2+\nu+1}}{\left(\frac{k+\nu+1}{e}\right)^{k+\nu+1}\sqrt{k+\nu+1}}\\
&\lesssim_\nu& \sqrt{k_1+\nu+1}e^{-\left[(k+\nu+1)\log(k+\nu+1)-(k_1+\nu+1)\log(k_1+\nu+1)-(k_2+\nu+1)\log(k_2+\nu+1)\right]}
\eee

\noindent{\bf step 1} Study of the Stirling phase. We compute:
\bee
&&(k+\nu+1)\log(k+\nu+1)-(k_1+\nu+1)\log(k_1+\nu+1)-(k_2+\nu+1)\log(k_2+\nu+1)\\
& = & k\left(1+\frac{\nu+1}{k}\right)\left[\log k+\log\left(1+\frac{\nu+1}{ k}\right)\right]\\
& - &  k\left(x+\frac{\nu+1}{ k}\right)\left[\log k+\log\left( x+\frac{\nu+1}{ k}\right)\right]\\
&-&  k\left((1-x)+\frac{\nu+1}{ k}\right)\left[\log k+\log\left((1-x)+\frac{\nu+1}{ k}\right)\right]\\
& = & -(\nu+1)\log k+  k\left(1+\frac{\nu+1}{ k}\right)\log\left(1+\frac{\nu+1}{ k}\right)\\
&-&  k\left[\left(x+\frac{\nu+1}{ k}\right)\log\left(x+\frac{\nu+1}{ k}\right) {+} \left((1-x)+\frac{\nu+1}{ k}\right)\log\left((1-x)+\frac{\nu+1}{ k}\right)\right]\\
& = & -(\nu+1)\log(\nu+1)+(\nu+1)\log\left(\frac{\nu+1}{ k}\right)+  k\left(1+\frac{\nu+1}{ k}\right)\log\left(1+\frac{\nu+1}{ k}\right)\\
&-&  k\left[\left( x+\frac{\nu+1}{ k}\right)\log\left( x+\frac{\nu+1}{ k}\right)+\left((1-x)+\frac{\nu+1}{ k}\right)\log\left((1-x)+\frac{\nu+1}{ k}\right)\right]\\
& {=} & -(\nu+1)\log(\nu+1)+ k\Phi^{(2)}_ k(x)
\eee
which yields
\be
\label{secondstirlingphase}
 \frac{\Gamma(k_1+\nu+2)\Gamma(k_2+\nu+2)}{\Gamma(k+\nu+2)}\le c_\nu\sqrt{1+k_1}e^{-k\Phi^{(2)}_k(x)}.
\ee
We have 
$$\Phi^{(2)}_{k}(0)=\Phi^{(2)}_{k}(1)=0$$ and
\bea
\label{copmtuatinphse}
\nonumber &&\pa_x\Phi^{(2)}_{k}(x)=-\log\left(x+\frac{\nu+1}{k}\right)+\log\left((1-x)+\frac{\nu+1}{k}\right)\\
& = & \log\left(\frac{(1-x)+\frac{\nu+1}{k}}{ x+\frac{\nu+1}{k}}\right)=\log\left(1+\frac{(1-2x)}{ x+\frac{\nu+1}{k}}\right)\ge 0
\eea
for $x\in[0,\frac 12]$. We conclude that $\Phi^{(2)}_{k}(x)$ is a non negative non decreasing function of $x\in[0,\frac 12]$.\\

\noindent{\bf step 4} Small $x$ lower bound.  Assume $$\frac{R}{k}<x\leq{\frac{1}{2}}$$ 
where $R>R_\nu\gg1$, then 
\bee
\Phi_k^{(2)}(x)&=&\frac{\nu+1}{k}\log\left(\frac{\nu+1}{k}\right)+  \left(1+\frac{\nu+1}{ k}\right)\log\left(1+\frac{\nu+1}{ k}\right)\\
&-&  \left( x+\frac{\nu+1}{ k}\right)\log\left( x+\frac{\nu+1}{ k}\right)-\left((1-x)+\frac{\nu+1}{ k}\right)\log\left((1-x)+\frac{\nu+1}{ k}\right)\\
&=&  {-\frac{\nu+1}{k}\left|\log\left(1+\frac{kx}{\nu+1}\right)\right|+  \left(1+\frac{\nu+1}{ k}\right)\left|\log\left(1+\frac{\nu+1}{ k}\right)\right| +  x\left|\log\left( x+\frac{\nu+1}{ k}\right)\right|}\\
&& {+\left((1-x)+\frac{\nu+1}{ k}\right)\left|\log\left((1-x)+\frac{\nu+1}{ k}\right)\right|}\\
&\geq & {-\frac{\nu+1}{k}\left|\log\left(1+\frac{R}{\nu+1}\right)\right| +  x\left|\log x+\log\left( 1+\frac{\nu+1}{ kx}\right)\right|}\\\\
&\geq & {x|\log x|\left\{1+O_{\nu}\left(\frac{\nu+1}{R}\left|\log\left(\frac{\nu+1}{R}\right)\right|\right)\right\}}\\
&\geq & {\frac 12 x|\log x|}.
\eee
\end{proof}

\begin{lemma}[Uniform convolution bound]
\label{estimateconflutionlemma}
Let $\nu>0$, then there exists $C_\nu>0$ such that for all $k\ge 1$, for all $j\ge 2$,
\be
\label{convolutionbound}
\sum_{k_1+\dots+k_j=k}\frac{\Pi_{i=1}^{{j}}\Gamma(k_i+\nu+2)}{\Gamma(k+\nu+2)}\leq C^j_{\nu}.
\ee
\end{lemma}

\begin{proof} We argue by induction on $j$.\\

\noindent{\bf step 1} Case $j={2}$. We apply Lemma \ref{lowerouvndstilignohase} with some large enough $R_\nu>1$ and may without loss of generality assume $k\ge k^*_\nu$, with $k^*_\nu$ large enough universal constant. This leads to
$$
\frac{\Gamma(k_1+\nu+2)\Gamma(k_2+\nu+2)}{\Gamma(k+\nu+2)}\le c_\nu\sqrt{1+k_1}e^{-k\left[\Phi^{(2)}_k(x)\right]},
$$
{where $x=k_1/k$. We may assume $k_1\leq k_2$ so that $x\leq 1/2$.} For $x\le\frac{R_\nu}{k}$, we use the lower bound $\Phi_k^{(2)}(x)\ge 0$. For $x\ge\frac{R_\nu}{k}$ the monotonicity of $\Phi_{k}^{(2)}$ {on $[0,1/2]$} implies
\be
\label{ceninlsmsoenneo}
\Phi_k^{(2)}(x)\ge \Phi^{(2)}_{k}\left(\frac{R_\nu}{k}\right)\ge{+}\frac12 \frac{R_\nu}{k}\left|\log\left(\frac{R_\nu}{k}\right)\right|\ge \frac{R_\nu}{4k}\log k,
\ee
where we required that $R_\nu\leq \sqrt{k_\nu^*}$,
from which
\bee
&&\sum_{k_1+k_2=k}\frac{\Gamma(k_1+\nu+2)\Gamma(k_2+\nu+2)}{\Gamma(k+\nu+2)}\\
&\lesssim_\nu& 1+\sum_{k_1+k_2=k,  \frac{R_\nu}{k}\le x\le \frac 12}\sqrt{1+k_1}e^{-\frac{R_\nu\log k}{{4}}}+\sum_{k_1+k_2=k, 0\le x\le \frac{R_\nu}{k}}\sqrt{1+k_1}\\
& \lesssim_\nu &1+R_\nu^{{\frac{3}{2}}}+\frac{k^{\frac32}}{k^{\frac{R_\nu}{{4}}}}\lesssim_\nu 1.
\eee 
since  the second sum has $k_1=xk\leq R_\nu$ terms.\\

\noindent{\bf step 2} Induction. We assume $j$ and prove $j+1$:
\bee
&&\sum_{k_1+\dots+k_{j+1}=k}\frac{\Pi_{i=1}^{{j+1}}\Gamma(k_i+\nu+2)}{\Gamma(k+\nu+2)}\\
&=&\sum_{k_1+m=k}\frac{\Gamma(k_1+\nu+2)\Gamma(m+\nu+2)}{\Gamma(k+\nu+2)}\sum_{k_2+\dots+k_{j+1}=m}\frac{\Pi_{i=2}^{j+1}\Gamma(k_i+\nu+2)}{\Gamma(m+\nu+2)}\\
& \leq & C_{\nu}^j\sum_{k_1+m=k}\frac{\Gamma(k_1+\nu+2)\Gamma(m+\nu+2)}{\Gamma(k+\nu+2)}\leq C_\nu^{j+1}
\eee
and \eqref{convolutionbound} follows.
 \end{proof}


\section{Study of the weight $w_{\gamma,\nu_b}$ for $k\le K$}


 We derive estimates and convolution bounds for the weight $w_{\gamma,\nu_b}(k)$ given by \eqref{defweight}. We start with estimating the weight.  
 
 \begin{lemma}[Estimates on the weight]
 \label{lemmaweight}
There {exist a universal constant $c>0$} and $\gamma^*\gg 1$ such that for al $\gamma>\gamma^*$, the following holds. Let $0\le k< \gamma-1$, then:
{ \be
 \label{lowkestimate}
 \frac{w_{\gamma,\nu}(k)}{\Gamma(k+\nu+2)}\ge \frac{c}{\gamma^k}.
 \ee}
 \end{lemma}
 
 \begin{proof} 
 From \eqref{defweight}, we have
 \bee
 \frac{w_{\gamma,\nu}(k)}{\Gamma(k+\nu+2)} &=& \frac{\Gamma(\gamma-1-k)}{\Gamma(\gamma-1)}.
 \eee
 Observe
\bee
\Gamma(\gamma-1)&=&(\gamma-2)\Gamma(\gamma-2)=\left[\Pi_{j=2}^{k+1}{(\gamma-j)}\right]\Gamma(\gamma-1-k)\\
&=& \left[\Pi_{j=1}^{k}{(\gamma-1-j)}\right]\Gamma(\gamma-1-k)
\eee 
and thus
\be
\label{vebibvebeibev}
\frac{(\gamma-1)^k\Gamma(\gamma-1-k)}{\Gamma(\gamma-1)}=\Pi_{j=1}^{k}\left(\frac{\gamma-1}{\gamma {-1} -j}\right)\ge 1.
\ee
\end{proof}

\begin{lemma}[Summation bound]
\label{lemmasummationbound}
For some $c_{\nu,a}>0$, $K_\nu\gg1$ and all $0<b<b^*(\nu)$ we have the following uniform bounds:
\be
\label{summationphi}
\sum_{k=1}^{K-K_{\nu}}w_{\gamma,\nu_b}(k)\le c_{\nu,a} b.
\ee
and for ${K}-K_\nu{+1}\le k\le K{-1}$:
\be
\label{weightnenoe}
w_{\gamma,\nu_b}(k)=e^{O_\nu(1)}\Gamma(\gamma-1-k)\gamma^{\nu_b+2-(\gamma-1-k)}.
\ee
\end{lemma}

\begin{proof}  
\noindent{\bf step 1} Away from the boundary. Pick a large enough universal integer $K_\nu\gg 1$ and assume first $k\le K-K_\nu$. Then, $$\gamma-k-2=K+1+\alpha_\gamma-k-2=\alpha_\gamma+(K-k)-1\ge K_\nu-1\gg 1.$$ From Stirling's 
formula:
\bee
w_{\gamma,\nu_b}(k)&=&\frac{\Gamma(\gamma-1-k)}{\Gamma(\gamma-1)}\Gamma(k+\nu_b+2)\\
&\le& c_\nu\frac{\left(\frac{\gamma-2-k}{e}\right)^{\gamma-2-k}\left(\frac{k+\nu_b+1}{e}\right)^{k+\nu_b+1}}{\left(\frac{\gamma-2}{e}\right)^{\gamma-2}}\left(\frac{(\gamma-2-k)(k+\nu_b+1)}{\gamma-2}\right)^{\frac 12}\\
& \le & c_\nu e^{-\Phi_\gamma(k)}\left(\frac{(\gamma-2-k)(k+\nu_b+1)}{\gamma-2}\right)^{\frac 12}
\eee
with
\bee
\Phi_{\gamma}(x)=(\gamma-2)\log(\gamma-2)-(\gamma-2-{x})\log(\gamma-2-{x})-({x}+\nu_b+1)\log({x}+\nu_b+1).
\eee
We compute
\bee
\Phi_\gamma'(x)&=&\log(\gamma-2-x)+1-\log(x+\nu_b+1)-1=\log\left(\frac{\gamma-2-x}{x+\nu_b+1}\right)\\
&=&\log\left(1+\frac{\gamma-\nu_b-3-2x}{x+\nu_b+1}\right) \left|\begin{array}{l} \ge 0 \ \ \mbox{for}\  \ 1\leq x\leq \frac{\gamma-\nu_b-3}{2}\\
\le 0 \ \ \mbox{for}\  \ \frac{\gamma-\nu_b-3}{2}\le x\le K-2.
\end{array}\right.
\eee
Pick a large enough universal constant $R_\nu\gg1$, then for $1\le x<R_\nu$,
\bee
\Phi_{\gamma}(x)&=&(\gamma-2)\log(\gamma-2)-(\gamma-2)\left(1-\frac{x}{\gamma-2}\right)\left[\log(\gamma-2)+\log\left(1-\frac{x}{\gamma-2}\right)\right]\\
&-&(x+\nu_b+1)\log(x+\nu_b+1)\\
& = & \frac{x}{\gamma-2}+O\left(\frac{x^2}{\gamma-2}\right)+x\log(\gamma-2)-(x+\nu_b+1)\log(x+\nu_b+1)\\
& = & x\log (\gamma-2)\left[1+O_{R_\nu}\left(\frac{1}{\log (\gamma-2)}\right)\right]+O_{\nu}(1).
\eee
This yields for $1\le k\le R_\nu$: $\Phi_\gamma(x)\ge \log(\gamma-2)-C_\nu$ and, as a result,
$$w_{\gamma,\nu_b}(k)\le c_\nu e^{-\log(\gamma-2)}\left(\frac{(\gamma-2-k)(k+\nu+1)}{\gamma-2}\right)^{\frac 12}\le \frac{c_\nu}{\gamma}.$$ For $R_\nu<x\leq \frac{\gamma-\nu_b-3}{2}$ 
$$\Phi_\gamma(x)\ge \Phi_\gamma(R_\nu)\ge \frac{R_\nu}{2}\log(\gamma-2)$$ 
and thus,
$$w_\gamma(k)\le \frac{c_\nu}{\gamma^{\frac{R_\nu}{4}}}.$$ We obtain the bound
$$\sum_{1\le k\le \frac{\gamma-\nu_b-3}{2}}w_{\gamma,\nu_b}(k)\le c_{\nu,a}b.$$
On the other hand,
\bee
&&\Phi_\gamma(\gamma-R_\nu)=(\gamma-2)\log(\gamma-2)-(R_\nu-2)\log(R_\nu-2)\\
&-&(\gamma-R_\nu+\nu_b+1)\log(\gamma-R_\nu+\nu_b+1)\\
& = & (\gamma-2)\log(\gamma-2)-(\gamma-2)\left[1-\frac{R_\nu-\nu_b-3}{\gamma-2}\right]\left[\log (\gamma-2)-\log\left(1-\frac{R_\nu-\nu_b-3}{\gamma-2}\right)\right]\\
&+&O(R_{{\nu}})\geq  \frac{R_\nu}{4}\log (\gamma-2).
\eee
The monotonicity  of $\Phi_\gamma$ then implies $$\Phi_\gamma(k)\ge \Phi_\gamma(\gamma-R_\nu)\ge \frac{R_\nu}{4}\log (\gamma-2)\ \ \mbox{for}\ \ \frac{\gamma-\nu_b {-3}}{{2}}\le k\le \gamma-R_\nu$$ which yields 
\bee
\sum_{\frac{\gamma-\nu_b-3}{2}\le k\le \gamma-K_\nu}w_{\gamma,\nu_b}(k)\le c_{\nu,a}b
\eee
and concludes the proof of \eqref{summationphi}.\\

\noindent{\bf step 2} {F}rom Stirling's formula, for ${K}-K_\nu {+1}\le k\le K-1$:
\bee
\frac{\Gamma(k+\nu_b+2)}{\Gamma(\gamma-1)}=(1+o_\nu(1))\frac{\left(\frac{k+\nu_b+1}{e}\right)^{k+\nu_b+1}}{\left(\frac{\gamma-2}{e}\right)^{\gamma-2}}\left(\frac{k+\nu_{{b}}+1}{\gamma-2}\right)^{\frac 12}= e^{O_\nu(1)} e^{-\Phi_\gamma(k)}
\eee
with 
$$\Phi_{\gamma}(k)=(\gamma-2)\log(\gamma-2)-(k+\nu_b+1)\log(k+\nu_b+1).$$
Let $$k=\gamma-x=K+1+\alpha_\gamma-x, \ \ 2+\alpha_\gamma\le x\le K_\nu {+\alpha_\gamma},$$ 
then 
\bee
 \Phi_{\gamma}(k)&=&(\gamma-2)\log(\gamma-2)-(\gamma-x+\nu_b+1)\log(\gamma-x+\nu_b+1)\\
 & = & (\gamma-2)\log(\gamma-2)-(\gamma-2)\left[1+\frac{\nu_b+3-x}{\gamma-2}\right]\left[\log(\gamma-2)+\log\left(1+\frac{\nu_b+3-x}{\gamma-2}\right)\right]\\
 & = & (x-\nu_b-3)\log (\gamma-2)+O_\nu(1)
 \eee
 and we obtain the formula 
 \bee
 w_{\gamma,\nu_b}(k)&=&\Gamma(\gamma-1-k) \frac{e^{O_\nu(1)}}{(\gamma-2)^{x-\nu_b-3}}=e^{O_\nu(1)}\Gamma(\gamma-1-k)\gamma^{\nu_b+3-(\gamma-k)}\\\
 &=& e^{O_\nu(1)}\Gamma(\gamma-1-k)\gamma^{\nu_b+2-(\gamma-1-k)}.
 \eee
\end{proof}

We now turn to the convolution type estimate on the weight.

\begin{lemma}[Convolution estimate]
\label{lemamconvoutitio}
There exist universal constants $c_{\nu,a}>0$, $0<b^*(\nu,a)\ll 1$ such that the following holds: for all $0<b<b^*$, for all $0\le k\le K$, 
\be
\label{tobeprovoeonor}
\sum_{k_1+k_2=k}w_{\gamma,\nu}(k_1)w_{\gamma,\nu}(k_2)\leq c_{\nu,a}w_{\gamma,\nu}(k). 
\ee
More generally,
\be
\label{tobeprovoeonorbibib}
\sum_{k_1+\dots+k_j=k}w_{\gamma,\nu}(k_1)\dots w_{\gamma,\nu}(k_j)\leq c^j_{\nu,a}w_{\gamma,\nu}(k). 
\ee
\end{lemma}

\begin{proof}[Proof of Lemma \ref{lemamconvoutitio}]  We prove \eqref{tobeprovoeonor}. The proof of \eqref{tobeprovoeonorbibib} follows by induction and is left to the reader.

 We parametrize
$$\left|\begin{array}{l}
k_1=k x{,\,\,}k_2=(1-x)k
\end{array}\right.\ \ \Leftrightarrow 
\left|\begin{array}{l}
\gamma-1-k=(\gamma-1)(1-\alpha)\\
\gamma-1-k_1=(\gamma-1)(1-\alpha x)\\
\gamma-1-k_2=(\gamma-1)(1-\alpha(1-x))
\end{array}\right.
$$
For $k_1=0$ or $k_2=0$, the claim is trivial and we therefore assume without loss of generality  $$1\le k_1\le k_2<\gamma-2.$$  We recall $$\gamma-1-k=\alpha_\gamma+(K-k), \ \ 0\le k\le K.$$

 \noindent{\bf step 1} Generic case. Assume first 
 \be
 \label{condiitoino}
 k\le K-2
 \ee 
 which ensures:
 \be
 \label{vneneoneovene} 
 \gamma-2-k=\alpha_\gamma+(K-k-1)\ge 1.
 \ee

\noindent\underline{Representation formula}. We have 
\bee
\frac{w_{\gamma,\nu}(k_1)w_{\gamma,\nu}(k_2)}{w_{\gamma,\nu}(k)}=\frac{\Gamma(\gamma-1-k_1)\Gamma(\gamma-1-k_2)}{\Gamma(\gamma-1-k)\Gamma(\gamma-1)}\frac{\Gamma(k_1+\nu+2)\Gamma(k_2+\nu+2)}{\Gamma(k+\nu+2)}.
\eee 
{The function $\Gamma(\gamma-1-z)$ is holomorphic and bounded for $0\le \mathcal{R} (z)\le \gamma-1$. 
Using the three line theorem, we obtain 
\begin{align*}
&\Gamma(\gamma-1-k_1)\le \Gamma(\gamma-1)^{1-x}\Gamma(\gamma-1-k)^x,\\
&\Gamma(\gamma-1-k_2)\le \Gamma(\gamma-1)^{x}\Gamma(\gamma-1-k)^{1-x}
\end{align*}
which implies that
$$
\frac{\Gamma(\gamma-1-k_1)\Gamma(\gamma-1-k_2)}{\Gamma(\gamma-1-k)\Gamma(\gamma-1)}\le 1
$$
}
Since the $\Gamma$ function is strictly positive and bounded on $[1,R]$ for all $R>0$, we may use Stirling's formula to upper bound:
\bee
&&\frac{\Gamma(\gamma-1-k_1)\Gamma(\gamma-1-k_2)}{\Gamma(\gamma-1-k)\Gamma(\gamma-1)}\\
& \le & c_{\nu}\frac{\left(\frac{\gamma-2-k_1}{e}\right)^{\gamma-2-k_1}\sqrt{2\pi (\gamma-2-k_1)}\left(\frac{\gamma-2-k_2}{e}\right)^{\gamma-2-k_2}\sqrt{2\pi (\gamma-2-k_2)}}{\left(\frac{\gamma-2-k}{e}\right)^{\gamma-2-k}\sqrt{2\pi (\gamma-2-k)}\left(\frac{\gamma-2}{e}\right)^{\gamma-2}\sqrt{2\pi(\gamma-2)}}.
\eee
We compute the Stirling phase:
\bea
\label{defphione}
\nonumber &&(\gamma-2)\log(\gamma-2)+(\gamma-2-k)\log(\gamma-2-k)\\
\nonumber &-& (\gamma-2-k_1)\log(\gamma-2-k_1)-(\gamma-2-k_2)\log(\gamma-2-k_2)\\
\nonumber &= & (\gamma-2)\log(\gamma-2)+(\gamma-2-k)\log(\gamma-2-k)\\
\nonumber &-& (\gamma-2-kx)\log(\gamma-2-kx)-(\gamma-2-k(1-x))\log(\gamma-2-k(1-x))\\
&=& \gamma\Phi^{(1)}_{k,\gamma}(x)
\eea

\noindent\underline{Monotonicity for $\Phi_{k,\gamma}^{(1)}(x).$} We compute $$\Phi_{k,\gamma}^{(1)}(0)=0$$ and 
\bea
\label{neononvenevnoev}
\nonumber &&\left[\Phi_{k,\gamma}^{(1)}\right]'(x)=k\log(\gamma-2-kx)+k-k\log(\gamma-2-k(1-x))-k\\
& = & k\log\left(\frac{\gamma-2-xk}{\gamma-2-(1-x)k}\right)=k\log\left(1+\frac{(1-2x)k}{\gamma-2-(1-x)k}\right)>0
\eea
for $0\le x<\frac 12$. In particular, 
\be
\label{cenineneonnenoeoejo}
\Phi_{k,\gamma}^{(1)}(x)\ge 0.
\ee

\noindent\underline{case $k_1$ large}. Pick a large enough universal constant $R_\nu\gg 1$ as in Lemma \ref{lowerouvndstilignohase}, then for $x\ge \frac{R_\nu}{k}$,  from \eqref{cneiocneonone}, \eqref{ceninlsmsoenneo}:
$$\frac{\Gamma(k_1+\nu+2)\Gamma(k_2+\nu+2)}{\Gamma(k+\nu+2)}\le c_\nu\sqrt{1+k_1}e^{-k\Phi^{(2)}_k(x)}\le \frac{c_{\nu}}{k^{\frac{R_\nu}{4}}}$$ and therefore, from \eqref{cenineneonnenoeoejo}:
\bee
\frac{w_{\gamma,\nu}(k_1)w_{\gamma,\nu}(k_2)}{w_{\gamma,\nu}(k)}&=&\frac{\Gamma(\gamma-1-k_1)\Gamma(\gamma-1-k_2)}{\Gamma(\gamma-1-k)\Gamma(\gamma-1)}\frac{\Gamma(k_1+\nu+2)\Gamma(k_2+\nu+2)}{\Gamma(k+\nu+2)}\\
& \leq & \frac{c_{\nu}}{k^{\frac{R_\nu}{4}}}\left(\frac{(\gamma-2-k_1)(\gamma-2-k_2)}{(\gamma-2-k)(\gamma-2)}\right)^{\frac 12}
\eee
If $k\le\frac{\gamma-2}{2}$ then $k_1,k_2\le \frac{\gamma-2}{2}$ and $$\frac{(\gamma-2-k_1)(\gamma-2-k_2)}{(\gamma-2-k)(\gamma-2)}\lesssim 1.$$ Otherwise, $\gamma-1>k=k_1+k_2\ge 2k_1$ and, using \eqref{vneneoneovene}:
$$\left(\frac{(\gamma-2-k_1)(\gamma-2-k_2)}{(\gamma-2-k)(\gamma-2)}\right)\le\gamma-2-k_1\le \gamma\le 2k $$ 
which implies 
$$\frac{w_{\gamma,\nu}(k_1)w_{\gamma,\nu}(k_2)}{w_{\gamma,\nu}(k)}\le \frac{c_{\nu}}{k^{\frac{R_\nu}{8}}}
$$
and gives
\be
\label{cbeveeoeneone}
\sum_{k_1+k_2=k,R_\nu{\leq}k_1{\le}\frac k2}\frac{w_{\gamma,\nu}(k_1)w_{\gamma,\nu}(k_2)}{w_\gamma(k)}\leq c_{\nu}.
\ee 

\noindent\underline{case $k_1$ small}. for $x<\frac{R_\nu}{k}$ i.e. $k_1<R_\nu$, using $\Phi_{k,\gamma}^{(1)}\ge0$, $\Phi_k^{(2)}\ge0$, we have the bound 
\bee
\frac{w_{\gamma,\nu}(k_1)w_{\gamma,\nu}(k_2)}{w_{\gamma,\nu}(k)}\le c_{\nu}\sqrt{1+k_1}\left(\frac{(\gamma-2-k_1)(\gamma-2-k_2)}{(\gamma-2-k)(\gamma-2)}\right)^{\frac 12},
\eee
and using \eqref{vneneoneovene}:
$$\frac{(\gamma-2-k_1)(\gamma-2-k_2)}{(\gamma-2-k)(\gamma-2)}\le \frac{\gamma-2-k_2}{\gamma-2-k}=\frac{\gamma-2-k+k_1}{\gamma-2-k}\le 1+\frac{R_\nu}{\gamma-2-k}\le 1+R_\nu\le c_\nu$$
Since the sum has at most $k_1\le R_\nu$ terms, we obtain
\be
\label{cneoineovlmvenenoe}
\sum_{k_1+k_2=k,0\le k_1\le R_\nu}\frac{w_{\gamma,\nu}(k_1)w_{\gamma,\nu}(k_2)}{w_\gamma(k)}\leq c_{\nu}.
\ee This concludes the proof of \eqref{tobeprovoeonor} in the regime {\eqref{condiitoino}}.\\

\noindent{\bf step 2} Boundary terms. We treat separately the cases 
\be
\label{cneineonenoeov}
k\in\{K-1,K\}.
\ee 
For $k_1=1$, we have 
\bee
&&\frac{\Gamma(\gamma-1-k_1)\Gamma(\gamma-1-k_2)}{\Gamma(\gamma-1-k)\Gamma(\gamma-1)}=\frac{\Gamma(\gamma-2)\Gamma(\gamma-1-(k-1))}{\Gamma(K-k+\alpha_\gamma)\Gamma(\gamma-1)}\\
& \le & \frac{c}{\gamma}\frac{\Gamma(K-k+\alpha_\gamma+1)}{\Gamma(K-k+\alpha_\gamma)}\le c
\eee
Then, using $\Phi_k^{(2)}\ge 0$:
\bee
\frac{w_{\gamma,\nu}(k_1)w_{\gamma,\nu}(k_2)}{w_{\gamma,\nu}(k)}&=&\frac{\Gamma(\gamma-1-k_1)\Gamma(\gamma-1-k_2)}{\Gamma(\gamma-1-k)\Gamma(\gamma-1)}\frac{\Gamma(k_1+\nu+2)\Gamma(k_2+\nu+2)}{\Gamma(k+\nu+2)}\\
& \le& {c_{\nu}}\sqrt{1+k_1}e^{-k\Phi_k^{(2)}(x)}\le c_{{\nu}}.
\eee
We may therefore assume 
\be
\label{lowerobundkone}
k_1\ge 2
\ee 
and therefore,
$$\gamma-2-k_2\ge \gamma-2-(k-2)=\alpha_\gamma+(K-k)-1+2\ge 1$$ $k_1\le k_2$ implies $k_1\le \frac{k}{2}\le \frac{\gamma-1}{2}$ and $$\gamma-1-k_1\gg1.$$ Moreover, $$\Gamma(\gamma-1-k)=\Gamma((K-k)+\alpha_\gamma)\gtrsim 1$$ independent of $\gamma$. As a result, {\eqref{cneiocneonone}} yields the upper bound:
\bee
\frac{w_{\gamma,\nu}(k_1)w_{\gamma,\nu}(k_2)}{w_{\gamma,\nu}(k)}&=&\frac{\Gamma(\gamma-1-k_1)\Gamma(\gamma-1-k_2)}{\Gamma(\gamma-1-k)\Gamma(\gamma-1)}\frac{\Gamma(k_1+\nu+2)\Gamma(k_2+\nu+2)}{\Gamma(k+\nu+2)}\\
& \le& c_\nu\frac{\Gamma(\gamma-1-k_1)\Gamma(\gamma-1-k_2)}{\Gamma(\gamma-1)}\sqrt{1+k_1}e^{-k\Phi_k^{(2)}(x)}
\eee
We compute the modified Stirling phase:
\bea
\label{eicbeivbiebvebve}
&&\frac{\Gamma(\gamma-1-k_1)\Gamma(\gamma-1-k_2)}{\Gamma(\gamma-1)}\\
\nonumber &\le&  c\frac{\left(\frac{\gamma-2-k_1}{e}\right)^{\gamma-2-k_1}\sqrt{2\pi (\gamma-2-k_1)}\left(\frac{\gamma-2-k_2}{e}\right)^{\gamma-2-k_2}\sqrt{2\pi (\gamma-2-k_2)}}{\left(\frac{\gamma-2}{e}\right)^{\gamma-2}\sqrt{2\pi(\gamma-2)}}\\
\nonumber& \le & c\left(\frac{(\gamma-2-k_1)(\gamma-2-k_2)}{\gamma-2}\right)^{\frac 12}e^{-\gamma\Phi_{\gamma}^{(3)}(x)}
\eea
with
\bee
&&\gamma\Phi^{(3)}_{k,\gamma}(x)\\
 &{=}&(\gamma-2)\log(\gamma-2)- (\gamma-2-k_1)\log(\gamma-2-k_1)-(\gamma-2-k_2)\log(\gamma-2-k_2)\\
 &= & (\gamma-2)\log(\gamma-2)- (\gamma-2-kx)\log(\gamma-2-kx)-(\gamma-2-k(1-x))\log(\gamma-2-k(1-x)).
\eee
We have verbatim as above 
$$\left[\Phi^{(3)}_\gamma\right]'(x)\ge0\ \ \mbox{for}\ \ 0\le x\le \frac 12$$
and, since $x=\frac{k_1}{k}\ge \frac{2}{k}$:
\bee
&&\Phi^{(3)}_\gamma(x)\ge \Phi^{(3)}_\gamma\left(\frac{2}{k}\right)\\
& = & (\gamma-2)\log(\gamma-2)- (\gamma-2-2)\log(\gamma-2-2)-(\gamma-2-k+2))\log(\gamma-2-k+2))\\
& = & (\gamma-2)\log(\gamma-2)- (\gamma-4)\log(\gamma-4)-(K-k+\alpha_\gamma+1)\log (K-k+\alpha_\gamma+1)\ge -c
\eee
for some universal constant $C$ independent of $\gamma$ from \eqref{cneineonenoeov}. We obtain the bound
\bee
\frac{w_{\gamma,\nu}(k_1)w_{\gamma,\nu}(k_2)}{w_{\gamma,\nu}(k)}\le  c\left(\frac{(\gamma-2-k_1)(\gamma-2-k_2)}{\gamma-2}\right)^{\frac 12}{\sqrt{1+k_1}}e^{-k\Phi_k^{(2)}(x)}.
\eee

\noindent\underline{case $k_1$ large}. Arguing verbatim as before, we pick a large enough universal constant $R_\nu\gg 1$ as in Lemma \ref{lowerouvndstilignohase} and estimate for $x\ge \frac{R_\nu}{k}$,
$$\frac{w_{\gamma,\nu}(k_1)w_{\gamma,\nu}(k_2)}{w_{\gamma,\nu}(k)}\le c\left(\frac{(\gamma-2-k_1)(\gamma-2-k_2)}{\gamma-2}\right)^{\frac 12}\frac{1}{k^{\frac{R_\nu}{4}}}$$ and 
$$\frac{(\gamma-2-k_1)(\gamma-2-k_2)}{\gamma-2}\le\gamma-2 \le 2k$$ \eqref{cbeveeoeneone} follows again.\\

\noindent\underline{case $k_1$ small}. Since $k_1\le R_\nu$, {and in view of \eqref{cneineonenoeov}}:
$$
\frac{(\gamma-2-k_1)(\gamma-2-k_2)}{\gamma-2}\le \gamma-2-k_2=\gamma-2-(k-k_1)\le R_\nu+2
$$
and \eqref{cneoineovlmvenenoe} follows again.  This concludes the proof of \eqref{tobeprovoeonor}.
\end{proof}

\begin{lemma}[Truncated convolution estimate]
\label{estimatetrucnatedsseries}
For $K+1\le k\le 2K$:
\be
\label{truncatedconvolutionestimate}
\sum_{k_1+k_2=k,0\le k_1,k_2\le K}w_{\gamma,\nu_b}(k_1)w_{\gamma,\nu_b}(k_2)\le c_{\nu,a}  w_{\gamma,\nu_b}(k-K)w_{\gamma,\nu_b}(K).
\ee
\end{lemma}

\begin{proof}[Proof of Lemma \ref{estimatetrucnatedsseries}]
Let 
$$\left|\begin{array}{l}K+1\leq k\le 2K\\
k_1+k_2=k\\
0\le k_1\leq k_2\leq K
\end{array}\right.
$$ 
then 
$$k-K\le k_1\le \frac{k}2.$$ 
We compute
{\bee
\frac{\Gamma(\nu_b+2+k_1)\Gamma(\nu_b+2+k_2)}{\Gamma(k-K)\Gamma(K)} &=& \frac{\Pi_{j=0}^{k_1}(\nu_b+2+j)\Pi_{j=0}^{k_2}(\nu_b+2+j)}{\Pi_{j=0}^{K}(\nu_b+2+j)\Pi_{j=0}^{k-K}(\nu_b+2+j)}\\
&=& \Pi_{j=k-K+1}^{k_1}\left(\frac{\nu_b+2+j}{\nu_b+2+j+(K-k_1)}\right)\leq 1
\eee
since $k_1\leq K$. It follows that}
\be
\label{cenvnovnenveonenvo}
\frac{w_{\gamma,\nu_b}(k_1)w_{\gamma,\nu_b}(k_2)}{w_{\gamma,\nu_b}(k-K)w_{\gamma,\nu_b}(K)}{\leq}  \frac{\Gamma(\gamma-1-k_1)\Gamma(\gamma-1-k_2)}{\Gamma(K+\alpha_\gamma-(k-K))\Gamma(\alpha_\gamma)}
\ee
For $k=2K$, we have $k_1=k_2=K$ and from \eqref{cenvnovnenveonenvo}:
\bee
\frac{w_{\gamma,\nu_b}(k_1)w_{\gamma,\nu_b}(k_2)}{w_{\gamma,\nu_b}(k-K)w_{\gamma,\nu_b}(K)}{\le}\frac{\Gamma(\gamma-1-k_1)\Gamma(\gamma-1-k_2)}{\Gamma(\alpha_\gamma)\Gamma(\alpha_\gamma)}=\frac{\Gamma(\alpha_\gamma)\Gamma(\alpha_\gamma)}{(\Gamma(\alpha_\gamma))^2}{=} 1.
\eee
We therefore assume $$k\le 2K-1$$ and 
$$k_2\le K-1, \ \ k_1\ge k-K+1$$
Using \eqref{cenvnovnenveonenvo}:
\bee
&&\frac{\Gamma(\gamma-1-k_1)\Gamma(\gamma-1-k_2)}{\Gamma(K+\alpha_\gamma-(k-K))\Gamma(\alpha_\gamma)}=\frac{\Gamma(\alpha_\gamma+K-k_1)\Gamma(\alpha_\gamma+K-k_2)}{\Gamma(2K+\alpha_\gamma-k)\Gamma(\alpha_\gamma)}\leq  \frac{1}{\Gamma(\alpha_\gamma)}\\
& \times &\frac{\left(\frac{\alpha_\gamma+K-k_1-1}{e}\right)^{\alpha_\gamma+K-k_1-1}\left(\frac{\alpha_\gamma+K-k_2-1}{e}\right)^{\alpha_\gamma+K-k_2-1}}{\left(\frac{\alpha_\gamma+2K-k-1}{e}\right)^{\alpha_\gamma+2K-k-1}}\left(\frac{(\alpha_\gamma+K-k_1-1)(\alpha_\gamma+K-k_2-1)}{\alpha_\gamma+2K-k-1}\right)^{\frac 12}\\
& \le & \frac{c}{\Gamma(\alpha_\gamma)}\left(\frac{(\alpha_\gamma+K-k_1-1)(\alpha_\gamma+K-k_2-1)}{\alpha_\gamma+2K-k-1}\right)^{\frac 12}e^{-\Phi_{\gamma,k}^{(5)}(k_1)}
\eee
where
\bee
&&\Phi_{\gamma,k}^{(5)}(x)=(\alpha_\gamma+2K-k-1)\log(\alpha_\gamma+2K-k-1)\\
&-&(\alpha_\gamma+K-x-1)\log (\alpha_\gamma+K-x-1)-(\alpha_\gamma+K-(k-x)-1)\log (\alpha_\gamma+K-(k-x)-1).
\eee
satisfies
\bee
&&\left[\Phi_{\gamma,k}^{(5)}\right]'(x)=\log\left(\frac{\alpha_\gamma+K-x-1}{\alpha_\gamma+K-(k-x)-1}\right)\\
& = & \log\left(1+\frac{k-2x}{\alpha_\gamma+K-(k-x)-1}\right)\ge 0.
\eee
Moreover, for $A>0$ and $|y|\ll A$:
\bea
\label{tayloerinvdkoneo}
\nonumber A\log A-(A+y)\log(A+y)&=&A\log A-A\left[1+\frac yA\right]\left[\log A+\frac yA+O\left[\left(\frac{y}{A}\right)^2\right]\right]\\
& = & -y\left[\log A+1+O\left(\frac{y}{A}\right)\right]
\eea
{Pick a large enough universal constant $K_\nu\gg 1$}. For 
\be
\label{cneknvkvnenenepneo}
{x=k-K+y}, \ \   {1}\le y {\le K_\nu}
\ee
we use \eqref{tayloerinvdkoneo} to estimate:
\bea
\label{ceninveivnonevone}
\nonumber &&\Phi_{\gamma,k}^{(5)}(x)=(\alpha_\gamma+2K-k-1)\log(\alpha_\gamma+2K-k-1)\\
\nonumber &-&(\alpha_\gamma+2K-k-1-y)\log (\alpha_\gamma+2K-k-1-y)-(\alpha_\gamma-1+y)\log (\alpha_\gamma-1+y)\\
\nonumber& = &y\left[\log \left(\alpha_\gamma+2K-k-1\right) +O(1)\right] {-(\alpha_\gamma-1+y)\log (\alpha_\gamma-1+y)}\\
&\ge& \frac y2\log \gamma
\eea
Similarly,
\bee
&&\frac{\Gamma(\nu_b+k_1+2)\Gamma(\nu_b+k_2+2)}{\Gamma(k-K+\nu_b+2)\Gamma(K+\nu_b+2)}\\
& \leq & \left(\frac{(\nu_b+k_1)(\nu_b+k_2)}{(k-K+\nu_b+1)(K+\nu_b+2)}\right)^{\frac 12}e^{-\Phi_{\gamma,k}^{(6)}(k_1)}
\eee
where
\bee
&&\Phi_{\gamma,k}^{(6)}(x)=(k-K+\nu_b+1)\log(k-K+\nu_b+1)+(K+\nu_b+1)\log(K+\nu_b+1)\\
& - & (\nu_b+x+1)\log(\nu_b+x+1)-(\nu_b+k-x+1)\log(\nu_b+k-x+1)
\eee
satisfies
$$\left[\Phi_{\gamma,k}^{(6)}\right]'(x)=\log\left[\frac{\nu_b+k-x+1}{\nu_b+x+1}\right]=\log\left(1+\frac{k-2x}{\nu_b+x+1}\right) {\ge} 0.
$$
Moreover 
\bee
\nonumber &&\Phi_{\gamma,k}^{(6)}(x)=(k-K+\nu_b+1)\log(k-K+\nu_b+1)+(K+\nu_b+1)\log(K+\nu_b+1)\\
\nonumber & - & (\nu_b+x+1)\log(\nu_b+x+1)-(\nu_b+k-x+1)\log(\nu_b+k-x+1)\\
\nonumber & = & (k-K+\nu_b+1)\log(k-K+\nu_b+1)+(K+\nu_b+1)\log(K+\nu_b+1)\\
\nonumber & - & (\nu_b+k-K+1+y)\log(\nu_b+k-K+1+y)-(\nu_b+K+1-y)\log(\nu_b+K+1-y)
\eee
{Let
\bee
\phi(z)=(z+\nu_b+1)\log(z+\nu_b+1) -  (\nu_b+z+1+y)\log(\nu_b+z+1+y)
\eee
then
\bee
\phi'(z)=\log\left(\frac{z+\nu_b+1}{\nu_b+z+1+y}\right) = \log\left(1-\frac{y}{\nu_b+z+1+y}\right)<0
\eee
so that $\phi$ is deceasing. Then, $\phi(k-K)\ge \phi(K)$, which implies
\bee
\nonumber &&\Phi_{\gamma,k}^{(6)}(x) \ge  (K+\nu_b+1)\log(K+\nu_b+1)+(K+\nu_b+1)\log(K+\nu_b+1)\\
\nonumber & - & (\nu_b+K+1+y)\log(\nu_b+K+1+y)-(\nu_b+K+1-y)\log(\nu_b+K+1-y)
\eee
Thus,} in the regime \eqref{cneknvkvnenenepneo}, recalling \eqref{tayloerinvdkoneo}:
\bea\label{cneioneoinveonveovieno}
 \Phi_{\gamma,k}^{(6)}(x)  \ge  y\left[\log\left(\frac{K+\nu_b+1}{{K}+\nu_b+1}\right)+O\left(\frac{y}{K}\right)\right] =  {O\left(\frac{y^2}{K}\right)}.
\eea
{Also, we have} the bound
\bee
&&\frac{w_{\gamma,\nu_b}(k_1)w_{\gamma,\nu_b}(k_2)}{w_{\gamma,\nu_b}(k-K)w_{\gamma,\nu_b}(K)}\\
&\le& \frac{c}{\Gamma(\alpha_\gamma)}\left(\frac{(\alpha_\gamma+K-k_1-1)(\alpha_\gamma+K-k_2-1)}{\alpha_\gamma+2K-k-1}\frac{(\nu_b+k_1)(\nu_b+k_2)}{(k-K+\nu_b+1)(K+\nu_b+2)}\right)^{\frac 12}\\
&\times& e^{-(\Phi_{\gamma,k}^{(5)}(x)+\Phi_{\gamma,k}^{(6)}(x))}.
\eee
Pick a large enough universal constant $K_\nu\gg 1$ and recall $ x=k-K+y$.\\

\noindent\underline{Case $y\ge K_\nu$}. Using the monotonicity of the Stirling phases and  \eqref{ceninveivnonevone}, \eqref{cneioneoinveonveovieno} we have the lower bound:
\bee
\left[\Phi_{\gamma,k}^{(5)}+\Phi_{\gamma,k}^{(6)}\right](x)\ge \left[\Phi_{\gamma,k}^{(5)}{+}\Phi_{\gamma,k}^{(6)}\right](y=K_\nu){\ge}\frac{K_\nu}{4}\log \gamma
\eee
which implies, for $K_\nu$ universal large enough:
\bee
\sum_{k_1+k_2=k, y\ge K_\nu,0\le k_1\leq k_2\leq K}\frac{w_{\gamma,\nu_b}(k_1)w_{\gamma,\nu_b}(k_2)}{w_{\gamma,\nu_b}(k-K)w_{\gamma,\nu_b}(K)}\le c_\nu\frac{\gamma}{\gamma^{\frac{R_\nu}8}}\le c_\nu.
\eee

\noindent\underline{Case ${1}\le y\le K_\nu$}. We estimate:
\bee
&&\frac{(\alpha_\gamma+K-k_1-1)(\alpha_\gamma+K-k_2-1)}{\alpha_\gamma+2K-k-1}\frac{(\nu_b+k_1)(\nu_b+k_2)}{(k-K+\nu_b+1)(K+\nu_b+2)}\\
& = &\frac{(\alpha_\gamma+2K-k{-}y)(\alpha_\gamma+y-1)}{\alpha_\gamma+2K-k-1}\frac{(\nu_b+k-K+y)(\nu_b+K-y)}{(k-K+\nu_b+1)(K+\nu_b+2)}\\
& \leq & \frac{c_\nu}{\alpha_\gamma}
\eee
where we used $2K-k-1\ge 0$. We then conclude, using the positivity of the total Stirling phase, since there are finitely many terms:
$$\sum_{k_1+k_2=k, y{\le} K_\nu,0\le k_1\leq k_2\leq K}\frac{w_{\gamma,\nu_b}(k_1)w_{\gamma,\nu_b}(k_2)}{w_{\gamma,\nu_b}(k-K)w_{\gamma,\nu_b}(K)}\le \frac{c_\nu}{\alpha_\gamma\Gamma(\alpha_\gamma)}{\le\frac{c_\nu}{\Gamma(1+\alpha_\gamma)}}\le c_\nu
$$ 
The collection of the above bounds yields \eqref{truncatedconvolutionestimate}.
\end{proof}


\section{Uniform convolution bounds for $b=0$ with $\nu$ complex}
\label{appendixconvolution:complex}


{This appendix is the extension of Appendix \ref{appendixconvolution} to the case where $\nu$ is complex. We use the standard formulas
\bee
\log(re^{i\theta}) &=& \log(r)+i\theta,\\
\arg(x+iy) &=& 2\arctan\left(\frac{y}{\sqrt{x^2+y^2}+x}\right) \ \ \textrm{ if } \ \ x>0 \ \ \textrm{ or } \ \ y\neq 0
\eee
where we have chosen the principal branch of the  logarithm, i.e., on $\mathbb{C}\setminus\mathbb{R}^-$.}

\begin{lemma}[Lower bound on the Stirling phase]
\label{lowerouvndstilignohase:complex}
Let $\nu\in\mathbb{C}\setminus\mathbb{Z}^-$. {Pick a large enough universal constant  $R>R_\nu\gg 1$}. Then there exist $c_\nu>1$ and $k^*=k^*(\nu,R)\gg1 $ such that for all $k>k^*$, for all $$k_1=kx, \ \ k_2=k(1-x), \ \ 0\le x\le \frac 12,$$ 
we have
\be
\label{cneiocneonone:complex}
\left|\frac{\Gamma(k_1+\nu+2)\Gamma(k_2+\nu+2)}{\Gamma(k+\nu+2)}\right|\le c_\nu\sqrt{1+k_1}e^{-k\left[\widetilde{\Phi}^{(2)}_k(x)\right]}
\ee
 where the phase function satisfies\footnote{In the case where $\Re(\nu)>0$,  $\widetilde{\Phi}^{(2)}_k$ is smooth. In the case 
 $\Re(\nu)\leq 0$, $\widetilde{\Phi}^{(2)}_k$ is smooth except at the point $x=-\frac{\Re(\nu)}{k}$ where it is only continuous, see Remark \ref{rem:nonsmoothnessoftildePhi2atonepointifRenuleq0}. In particular, at that point, $(\widetilde{\Phi}^{(2)}_k)'(x)\ge 0$ means that  the left and right derivative, although not equal, are both positive.} 
 $$\widetilde{\Phi}^{(2)}_k(x)\ge 0, \ \ (\widetilde{\Phi}^{(2)}_k)'(x)\ge 0.$$  

Moreover, for {$\frac{R}{k}<x\leq\frac{1}{2}$}, 
\be
\label{cneovneoneo:complex}
\widetilde{\Phi}^{(2)}_k(x)\ge \frac 12x|\log x|.
\ee
\end{lemma}

\begin{proof}[Proof of Lemma \ref{lowerouvndstilignohase:complex}]  We assume $k\ge k^*(R,
\nu)$ large enough, and consider two cases: $k_1+\Re(\nu)\leq 0$ and $k_1+\Re(\nu)> 0$.\\

\noindent{\bf step 1} We consider first the case $k_1+\Re(\nu)\leq 0$. In this case, we have
\bee
x=\frac{k_1}{k}\leq -\frac{\Re(\nu)}{k}, \ \ \ \ \Re(\nu)\leq 0.
\eee
Note that, since $k\ge k^*(R,
\nu)$ is  large enough,  we have $k_2+\Re(\nu)>0$ and $k+\Re(\nu)>0$. 
We may thus apply Stirling's formula\footnote{In what follows, we apply Stirling's formula for $z$ satisfying $\Re(z)> 0$. In particular, $\sqrt{z}$ and $\log(z)$ are both defined.} 
$$\Gamma(z+1)=(1+o_{z\to +\infty}(1))\left(\frac{z}{e}\right)^z\sqrt{2\pi z}$$
to upper bound, using also $0\leq k_1\leq |\nu|$ and $k_1+\nu+2\in\mathbb{C}\setminus\mathbb{Z}^-$, 
\bee
&&\left|\frac{\Gamma(k_1+\nu+2)\Gamma(k_2+\nu+2)}{\Gamma(k+\nu+2)}\right|\\
& \lesssim_\nu & \left|\frac{\left(\frac{k_2+\nu+1}{e}\right)^{k_2+\nu+1}\sqrt{k_2+\nu+1}}{\left(\frac{k+\nu+1}{e}\right)^{k+\nu+1}\sqrt{k+\nu+1}}\right|\\
&\lesssim_\nu& e^{-\Re\left[(k+\nu+1)\log(k+\nu+1)-(k_2+\nu+1)\log(k_2+\nu+1)\right]}
\eee
We compute:
\bee
&&(k+\nu+1)\log(k+\nu+1)-(k_2+\nu+1)\log(k_2+\nu+1)\\
& = & k\left(1+\frac{\nu+1}{k}\right)\left[\log k+\log\left(1+\frac{\nu+1}{ k}\right)\right]\\
&-&  k\left((1-x)+\frac{\nu+1}{ k}\right)\left[\log k+\log\left((1-x)+\frac{\nu+1}{ k}\right)\right]\\
& = & kx\,\log k+  k\left(1+\frac{\nu+1}{ k}\right)\log\left(1+\frac{\nu+1}{ k}\right)\\
&-&  k \left((1-x)+\frac{\nu+1}{ k}\right)\log\left((1-x)+\frac{\nu+1}{ k}\right).
\eee
Therefore,
\bee
\Re\left[(k+\nu+1)\log(k+\nu+1)-(k_2+\nu+1)\log(k_2+\nu+1)\right] & {=} &  k\widetilde{\Phi}^{(2)}_k(x)
\eee
which yields, using also that $k_1\geq 0$,  
\be
\label{secondstirlingphase:complex:00}
 \left|\frac{\Gamma(k_1+\nu+2)\Gamma(k_2+\nu+2)}{\Gamma(k+\nu+2)}\right|\le c_\nu\sqrt{1+k_1}e^{-k\widetilde{\Phi}^{(2)}_k(x)}.
\ee
We have 
$$\widetilde{\Phi}^{(2)}_k(0)=0$$ and
\bea
\label{copmtuatinphse:complex}
\nonumber &&\pa_x\widetilde{\Phi}^{(2)}_k(x)=\log(k)+1+\Re\left[\log\left((1-x)+\frac{\nu+1}{k}\right)\right]\\
\nonumber& = & \log(k)+1+\log\left|(1-x)+\frac{\nu+1}{k}\right|\\
\nonumber&=&\log(k)+1+\log\left(\sqrt{\left(1-x+\frac{\Re(\nu)+1}{k}\right)^2+\left(\frac{\Im(\nu)}{k}\right)^2}\right)\\
&\ge &  0
\eea
for $x\in[0, -\Re(\nu)/k]$. We conclude,  in the case $k_1+\Re(\nu)\leq 0$,  that $\widetilde{\Phi}^{(2)}_k(x)$ is a non-negative non-decreasing function of $x\in[0, -\Re(\nu)/k]$. Also, note that for $k_1+\Re(\nu)\leq 0$ and provided we choose $R>R_\nu$ large enough, we have $kx=k_1\leq |\nu|<R$ so that we do not have to prove \eqref{cneovneoneo:complex} in that case. This concludes the proof of the case $k_1+\Re(\nu)\leq 0$. \\ 
 
\noindent{\bf step 2} From now on, we focus on the case $k_1+\Re(\nu)> 0$.  In this case, we have
\begin{itemize}
\item either $\Re(\nu)>0$ and $x\in [0, \frac{1}{2}]$,

\item or $\Re(\nu)\leq 0$ and $x\in (-\frac{\Re(\nu)}{k}, \frac{1}{2}]$.
\end{itemize}
Since we also have  $k_2+\Re(\nu)>0$ and $k+\Re(\nu)>0$, we may thus apply Stirling's formula to upper bound
\bee
&&\left|\frac{\Gamma(k_1+\nu+2)\Gamma(k_2+\nu+2)}{\Gamma(k+\nu+2)}\right|\\
& \lesssim_\nu & \left|\frac{\left(\frac{k_1+\nu+1}{e}\right)^{k_1+\nu+1}\sqrt{k_1+\nu+1}\left(\frac{k_2+\nu+1}{e}\right)^{k_2+\nu+1}\sqrt{k_2+\nu+1}}{\left(\frac{k+\nu+1}{e}\right)^{k+\nu+1}\sqrt{k+\nu+1}}\right|\\
&\lesssim_\nu& \sqrt{\left|k_1+\nu+1\right|}e^{-\Re\left[(k+\nu+1)\log(k+\nu+1)-(k_1+\nu+1)\log(k_1+\nu+1)-(k_2+\nu+1)\log(k_2+\nu+1)\right]}
\eee
We compute:
\bee
&&(k+\nu+1)\log(k+\nu+1)-(k_1+\nu+1)\log(k_1+\nu+1)-(k_2+\nu+1)\log(k_2+\nu+1)\\
& = & k\left(1+\frac{\nu+1}{k}\right)\left[\log k+\log\left(1+\frac{\nu+1}{ k}\right)\right]\\
& - &  k\left(x+\frac{\nu+1}{ k}\right)\left[\log k+\log\left( x+\frac{\nu+1}{ k}\right)\right]\\
&-&  k\left((1-x)+\frac{\nu+1}{ k}\right)\left[\log k+\log\left((1-x)+\frac{\nu+1}{ k}\right)\right]\\
& = & -(\nu+1)\log k+  k\left(1+\frac{\nu+1}{ k}\right)\log\left(1+\frac{\nu+1}{ k}\right)\\
&-&  k\left[\left(x+\frac{\nu+1}{ k}\right)\log\left(x+\frac{\nu+1}{ k}\right) {+} \left((1-x)+\frac{\nu+1}{ k}\right)\log\left((1-x)+\frac{\nu+1}{ k}\right)\right].
\eee
Therefore,
\begin{itemize}
\item if $\Re(\nu)>0$, then 
\bee
&&\Re\left[(k+\nu+1)\log(k+\nu+1)-(k_2+\nu+1)\log(k_2+\nu+1)\right]\\
& {=} & -\Re\left[(\nu+1)\log(\nu+1)\right]+ k\widetilde{\Phi}^{(2)}_k(x),
\eee
with the choice
\bee
\widetilde{\Phi}^{(2)}_k(x) &=& \Re\Bigg\{\frac{\nu+1}{k}\log\left(\frac{\nu+1}{ k}\right)+\left(1+\frac{\nu+1}{ k}\right)\log\left(1+\frac{\nu+1}{ k}\right)\\
&-&  \left[\left( x+\frac{\nu+1}{ k}\right)\log\left( x+\frac{\nu+1}{ k}\right)+\left((1-x)+\frac{\nu+1}{ k}\right)\log\left((1-x)+\frac{\nu+1}{ k}\right)\right]\Bigg\},
\eee

\item if $\Re(\nu)\leq 0$, then
\bee
&&\Re\left[(k+\nu+1)\log(k+\nu+1)-(k_2+\nu+1)\log(k_2+\nu+1)\right]\\
& {=} &  -\log k -  k\Re\left[\left(\frac{-\Re(\nu)+\nu+1}{ k}\right)\log\left(\frac{-\Re(\nu)+\nu+1}{ k}\right)\right]+ k\widetilde{\Phi}^{(2)}_k(x)\\
& {=} &  - \Re\Big[ \left(-\Re(\nu)+\nu+1\right)\log\left(-\Re(\nu)+\nu+1\right)\Big]+ k\widetilde{\Phi}^{(2)}_k(x)
\eee
with the choice
\bee
\widetilde{\Phi}^{(2)}_k(x) &=&  -\frac{\Re(\nu)\,\log k}{k} + \Re\left[\left(\frac{-\Re(\nu)+\nu+1}{ k}\right)\log\left(\frac{-\Re(\nu)+\nu+1}{ k}\right)\right]\\
&&+\Re\Bigg\{\left(1+\frac{\nu+1}{ k}\right)\log\left(1+\frac{\nu+1}{ k}\right)\\
&-&  \left[\left( x+\frac{\nu+1}{ k}\right)\log\left( x+\frac{\nu+1}{ k}\right)+\left((1-x)+\frac{\nu+1}{ k}\right)\log\left((1-x)+\frac{\nu+1}{ k}\right)\right]\Bigg\}.
\eee
\end{itemize}

\begin{remark}\label{rem:nonsmoothnessoftildePhi2atonepointifRenuleq0}
Note that the above choice for $\Re(\nu)\leq 0$ is such that $\widetilde{\Phi}^{(2)}_k(x)$, which is defined in that case on $x\in (-\frac{\Re(\nu)}{k}, \frac{1}{2}]$, and in Step 1 on $x\in [0, -\frac{\Re(\nu)}{k}]$, is smooth on $x\in [0,\frac{1}{2}]\setminus\{-\frac{\Re(\nu)}{k}\}$ and continuous at $x=-\frac{\Re(\nu)}{k}$.
\end{remark}

The above choice for $\widetilde{\Phi}^{(2)}_k(x)$ yields
\be
\label{secondstirlingphase:complex}
 \left|\frac{\Gamma(k_1+\nu+2)\Gamma(k_2+\nu+2)}{\Gamma(k+\nu+2)}\right|\le c_\nu\sqrt{1+k_1}e^{-k\widetilde{\Phi}^{(2)}_k(x)}.
\ee
Recall that we are in the case $k_1+\Re(\nu)>0$, i.e. $x>-\Re(\nu)/k$. We have 
\bee
\nonumber &&\pa_x\widetilde{\Phi}^{(2)}_k(x)=\Re\left[-\log\left(x+\frac{\nu+1}{k}\right)+\log\left((1-x)+\frac{\nu+1}{k}\right)\right]\\
\nonumber& = & \Re\left[\log\left(\frac{(1-x)+\frac{\nu+1}{k}}{ x+\frac{\nu+1}{k}}\right)\right]=\log\left|1+\frac{(1-2x)}{ x+\frac{\nu+1}{k}}\right|\\
\nonumber&=& \log\left(\sqrt{\left(1+\frac{(1-2x)\left(x+\frac{\Re(\nu)+1}{k}\right)}{ \left(x+\frac{\Re(\nu)+1}{k}\right)^2+\left(\frac{\Im(\nu)}{k}\right)^2}\right)^2+\left(\frac{(1-2x)\frac{\Im(\nu)}{k}}{ \left(x+\frac{\Re(\nu)+1}{k}\right)^2+\left(\frac{\Im(\nu)}{k}\right)^2}\right)^2}\right)\\
&\ge &  0
\eee
for $x\in[0,\frac 12]$ if $\Re(\nu)$>0 and for $x\in (-\frac{\Re(\nu)}{k}, \frac{1}{2}]$ if $\Re(\nu)\leq 0$. Also, we have $\widetilde{\Phi}^{(2)}_k(0)=0$ if $\Re(\nu)>0$, and $\widetilde{\Phi}^{(2)}_k(-\frac{\Re(\nu)}{k})\geq 0$ if $\Re(\nu)\leq 0$ in view of the continuity of $\widetilde{\Phi}^{(2)}_k$ at $-\frac{\Re(\nu)}{k}$ and its positivity on $x\in [0, -\frac{\Re(\nu)}{k}]$ established in step 1. Together with step 1, we conclude that $\widetilde{\Phi}^{(2)}_k(x)$ is a non-negative non-decreasing function of $x\in[0,\frac 12]$.\\

\noindent{\bf step 3} Lower bound for small $x$.  Assume $$\frac{R}{k}<x\leq{\frac{1}{2}}$$ 
where $R>R_\nu\gg1$. We start with the case $\Re(\nu)>0$ for which we have in view of the definition of $\widetilde{\Phi}^{(2)}_k$ in that case
\bee
\widetilde{\Phi}^{(2)}_k(x)&=&\Re\Bigg[\frac{\nu+1}{k}\log\left(\frac{\nu+1}{k}\right)+  \left(1+\frac{\nu+1}{ k}\right)\log\left(1+\frac{\nu+1}{ k}\right)\\
&-&  \left( x+\frac{\nu+1}{ k}\right)\log\left( x+\frac{\nu+1}{ k}\right)-\left((1-x)+\frac{\nu+1}{ k}\right)\log\left((1-x)+\frac{\nu+1}{ k}\right)\Bigg]\\
&=&  \frac{\Re(\nu)+1}{k}\log\left|\frac{\nu+1}{k}\right|+  \left(1+\frac{\Re(\nu)+1}{ k}\right)\log\left|1+\frac{\nu+1}{ k}\right|\\
&-&  \left( x+\frac{\Re(\nu)+1}{ k}\right)\log\left| x+\frac{\nu+1}{ k}\right| -\left((1-x)+\frac{\Re(\nu)+1}{ k}\right)\log\left|(1-x)+\frac{\nu+1}{ k}\right|\\
&& -\frac{\Im(\nu)}{k}\Im\Bigg[\log\left(\frac{\nu+1}{k}\right)+  \log\left(1+\frac{\nu+1}{ k}\right)-  \log\left( x+\frac{\nu+1}{ k}\right)- \log\left((1-x)+\frac{\nu+1}{ k}\right)\Bigg]\\\
&=&  -\frac{\Re(\nu)+1}{k}\log\left(\frac{\left| x+\frac{\nu+1}{ k}\right|}{\left|\frac{\nu+1}{k}\right|}\right)+  \left(1+\frac{\Re(\nu)+1}{ k}\right)\log\left|1+\frac{\nu+1}{ k}\right|\\
&-&  x\log\left| x+\frac{\nu+1}{ k}\right| -\left((1-x)+\frac{\Re(\nu)+1}{ k}\right)\log\left|(1-x)+\frac{\nu+1}{ k}\right| +O(1)\frac{\Im(\nu)}{k}
\eee
where we used in the last inequality that $|\Im(\log(z))|\leq\pi$ for $\Re(z)>0$ or $\Im(z)\neq 0$. We deduce
\bee
\widetilde{\Phi}^{(2)}_k(x) &=&  -\frac{\Re(\nu)+1}{k}\left|\log\left(\sqrt{\frac{(kx+\Re(\nu)+1)^2+(\Im(\nu))^2}{(\Re(\nu)+1)^2+(\Im(\nu))^2}}\right)\right|\\
&&+  \left(1+\frac{\Re(\nu)+1}{ k}\right)\left|\log\left(\sqrt{\left(1+\frac{\Re(\nu)+1}{k}\right)^2+\left(\frac{\Im(\nu)}{k}\right)^2}\right)\right|\\
&+&  x\left|\log\left(\sqrt{\left(x+\frac{\Re(\nu)+1}{k}\right)^2+\left(\frac{\Im(\nu)}{k}\right)^2}\right)\right|\\
&& +\left((1-x)+\frac{\Re(\nu)+1}{ k}\right)\left|\log\left(\sqrt{\left(1-x+\frac{\Re(\nu)+1}{k}\right)^2+\left(\frac{\Im(\nu)}{k}\right)^2}\right)\right| +O(1)\frac{\Im(\nu)}{k}\\
&\geq& x\Bigg[ -\frac{\Re(\nu)+1}{kx}\left|\log\left(\sqrt{\frac{(R+\Re(\nu)+1)^2+(\Im(\nu))^2}{(\Re(\nu)+1)^2+(\Im(\nu))^2}}\right)\right|\\
&+&  \left|\log(x)+\log\left(\sqrt{\left(1+\frac{\Re(\nu)+1}{kx}\right)^2+\left(\frac{\Im(\nu)}{kx}\right)^2}\right)\right|  +O\left(\frac{\Im(\nu)}{R}\right)\Bigg]\\
&\geq & {x|\log x|\left\{1+O_{\nu}\left(\frac{|\nu+1|}{R}\left|\log\left(\frac{|\nu+1|}{R}\right)\right|\right)\right\}}\\
&\geq & {\frac 12 x|\log x|}
\eee
as desired.

It remains to treat the case $\Re(\nu)\leq 0$. Recall that since $x>\frac{R}{k}$, we are in the case $k_1+\Re(\nu)> 0$. Then, in view of the definition of $\widetilde{\Phi}^{(2)}_k$ in that case, we have
\bee
\widetilde{\Phi}^{(2)}_k(x) &=&  -\frac{\Re(\nu)\,\log k}{k} + \Re\left[\left(\frac{-\Re(\nu)+\nu+1}{ k}\right)\log\left(\frac{-\Re(\nu)+\nu+1}{ k}\right)\right]\\
&&+\Re\Bigg\{\left(1+\frac{\nu+1}{ k}\right)\log\left(1+\frac{\nu+1}{ k}\right)\\
&-&  \left[\left( x+\frac{\nu+1}{ k}\right)\log\left( x+\frac{\nu+1}{ k}\right)+\left((1-x)+\frac{\nu+1}{ k}\right)\log\left((1-x)+\frac{\nu+1}{ k}\right)\right]\Bigg\}\\
&=&  -\frac{\Re(\nu)\,\log k}{k} +\frac{1}{ k}\log\left|\frac{-\Re(\nu)+\nu+1}{ k}\right| +  \left(1+\frac{\Re(\nu)+1}{ k}\right)\log\left|1+\frac{\nu+1}{ k}\right|\\
&-&  \left( x+\frac{\Re(\nu)+1}{ k}\right)\log\left| x+\frac{\nu+1}{ k}\right| -\left((1-x)+\frac{\Re(\nu)+1}{ k}\right)\log\left|(1-x)+\frac{\nu+1}{ k}\right|\\
&& -\frac{\Im(\nu)}{k}\Im\Bigg[\log\left(\frac{-\Re(\nu)+\nu+1}{ k}\right)+  \log\left(1+\frac{\nu+1}{ k}\right)\\
&&-  \log\left( x+\frac{\nu+1}{ k}\right)- \log\left((1-x)+\frac{\nu+1}{ k}\right)\Bigg]\\
&=&  -\frac{\Re(\nu)}{k}\log\left(\frac{\left| x+\frac{\nu+1}{ k}\right|}{\frac{1}{k}}\right) -\frac{1}{k}\log\left(\frac{\left| x+\frac{\nu+1}{ k}\right|}{\left|\frac{-\Re(\nu)+\nu+1}{k}\right|}\right)+  \left(1+\frac{\Re(\nu)+1}{ k}\right)\log\left|1+\frac{\nu+1}{ k}\right|\\
&-&  x\log\left| x+\frac{\nu+1}{ k}\right| -\left((1-x)+\frac{\Re(\nu)+1}{ k}\right)\log\left|(1-x)+\frac{\nu+1}{ k}\right| +O(1)\frac{\Im(\nu)}{k}
\eee
where we used in the last inequality that $|\Im(\log(z))|\leq\pi$ for $\Re(z)>0$ or $\Im(z)\neq 0$. We deduce
\bee
\widetilde{\Phi}^{(2)}_k(x) &=&  -\frac{\Re(\nu)}{k}\left|\log\left(\sqrt{(kx+\Re(\nu)+1)^2+(\Im(\nu))^2}\right)\right| -\frac{1}{k}\left|\log\left(\sqrt{\frac{(kx+\Re(\nu)+1)^2+(\Im(\nu))^2}{1+(\Im(\nu))^2}}\right)\right|
\\
&&+  \left(1+\frac{\Re(\nu)+1}{ k}\right)\left|\log\left(\sqrt{\left(1+\frac{\Re(\nu)+1}{k}\right)^2+\left(\frac{\Im(\nu)}{k}\right)^2}\right)\right|\\
&+&  x\left|\log\left(\sqrt{\left(x+\frac{\Re(\nu)+1}{k}\right)^2+\left(\frac{\Im(\nu)}{k}\right)^2}\right)\right|\\
&& +\left((1-x)+\frac{\Re(\nu)+1}{ k}\right)\left|\log\left(\sqrt{\left(1-x+\frac{\Re(\nu)+1}{k}\right)^2+\left(\frac{\Im(\nu)}{k}\right)^2}\right)\right| +O(1)\frac{\Im(\nu)}{k}\\
&\geq& x\Bigg[ -\frac{\Re(\nu)}{kx}\left|\log\left(\sqrt{(R+\Re(\nu)+1)^2+(\Im(\nu))^2}\right)\right| \\
&&-\frac{1}{kx}\left|\log\left(\sqrt{\frac{(R+\Re(\nu)+1)^2+(\Im(\nu))^2}{1+(\Im(\nu))^2}}\right)\right|\\
&+&  \left|\log(x)+\log\left(\sqrt{\left(1+\frac{\Re(\nu)+1}{kx}\right)^2+\left(\frac{\Im(\nu)}{kx}\right)^2}\right)\right|  +O\left(\frac{\Im(\nu)}{R}\right)\Bigg]\\
&\geq & {x|\log x|\left\{1+O_{\nu}\left(\frac{|\nu|+1}{R}\left|\log\left(\frac{|\nu|+1}{R}\right)\right|\right)\right\}}\\
&\geq & {\frac 12 x|\log x|}
\eee
as desired.
\end{proof}

\begin{lemma}[Uniform convolution bound]
\label{estimateconflutionlemma:complex}
Let $\nu\in\mathbb{C}\setminus\mathbb{Z}^-$. Then there exists $C_\nu>0$ such that for all $k\ge 1$, for all $j\ge 2$,
\be
\label{convolutionbound:complex}
\sum_{k_1+\dots+k_j=k}\left|\frac{\Pi_{i=1}^{{j}}\Gamma(k_i+\nu+2)}{\Gamma(k+\nu+2)}\right|\leq C^j_{\nu}.
\ee
\end{lemma}

\begin{proof} We argue by induction on $j$.\\

\noindent{\bf step 1} Case $j={2}$. We apply Lemma \ref{lowerouvndstilignohase:complex} with some large enough $R_\nu>1$ and assume, without loss of generality, $k\ge k^*_\nu$, with $k^*_\nu$ a universal, large enough constant. This leads to
$$
\left|\frac{\Gamma(k_1+\nu+2)\Gamma(k_2+\nu+2)}{\Gamma(k+\nu+2)}\right|\le c_\nu\sqrt{1+k_1}e^{-k\left[\widetilde{\Phi}^{(2)}_k(x)\right]},
$$
{where $x=k_1/k$. We may assume $k_1\leq k_2$ so that $x\leq 1/2$.} For $x\le\frac{R_\nu}{k}$, we use the lower bound $\widetilde{\Phi}^{(2)}_k(x)\ge 0$. For $x\ge\frac{R_\nu}{k}$ the monotonicity of $\widetilde{\Phi}^{(2)}_k$ {on $[0,1/2]$} implies
$$\widetilde{\Phi}^{(2)}_k(x)\ge \widetilde{\Phi}^{(2)}_k\left(\frac{R_\nu}{k}\right)\ge{+}\frac12 \frac{R_\nu}{k}\left|\log\left(\frac{R_\nu}{k}\right)\right|\ge \frac{R_\nu}{4k}\log k
$$
for $R_\nu\leq \sqrt{k_\nu^*}$, from which
\bee
&&\sum_{k_1+k_2=k}\left|\frac{\Gamma(k_1+\nu+2)\Gamma(k_2+\nu+2)}{\Gamma(k+\nu+2)}\right|\\
&\lesssim_\nu& 1+\sum_{k_1+k_2=k,  \frac{R_\nu}{k}\le x\le \frac 12}\sqrt{1+k_1}e^{-\frac{R_\nu\log k}{{4}}}+\sum_{k_1+k_2=k, 0\le x\le \frac{R_\nu}{k}}\sqrt{1+k_1}\\
& \lesssim_\nu &1+R_\nu^{{\frac{3}{2}}}+\frac{k^{\frac32}}{k^{\frac{R_\nu}{{4}}}}\lesssim_\nu 1,
\eee 
since  the second sum has $k_1=xk\leq R_\nu$ terms.\\

\noindent{\bf step 2} Induction. We assume $j$ and prove $j+1$:
\bee
&&\sum_{k_1+\dots+k_{j+1}=k}\left|\frac{\Pi_{i=1}^{{j+1}}\Gamma(k_i+\nu+2)}{\Gamma(k+\nu+2)}\right|\\
&=&\sum_{k_1+m=k}\left|\frac{\Gamma(k_1+\nu+2)\Gamma(m+\nu+2)}{\Gamma(k+\nu+2)}\right|\sum_{k_2+\dots+k_{j+1}=m}\left|\frac{\Pi_{i=2}^{j+1}\Gamma(k_i+\nu+2)}{\Gamma(m+\nu+2)}\right|\\
& \leq & C_{\nu}^j\sum_{k_1+m=k}\left|\frac{\Gamma(k_1+\nu+2)\Gamma(m+\nu+2)}{\Gamma(k+\nu+2)}\right|\leq C_\nu^{j+1}
\eee
and \eqref{convolutionbound:complex} is proved.
 \end{proof}


\section{Proof of Lemma \ref{lemmaisolated}}
\label{appendixanalyticityofSinftyinell}


This Appendix is devoted to the proof of Lemma \ref{lemmaisolated} which is a direct consequence of analyticity.
We study the function $S_\infty(d,\ell)$ from \eqref{Sinf} which arises from the 
limiting problem \eqref{limitingxequation}. We recall that there are in fact two different limiting problems which 
correspond to the parameters $r=r_*, r^+$ associated to the respective ranges $0<\ell<d$ and $\ell>d$. Each 
of the limiting problem gives rise to a collections of coefficients: $\nu_\infty, \mu_+,  c_-, , c_+,...$ 
each of which is a function of $\ell$. In fact, each of them is a different function of $\ell$ depending on the case
$r_*$ or $r^+$, see Appendix F.  Let us associate superscript $*$ to the former and $+$ to the later.
Since $r_+$ corresponds to the range $\ell>d$, the coefficients $\nu_\infty^+,...$ are originally defined for the same
range of $\ell$ but, by the direct examination of Appendix F,  can be extended {\it through the same formula} 
to the interval $\ell\in (0,d)$. In fact, they can be similarly extended as holomorphic functions to a small complex
neighborhood of $\mathbb{R}^+\setminus \{d\}$. In the $*$ case, the function $\nu_\infty^*$ is originally 
positive on the subset $\mathcal  O_d^*\subset (0,d)$. Again, by examining the formulas in Appendix F, 
we can conclude that $\nu_\infty^*,...$ are holomorphic functions of $\ell$ in a small complex neighborhood of
$\mathcal O_d^*$ (we do not need to extend them to $\ell>d$).\\

We now similarly define $S^*_\infty(d,\ell)$ and $S^+_\infty(d,\ell)$, originally on their respective sets 
$\mathcal O_d^*$ and $\ell>d$, and then argue that $S^*_\infty(d,\ell)$ is actually holomorphic in a neighborhood 
of $\mathcal O_d^*$, while for $d=2,3$ $S^+_\infty(d,\ell)$ is holomorphic in a small complex
neighborhood of $\mathbb{R}^+\setminus \{d\}$. More precisely,
\\
 \begin{lemma}
 We have
\begin{enumerate}
\item if $d\geq 4$, each function $S_\infty^*(d, \ell)$ extends holomorphically in $\ell$ to a complex neighborhood of  $\mathcal O_d^*$,
\item if $d\geq 4$, each function $S_\infty^+(d, \ell)$ extends holomorphically in $\ell$ to a complex neighborhood of  
$(d,\infty)$,

\item if $d=2, 3$, $S_\infty^*(d, \ell)$ extends holomorphically in $\ell$ to a complex neighborhood of $\{0< \ell<d\}$,

\item if $d=2, 3$, $S_\infty^+(d, \ell)$ extends holomorphically in $\ell$ to an open connected set  in $\Bbb C$ containing neighborhoods of $\{0< \ell<d\}$ and $\{d< \ell<+\infty\}$.
\end{enumerate} 
\end{lemma}
We emphasize that the extensions above are not abstract but follow from extending the values of $\ell$ to the complex 
plane in {\it explicit} formulas.

\begin{proof} 
Note from their definition that all the constants appearing in the definition of the holomorphic functions $\mu_0$, $\mu_j$ and $\nu_j$, i.e.,
\bee
&& \re, \ \ \sigma_2, \ \ \l_-, \ \ \mu_+, \ \ c_-, \ \, c_+, \ \ c_j,\ \ d_{ij}, \ \ e_{ij}, \ \ \dt_{ij}, \ \ \et_{ij}, \ \  \wte, \ \ \psite, \\
&&\Dt_{ij}, \ \ \Et_{ij}, \ \ a, \ \ \nu,
\eee
are rational functions of $\ell$, $\sqrt{\ell}$ and $\ell^{\frac{1}{4}}$, and hence are holomorphic in $\ell$ for $\ell\in\mathbb{C}\setminus\mathbb{R}^-$  wherever the denominators do not vanish. These denominators are given by the following list
\bee
\ell, \ \ \ell+\sqrt{d}, \ \ 1+\sqrt{\ell}, \ \ 1+(d+1)\sqrt{\ell}+\ell, \ \ \mu_+, \ \ c_+-c_-, \ \ \l_-, \ \ \dt_{20}. 
\eee
For real values of $\ell$, these denominators vanish at $\ell=0$, $\ell=d$, as well as at certain negative values of $\ell$ which explicitly depend on $d$. In particular, we deduce that all the above constants are holomorphic in $\ell$ in a small
neighborhood of $\mathbb{R}^+\setminus\{d\}$. This applies both to the $*$ and the $+$ case
and immediately implies that all the Taylor coefficients $(\mu_j)_k$, $(\nu_j)_k$ and $(\mu_0)_k$ of $\mu_j$, $\nu_j$ and $\mu_0$ are holomorphic in $\ell$ on the same set.\\

We now consider the set $\Omega_d^+$ obtained the intersection of the above small neighborhood of  $\mathbb{R}^+\setminus\{d\}$ with the set $(\nu_\infty^+)^{-1}\big(\mathbb{C}\setminus\mathbb{Z}^-\big)$. 
Note that since by \eqref{eq:algebricformulafornubis} the function $\nu_\infty^+>0$ for all $\ell>0$, 
the set  $\Omega_d^+$ contains $\mathbb{R}^+\setminus\{d\}$.\\

 In the $*$ case, we define the set $\Omega_d^*$ to be simply a small complex neighborhood of $\mathcal O_d^*$. 
 Note that since $\nu^*_\infty>0$ on $\mathcal O_d^*$, the condition that $\nu\in \mathbb{C}\setminus\mathbb{Z}^-$
 is automatically satisfied on $\Omega_d^*$, provided the neighborhood is chosen to be sufficiently small.\\

Our goal is now is to show that $S_\infty^*$ and $S_\infty^+$ are holomorphic respectively on $\Omega_d^*$ and 
$\Omega_d^+$ and that $\mathbb{R}^+\setminus\{d\}$ belongs to the same connected component of 
$\Omega_d^+$. The next argument applies to both, so will simply use the 
notations $S_\infty$ and $\Omega_d$.

In view of the definitions of $\Omega_d$, in particular, for any $k$, $\nu+k+3\in \mathbb{C}\setminus\mathbb{Z}^-$ and hence $\Gamma(\nu+k+3)$ is holomorphic in $\ell$ on $\Omega_d$. We deduce that  all the Taylor coefficients $(h_j)_k$, $(\th_j)_k$ and $(h_0)_k$ of $h_j$, $\th_j$ and $h_0$ are holomorphic in $\ell$ on $\Omega_d$. In view of the definition of $w_k$, we conclude that for any $k$,  $w_k$ is  holomorphic in $\ell$ on $\Omega_d$. In particular, this is also true for $(-1)^kw_k$. As a result, for any $k$, $S_k$ is holomorphic in $\ell$ on $\Omega_d$.

To conclude that $S_\infty(d, \ell)$ is holomorphic in $\ell$ on $\Omega_d$, it suffices to prove that $S_k(d, \ell)$ converges uniformly to $S_\infty(d, \ell)$ on any compact of $\Omega_d$. Note that Lemma \ref{propositionboundedness}  implies that $S_k(d, \ell)$ converges uniformly to $S_\infty(d, \ell)$ on any compact of $0<\ell<d$ in $\mathbb{R}$, so it suffices to prove that the conclusion of Lemma \ref{propositionboundedness} still holds on $\Omega_d$. Now, a quick inspection of the proof of Lemma \ref{propositionboundedness} reveals that  the only obstruction is that the conclusions  of Appendix \ref{appendixconvolution} do not hold if $\nu\in \mathbb{C}\setminus\mathbb{Z}^-$. The fact that these conclusions hold even if $\nu\in \mathbb{C}\setminus\mathbb{Z}^-$ has been checked in Appendix \ref{appendixconvolution:complex}. Thus $S_\infty(d, \ell)$ is holomorphic on $\Omega_d$. 

Finally, it remains to prove that, for $d=2, 3$, in the case $r=r_+$, $\{0< \ell<d\}$ and $\{d< \ell<+\infty\}$ belong to the same connected component of $\Omega_d^+$. First, recall from Lemma \ref{lemma:signofdt:r=rpluscase} that, for $d=2, 3$, in the case $r=r_+$, $\nu>0$ on $\{0< \ell<d\}$ and $\{d< \ell<+\infty\}$ so that $\{0< \ell<d\}$ and $\{d< \ell<+\infty\}$ belong to $\Omega_d$. It thus remains to prove that they belong to the same connected component of $\Omega_d$. To this end, $\Omega_d^+$ being open, it suffices to exhibit a continuous curve in $\Omega_d^+$ with one end on $\{0< \ell<d\}$ and the other on $\{d< \ell<+\infty\}$. Now,  from the explicit formula  \eqref{eq:algebricformulafornubis}, we have  the following asymptotic as $\ell\to d$
\bee
\nu &=& \frac{\nu_0(d)}{(\ell-d)^2}\Big(1+a_1(d)(\ell-d)+a_2(d)(\ell-d)^2+O\big((\ell-d)^3\big)\Big), \ \ \ \ \nu_0(d)>0.
\eee
We then choose the suitable curve in the complex plane as
\bee
\gamma_{r_0}=\left\{d+\sqrt{\frac{\nu_0(d)}{r_0}}e^{i\left(\frac{\pi}{2}+\theta\right)}, \ \ \ \ -\frac{\pi}{2}\leq\theta\leq \frac{\pi}{2}\right\}, \ \ \ \ r_0\gg 1. 
\eee
For $r_0$ large enough, $\gamma_{r_0}$ clearly has one end on $\{0< \ell<d\}$ and the other on $\{d< \ell<+\infty\}$ and is also a curve in $\mathbb{C}\setminus\big(\{d\}\cup\mathbb{R}^-\big)$. So, in view of the definition of $\Omega_d^+$, it suffices to prove that $\nu\notin\mathbb{Z}^-$ for any $\ell\in \gamma_{r_0}$ for a suitable choice of $r_0$. Now, in view of the above asymptotic for $\nu$, we have for $r_0$ sufficiently large
\bee
&&\nu\left(d+\sqrt{\frac{\nu_0(d)}{r_0}}e^{i\left(\frac{\pi}{2}+\theta\right)}\right)\\
 &=& -e^{-2i\theta}\Big(r_0+ia_1(d)\sqrt{\nu_0(d)}e^{i\theta}\sqrt{r_0}-a_2(d)\nu_0(d)e^{2i\theta}+O\big(r_0^{-\frac{1}{2}}\big)\Big).
\eee
Taking the imaginary part, we see that the above expression crosses $\mathbb{R}^-$ for $\theta=\theta_0$ with $\theta_0$ satisfying
\bee
r_0\sin(2\theta_0) -a_1(d)\sqrt{\nu_0(d)}\cos(\theta_0)\sqrt{r_0} &=& O\big(r_0^{-\frac{1}{2}}\big)
\eee 
and hence
\bee
\sin(\theta_0) &=& \frac{a_1(d)\sqrt{\nu_0(d)}}{2\sqrt{r_0}}+O\big(r_0^{-\frac{3}{2}}\big).
\eee
Plugging back in the above expression, we infer at $\theta=\theta_0$
\bee
&&\nu\left(d+\sqrt{\frac{\nu_0(d)}{r_0}}e^{i\left(\frac{\pi}{2}+\theta_0\right)}\right)\\
 &=& -r_0\cos(2\theta_0) -a_1(d)\sqrt{\nu_0(d)}\sin(\theta_0)\sqrt{r_0}+a_2(d)\nu_0(d)+O\big(r_0^{-\frac{1}{2}}\big)\\
 &=& -r_0 +a_2(d)\nu_0(d)+O\big(r_0^{-\frac{1}{2}}\big).
\eee
Thus, choosing 
\bee
r_0(k) &:=& k+a_2(d)\nu_0(d)-\frac{1}{2}, \ \ \ \ k\in \mathbb{N}, \ \ \ \ k\gg 1,
\eee
we obtain 
\bee
\nu\left(d+\sqrt{\frac{\nu_0(d)}{r_0(k)}}e^{i\left(\frac{\pi}{2}+\theta_0(k)\right)}\right) &=& -k+\frac{1}{2} +O\big(k^{-\frac{1}{2}}\big)
\eee
and hence $\nu\notin\mathbb{Z}^-$ for any $\ell\in \gamma_{r_0(k)}$ provided $k\in \mathbb{N}$ is chosen large enough.
\end{proof}


\section{Slopes and eigenvalues at $P_2$ near the critical value}
\label{appendixconstants}

In this appendix we collect the values of the parameters which appear in the computations  near $P_2$ for $r<\re$ close to $\re$.


\subsection{Values of the parameters}


For $1<r<r_+(d,\ell)$:
\be
\label{valueparameters}
\left|\begin{array}{l}
c_1=3w_2^2-2(r+1)w_2+r-d\sigma_2^2\\
c_2=\frac{\sigma_2}{\ell}[2w_2(\ell+d-1)-(\ell+d+\ell r-r)]\\
c_3=-2d\sigma_2w_2+2\ell(r-1)\sigma_2\\
c_4=-2\sigma_2^2\\
d_{20}=3w_2-(r+1)\\
d_{11}=-2d\sigma_2\\
d_{02}=\ell(r-1)-dw_2\\
e_{20}=\frac{\sigma_2(\ell+d-1)}{\ell}\\
e_{11}=\frac{2w_2(\ell+d-1)-(\ell+d+\ell r-r)}{\ell}\\
e_{02}=-3\sigma_2\\
e_{21}=\frac{\ell+d-1}{\ell}.
\end{array}\right.
\ee
Moreover,
\be
\label{ienveovnovne}
\left|\begin{array}{l}
\dt_{20}=(c_+e_{20}-d_{20})c_-^2+(c_+e_{11}-d_{11})c_-+c_+e_{02}-d_{02}\\
\dt_{11}= 2c_-c_+(c_+e_{20}-d_{20})+(c_-+c_+)(c_+e_{11}-d_{11})+2(c_+e_{02}-d_{02})\\
\dt_{02}=(c_+e_{20}-d_{20})c_+^2+(c_+e_{11}-d_{11})c_++c_+e_{02}-d_{02}\\
\dt_{30}=-(c_-^3-dc_-)+c_+\left(\frac{\ell+d-1}{\ell}c_-^2-1\right)\\
\dt_{21}= -(3c_-c_+^2-dc_--2dc_+)+c_+\left(\frac{\ell+d-1}{\ell}(2c_-c_++c_+^2)-3\right)\\
\dt_{12}=  -(3c_-c_+^2-dc_--2dc_+)+c_+\left(\frac{\ell+d-1}{\ell}(2c_-c_++c_+^2)-3\right)\\
\et_{20}=(d_{20}-c_-e_{20})c_-^2+(d_{11}-c_-e_{11})c_-+d_{02}-c_-e_{02}\\
\et_{11}=2c_-c_+(d_{20}-c_-e_{20})+(c_-+c_+)(d_{11}-c_-e_{11})+2(d_{02}-c_-e_{02})\\
\et_{02}=(d_{20}-c_-e_{20})c_+^2+(d_{11}-c_-e_{11})c_++d_{02}-c_-e_{02}\\
\et_{21}=(3c_-^2c_+-2dc_--dc_+)-c_-\left(\frac{\ell+d-1}{\ell}(c_-^2+2c_-c_+)-3\right).
\end{array}\right.
\ee


\subsection{Slopes and eigenvalues at $P_2$ for $r=r^*(d,\ell)$}


We compute the slopes and eigenvalues at $P_2$ for the critical speed $r=r^*(d,\ell)$, $0<\ell<d$. 

\begin{lemma}[Critical values of the slopes at $P_2$]
\label{lemmakvnonoe}
 Let $$0<\ell<d, \ \ r=r^*(d,\ell)=\frac{d+\ell}{\ell+\sqrt{d}},$$ then at $P_2$:
\be
\label{calculparametres}
\left|\begin{array}{llll}
\sigma^\infty_2=\frac{\sqrt{d}}{\ell+\sqrt{d}}\\
\l^\infty_-=-\frac{d(d-\sqrt{d})+2d+\ell(d+\sqrt{d})}{(\ell+\sqrt{d})^2}<0\\
\l^\infty_+=0\\
1+c^\infty_-=\frac{(\sqrt{d}-1)(d-\ell)}{d(\sqrt{d}-1)+\ell(\sqrt{d}+1)}\\
c^\infty_+=\ell\\
c^\infty_1=-\frac{\sqrt{d}(d(\sqrt{d}-1)+\ell(\sqrt{d}+1)}{(\ell+\sqrt{d})^2}\\
c^\infty_2=-\frac{\sqrt{d}[d(\sqrt{d}-1)+\ell(\sqrt{d}+1)]}{\ell(\ell+\sqrt{d})^2}\\
c^\infty_3=-\frac{2\ell d}{(\ell+\sqrt{d})^2}\\
c^\infty_4=-\frac{2d}{(\ell+\sqrt{d})^2}.
\end{array}\right.
\ee
Moreover, 
\be
\label{fpormoaurmalf}
\mu^\infty_+(r^*)=-\pa_r\l_+(r^*)=-\frac{4d}{-d+d^\frac{3}{2}+\ell+\sqrt{d}(2+\ell)}<0.
\ee
\end{lemma}

\begin{proof}[Proof of Lemma \ref{lemmakvnonoe}]
The proof of \eqref{calculparametres} follows from \eqref{cooridntate}, \eqref{poitnprtwppthre} which ensure $$\sigma^\infty_3<\sigma^\infty_2=\sigma^\infty_5=\frac{r^*(d,\ell)\sqrt{d}}{d+\ell}=\frac{\sqrt{d}}{\ell+\sqrt{d}}.$$ Plugging this into \eqref{defvalueci} yields \eqref{calculparametres} via a direct computation. The computation of \eqref{fpormoaurmalf} can also be done analytically but is more involved and  has been computed with  Mathematica.
\end{proof}

Observe that $r=r^*(d,\ell)$ for $0<\ell<d$  corresponds to {\em parallel slopes} at $P_2$: 
\be
\label{parallelslopes}
c^\infty_-=-\frac{c^\infty_3}{c^\infty_1}=-\frac{c^\infty_4}{c^\infty_2}.
\ee 
We observe the formulas
\be
\label{signcminus}
c^\infty_-=\frac{-2\ell\sqrt{d}}{d(\sqrt{d}-1)+\ell(\sqrt{d}+1)}
\ee 
and the algebraic relation at $r^*(d,\ell)$:
\be
\label{relationparametres}
2(\sigma^\infty_2)^2\ell=c^\infty_-(2(\sigma^\infty_2)^2+\l_-).
\ee


\subsection{Signs of $\dt_{20}$, $\et_{30}$ and $\nu$ at $r^*$}
\label{section:signofdt20et30andnu}


We compute explicitly {the sign} of the coefficients $\dt_{20}$, $\et_{30}$ and $\nu$ at $r^*(\ell)$.

\begin{lemma}
\label{lemma:signofdt}
For all $d\geq 2$ and $0<\ell< d$, we have
\be
\label{signofdt20andet20}
\dt^\infty_{20}<0, \ \ \et^\infty_{20}>0.
\ee
Also recalling \eqref{nolienartiyt} $p=1+\frac 4\ell$, \eqref{defparameterscneoevn}:
\be
\label{signofnu}
\nu_\infty(d,\ell)>0\ \ \mbox{for}\ \ \left|\begin{array}{l}
d=2, \ \ 0<\ell<2\\
d=3, \ \ 0<\ell<3\\
d=5, \ \ p\leq 10\\
d=6, \ \ p\leq 6\\
d=7, \ \ p\leq 4\\
d=8, \ \ p\leq 3\\
d=9, \ \ p\leq 3.
\end{array}\right.
\ee
\end{lemma}

\begin{proof}

\noindent{\bf step 1} Quadratic terms. We compute the values at $r^*(\ell)$ from \eqref{valueparameters}:
$$
\left|\begin{array}{l}
d^\infty_{20}=\frac{-\sqrt{d}-(d-\ell)}{\ell+\sqrt{d}}\\
d^\infty_{11}=-\frac{2d\sqrt{d}}{\ell+\sqrt{d}}\\
d^\infty_{02}=\frac{-d\sqrt{d}+(d-\ell)\sqrt{d}}{\ell+\sqrt{d}}\\
e^\infty_{20}=\frac{(2d-1)\sqrt{d}-(d-\ell)\sqrt{d}}{\ell(\ell+\sqrt{d})}\\
e^\infty_{11}=\frac{-2d\sqrt{d}+(1+\sqrt{d})(d-\ell)}{\ell(\ell+\sqrt{d})}\\
e^\infty_{02}=\frac{-3\sqrt{d}}{\sqrt{d}+\ell}.
\end{array}\right.
$$
Also, recall that we have
$$
c^\infty_+=\ell, \ \ c^\infty_-=\frac{-2\ell\sqrt{d}}{d(\sqrt{d}-1)+\ell(\sqrt{d}+1)}, \ \ \l^\infty_-=-\frac{d(d-\sqrt{d})+2d+\ell(d+\sqrt{d})}{(\ell+\sqrt{d})^2},
$$ 

\noindent{\bf step 2} Computation of $\dt^\infty_{20}$, $\et^\infty_{20}$ and $\nu_\infty$. Recall \eqref{ienveovnovne} which together with \eqref{calculparametres}, \eqref{signcminus} yields at $r^*(\ell)$:
\bee
\dt^\infty_{20} &=& -\frac{4\ell d(\sqrt{d}-1)(d-\ell)(d+\ell)}{(\ell+\sqrt{d})(d(\sqrt{d}-1)+\ell(1+\sqrt {d}))^2},\\
\et ^\infty_{20}&=& -\frac{\ell(\sqrt{d}-1)\sqrt{d}(1+\sqrt{d})(\ell+d)\Big(\ell^2(\sqrt{d}-1)-2\ell\sqrt{d}(1+\sqrt{d})^2+d^\frac{3}{2}(-2+5\sqrt{d}-3d)\Big)}{(\ell+\sqrt{d})(d(\sqrt{d}-1)+\ell(1+\sqrt {d}))^3}.
\eee
We now compute from \eqref{defparameterscneoevn}, \eqref{defdtij}, \eqref{nioneinevioohve}:
\bea
\label{eq:algebricformulafornu}
\nonumber \nu_\infty& =& -\gamma b(\Dt_{11}+\Dt_{30}-\Et_{11})\\
\nonumber & = &  -\frac{1}{(\dt_{20})^2}\Big[\et_{20}\dt_{11}+(c_+-c_-)|\l_-|\dt_{30}-\et_{11}|\dt_{20}|\Big]\\
\nonumber &=& -(1+\sqrt{d})\Big(2(\sqrt{d}-1)(\ell+d)d(d-\ell)^2\Big)^{-1}\\
\nonumber&&\times\Big\{\ell^4(\sqrt{d}-1)^2-\ell^3d(1+\sqrt{d})^2-4\ell^2d(1-\sqrt{d}+2d -d^\frac{3}{2}+d^2)\\
&&+\ell d^2(4-12\sqrt{d}+3d+2d^\frac{3}{2}-d^2)+(\sqrt{d}-1)^2d^3(d-4)\Big\}.
\eea

\noindent{\bf step 3} Sign of $\dt^\infty_{20}$, $\et^\infty_{20}$ and $\nu_\infty$. The sign $\dt^\infty_{20}<0$ follows from its formula. Concerning $\et^\infty_{20}$, we see from its formula that $-\et^\infty_{20}$ has the same sign as
\bee
\mathcal{E}_{20} &:=& \ell^2(\sqrt{d}-1)-2\ell\sqrt{d}(1+\sqrt{d})^2+d^\frac{3}{2}(-2+5\sqrt{d}-3d).
\eee
Now, we have
\bee
\mathcal{E}_{20}'(\ell) &=& 2(\sqrt{d}-1)\left(\ell - \frac{\sqrt{d}(1+\sqrt{d})^2}{\sqrt{d}-1}\right)
\eee
and since for $d\geq 2$ we have
\bee
\frac{\sqrt{d}(1+\sqrt{d})^2}{\sqrt{d}-1}> d,
\eee
we infer for $\ell\leq d$ that $\mathcal{E}_{20}$ is decreasing and hence for $0\leq \ell\leq d$,
\bee
\mathcal{E}_{20}(\ell)\leq \mathcal{E}_{20}(0)=-d^\frac{3}{2}(2-5\sqrt{d}+3d)=-3d^\frac{3}{2}(\sqrt{d}-1)\left(\sqrt{d}-\frac{2}{3}\right)<0
\eee
so that $\et_{20}>0$ for all $d\geq 2$ and $0\leq \ell\leq d$.\\
Concerning $\nu$, we see from its formula that $\nu$ has the opposite sign to 
\bee
N_0 &:=& \ell^4(\sqrt{d}-1)^2-\ell^3d(1+\sqrt{d})^2-4\ell^2d(1-\sqrt{d}+2d -d^\frac{3}{2}+d^2)\\
&&+\ell d^2(4-12\sqrt{d}+3d+2d^\frac{3}{2}-d^2)+(\sqrt{d}-1)^2d^3(d-4).
\eee
We then check the sign of $N_0$ numerically and confirm that $\nu>0$ for the above claimed range.
\end{proof}


\subsection{Slopes and eigenvalues at $P_2$ for $r=r_+(d,\ell)$}


\begin{lemma}[Critical values of the slopes at $P_2$]
\label{lemmakvnonoebis}
 Let $$\ell>d, \ \ r=r_+(d,\ell)=1+\frac{d-1}{(1+\sqrt{\ell})^2},$$ then at $P_2$:
 \be
\label{calculparametresbis}
\left|\begin{array}{l}
\sigma^\infty_2=\frac{1}{1+\sqrt{\ell}}\\
w^\infty_2=\frac{\sqrt{\ell}}{1+\sqrt{\ell}}\\
c^\infty_1=-\frac{2\sqrt{\ell}(d+\sqrt{\ell})}{(1+\sqrt{\ell})^3}\\
c^\infty_2=-\frac{2}{(1+\sqrt{\ell})^2}\\
c^\infty_3=-\frac{2\sqrt{\ell}(d+\sqrt{\ell})}{(1+\sqrt{\ell})^3}\\
c^\infty_4=-\frac{2}{(1+\sqrt{\ell})^2}\\
c^\infty_-=-1\\
c^\infty_+=\frac{\sqrt{\ell}(d+\sqrt{\ell})}{1+\sqrt{\ell}}\\
\l^\infty_+=0\\
\l^\infty_-=-\frac{2\left[\ell+(d+1)\sqrt{\ell}+1\right]}{(1+\sqrt{\ell})^3}.
\end{array}\right.
\ee
Moreover,
\be
\label{fpormoaurmalfbis}
\mu^\infty_+=(\pa_b\l_+)_{b=0}=-\frac{2\sqrt{d-1}(\ell-d)}{(1+\sqrt{\ell})\ell^{\frac{1}{4}}(1+(d+1)\sqrt{\ell}+\ell)}<0.
\ee
\end{lemma}

\begin{proof} This follows from a direct computation. \eqref{fpormoaurmalfbis} has been computed with Mathematica.
\end{proof}


\subsection{Signs of $\dt_{20}$, $\et_{30}$ and $\nu$ at $r_+$}
\label{section:signofdt20et30andnubis}


We compute explicitly {the sign} of the coefficients $\dt_{20}$, $\et_{30}$ and $\nu$ at $r^*(\ell)$.

\begin{lemma}\label{lemma:signofdt:r=rpluscase}
For all $d\geq 2$ and $\ell>d$, we have
\be\label{signofdt20andet20bis}
\dt^\infty_{20}<0, \ \ \et^\infty_{20}>0, \ \ \nu_\infty>0.
\ee
\end{lemma}

\begin{proof} We collect the values
\be
\label{valuedijbis}
\left|\begin{array}{l}
d^\infty_{20}=\frac{\ell-\sqrt{\ell}-d-1}{(1+\sqrt{\ell})^2}\\
d^\infty_{11}=\frac{-2d}{1+\sqrt{\ell}}\\
d^\infty_{02}=\frac{-\ell-d\sqrt{\ell}}{(1+\sqrt{\ell})^2}\\
e^\infty_{20}=\frac{\ell+d-1}{\ell(1+\sqrt{\ell})}\\
e^\infty_{11}=\frac{-2}{1+\sqrt{\ell}}\\
e^\infty_{02}=-\frac{3}{1+\sqrt{\ell}}\\
e^\infty_{21}=\frac{\ell+d-1}{\ell}.
\end{array}\right.
\ee

We compute directly
$$
\left|\begin{array}{ll}
\dt^\infty_{20}=-\frac{(d-1)(\ell-d)}{\sqrt{\ell}(1+\sqrt{\ell})^2}\\
\et^\infty_{20}=\frac{(d-1)(\ell+1)}{\ell(1+\sqrt{\ell})}.
\end{array}\right.
$$
and 
\bea
\label{eq:algebricformulafornubis}
\nonumber \nu_\infty &=& \frac{2}{(d-1)(\ell-d)^2}\Bigg[d^3(\ell+2)+d^2(\ell^2+2\ell^{\frac{3}{2}}-2\ell+2\sqrt{\ell}-2)+d(2\ell^{\frac{5}{2}}+5\ell^2+4\ell^{\frac{3}{2}}+7\ell+2\sqrt{\ell}+1)\\
&&+\ell(\ell^2+2\ell^{\frac{3}{2}}+\ell+2\sqrt{\ell}+1)\Bigg]>0
\eea
and the claim is proved. We note here that this function $\nu_\infty$ is positive for {\it every} value of 
$\ell>0$.
\end{proof}


\section{Expansion of the functionals in \eqref{equaitoncompee}}
\label{apneidincoien}

In this Appendix we collect all the formulas for all the terms appearing in \eqref{equaitoncompee}. We recall $x=bu$.\\

\noindent\underline{$\Dt_{ij},\Et_{ij}$ coefficients}
\be
\label{defdtij}
\left|\begin{array}{l}
\Dt_{20}=\frac{\dt_{20}\wte}{|\mu_+|(c_+-c_-)}\\
\Dt_{11}=\frac{\dt_{11}\psite\wte}{|\mu_+|(c_+-c_-)}\\
\Dt_{02}=\frac{\dt_{02}\psite^2\wte}{|\mu_+|(c_+-c_-)}\\
\Dt_{30}=\frac{\dt_{30}\wte^2}{|\mu_+|(c_+-c_-)}\\
\Dt_{21}=\frac{\dt_{21}\psite\wte^2}{|\mu_+|(c_+-c_-)}\\
\Dt_{12}=\frac{\dt_{12}\psite^2\wte^2}{|\mu_+|(c_+-c_-)}\\
\Dt_{03}=\frac{\dt_{03}\psite^3\wte^2}{|\mu_+|(c_+-c_-)}
\end{array}\right., \ \ 
\left|\begin{array}{l}
\Et_{20}=\frac{\et_{20}\wte}{(c_+-c_-)|\l_-|\psite}\\
\Et_{11}=\frac{\et_{11}\wte}{(c_+-c_-)|\l_-|}\\
\Et_{02}=\frac{\et_{02}\psite\wte}{(c_+-c_-)|\l_-|}\\
\Et_{30}=\frac{\et_{30}\wte^2}{(c_+-c_-)|\l_-|\psite}\\
\Et_{21}=\frac{\et_{21}\wte^2}{(c_+-c_-)|\l_-|}\\
\Et_{12}=\frac{\et_{12}\psite\wte^2}{(c_+-c_-)|\l_-|}\\
\Et_{03}=\frac{\et_{03}\psite^2\wte^2}{(c_+-c_-)|\l_-|}
\end{array}\right.
\ee

\begin{remark}
\label{limitnigncoie}
 Note that the coefficients $\Dt_{ij},\Et_{ij}$ have a well defined limit $\Dt^\infty_{ij},\Et^\infty_{ij}$ as $b\to 0$ from \eqref{valuelimitsfhihs}.
\end{remark}

\noindent\underline{Polynomials $H_1,H_2$}
\be
\label{decopmositionhtwo}
\left|\begin{array}{l}
H_1(b,u)=\sum_{j=0}^3b^jH_{1,j}(x)\\
H_2(b,u)=\sum_{j=0}^3 b^jH_{2,j}(x)
\end{array}\right.
\ee
with
\be
\label{hijformulas}
\left|\begin{array}{l}
H_{1,0}(x)=-(\Et_{11}+\Et_{30})+(\Et_{02}+\Et_{21})x-\Et_{12}x^2+\Et_{03}x^3\\
H_{1,1}(x)=(\Et_{02}+\Et_{21})-\Et_{12}x+\Et_{03}x^2\\
H_{1,2}(x)=-\Et_{12}+\Et_{03}x\\
H_{1,3}(x)=\Et_{03},\\
H_{2,0}(x)=(\Dt_{11}+\Dt_{30})x-(\Dt_{02}+\Dt_{21})x^2+\Dt_{12}x^3-\Dt_{03}x^4\\
H_{2,1}(x)=-(\Dt_{02}+\Dt_{21})x+\Dt_{12}x^2-\Dt_{03}x^3\\
H_{2,2}(x)=\Dt_{12}x-\Dt_{03}x^2\\
H_{2,3}=-\Dt_{03}x.
\end{array}\right.
\ee
\noindent\underline{Polynomials $G_1,G_2$}
\be
\label{polynomialsgi}
\left|\begin{array}{l}
G_1(b,u)=\Et_{11}x-(2\Et_{02}+\Et_{21})x^2+2\Et_{12}x^3-3\Et_{03}x^4\\
G_2(b,u)=-\Dt_{11}x+(2\Dt_{02}+\Dt_{21})x^2-2\Dt_{12}x^3+3\Dt_{03}x^4
\end{array}\right.
\ee
\noindent\underline{Nonlinear terms}
\be
\label{nolinearterms}
\left|\begin{array}{l}
\NLt_1=\sum_{j=0}^2b^j\NLt_{1j}\\
\NLt_2=\sum_{j=0}^2b^j\NLt_{2j}
\end{array}\right.
\ee
with
$$\left|\begin{array}{l}
\NLt_{10}=-xM_{11}\Psi^2+x^2M_{12}\Psi^3\\
\NLt_{11}=M_{11}\Psi^2-2xM_{12}\Psi^3\\
\NLt_{12}=M_{12}\Psi^3
\end{array}\right., \ \  \left|\begin{array}{l}
M_{11}= -\Et_{02}x+\Et_{12}x^2-3\Et_{03}x^3\\
M_{12}=\Et_{03}x^2
\end{array}\right.
$$
$$\left|\begin{array}{l}
\NLt_{20}=-xM_{21}\Psi^2+x^2M_{22}\Psi^3\\
\NLt_{21}=M_{21}\Psi^2-2xM_{22}\Psi^3\\
\NLt_{22}=M_{22}\Psi^3
\end{array}\right., \ \ \left|\begin{array}{l}
M_{21}=\Dt_{02}x-\Dt_{12}x^2+3\Dt_{03}x^3\\
M_{22}=\Dt_{03}x^2.
\end{array}\right.
$$


\section{Proof of \eqref{vnineoneonve:rstar} and \eqref{vnineoneonve}}
\label{aioenoenoneieon}


This section is devoted to the derivation of the estimates \eqref{vnineoneonve:rstar} and \eqref{vnineoneonve} by computing both the separatrix $\Phi_S$ and the root $\Phi_+$ in the variables or the renormalization \eqref{changevariables}.

\subsection{Computation of $\Phi_S$}

\begin{lemma}[Computation the separatrix $\Phi_S$ in u]
\label{ionwineioenoive}
Let 
\be
\label{valuecoeficcients}
\left|\begin{array}{l}
a=\gamma b=\frac{|\l_-|}{|\mu_+|}\\
\Theta_0=-\Et_{11}-\Et_{30}-\frac 2a\\
a_1=-\frac{(c_+-c_-)\mu_+}{\dt_{20}}\\
a_2=-\frac{1}{a_1}\frac{1}{\dt_{20}}\left(-\frac{\dt_{11}\et_{20}}{(c_+-c_-)\l_-}+\dt_{30}\right)\left(\frac{(c_+-c_-)\mu_+}{\dt_{20}}\right)^2\\
g_1= \frac{a_1\et_{20}}{(c_+-c_-)\l_-}=-\frac{\et_{20}\mu_+}{\l_-\dt_{20}}\\
g_2= \frac{a_1\et_{20}}{(c_+-c_-)\l_-}\left[a_2+\left(\frac{\et_{30}}{\et_{20}}-\frac{\et_{11}}{(c_+-c_-)\l_-}\right)a_1\right].
\end{array}\right.
\ee
then 
\bea
\label{firovhineioeohis}
\Phi_S& = & c_--b\left[(c_+-c_-)g_1u\right]\\
\nonumber & + & b^2\left\{-(c_+-c_-)(g_2+g_1\Theta_0)u+\left[(c_+-c_-)g_1\Theta_0-\left(c_+-c_-\right)g_1^2\right]u^2\right\}+O(b^3)
\eea
where the $O(b^3)$ is uniform in $0<b<b^*\ll 1$ small enough and $u\in [0,1]$.
\end{lemma}

\begin{proof}[Proof of Lemma \ref{ionwineioenoive}]

\noindent{\bf step 1} Unfolding the change of variables. Recall:
$$
\left|\begin{array}{l}
\Wt=-b\wt\\
\Sigmat=b^2\sigmat
\end{array}\right., \ \ 
\left|\begin{array}{l}
\psit=\frac{\sigmat}{\wt}=\psite \phi\\
\wt=\wte u\\
 \end{array}\right., \ \
$$
and
$$\left|\begin{array}{l}W\\\Sigma\end{array}\right.=P\left|\begin{array}{l}\Wt\\\Sigmat\end{array}\right.=\left(\begin{array}{ll} c_-&c_+\\ 1&1\end{array}\right)\left|\begin{array}{l}\Wt\\\Sigmat\end{array}\right.=\left|\begin{array}{l} c_-\Wt+c_+\Sigmat \\ \Wt+\Sigmat.\end{array}\right.$$
Then,
\be
\label{vneioneneenovne}
\left|\begin{array}{l}
\Wt=-b \wt=-b\wte u\\
\Sigmat=b^2\sigmat=b^2\wte\psite u\phi
\end{array}\right., \ \ \left|\begin{array}{l}
W=-c_-b\wte u+c_+b^2\wte\psite u\phi\\
\Sigma=-b\wte u+b^2\wte\psite u\phi
\end{array}\right.
\ee
which yields $$\Phi=\frac{W}{\Sigma}=\frac{-c_-b\wte u+c_+b^2\wte \psite u\phi}{-b\wte u+b^2\wte\psite u\phi}=\frac{c_--bc_+\psite \phi}{1-b\psite \phi}.$$
For $0\le u\le 1$, recalling \eqref{estimateseprpapt}, \eqref{suihfoenioeneoi}, \eqref{defphi}, \eqref{changevariables}, \eqref{deflkenrnal}, \eqref{vnovoeonven}: 
\be
\label{vneineoineon}
\left|\begin{array}{l}
\phi=u+(1-u)\Psi=u+bM_bu(1-u)\Theta\\
M_b=1+O(b)\\
0\le u\le \frac 34
\end{array}\right.
\ee
such that for the separatrix $$|\Theta_S|\leq c(d,\ell)$$ uniformly as $b\to 0 $ from \eqref{estimateseprpapt}.
Therefore, we may Taylor expand uniformly in $b$ and $u\in [0,1]$, to obtain for the separatrix curve:
\bea
\label{vnevnneoenenvn}
\nonumber \Phi&=&(c_--bc_+\psite \phi)(1+b\psite\phi+b^2\psite^2\phi^2+O(b^3))\\
&=& c_--b\psite(c_+-c_-)\phi-b^2\phi^2\psite^2\left(c_+-c_-\right)+O(b^3).
\eea

\noindent{\bf step 2} Computation of $\psite$. Recall 
\eqref{fundvonnnrelation}:
$$
\left|\begin{array}{l}
\Wte=-b\frac{(c_+-c_-)\mu_+}{\dt_{20}}+O(b^2)\\
\Sigmate=-\frac{\et_{20}\Wte^2}{(c_+-c_-)\l_-}+O(b^3)
\end{array}\right.
$$
and
\bee
\nonumber \mathcal G_1&=&(c_+-c_-)\l_+\Wt+\dt_{20}\Wt^2+\dt_{11}\Wt\Sigmat+\dt_{02}\Sigmat^2+\dt_{30}\Wt^3+\dt_{21}\Wt^2\Sigmat+\dt_{12}\Wt\Sigmat^2+\dt_{03}\Sigma^3\\
& = & -\Delta_1+c_+\Delta_2,
\eee
\bee
\nonumber \mathcal G_2&=&(c_+-c_-)\l_-\Sigmat+\et_{20}\Wt^2+\et_{11}\Wt\Sigmat+\et_{02}\Sigmat^2+\et_{30}\Wt^3+\et_{21}\Wt^2\Sigmat+\et_{12}\Wt\Sigmat^2+\et_{03}\Sigma^3\\
& = & \Delta_1-c_-\Delta_2.
\eee
This yields
\bee
&&\Sigmate[(c_+-c_-)\l_-+\et_{11}\Wte]=-\et_{20}\Wte^2-\et_{30}\Wte^3+O(b^4)\\
&\Rightarrow & \Sigmate=\frac{-\et_{20}\Wte^2-\et_{30}\Wte^3+O(b^4)}{(c_+-c_-)\l_-+\et_{11}\Wte}\\
&=&-\frac{1}{(c_+-c_-)\l_-}\left[\et_{20}\Wte^2+\et_{30}\Wte^3+O(b^4)\right]\left(1-\frac{\et_{11}\Wte}{(c_+-c_-)\l_-}+O(b^2)\right)\\
& = & -\frac{\et_{20}\Wte^2}{(c_+-c_-)\l_-}\left[1+\frac{\et_{30}}{\et_{20}}\Wte+O(b^2)\right]\left(1-\frac{\et_{11}\Wte}{(c_+-c_-)\l_-}+O(b^2)\right)\\
& = - & \frac{\et_{20}\Wte^2}{(c_+-c_-)\l_-}\left[1+\left(\frac{\et_{30}}{\et_{20}}-\frac{\et_{11}}{(c_+-c_-)\l_-}\right)\Wte+O(b^2)\right]
\eee 
and
\bee
&&(c_+-c_-)\mu_+b+\dt_{20}\Wte+\dt_{11}\Sigmate+\dt_{30}\Wte^2+O(b^3)=0\\
&\Rightarrow &\Wte=-\frac{1}{\dt_{20}}\left[(c_+-c_-)\mu_+b-\dt_{11}\frac{\et_{20}\Wte^2}{(c_+-c_-)\l_-}+\dt_{30}\Wte^2\right]+O(b^3)\\
& = & -\frac{1}{\dt_{20}}\left[(c_+-c_-)\mu_+b+\left(-\frac{\dt_{11}\et_{20}}{(c_+-c_-)\l_-}+\dt_{30}\right)\left(b\frac{(c_+-c_-)\mu_+}{\dt_{20}}\right)^2\right]+O(b^3)\\
& = & \left(-\frac{(c_+-c_-)\mu_+}{\dt_{20}}\right)b-\frac{1}{\dt_{20}}\left(-\frac{\dt_{11}\et_{20}}{(c_+-c_-)\l_-}+\dt_{30}\right)\left(\frac{(c_+-c_-)\mu_+}{\dt_{20}}\right)^2b^2+O(b^3)\\
& = & a_1b\left(1+a_2b+O(b^2)\right)
\eee
with $(a_1,a_2)$ given by \eqref{valuecoeficcients}. This yields
\bee
\psite&=&-\frac{\Sigmate}{b\Wte}=\frac{1}{b\Wte} \frac{\et_{20}\Wte^2}{(c_+-c_-)\l_-}\left[1+\left(\frac{\et_{30}}{\et_{20}}-\frac{\et_{11}}{(c_+-c_-)\l_-}\right)\Wte+O(b^2)\right]\\
& = & \frac{\et_{20}\Wte}{b(c_+-c_-)\l_-}\left[1+\left(\frac{\et_{30}}{\et_{20}}-\frac{\et_{11}}{(c_+-c_-)\l_-}\right)\Wte+O(b^2)\right]\\
& = & \frac{a_1\et_{20}}{(c_+-c_-)\l_-}\left(1+a_2b+O(b^2)\right)\left[1+\left(\frac{\et_{30}}{\et_{20}}-\frac{\et_{11}}{(c_+-c_-)\l_-}\right)a_1b+O(b^2)\right]\\
& = & g_1+bg_2+O(b^2)
\eee
with $(g_1,g_2)$ given by \eqref{valuecoeficcients}.\\

\noindent{\bf step 3} Computation of $\Phi_S$. We have for $0\le u\le 1$ recalling \eqref{valuecoeficcients}, \eqref{vneineoineon}: $$\phi_S=u+bu(1-u)M_b\Theta_S(u)=u+bu(1-u)\Theta_0+O(u b^2)$$
and hence from \eqref{vnevnneoenenvn}:
\bee
&&\Phi_S= c_--b\psite(c_+-c_-)\phi_S-b^2\phi_S^2\psite^2\left(c_+-c_-\right)+O(b^3)\\
& = & c_--b(c_+-c_-)\left[g_1+bg_2+O(b^2)\right]\left\{u+bu(1-u)\Theta_0+O(ub^2)\right\}\\
& - & b^2\left(c_+-c_-\right)g_1^2u^2 +O(b^3)\\
& = & c_--b(c_+-c_-)u\left\{g_1+b\left(g_2+g_1\Theta_0(1-u)\right)+O(b^2)\right\}-b^2\left(c_+-c_-\right)g_1^2u^2+O(b^3)\\
& = & c_--b\left[(c_+-c_-)g_1u\right]+b^2\left[-(c_+-c_-)[(g_2+g_1\Theta_0)u-g_1\Theta_0u^2]-\left(c_+-c_-\right)g_1^2u^2\right]+O(b^3)\\
& = & c_--b\left[(c_+-c_-)g_1u\right]+b^2\left\{-(c_+-c_-)(g_2+g_1\Theta_0)u+\left[(c_+-c_-)g_1\Theta_0-\left(c_+-c_-\right)g_1^2\right]u^2\right\}\\
&+& O(b^3)
\eee
 and \eqref{firovhineioeohis} is proved.
\end{proof}

\subsection{Computation of $\Phi_+$} We now use the separatrix curve $\Phi_S(u)$ to parametrize $\Sigma$ in the eye and compute the root $\Phit_+$.

\begin{lemma}[Computation of $\Phit_+$]
\label{lemmaphieoioet}
Let $\Phit_+$ be given by \eqref{rofftioegopih}. Let
\be
\label{eionvioneonenen}
\left|\begin{array}{l}
f_{00}=-\l_+(1+c_-)\\
f_{01}=-2\sigma_2-e_{02}-d_{02} +(-4\sigma_2-e_{11}-d_{11})c_- +(-2\sigma_2-e_{20}-d_{20})c_-^2\\
f_{02}=(d-1)c_- +(1-e_{21})c_-^2\\
f_{10}=-(c_2+c_1)\\
f_{11}=(-4\sigma_2-e_{11}-d_{11})+2c_-(-2\sigma_2-e_{20}-d_{20})\\
f_{21}=-2\sigma_2-e_{20}-d_{20}\\
\end{array}\right.
\ee
then, uniformyl for $0<b<b^*\ll1 $ small enough and $0\le u\le 1$ for  $\re=r^*$:
\bea
\label{enineineieonoe:rstar}
\Phit_+ & = & -\frac{f_{00}}{f_{10}}-\left(\frac{f_{01}}{f_{10}}a_1\right)bu+O(b^2)
\eea
and  for $\re=r_+$:
\bea
\label{enineineieonoe}
&&\Phit_+\\
\nonumber & = & -\frac{f_{00}}{f_{10}}-\left(\frac{f_{01}}{f_{10}}a_1\right)bu+\left[-\frac{f_{01}a_1a_2}{f_{10}}u+\left(\frac{f_{01}f_{11}-f_{02}f_{10}}{f_{10}^2}a_1^2+\frac{f_{01}}{f_{10}}a_1g_1\right)u^2\right]b^2+O(b^3).
\eea
\end{lemma}

\begin{proof}[Proof of Lemma \ref{lemmaphieoioet}] We recall the formulas using the notation \eqref{eionvioneonenen}
$$\left|\begin{array}{l}
F_0=f_{00}+f_{01}\Sigma+f_{02}\Sigma^2\\
F_1=f_{10}+f_{11}\Sigma+O(\Sigma^2)\\
F_2=f_{21}\Sigma+O(\Sigma^2)
\end{array}\right.
$$
and 
$$
\Phit_\pm = \frac{-|F_1(\Sigma)| \pm \sqrt{F_1(\Sigma)^2-4|F_2(\Sigma)|F_0(\Sigma)}}{2|F_2(\Sigma)|}.
$$
The difference between the case $\re=r^*$ and $\re=r_+$ is that $f_{00}=O(b)$ in the first case while $f_{00}=O(b^2)$ in the second case. The case $\re=r^*$ being similar and simpler is left to the reader. From now on, we focus the case $\re=r_+$ for which we need to prove \eqref{enineineieonoe}.

We compute explicitely 
$$\left|\begin{array}{l}
f_{00}=O(b^2)\\
f_{10}=|c_1^\infty+c_2^\infty|+O(b)>0\\
f_{21}=-2\sigma_2^\infty-e_{20}^\infty -d_{20}^\infty<0
\end{array}\right.
$$ Then, remembering $\Sigma<0$ and for $0<b<b^*$ small enough, yields $F_2>0$, $F_1<0$ for $u\in [0,1]$.\\

\noindent{\bf step 1} Taylor expansion. We compute using $f_{00}=O(b^2)$:
\bee
&&F_1(\Sigma)^2-4F_2(\Sigma)F_0(\Sigma)=(f_{10}+f_{11}\Sigma+O(\Sigma^2))^2-4(f_{21}\Sigma+O(\Sigma^2))\left(f_{00}+f_{01}\Sigma+O(\Sigma^2)\right)\\
& =& f_{10}^2+2f_{10}f_{11}\Sigma+O(b^2)=  f_{10}^2\left[1+\left(\frac{2f_{11}}{f_{10}}\right)\Sigma+O(b^2)\right]
\eee
and 
\bee
&&F_1(\Sigma)+ \sqrt{F_1(\Sigma)^2-4F_2(\Sigma)F_0(\Sigma)}=f_{10}+f_{11}\Sigma+O(b^2)+ f_{10}\left[1+\left(\frac{2f_{11}}{f_{10}}\right)\Sigma+O(b^2)\right]^{\frac 12}\\
& = & f_{10}+f_{11}\Sigma+O(b^2)+f_{10}\left(1+\frac{f_{11}}{f_{10}}\Sigma\right)+O(b^2)=2f_{10}+2f_{11}\Sigma+O(b^2)\\
&=& 2f_{10}\left(1+\frac{f_{11}}{f_{10}}\Sigma+O(b^2)\right).
\eee
We conclude
\bee
\Phit_+& = & -\frac{2F_0(\Sigma)}{F_1(\Sigma)+ \sqrt{F_1(\Sigma)^2-4F_2(\Sigma)F_0(\Sigma)}}=  \frac{-2(f_{00}+f_{01}\Sigma+f_{02}\Sigma^2)}{2f_{10}\left[1+\frac{f_{11}}{f_{10}}\Sigma+O(b^2)\right]}\\
& =& -\frac{1}{f_{10}}\left[f_{00}+f_{01}\Sigma+f_{02}\Sigma^2\right]\left[1-\frac{f_{11}}{f_{10}}\Sigma+O(b^2)\right]\\
& = & -\frac{1}{f_{10}}\left[f_{00}+f_{01}\Sigma+\left(-\frac{f_{01}f_{11}}{f_{10}}+f_{02}\right)\Sigma^2+O(b^3)\right]\\
& = & -\frac{f_{00}}{f_{10}}-\frac{f_{01}}{f_{10}}\Sigma +\frac{f_{01}f_{11}-f_{02}f_{10}}{f_{10}^2}\Sigma^2+O(b^3).
\eee

\noindent{\bf step 2} Expression in terms of $u$. We recall from \eqref{vneioneneenovne}:
\bee
\Sigma&=&-b\wt_\infty u+b^2\wt_\infty\psit_\infty u\phi=\Wt_\infty u-b\Wt_\infty\psit_\infty u\phi\\
& = & a_1b\left(1+a_2b+O(b^2)\right)u-b^2a_1g_1u^2+O(b^3)\\
&=&[a_1u]b+\left[a_1a_2u-a_1g_1u^2\right]b^2+O(b^3)
\eee
and hence
\bee
\Phit_+&=& -\frac{f_{00}}{f_{10}}-\frac{f_{01}}{f_{10}}\left\{[a_1u]b+\left[a_1a_2u-a_1g_1u^2\right]b^2\right\} +\frac{f_{01}f_{11}-f_{02}f_{10}}{f_{10}^2}a_1^2b^2u^2+O(b^3)\\
& = & -\frac{f_{00}}{f_{10}}-\left(\frac{f_{01}}{f_{10}}a_1\right)bu+\left[\frac{f_{01}f_{11}-f_{02}f_{10}}{f_{10}^2}a_1^2u^2-\frac{f_{01}}{f_{10}}(a_1a_2u-a_1g_1u^2)\right]b^2+O(b^3)\\
& = & -\frac{f_{00}}{f_{10}}-\left(\frac{f_{01}}{f_{10}}a_1\right)bu+\left\{-\frac{f_{01}a_1a_2}{f_{10}}u+\left[\frac{f_{01}f_{11}-f_{02}f_{10}}{f_{10}^2}a_1^2+\frac{f_{01}}{f_{10}}a_1g_1\right]u^2\right\}b^2+O(b^3)
\eee
and \eqref{enineineieonoe} is proved.
\end{proof}


\subsection{The positivity condition and proof of \eqref{vnineoneonve:rstar} and \eqref{vnineoneonve}}


Again, the case $\re=r^*$ being similar and simpler is left to the reader, we focus on the case $\re=r_+$ for which we need to prove \eqref{vnineoneonve}. We study the positivity condition
\bee
\Phi_S>\Phi_+
\eee
in the zone $0\le u\le 1$ which becomes
\bee
 &&-b\left[(c_+-c_-)g_1u\right]+b^2\left\{-(c_+-c_-)(g_2+g_1\Theta_0)u+\left[(c_+-c_-)g_1\Theta_0-\left(c_+-c_-\right)g_1^2\right]u^2\right\}\\
 & > & -\frac{f_{00}}{f_{10}}-\left(\frac{f_{01}}{f_{10}}a_1\right)bu+\left\{-\frac{f_{01}a_1a_2}{f_{10}}u+\left[\frac{f_{01}f_{11}-f_{02}f_{10}}{f_{10}^2}a_1^2+\frac{f_{01}}{f_{10}}a_1g_1\right]u^2\right\}b^2 +O(b^3)\\
 & \Leftrightarrow& A_0+bA_1u+b^2(A_3u+A_4u^2)+O(b^3)>0
\eee
with
$$\left|\begin{array}{l}
A_0=\frac{f_{00}}{f_{10}}\\
A_1=-(c_+-c_-)g_1+\frac{f_{01}}{f_{10}}a_1\\
A_3=\frac{f_{01}}{f_{10}}a_1a_2-(c_+-c_-)(g_2+g_1\Theta_0)\\
A_4=-\frac{f_{01}}{f_{10}}a_1g_1+\frac{f_{02}f_{10}-f_{01}f_{11}}{f_{10}^2}a_1^2-\left(c_+-c_-\right)g_1^2+(c_+-c_-)g_1\Theta_0\end{array}\right.
$$
We compute 
$$A_0=\frac{f_{00}}{f_{10}}=\frac{-\l_+(1+c_-)}{-(c_1+c_2)}=b\mu_+\frac{1+c_-}{c_1+c_2}$$ and 
\bee
&&A_1=-(c_+-c_-)g_1+\frac{f_{01}}{f_{10}}a_1=(c_+-c_-)\left(\frac{\et_{20}\mu_+}{\l_-\dt_{20}}\right)+\frac{f_{01}}{f_{10}}\left(-\frac{(c_+-c_-)\mu_+}{\dt_{20}}\right)\\
& = & \frac{(c_+-c_-)\mu_+}{\dt_{20}}\left[\frac{\et_{20}}{\l_-}-\frac{f_{01}}{f_{10}}\right]
\eee
We therefore divide by $|\mu_+|=-\mu_+$ and obtain the condition
\be
\label{confititnotob}
B_0+bB_1u+b^2(B_3u+B_4u^2)+O(b^3)>0
\ee
with
\be
\label{vnvemopempeioneneo}
\left|\begin{array}{l}
B_0=-\frac{b(1+c_-)}{c_1+c_2}\\
B_1=\frac{c_+-c_-}{\dt_{20}}\left[-\frac{\et_{20}}{\l_-}+\frac{f_{01}}{f_{10}}\right]\\
B_3=\frac{A_3}{|\mu_+|}\\
B_4=\frac{A_4}{|\mu_+|},
\end{array}\right.
\ee
this is \eqref{vnineoneonve}.


\section{Numerical computation  of specific values of $S_\infty(d, \ell)$}
\label{appendixnumericsofSinftyinell}


Let $\mu_0$, $\mu_j$ and $\nu_j$, for $j=1, 2, 3, 4$ the holomorphic functions introduced in \eqref{venonveoneonoenv}
and  
\bee
\left|\begin{array}{l}
(h_0)_k=\frac{a^k}{\Gamma(\nu+k+3)}(\mu_0)_k\\
h_k=\frac{a^k}{\Gamma(\nu+k+3)}\mu_k\\
\th_k=\frac{a^k}{\Gamma(\nu+k+3)}\nu_k,
\end{array}\right.
\eee
introduced  in \eqref{bfeioeoneonneon}.
Let $(w_k)_{k\geq 0}$ be the sequence defined by
\bee
&& w_{k+1}+w_k=\frac{1}{a}(h_0)_{k+1}\\
& + & \frac{a}{(k+\nu+3)(k+\nu+2)}\sum_{j=1}^4\sum_{k_1+\dots k_{j+1}=k-1}h_{jk_1}w_{k_2}\dots w_{k_{j+1}}\frac{\Pi_{i=1}^{j+1}\Gamma(\nu+k_i+3)}{\Gamma(k-1+\nu+3)}\\
& + &  \frac{a^2}{\Pi_{j=1}^3(k+\nu+j)}\sum_{j=1}^4\sum_{k_1+\dots k_{j+2}=k-2}\th_{jk_1}w_{k_2}\dots w_{k_{j+1}}(k_{j+2}w_{k_{j+2}})\frac{\Pi_{i=1}^{j+2}\Gamma(\nu+k_i+3)}{\Gamma(k-2+\nu+3)}
\eee
and let 
\bee
S_k &=& \frac{1}{a}(-1)^kw_k.
\eee
Then, according to Lemma \ref{propositionboundedness}, $S_k$ converges to a limit as $k\to \infty$ which is
\bee
S_\infty(d, \ell) &=& \frac{1}{a}\lim_{k\to +\infty}(-1)^kw_k.
\eee

We check numerically that $S_\infty(d, \ell) \neq 0$ in the following cases:\\

\noindent\underline{cases $\ell<d$}.\\

\begin{enumerate}
\item For $d=2$, $\ell=0.1$, $S_\infty=0.1236$.
\item For $d=3$, $\ell=0.1$, $S_\infty=0.0948$.
\item For $d=5$, $p=9$, $S_\infty=-0.0098$.
\item For $d=5$, $p=10$, $S_\infty=-0.0119$.
\item For $d=6$, $p=5$, $S_\infty=0.0012$.
\item For $d=7$, $p=4$, $S_\infty=-0.0006$.
\item For $d=8$, $p=3$, $S_\infty=6.1871\times 10^{-6}$.
\item For $d=9$, $p=3$, $S_\infty=-0.00024$.
\end{enumerate}

\medskip

\noindent\underline{cases $\ell>d$}.
\begin{enumerate}
\item For $d=2$, $\ell=0.1$, $S_\infty=3.0557\times 10^{-8}$.
\item For $d=3$, $\ell=0.1$, $S_\infty=2.8518\times 10^{-4}$.
\end{enumerate}

\end{appendix}

\end{document}